\documentclass[11pt,epsfig,amsfonts]{amsart}
\newtheorem{definition}{Definition}[section]
\newtheorem{lemma}{Lemma}[section]
\newtheorem{theorem}{Theorem}

\newtheorem{remark}{Remark}[section]

\newtheorem{proposition}{Proposition}[section]
\numberwithin{equation}{section}
\usepackage{epsfig}

\usepackage{stmaryrd}
\usepackage{xcolor}
\usepackage{epsfig}
\usepackage{enumerate}
\usepackage{enumitem}
\usepackage{amsmath}
\usepackage{amssymb}
\usepackage{amscd}
\usepackage{graphicx}
\usepackage{lineno}
\usepackage{pstricks}
\usepackage{tocvsec2}
\usepackage{mathtools}
\allowdisplaybreaks[4]

\usepackage [backref, colorlinks, citecolor=red, anchorcolor=black, linkcolor=blue]{hyperref}

\usepackage{amsfonts}
\usepackage[top=35mm, bottom=35mm, left=30mm, right=30mm]{geometry}
\usepackage{verbatim}
\usepackage{graphicx, amssymb, amsmath,amsthm}
\allowdisplaybreaks[4]
\usepackage{appendix}
 \usepackage{mathrsfs}
 \usepackage{cite}

\newtheorem{sublemma}{Sublemma}[section]
\usepackage [backref, colorlinks, citecolor=red, anchorcolor=black, linkcolor=blue]{hyperref}

\def \P{\mathbb{P}}

\renewcommand*{\backrefalt}[4]{\quad \tiny
  \ifcase #1 (\textbf{NOT CITED.})%
  \or    (Cited on Section~#2.)%
  \else   (Cited on Section~#2.)%
  \fi}



\author{Xue Liu}
\address[Xue Liu]
{CAS Wu Wen-Tsun Key Laboratory of Mathematics\\
 University of Science and Technology of China\\
 Hefei, Anhui 230026, PR China}
\email[X.~Liu]{xueliu21@ustc.edu.cn}

\begin{document}

\begin{abstract}
In this paper, we study the statistical property of Anosov systems on surface driven by an external force. By utilizing the Birkhoff cone method, we show that if the systems on surface satisfying the Anosov and topological mixing on fibers property, then the quenched random correlation for H\"older observables with respect to the unique random SRB measures decays exponentially.
\end{abstract}


\title[Exponential decay of random correlations]{Exponential decay of random correlations for random Anosov systems mixing on fibers}
\maketitle

\tableofcontents

\parskip 10pt         

\section{Introduction}
In this paper, we investigate the decay rate of quenched random correlation for Anosov system on surface driven by an external force with respect to the random SRB measure. Let $M$ be a surface, by which we mean a connected closed smooth 2-dimensional Riemannian manifold. Let $\theta:\Omega\rightarrow\Omega$ be a homeomorphism on a compact metric space $(\Omega,d_\Omega)$, which will describe the external force.  A dynamical system driven by $(\Omega,\theta)$ is a mapping
\begin{equation*}
  F: \mathbb{Z}\times \Omega\times M\rightarrow M,\ (n,\omega,x)\mapsto F(n,\omega,x)
\end{equation*}satisfying for each $n\in\mathbb{Z}$, $(\omega,x)\mapsto F(n,\omega,x)$ is continuous and the mappings $F(n,\omega):=F(n,\omega,\cdot):M\rightarrow M$ form a cocycle over $\theta$, i.e.,
      \begin{align*}
         F(0,\omega)&=id_M\mbox{ for all }\omega\in \Omega,\\
         F(n+m,\omega)&=F(n,\theta^m\omega)\circ F(m,\omega)
         \mbox{ for all } n,m\in\mathbb{Z},\ \omega\in\Omega.
      \end{align*}
When $(\Omega,\mathcal{B}(\Omega))$ is equipped with a $\theta-$invariant probability measure $\P$, $F$ is a continuous random dynamical system (abbr. RDS) \cite{Arnold98}. We denote $f_\omega$ by the time-one map $F(1,\omega)$ of the RDS, and assume $f_\omega$ to be a diffeomorphism for all $\omega\in\Omega$. Putting $f_\omega$ and $\theta$ together forms a skew product transformation $\phi:M\times \Omega\to M\times \Omega$ by $\phi(x,\omega)=(f_\omega(x),\theta\omega)$. The skew product system $\phi$ is called Anosov on fibers if for every $(x,\omega)\in M\times \Omega$, there is a splitting of the tangent bundle of $M$ at $x$
\begin{equation*}
  T_xM=E^s{(x,\omega)}\oplus E^u{(x,\omega)}
\end{equation*}which depends continuously on $(x,\omega)\in M\times \Omega$ with $\dim E^s{(x,\omega)}=\dim E^u{(x,\omega)}=1$, and the splitting is invariant in the sense that
\begin{equation*}
  D_xf_\omega E^u{(x,\omega)}=E^u{(f_\omega (x),\theta\omega)},\ \ D_xf_\omega E^s{(x,\omega)}=E^s{(f_\omega(x),\theta\omega)},
\end{equation*}and
\begin{equation*}
  \begin{cases}
    |D_xf_\omega\xi|\geq e^{\lambda_0}|\xi|, & \ \forall \xi\in E^u{(x,\omega)}, \\
    |D_xf_\omega\eta|\leq e^{-\lambda_0}|\eta|, &\ \forall \eta\in E^s{(x,\omega)},
  \end{cases}
\end{equation*}where $\lambda_0>0$ is a constant. We assume that $\phi$ is topological mixing on fibers, that is, for any nonempty open sets $U,\ V\subset M$, there exists $N>0$ such that for any $n\geq N$ and $\omega\in\Omega$, $\phi^n( U\times \{\omega\})\cap(V\times \{\theta^n\omega\})\not=\emptyset$.

A random probability measure $\omega\mapsto\mu_\omega$ on $M$ is a measurable map $\mu:\Omega\rightarrow Pr(M)$, where $Pr(M)$ is the space of Borel probability measures on $M$ equipped with the Borel $\sigma-$algebra generated by the weak$^*$ topology. A random probability measure $\omega\mapsto\mu_\omega$ is said to be $\phi-$invariant if $(f_\omega)_*\mu_\omega=\mu_{\theta\omega}$ for $\P$-a.s. $\omega\in\Omega$. Given a pair of regular observables $\varphi$ and $\psi$ on $M$, we say the (quenched) past random correlation function of $\varphi$ and $\psi$ with respect to the system $\phi$ and an invariant random probability measure $(\mu_\omega)_{\omega\in\Omega}$ decays exponentially if
\begin{equation*}
  \left|\int_{M}\psi(F(n,\theta^{-n}\omega,x))\varphi(x)d\mu_{\theta^{-n}\omega}-\int_M\psi(x)d\mu_\omega\int_M\varphi(x)d\mu_{\theta^{-n}\omega}\right|\rightarrow 0
\end{equation*}exponentially as $n\rightarrow\infty$. We say that the (quenched) future random correlation function of $\varphi$ and $\psi$ decays exponentially if
\begin{equation*}
 \left|\int_{M}\psi(F(n,\omega,x))\varphi(x)d\mu_{\omega}-\int_M\psi(x)d\mu_{\theta^n\omega}\int_M\varphi(x)d\mu_{\omega}\right|\rightarrow 0
\end{equation*}exponentially as $n\rightarrow\infty$. If $\phi$ is Anosov and topological mixing on fibers, then $\phi$ is a random topological transitive hyperbolic systems \cite{HLL}, so there exists a unique random SRB measure and the unique random SRB measure is given by $\mu_\omega:=\lim_{n\rightarrow\infty}(f_{\theta^{-n}\omega}^n)_*m$, where $m$ is the normalized Riemannian volume measure \cite{Gund99}. In this paper, we prove that such system $\phi$ and the unique random SRB measure have exponential decay of both past and future random correlations for H\"older observables $\psi$ and $\varphi$. We emphasis that our result holds for all $\omega\in\Omega$, while most result of decay of quenched random correlation only holds for $\P$-a.s. $\omega\in\Omega$.

 Recently, Anosov on fibers systems with the topological mixing on fibers property have been studied in \cite{HLL} from the topological complexity perspective, in which the authors proved dynamical complexity such as the density of random periodic points, strong random horseshoe, the density of measure equi-distributed on a random periodic orbit, and a simplified random Liv\v sic theorem. Examples such as fiber Anosov maps on 2-dimension torus driven by irrational rotation on the torus and random composition of $2\times 2 $ area-preserving positive matrices are under consideration. Moreover, the Anosov on fibers systems contain a class of partially hyperbolic systems. In fact, let $\Omega$ be a compact differentiable manifold, and let $\theta:\Omega\rightarrow\Omega$ be a diffeomorphism such that the expansion of $D\theta$ is weaker than $e^{\lambda_0}$ and contraction of $D\theta$ is weaker than $e^{-\lambda_0}$. Furthermore, we assume $f_\omega(x)$ and $f_\omega^{-1}(x)$ are $C^1$ in $\omega$. Then the system $\phi$ is a partially hyperbolic system, and the dimension of central direction coincides with $\dim\Omega$ (see Proposition \ref{proposition A.1} in Appendix).

For deterministic dynamical systems, there are a large number of results considering the exponential decay of correlations, for instance \cite{BW75,Ru78,Liv95A,Young98,Young99,Dol00,Ch06,Alv08,CD09,Ch09,WY13,Liv16,Pes19,AZI19}. For RDS, the exponential decay of (quenched) random correlations was obtained for random Lasota–Yorke maps on intervals \cite{Buz99}, for random perturbations of expanding maps \cite{Baladi96}, for i.i.d unimodal maps \cite{Baladiunimodal2002} and for a class of non-uniformly expanding random dynamical systems \cite{MSP20}. In \cite{Kif08}, the topological one-sided random shift of finite type with the fiber Gibbs measure was proved to have certain nonuniform $\omega-$wise decay of correlations, and similar results held for random expanding in average transformations. Note that the system in this paper belongs to random hyperbolic system and it is randomly conjugated to a two-sided random shift \cite{Gund99}. Recently, the authors in \cite{AlvesWael2022} obtain the exponential decay of  quenched future random correlations for random perturbation of topological mixing uniformly hyperbolic system and a family of equivariant physical measure. We note that our results are independent, since our main example can not be obtained by perturbation, and it is unknown whether the family of equivariant physical measure coincides the random SRB measure. Other decay rates of random correlations such as stretched exponential and polynomial decay were also considered for certain random dynamical systems \cite{Xin18,Marks19,Marks20,AlvesWael2022}.

In this paper, we prove the exponential decay of (quenched) random correlations for Anosov and topological mixing on fibers system by directly studying the fiber transfer operator $L_{\omega}$, which is defined by
\begin{equation}\label{L varphi}
  L_{\omega}\varphi: M \rightarrow\mathbb{R},\ (L_{\omega}\varphi)(x):=\frac{\varphi((f_{\omega})^{-1}x)}{|\det D_{(f_{\omega})^{-1}(x)}f_{\omega}|}
\end{equation}for any measurable observables $\varphi:M\rightarrow \mathbb{R}$. We construct a Birkhoff cone on each fiber and introduce the projective metric on each fiber Birkhoff cone. The construction of the Birkhoff cone on each fiber is inspired by the construction in \cite{Liv95A} and \cite{Viana}. The most technical analysis in this paper lies in estimating the diameter of the image of $L_\omega^N=L_{\theta^{N-1}\omega}\circ \cdots\circ L_\omega$ on fiber Birkhoff cone, where $N$ comes from the mixing on fibers property. We prove that the image of $L_\omega^N$ acting on the $\omega-$fiber Birkhoff cone has finite diameter with respect to the projective metric on $\theta^N\omega-$fiber Birkhoff cone. Moreover, this diameter is uniformly finite for all $\omega\in\Omega$. As a consequence of Birkhoff's inequality, $L_\omega^N$ is a contraction and the contraction rate is independent of $\omega\in\Omega$. The contraction implies weak$^*$-limit of $(L_{\theta^{-n}\omega}^n1)dm$ exists and it gives the unique random SRB measure, where $m$ is the normalized Riemannian volume measure. The exponential decay of random correlations can be obtained from the contraction and the usual techniques in deterministic systems.

The Birkhoff cone approach has been used extensively to study the transfer operator and exponential decay of correlations.
Liverani in \cite{Liv95A} used it to prove the exponential decay of correlations for smooth uniformly hyperbolic area-preserving cases. Later, it was generalized to general Axiom A attractors in \cite{Viana,Baladi00}, and some partially hyperbolic systems \cite{Cas02,Cas17}. For RDS, the Birkhoff cone approach was used in \cite{Baladi96} and \cite{MSP20} for exponential decay of (quenched) random correlations.


This paper is organized as follows. In Section \ref{section 2}, we state the settings and formulate the main result. In Section \ref{section 3}, we introduce several preliminary lemmas and propositions to pave the way for the proof of the main result. Section \ref{section 4} addresses the proof of the main result based on the Birkhoff cone. We recall the definitions of convex cone, projective metric and Birkhoff's inequality in the Appendix.

\section{Setting and Main Result}\label{section 2}
In this section, we begin with the setting of Anosov and mixing on fibers systems. After introducing several necessary notations, we formulate the main result.
\subsection{Anosov and mixing on fibers systems} Let $M$ be a connected closed smooth Riemannian manifold with $\dim M=2$, and $d_M$ be the induced Riemannian metric on $M$, $(\Omega,d_\Omega)$ be a compact metric space, and $\theta:\Omega\rightarrow \Omega$ be a homeomorphism. Let $\P$ be an ergodic Borel probability measure on $\Omega$. $M\times\Omega$ is a compact metric space with distance $d((x_1,\omega_1),(x_2,\omega_2))=d_M(x_1,x_2)+d_\Omega(\omega_1,\omega_2)$ for any $x_1,x_2\in M$ and $\omega_1,\omega_2\in\Omega$. Let Diff$^2(M)$ be the space of $C^2$ diffeomorphisms on $M$ with $C^2-$topology (see, e.g., \cite{Hirsch}), and $f:\Omega\rightarrow$Diff$^2(M)$ be a continuous map. The skew product $\phi:M\times\Omega\rightarrow M\times \Omega$ induced by $f$ and $\theta$ is defined by:
\begin{equation*}
  \phi(x,\omega)=(f(\omega)x,\theta\omega)=(f_\omega x,\theta\omega),\ \forall\omega\in\Omega,\ x\in M.
\end{equation*}
where we rewrite $f(\omega) $ as $f_\omega$. Then inductively:
\begin{equation}\label{def f omega n}
  \phi^n(x,\omega)=(f_\omega^nx,\theta^n\omega):=\begin{cases}
                        (f_{\theta^{n-1}\omega}\circ\cdots\circ f_\omega x,\theta^n\omega), & \mbox{if } n>0 \\
                                                (x,\omega), & \mbox{if } n=0 \\
                        ((f_{\theta^n\omega})^{-1}\circ\cdots\circ (f_{\theta^{-1}\omega})^{-1}x,\theta^n\omega), & \mbox{if } n<0.
                      \end{cases}
\end{equation}

The system $\phi$ is called Anosov on fibers if the following is true: for every $(x,\omega)\in M\times \Omega$, there is a splitting of the tangent bundle of $M$ at $x$
\begin{equation*}
  T_xM=E^s{(x,\omega)}\oplus E^u{(x,\omega)}
\end{equation*}which depends continuously on $(x,\omega)\in M\times \Omega$ with $\dim E^s{(x,\omega)}=\dim E^u{(x,\omega)}=1$ and satisfies that
\begin{equation*}
  D_xf_\omega E^u{(x,\omega)}=E^u{(f_\omega(x),\theta\omega)},\ \ D_xf_\omega E^s{(x,\omega)}=E^s{(f_\omega(x),\theta\omega)}
\end{equation*}and
\begin{equation*}
  \begin{cases}
    |D_xf_\omega\xi|\geq e^{\lambda_0}|\xi|, & \ \forall \xi\in E^u{(x,\omega)}, \\
    |D_xf_\omega\eta|\leq e^{-\lambda_0}|\eta|, &\ \forall \eta\in E^s{(x,\omega)},
  \end{cases}
\end{equation*}where $\lambda_0>0$ is a constant. We say that $\phi: M\times \Omega\rightarrow M\times \Omega$ is topological mixing on fibers if for any nonempty open sets $U,V\subset M$, there exists  $N>0$ such that for any $n\geq N$ and $\omega\in\Omega$
\begin{equation*}
  \phi^n(U\times \{\omega\})\cap V\times \{\theta^n\omega\}\not=\emptyset.
\end{equation*}
Examples of Anosov and topological mixing on fibers systems are given in Appendix \ref{Appendix Example}.
\subsection{Random probability measures}\label{subsection 2.2}
The following notations are borrowed from \cite{Hans02}.
Denote $Pr(M)$ to be the space of probability measures on $(M,\mathcal{B}(M))$ equipped with weak$^*$ topology.

 A map $\mu:\mathcal{B}(M)\times\Omega\rightarrow [0,1]$ by $(B,\omega)\mapsto \mu_\omega(B)$ is said to be a random probability measure on $M$ if it satisfies: for every $B\in\mathcal{B}(M)$, $\omega\mapsto\mu_\omega(B)$ is measurable; and for $\P$-almost every $\omega\in\Omega$, $B\rightarrow\mu_\omega(B)$ is a Borel probability measure.  By Remark 3.2 in \cite{Hans02}, the second condition can be relaxed to: for every $D$ from a $\cap-$stable family $\mathfrak{F}$ of Borel sets of $M$ which generates $\mathcal{B}(M)$ (i.e. $\sigma(\mathfrak{F})=\mathcal{B}(M)$), $\omega\mapsto \mu_\omega(D)$ is measurable. A typical example of $\mathfrak{F}$ is the family of all closed sets in $M$. We denote a random probability measure by $\omega\mapsto \mu_\omega$ or $(\mu_\omega)_{\omega\in\Omega}$.

A random probability measure $\omega\mapsto\mu_\omega$ is said to be $\phi-$invariant if $(f_\omega)_*\mu_{\omega}=\mu_{\theta\omega}$ for $\P$-a.s. $\omega\in\Omega$.
\subsection{Main result}

Let $C^0(M)$ be the collection of all continuous functions $\varphi:M\rightarrow\mathbb{R}$.
For $\alpha\in(0,1)$, and $\varphi\in C^0(M)$, denote
\begin{equation*}
  \|\varphi\|_{C^0(M)}:=\sup_{x\in M}|\varphi(x)|\mbox{ and }|\varphi|_\alpha:=\sup_{x,y\in M, x\not=y}\frac{|\varphi(x)-\varphi(y)|}{d(x,y)^\alpha}.
\end{equation*} We denote $C^{0,\alpha}(M):=\{\varphi\in C^0(M):\ |\varphi|_\alpha<\infty\}$ to be the space of $\alpha-$H\"older continuous functions on $M$. For $\varphi\in C^{0,\alpha}(M)$, we let
\begin{equation*}
  \|\varphi\|_{C^{0,\alpha}(M)}:=\|\varphi\|_{C^0(M)}+|\varphi|_{\alpha}.
\end{equation*}

\begin{theorem}\label{random exponential decay}
Assume that $\phi$ satisfies Anosov on fibers and topological mixing on fibers. Then
\begin{enumerate}
  \item there exists an invariant random probability measure $\omega\mapsto \mu_\omega $ given by $\mu_\omega=\lim_{n\rightarrow\infty}(f_{\theta^{-n}\omega}^{n})_*m$ for all $\omega\in\Omega$, where $m$ is the normalized Riemannian volume measure,
  \item there exists a constant $\nu_0$ that only depends on the system $\phi$ such that, for H\"older exponents $\kappa,\nu\in(0,1)$ with
\begin{equation*}
  0<\kappa+\nu<\nu_0
\end{equation*}and $\psi\in C^{0,\kappa}(M)$, $\varphi\in C^{0,\nu}(M)$, the (quenched) past and future random correlation between $\varphi$ and $\psi$ exponential decay with respect to  the random probability measure $(\mu_\omega)_{\omega\in\Omega}$ defined in $(1)$, i.e. for any $n\in\mathbb{N}$, $\omega\in\Omega$,
\begin{align*}
& \left|\int_{M}\psi(f_{\theta^{-n}\omega}^n x)\varphi(x)d\mu_{\theta^{-n}\omega}-\int_M\psi(x)d\mu_\omega\int_M\varphi(x)d\mu_{\theta^{-n}\omega}\right|
  \leq K\|\psi\|_{C^{0,\kappa}(M)}\cdot\|\varphi\|_{C^{0,\nu}(M)}\cdot\Lambda^{n};\\
 & \left|\int_{M}\psi(f_{\omega}^n x)\varphi(x)d\mu_{\omega}-\int_M\psi(x)d\mu_{\theta^n\omega}\int_M\varphi(x)d\mu_{\omega}\right|
  \leq K\|\psi\|_{C^{0,\kappa}(M)}\cdot\|\varphi\|_{C^{0,\nu}(M)}\cdot\Lambda^{n},
\end{align*}
where constants $K>1$ and $\Lambda\in(0,1)$ only depend on $\kappa$, $\nu$ and system $\phi$.
\end{enumerate}
\end{theorem}
We note that topological mixing on fibers property implies random topological transitivity by \cite[Lemma A.1]{HLL}. Then according to \cite[Theorem 4.3]{Gund99}, the measure $\mu_\omega$ we constructed above is the unique random SRB measure (we state this lemma and this theorem in the Appendix).

\section{Preliminary Lemmas and Propositions}\label{section 3}

In this section, we introduce several technical lemmas and propositions that will be used in the proof of the main result. Lemmas in Subsections \ref{subsection 3.1} and \ref{subsection 3.2} can be viewed as the combination of Lusin's theorem and corresponding results in general RDS \cite{PDQM} by noticing that our system $f_\omega$, $E^s(x,\omega)$ and $E^u(x,\omega)$ are continuous depending on $\omega$. We state and prove two distortion lemmas in Subsection \ref{subsection 3.3}. We formulate and prove the absolute continuity and H\"older continuity of the stable and unstable foliations on each fiber in Subsection \ref{subsection 3.4} and \ref{subsection 3.5} respectively. We discuss properties of holonomy maps between local stable leaves in Subsection \ref{subsection 3.6}. In Subsection \ref{subsection 3.7}, we prove a version of Fubini's theorem on any rectangle on each fiber.
\subsection{H\"older continuity of stable and unstable subbundle on each fiber}\label{subsection 3.1}
In this subsection, we will formulate the H\"older continuity of $E^s(x,\omega)$  and $E^u(x,\omega)$ on $x\in M$ for any $\omega\in\Omega$, which is borrowed from \cite{PDQM}.

For nontrivial closed subspaces $A,B$ in a Hilbert space $H$ with given inner product and induced norm $|\cdot|$, define the aperture between $A$ and $B$ by
\begin{equation*}
  \Gamma(A,B):=\max\left\{\sup_{v\in A,|v|=1}\inf_{w\in B}|v-w|,\sup_{w\in B,|w|=1}\inf_{v\in A}|v-w|\right\}.
\end{equation*}
Then $\Gamma(A,B)\in [0,1]$. One also has
$\Gamma(A,B)=\|P_A-P_B\|,$ where $P_A$ and $P_B$ are the orthogonal projections on $A$ and $B$ respectively (see more details in  \cite[Chap. Four sec. 2]{Kato1995}).

By the compactness of $M$, there exists a $\rho_0>0$ such that for any $x,y\in M$ with $d(x,y)<\rho_0$, there exists a unique geodesic connecting $x$ and $y$.  For any $x,y\in M$, if $d(x,y)<\rho_0$, then there exists an isometry from $T_xM$ to $T_y M$ given by the parallel transport on the unique geodesic connecting $x$ and $y$, named $P(x,y)$. Then for any $x,y\in M$, $E(x)\subset T_xM$, $E(y)\subset T_yM$ subspaces, we can define
\begin{equation}\label{def distance between bundle}
  d(E(x),E(y)):=\begin{cases}
                  \Gamma_x(E(x),P(y,x)E(y)), & \mbox{if } d(x,y)<\rho_0 \\
                  1, & \mbox{otherwise}.
                \end{cases}
\end{equation}where $\Gamma_x$ is the aperture between subspaces in $T_xM$ defined by the given Riemannian metric. The following lemma is an adapted version of Theorem 4.1 in \cite{PDQM} by noticing that $f_\omega\in$Diff$^2(M)$ is continuously depending on $\omega\in\Omega$.

\begin{lemma}\label{Holder continuity of stableunstable distribution}
  There are constants $C_1>0$ and $\nu_1\in (0,1)$ such that
  for each $\omega\in\Omega,$  $E^s(x,\omega)$ and $E^u(x,\omega)$ are $(C_1,\nu_1)$-H\"older continuous on $x$, i.e.,
  \begin{equation}\label{Holder continuous bundle}
    d(E^\tau(x,\omega),E^\tau(y,\omega))\leq C_1d(x,y)^{\nu_1},\ \ \tau=s,u.
  \end{equation}
\end{lemma}
\subsection{Stable and unstable invariant manifolds}\label{subsection 3.2}
 We define the local stable and unstable manifolds as the following:
\begin{align*}
  W_\epsilon^s(x,\omega) &=\{y\in M|\ d(\phi^n(y,\omega),\phi^n(x,\omega))\leq \epsilon\mbox{ for all }n\geq 0\},  \\
   W_\epsilon^u(x,\omega)& =\{y\in M|\ d(\phi^n(y,\omega),\phi^n(x,\omega))\leq \epsilon\mbox{ for all }n\leq 0\}.
\end{align*}
The following lemma can be found in \cite[Lemma 3.1]{HLL}, and it is a special version of \cite[Theorem 3.1]{Gund99}.

 For $\tau=s,u$, denote $P(E^\tau(x,\omega))$ to be the projection from $T_xM$ to $E^\tau(x,\omega)$ with respect to the splitting $T_xM=E^s(x,\omega)\oplus E^u(x,\omega)$. Since $E^s(x,\omega),\ E^u(x,\omega)$ are uniformly continuous on $x$ and $\omega$, there exists a number $\mathcal{P}>1$ such that
\begin{equation}\label{def of P}
  \sup\{\|P(E^s(x,\omega))\|,\|P(E^u(x,\omega))\|:\ (x,\omega)\in M\times\Omega\}<\mathcal{P}.
\end{equation}
\begin{lemma}[Stable and unstable invariant manifolds]\label{lemma stable unstable}
For any $\lambda\in(0,\lambda_0),$ there exists $\epsilon_0>0$ such that for any $\epsilon\in(0,\epsilon_0],$ the followings hold:
\begin{enumerate}
  \item $W_\epsilon^\tau(x,\omega)$ are $C^2$ embedded discs for all $(x,\omega)\in M\times\Omega$ with $T_xW^\tau(x,\omega)=E^\tau(x,\omega)$ for $\tau=u,s$. Moreover, there exist a constant $L>1$ and $C^2$ maps
      \begin{equation*}
        h^u_{(x,\omega)}:E^u(x,\omega)(\mathcal{P}\epsilon)\rightarrow E^s(x,\omega),\ h^s_{(x,\omega)}:E^s(x,\omega)(\mathcal{P}\epsilon)\rightarrow E^u(x,\omega)
      \end{equation*}such that $W_\epsilon^\tau(x,\omega)\subset Exp_{x}(graph (h^\tau_{(x,\omega)}))$ and $\|Dh^\tau_{(x,\omega)}\|<\frac{1}{3},\ Lip(Dh^\tau_{(x,\omega)})<L$ for $\tau=u,s$, where $E^\tau(x,\omega)(\mathcal{P}\epsilon)$ is the $\mathcal{P}\epsilon$-ball of $E^\tau(x,\omega)$ centered at origin for $\tau=u,s$.
  \item $d_M(f_\omega^nx,f_\omega^ny)\leq e^{-n\lambda}d_M(x,y)$ for $y\in W^s_\epsilon(x,\omega)$ and $n\geq 0$, and $d_M(f_\omega^{-n}x,f_\omega^{-n}y)\leq e^{-n\lambda}d_M(x,y)$ for $y\in W^u_\epsilon(x,\omega)$ and $n\geq 0$.
  \item $W^s_\epsilon(x,\omega), W^u_\epsilon(x,\omega)$ vary continuously on $(x,\omega)$ in $C^1$ topology.
\end{enumerate}
\end{lemma}

The following lemma is about the local product structure on each fiber.
\begin{lemma}\label{local product structure}\cite[Lemma 3.2]{HLL}
  For any $\epsilon\in(0,\epsilon_0)$, there is a $\delta\in(0,\epsilon)$ such that for any $x,y\in M$ with $d_M(x,y)<\delta$, $W_\epsilon^s(x,\omega)\cap W_\epsilon^u(y,\omega) $ consists of a single point, which is denoted by $[x,y]_\omega$. Furthermore
  \begin{equation*}
    [\cdot,\cdot]_\cdot:\{(x,y,\omega)\in M\times M\times \Omega|\ d_M(x,y)<\delta\}\rightarrow M
  \end{equation*}is continuous.
\end{lemma}

\subsection{Two distortion lemmas}\label{subsection 3.3}
In this subsection, we state two distortion lemmas. The following Lemma is clear, and it is used for proving the absolutely continuity of stable and unstable foliations on each fiber.

\begin{lemma}\label{detDfE-detDfE}
  For any $x,y\in M$, $E(x,\omega)\in T_xM$, $E(y,\omega)\in T_yM$ with $\dim E(x,\omega)=E(y,\omega)$, we have
  \begin{align}
    ||\det(D_xf_\omega|_{E(x,\omega)})|-|\det(D_yf_\omega|_{E(y,\omega)})||&\leq C_2d(x,y)+C_2d(E(x,\omega),E(y,\omega)),\label{psoitive time}\\
    ||\det(D_xf_\omega^{-1}|_{E(x,\omega)})|-|\det(D_yf_\omega^{-1}|_{E(y,\omega)})||&\leq C_2d(x,y)+C_2d(E(x,\omega),E(y,\omega)),\label{negative time}
  \end{align}where $C_2>0$ is a constant depending on  $\sup_{\omega\in\Omega}\|f_\omega\|_{C^2}$ and $\sup_{\omega\in\Omega}\|f_\omega^{-1}\|_{C^2}$.
\end{lemma}

 Lemma \ref{Lip jac Es} is used for the construction of the fiber convex cone of observables in Section \ref{section 4}, and its proof is parallel to the deterministic case, see \cite[Lemma 3.2]{LianSRB2016}.
\begin{lemma}\label{Lip jac Es}
  $J^s(x,\omega)=|\det(D_xf_\omega|_{E^s(x,\omega)})|$ has a uniform Lipschitz variation on the local stable manifolds, i.e., there is a constant $K_1>0$ independent of $\omega$ such that for any $z\in M$, and $x,y\in W^s_\epsilon(z,\omega)$,
  \begin{equation*}
    |J^s(x,\omega)-J^s(y,\omega)|\leq K_1d(x,y),
  \end{equation*}and
  \begin{equation*}
    |\log J^s(x,\omega)-\log J^s(y,\omega)|\leq K_1d(x,y).
  \end{equation*}
\end{lemma}
\begin{proof}[Proof of Lemma \ref{Lip jac Es}]
  Since $M\times\Omega$ is a compact space and $f:\Omega\rightarrow$ Diff$^2(M)$ is continuous, $|D_xf_\omega|$ and $|D_x^2f_\omega|$ are uniformly bounded. Let $K\geq 1$ be a constant such that
\begin{equation*}
  \max\left\{\sup_{(x,\omega)\in M\times\Omega}\|D_xf_\omega\|,\sup_{(x,\omega)\in M\times\Omega}\|D_x^2f_\omega\|,Lip Dh^s_{(x,\omega)}\right\}\leq K.
\end{equation*}

Notice that if $y,z\in W_\epsilon^s(x,\omega)$, and $d(y,z)<\frac{\epsilon}{2\mathcal{P}K^2}$, then
\begin{align*}
  \|P(E^s(x,\omega))(Exp_x^{-1}(z)-Exp_x^{-1}(y))\| & \leq \mathcal{P}d(y,z)<\frac{\epsilon}{2K^2}; \\
  \|P(E^s(f_\omega y,\theta\omega))(Exp_{f_\omega(x)}^{-1}(f_\omega(z))-Exp_{f_\omega(x)}^{-1}(f_\omega(y)))\| &\leq \mathcal{P}Kd(y,z)\leq \frac{\epsilon}{2K}.
\end{align*} Therefore, $z\in W_{\epsilon}^s(y,\omega)$ and $f_\omega(z)\in W_{\epsilon}^s(f_\omega(y),\theta\omega)$. So it is sufficient to prove that there exists a constant $K_1>0$ independent of $x$ and $\omega$ such that for any $y\in W_{\frac{\epsilon}{2\mathcal{P}K^2}}^s(x,\omega)$,
\begin{equation}
  |J^s(x,\omega)-J^s(y,\omega)|\leq K_1d(x,y).
\end{equation} With the help of the normal coordinate chart, and notice that $d(x,y)<\epsilon<\rho_0$ and $d(f_\omega x,f_\omega y)\leq \epsilon<\rho_0$, we may pretend that $x$, $y$ together with $W^s_{\frac{\epsilon}{2\mathcal{P}K}}(x,\omega)$ lie in a same Euclidean space, and $f_\omega x$, $f_\omega y$ together with $W^s_\epsilon(f_\omega x,\theta\omega)$ lie in a same Euclidean space.
By the invariant stable manifolds theorem, there exists $\xi_y\in E^s(x,\omega)(\frac{\epsilon}{2K^2})$ and $\xi_{f_\omega(y)}\in E^u(f_\omega(x),\theta\omega)(\epsilon)$ such that
\begin{align}
  y & =x+\xi_y+h^s_{(x,\omega)}(\xi_y) ;\label{xiy}\\
  f_\omega (y) &=f_\omega(x)+\xi_{f_\omega(y)}+h^s_{(f_\omega(x),\theta\omega)}(\xi_{f_\omega(y)}),\label{xify}
\end{align}and $E^s(y,\omega)=graph((Dh_{(x,\omega)}^s)_{\xi_y})$, $E^s(f_\omega(y),\theta\omega)=graph((Dh^s_{(f_\omega(x),\theta\omega)})_{\xi_{f_\omega(y)}})$.
From $(\ref{xiy})$ and $(\ref{xify})$, we have
\begin{equation}\label{d xy xiy}
  d(x,y)=\|x-y\|=\|\xi_y+h_{(x,\omega)}^s(\xi_y)\|\geq \|\xi_y\|-\frac{1}{3}\|\xi_y\|=\frac{2}{3}\|\xi_y\|
\end{equation}and
\begin{equation*}
(1-\frac{1}{3})\|\xi_{f_\omega(y)}\|\leq \|\xi_{f_\omega(y)}+h^s_{(f_\omega(x),\theta\omega)}(\xi_{f_\omega(y)})\|=\|f_\omega(y)-f_\omega(x)\|\leq K\|y-x\|\leq K(1+\frac{1}{3})\|\xi_y\|,
\end{equation*}so $\|\xi_{f_\omega(y)}\|\leq 2K\|\xi_y\|$.

Now, we define the following linear maps $L_{(x,\omega)}, L_{(y,\omega)}:E^s(x,\omega)\rightarrow E^s(f_\omega(x),\theta\omega)$ by
\begin{align*}
  L_{(x,\omega)} & =D_xf_\omega|_{E^s(x,\omega)}; \\
  L_{(y,\omega)} &= P(E^s(f_\omega x,\theta\omega))D_yf_\omega|_{E^s(y,\omega)}(I+(Dh_{(x,\omega)}^s)_{\xi_y}).
\end{align*}We have $\|L_{(x,\omega)}\|,\|L_{(y,\omega)}\|\leq \frac{4}{3}\mathcal{P}K$. Hence, we have
\begin{align*}
  &\ \ \ \ \sup_{v\in E^s(x,\omega), \|v\|=1}\|P(E^s(f_\omega x,\theta\omega))D_xf_\omega v-P(E^s(f_\omega x,\theta\omega))D_yf_\omega(I+(Dh_{(x,\omega)}^s)_{\xi_y})v\|\\
  & \leq \mathcal{P}(\|D_xf_\omega-D_yf_\omega\|+\|D_yf_\omega (Dh_{(x,\omega)}^s)_{\xi_y}\|)\\
  &\leq \mathcal{P}K\|y-x\|+\mathcal{P}K^2\|\xi_y\|\\
  &=\mathcal{P}Kd(x,y)+\mathcal{P}K^2\|\xi_y\|\\
  &\overset{\eqref{d xy xiy}}\leq (\mathcal{P}K+\frac{3}{2}\mathcal{P}K^2)d(x,y).
\end{align*}So $\|L_{(x,\omega)}-L_{(y,\omega)}\| \leq C(\mathcal{P}K+\frac{3}{2}\mathcal{P}K^2)d(x,y)$, where the constant $C$ only depends on the normal coordinate chart.
Then by properties of the determinant, we have that
\begin{equation}\label{lip1}
  |det(L_{(x,\omega)})-det(L_{(y,\omega)})|\leq R_1d(x,y),
\end{equation}where $R_1$ is a polynomial of $K,\mathcal{P}$ and dim$E^s(x,\omega)$.

Notice that for $\xi_y\in E^s(x,\omega)(\frac{\epsilon}{2K^2})$
\begin{align*}
  &\ \ \ \ \ \|P(E^s(f_\omega (x),\theta\omega))|_{E^s(f_\omega(y),\theta\omega)}-I\| \leq \frac{\|(Dh^s_{(f_\omega(x),\theta\omega)})_{\xi_{f_\omega(y)}}\|}{1-\|(Dh^s_{(f_\omega(x),\theta\omega)})_{\xi_{f_\omega(y)}}\|}\\
  &\leq \frac{K\|\xi_{f_\omega(y)}\|}{1-K\|\xi_{f_\omega(y)}\|}
  \leq \frac{2K^2\|\xi_y\|}{1-2K^2\|\xi_y\|}
  \leq \frac{2K^2\|\xi_y\|}{1-2K^2\frac{\epsilon}{2K^2}}\\
  &\leq 4K^2\|\xi_y\|\leq 6K^2d(x,y).
\end{align*}So we have
\begin{equation}\label{lip2}
  |det(P(E^s(f_\omega(x),\theta\omega))|_{E^s(f_\omega(y),\theta\omega)})-1|\leq R_2d(x,y),
\end{equation}where $R_2$ is a polynomial of $K$ and dim$E^s(x,\omega)$.
Also
\begin{equation*}
  \|I+(Dh_{(x,\omega)}^s)_{\xi_y}-I\|\leq K\|\xi_y\|\leq \frac{3}{2}Kd(x,y)
\end{equation*}implies that there exists a constant $R_3$ such that
\begin{equation}\label{lip3}
  |det(I+(Dh_{(x,\omega)}^s)_{\xi_y})-1|\leq R_3 d(x,y).
\end{equation} Combining $(\ref{lip1})$, $(\ref{lip2})$, and $(\ref{lip3})$, we have
\begin{equation*}
  |J^s(x,\omega)-J^s(y,\omega)|\leq K_0d(x,y),
\end{equation*}where $K_0$ only depends on $K$, $\mathcal{P}$ and $\dim E^s$. Notice that $\inf_{(x,\omega)\in M\times\Omega}|J^s(x,\omega)|>0$, as a consequence, there exists a $K_1>K_0$ such that
\begin{equation*}
  |\log J^s(x,\omega)-\log J^s(y,\omega)|\leq K_1d(x,y).
\end{equation*}The proof of Lemma \ref{Lip jac Es} is complete.
\end{proof}

\subsection{Absolute continuity of the stable and unstable foliations on each fiber}\label{subsection 3.4}
The absolute continuity of $\{ W^\tau_\epsilon(x,\omega)\}$ for fixed $\omega$ and $\tau=s,u$ was stated in \cite{PDQM} for general random dynamical systems without proof. In this subsection, we give a proof in our settings. Besides, we give a specific formula for the Jacobian of holonomy maps.

For any $\omega\in\Omega$, a smooth submanifold $U(\omega)\subset M$ is said to be transverse to the local stable manifolds if for any $x\in U(\omega)$, $T_xU(\omega)\oplus E^s(x,\omega)=T_xM$. The transversal angel can be measured by the following quantity.  For subspaces $A,B\subset\mathbb{R}^N$ with a given norm $\|\cdot\|$, set
\begin{equation}\label{def of Theta}
  \Theta(A,B)=\min\left\{\min_{v\in A,\|v\|=1;w\in B}\|v-w\|; \min_{w\in B,\|w\|=1;v\in A}\|v-w\|\right\}.
\end{equation}For $\theta\in(0,1]$, we say that a subspace $A\subset\mathbb{R}^N$ is $\theta-$transverse to a subspace $B\subset\mathbb{R}^N$ if $\Theta(A,B)\geq \theta.$
Given smooth submanifolds $U(\omega)$ and $V(\omega)$ transversal to the local stable manifolds, we say that $\psi_\omega:U(\omega)\rightarrow V(\omega)$ is a fiber holonomy map induced by the local stable manifolds if $\psi_\omega$ is injective and continuous, and
\begin{equation*}
  \psi_\omega(x)\in W_\epsilon^s(x,\omega)\cap V(\omega)\mbox{  for every }x\in U(\omega).
\end{equation*}
We say that $\{W_\epsilon^s(x,\omega)\}$ is absolutely continuous on each fiber if every fiber holonomy map $\psi_\omega$ induced by the local stable manifolds is absolutely continuous with respect to $m_{U(\omega)}$ and $m_{V(\omega)}$, where $m_{V(\omega)} $ and $m_{U(\omega)}$ are the intrinsic Riemannian volume measure on manifolds $V(\omega)$ and $U(\omega)$ respectively. The absolute continuity of $\{W_\epsilon^u(x,\omega)\}$ can be defined similarly. Our proof of absolute continuity follows the idea given in \cite[Theorem  6.2.6]{Young1995}.
\begin{proposition}\label{Theorem fiberwise absolutely continuous}
  Suppose $\phi$ is $C^2$ Anosov on fibers, then $\{W_\epsilon^s(x,\omega)\}$ and $\{W^u_\epsilon(x,\omega)\}$ are absolutely continuous on each fiber respectively.
\end{proposition}
\begin{proof}[Proof of Proposition \ref{Theorem fiberwise absolutely continuous}]
We only prove that $\{W_\epsilon^s(x,\omega)\}$ is absolutely continuous on each fiber since the case for $\{W_\epsilon^u(x,\omega)\}$ is similar.
For any $\omega\in\Omega$ fixed, let $\psi_\omega:U(\omega)\rightarrow V(\omega)$ be the fiber holonomy map between two smooth pre-compact submanifolds that are transverse to the local stable manifolds. By the regularity of Riemannian volume measure, to prove the absolute continuity of $\psi_\omega$, it is sufficient to prove that there exists a number $C(\omega)>0$ such that
\begin{equation}\label{eq aim ac}
  m_{V(\omega)}(\psi_\omega(A))\leq C(\omega)m_{U(\omega)}(A)\mbox{ for any compact subset } A\subset U(\omega).
\end{equation}Before starting the proof, we need  some preparations.
\begin{lemma}\label{lemma C5C6}
  There exists $C_3,\ C_4(\omega)\in (0,1)$ such that for any $n\in\mathbb{N}$, one has
  \begin{enumerate}
      \item for any $(x,\omega)\in M\times \Omega$, $v\in (E^s(x,\omega))^\perp$,
    \begin{equation}
      \|D_xf_\omega^n v\|\geq  C_3 e^{\lambda n}\|v\|;\label{eq C3}
    \end{equation}
    \item for any $x\in U(\omega)$, $v\in T_x U(\omega)$ and $y\in V(\omega)$, $\beta\in T_yV(\omega)$,
    \begin{equation}\label{dxfn v}
      \|D_xf_\omega^n v\| \geq C_4(\omega)e^{\lambda n}\|v\|,\mbox{ and }\|D_yf_\omega^n \beta\| \geq C_4(\omega)e^{\lambda n}\|\beta\| ;
    \end{equation}
  \end{enumerate}
\end{lemma}
\begin{proof}[Proof of Lemma \ref{lemma C5C6}]
First, we prove (1). We pick $N_0\in\mathbb{N}$ such that for any $n\geq N_0$,
\begin{equation*}
  e^{\lambda n}-e^{-\lambda n}\mathcal{P}\geq \frac{1}{2}e^{\lambda n},
\end{equation*}where $\mathcal{P}$ is the constant in \eqref{def of P}. Now for any $(x,\omega)\in M\times \Omega$,   $v\in (E^s(x,\omega))^\perp$, then $v$ has decomposition $v=v_1+v_2$ for $v_1\in E^u(x,\omega)$ and $v_2\in E^s(x,\omega)$. Note that
\begin{equation*}
  \langle v,v\rangle=\langle v,v_1+v_2\rangle=\langle v,v_1\rangle\leq \|v\|\cdot \|v_1\|,
\end{equation*}which implies $\|v_1\|\geq \|v\|$.
 If $0\leq n<N_0$, we have
\begin{equation*}
   \|D_xf_\omega^n v\|\geq  (\inf_{(x,\omega)\in M\times \Omega}m(D_xf_\omega))^{N_0}\|v\|\geq \frac{(\inf_{(x,\omega)\in M\times \Omega}m(D_xf_\omega))^{N_0}}{e^{\lambda N_0}}e^{\lambda n}\|v\|,
\end{equation*}where $m(D_xf_\omega)=\|(D_xf_\omega)^{-1}\|^{-1}<1$ denotes the co-norm of $D_xf_\omega$. If $n\geq N_0$, then
\begin{align*}
   \|D_xf_\omega^n v\| & \geq  \|D_xf_\omega^n v_1\|- \|D_xf_\omega^n v_2\|\geq e^{\lambda n}\|v_1\|-e^{-\lambda n}\|v_2\|\\
   &\geq e^{\lambda n}\|v\|-e^{-\lambda n}\mathcal{P}\|v\|\geq\frac{1}{2}e^{\lambda n}\|v\|.
\end{align*}These two cases imply that \eqref{eq C3} holds for
\begin{equation*}
   C_3:= \min\left\{\frac{1}{2},\frac{(\inf_{(x,\omega)\in M\times \Omega}m(D_xf_\omega))^{N_0}}{e^{\lambda N_0}}\right\}.
\end{equation*}

Second, we prove (2). We only prove the first inequality of \eqref{dxfn v}, as the second one is similar. Note that $U(\omega)$ and $V(\omega)$ are transverse to the local stable manifolds and pre-compact, the following quantity is positive
\begin{equation}\label{transverse angle}
  \gamma_1(\omega):=\min\{\inf\{\Theta(T_xU(\omega),E^s(x,\omega))|\ x\in U(\omega)\},\inf\{\Theta(T_y V(\omega),E^s(y,\omega))|\ y\in V(\omega)\}\}.
\end{equation}
Now for any $x\in U(\omega)$ and any $v\in T_xU(\omega)$, then $v=v_1+v_2$ for $v_1\in (E^s(x,\omega))^\perp$ and $v_2\in E^s(x,\omega)$. Then one has
\begin{equation}\label{v1 v}
  \|v_1\|=\|v-v_2\|=\min_{\eta\in E^s(x,\omega)}\|v-\eta\|=\min_{\eta\in E^s(x,\omega)}\left\|\frac{v}{\|v\|}-\frac{\eta}{\|v\|}\right\|\cdot\|v\|\geq \gamma_1(\omega)\|v\|.
\end{equation}
 We pick $N_1(\omega)$ sufficiently large such that for any $n\geq N_1(\omega)$
\begin{equation}\label{N0 sufficiently large}
  (C_3e^{\lambda n}-e^{-\lambda n})\gamma_1(\omega)-e^{-\lambda n}\geq \frac{1}{2}e^{\lambda n}C_3\gamma_1(\omega).
\end{equation}If $0<n<N_1(\omega)$, then
\begin{equation*}
  \|D_xf_\omega^nv\|\geq (\inf_{(x,\omega)\in M\times \Omega}m(D_xf_\omega))^{N_1(\omega)}\|v\|\geq \frac{(\inf_{(x,\omega)\in M\times \Omega}m(D_xf_\omega))^{N_1(\omega)}}{e^{\lambda N_1(\omega)}}e^{\lambda n}\|v\|.
\end{equation*} If $n\geq N_1(\omega)$, we have
\begin{align*}
  \|D_xf_\omega^{-n} v\|&=\|D_xf_\omega^n(v_1+v_2)\|\geq \|D_xf_\omega^n v_1\|-\|D_xf_\omega^nv_2\|\\
  &\overset{\eqref{eq C3}}\geq C_3e^{\lambda n}\|v_1\|-e^{-\lambda n}\|v_2\|\geq C_3 e^{\lambda n}\|v_1\|-e^{-\lambda n}(\|v\|+\|v_1\|)\\
  &= (C_3e^{\lambda n}-e^{-\lambda n})\|v_1\|-e^{-\lambda n}\|v\|\overset{\eqref{v1 v}}\geq (C_3e^{\lambda n}-e^{-\lambda n})\gamma_1(\omega)\|v\|-e^{-\lambda n}\|v\|\\
  &\overset{\eqref{N0 sufficiently large}}\geq \frac{1}{2}\gamma_1(\omega)C_3e^{\lambda n}\|v\|.
\end{align*} These two cases imply \eqref{dxfn v} holds for
\begin{equation*}
  C_4(\omega)=\min\left\{ \frac{1}{2}\gamma_1(\omega)C_3, \frac{(\inf_{(x,\omega)\in M\times \Omega}m(D_xf_\omega))^{N_1(\omega)}}{e^{\lambda N_1(\omega)}}\right\}\in(0,1).
\end{equation*}
The proof of Lemma \ref{lemma C5C6} is complete.
\end{proof}
\begin{lemma}\label{lemma new C5C6}
Recall that the distance between subbundle is defined in \eqref{def distance between bundle}. There exists $C_6(\omega)>1$ such that  for any $n\in\mathbb{N}$,
  \begin{align}
  &d(D_xf_\omega^n T_xU(\omega),D_xf_\omega^n E^u(x,\omega))\leq C_6(\omega)e^{-\lambda n}d(T_x U(\omega), E^u(x,\omega))\mbox{ for }x\in U(\omega);\label{d(fnTxU,Eu)}\\
  &d(D_yf_\omega^n T_yV(\omega),D_yf_\omega^n E^u(y,\omega))\leq C_6(\omega)e^{-\lambda n}d(T_y V(\omega), E^u(y,\omega))\mbox{ for }y\in V(\omega)\label{d(fnTxV,Eu)}.
\end{align}
\end{lemma}
\begin{proof}[Proof of Lemma \ref{lemma new C5C6}]
The proof of inequalities \eqref{d(fnTxU,Eu)} and \eqref{d(fnTxV,Eu)} are similar, therefore we only prove \eqref{d(fnTxU,Eu)}. Recall that $\gamma_1(\omega)$ is  defined in \eqref{transverse angle}.  For any $v\in T_xU(\omega)$ with $\|v\|=1$, suppose $v$ has decomposition $v=v_1+v_2$ for some $v_1\in E^u(x,\omega)$ $v_2\in E^u(x,\omega)^\perp$, and decomposition $v=\beta_1+\beta_2$ for some $\beta_1\in E^u(x,\omega)$ and $\beta_2\in E^s(x,\omega)$, then by trigonometry, we have
\begin{equation}\label{eq C difference1}
  \|\beta_2\|\leq \gamma_1(\omega)^{-1}\|v_2\|=\gamma_1(\omega)^{-1}\inf_{\eta\in E^u(x,\omega)}\|v-\eta\|.
\end{equation}Let $\theta_0:=\inf_{(x,\omega)\in M\times\Omega}\Theta(E^u(x,\omega),E^s(x,\omega))>0$. For any $v\in E^u(x,\omega)$ with $\|v\|=1$, suppose $v$ has decomposition $v=v_1+v_2$ for some $v_1\in T_xU(\omega)$, $v_2\in (T_xU(\omega))^\perp$, and decomposition $v=\beta_1+\beta_2$ for some $\beta_1\in T_xU(\omega)$ and $\beta_2\in E^s(x,\omega)$, then by trigonometry, we have
\begin{equation}\label{eq C difference2}
  \|\beta_2\|\leq \theta_0^{-1}\|v_2\|=\theta_0^{-1}\inf_{\eta\in T_xU(\omega)}\|v-\eta\|.
\end{equation}

For any $n\in\mathbb{N}$, pick any $v\in D_xf_\omega^nT_xU(\omega)$ with $\|v\|=1$. Denote $v=D_xf_\omega^n v^\prime$ for $v^\prime\in T_xU(\omega) $. Then there is decomposition $v^\prime=v_1^\prime+v_2^\prime$ for $v_1^\prime\in E^u(x,\omega)$ and $v_2^\prime\in E^s(x,\omega)$. Then we have
\begin{align*}
\inf_{\eta\in D_xf_\omega^nE^u(x,\omega)}\|v-\eta\|&\leq \|v-D_xf_\omega^n v_1^\prime\|=\|D_xf_\omega^n v_2^\prime\|\leq e^{-\lambda n}\|v_2^\prime\|\\
&= e^{-\lambda n}\|v^\prime-v_1^\prime\| =  e^{-\lambda n}\left\|\frac{v^\prime}{\|v^\prime\|}-\frac{v_1^\prime}{\|v^\prime\|}\right\|\cdot\|v^\prime\|\\
&\overset{\eqref{dxfn v}}\leq  e^{-\lambda n}\cdot (C_4(\omega))^{-1}e^{-\lambda n}\|v\|\left\|\frac{v^\prime}{\|v^\prime\|}-\frac{v_1^\prime}{\|v^\prime\|}\right\|\\
&\overset{\eqref{eq C difference1}}\leq  e^{-\lambda n}\cdot (C_4(\omega))^{-1}\cdot e^{-\lambda n}\cdot 1\cdot \gamma_1(\omega)^{-1}\cdot \inf_{\beta\in E^u(x,\omega)}\left\|\frac{v^\prime}{\|v^\prime\|}-\beta\right\|\\
&\leq e^{-\lambda n}\cdot (C_4(\omega))^{-1}\cdot \gamma_1(\omega)^{-1}\cdot\sup_{\zeta\in T_xU(\omega),\|\zeta\|=1}\inf_{\beta\in E^u(x,\omega)}\left\|\zeta-\beta\right\|.
\end{align*}Therefore, we obtain
\begin{equation}\label{3.14 1}
\begin{split}
&\sup_{v\in D_xf_\omega^nT_xU(\omega),\|v\|=1}\inf_{\eta\in D_xf_\omega^nE^u(x,\omega)}\|v-\eta\|\\
\leq &e^{-\lambda n}\cdot (C_4(\omega))^{-1}\cdot \gamma_1(\omega)^{-1}\cdot\sup_{\zeta\in T_xU(\omega),\|\zeta\|=1}\inf_{\beta\in E^u(x,\omega)}\left\|\zeta-\beta\right\|.
\end{split}
\end{equation}Pick any $\xi \in D_xf_\omega^nE^u(x,\omega)$ with $\|\xi\|=1$, and denote $\xi=D_xf_\omega^n \xi^\prime$ for $\xi^\prime\in E^u(x,\omega)$. Then there exists $\xi_1^\prime\in T_xU(\omega)$ and $\xi_2^\prime \in E^s(x,\omega)$ such that $\xi^\prime=\xi_1^\prime+\xi_2^\prime$. Now,
\begin{align*}
 \inf_{\eta\in D_xf_\omega^n T_xU(\omega)}\|\xi-\eta\|&\leq   \|\xi-D_xf_\omega^n\xi_1^\prime\|=\|D_xf_\omega^n\xi_2^\prime\|\leq e^{-\lambda n}\|\xi_2^\prime\|\\
 &= e^{-\lambda n}\|\xi^\prime-\xi_1^\prime\|= e^{-\lambda n}\left\|\frac{\xi^\prime}{\|\xi^\prime\|}-\frac{\xi_1^\prime}{\|\xi^\prime\|}\right\|\cdot\|\xi^\prime\|\\
 &\overset{\eqref{eq C difference2}}\leq  e^{-\lambda n}\cdot \|\xi^\prime\|\cdot  \theta_0^{-1}\cdot \inf_{\zeta\in T_xU(\omega)}\left\|\frac{\xi^\prime}{\|\xi^\prime\|}-\zeta\right\|\\
 &\leq e^{-\lambda n}\cdot e^{-\lambda n}\|\xi\|\cdot \theta_0^{-1}\cdot\sup_{\beta\in E^u(\omega,x),\|\beta\|=1}\inf_{\zeta\in T_xU(\omega)}\left\|\beta-\zeta\right\|\\
  &\leq e^{-\lambda n}\cdot 1 \cdot \theta_0^{-1}\cdot\sup_{\beta\in E^u(\omega,x),\|\beta\|=1}\inf_{\zeta\in T_xU(\omega)}\left\|\beta-\zeta\right\|.
\end{align*}Therefore, we have
\begin{equation}\label{3.14 2}
\begin{split}
  \sup_{\xi\in D_xf_\omega^nE^u(x,\omega),\|\xi\|=1}\inf_{\eta\in D_xf_\omega^n T_xU(\omega)}\|\xi-\eta\|\leq & e^{-\lambda n}\cdot \theta_0^{-1}\cdot\sup_{\beta\in E^u(\omega,x),\|\beta\|=1}\inf_{\zeta\in T_xU(\omega)}\left\|\beta-\zeta\right\|.
\end{split}
\end{equation}By checking definitions and letting $C_6(\omega)=(C_4(\omega))^{-1}\max\{\gamma_1(\omega)^{-1},\theta_0^{-1}\}$, then \eqref{d(fnTxU,Eu)} is a corollary of \eqref{3.14 1}, \eqref{3.14 2}. The proof of Lemma \ref{lemma new C5C6} is complete.
\end{proof}

For any compact $A\subset U(\omega)$, let $\mathcal{O}$ be a small neighborhood of $A$ in $U(\omega)$ such that
\begin{equation}\label{measure O leq 2 m A}
  m_{U(\omega)}(\mathcal{O})\leq 2m_{U(\omega)}(A).
\end{equation} By $(\ref{dxfn v})$, let $\delta_0\in (0,\epsilon_0)$ be a sufficiently small number, then there exists a number $N_1(\omega)$ such that for any $n\geq N_1(\omega)$ and $\delta\in(0,\delta_0)$, we have
\begin{equation}\label{contained in O}
  f_{\theta^n\omega}^{-n}B_{f_\omega^nU(\omega)}(f_\omega^nx,\delta)\subset\mathcal{O}\mbox{ for any }x\in A,
\end{equation}where $B_{f_\omega^nU(\omega)}(f_\omega^nx,\delta)$ is the $\delta-$neighborhood of $f_\omega^nx$ on $f_\omega^nU(\omega)$.

\begin{lemma}\label{Lemma close}
For  any $\delta\in(0,\delta_0)$ and any constant $C_7>1$, there exists  $N_2(\omega)\in\mathbb{N}$ such that for any $n\geq N_2(\omega)$ and $x\in f_\omega^nU(\omega)$, we have
\begin{equation}\label{map delta into 2delta}
  \bar{\psi}_{\theta^n\omega}(B_{f_\omega^nU(\omega)}(x,\delta))\subset B_{f_\omega^nV(\omega)}(\bar{\psi}_{\theta^n\omega} (x),2\delta),
\end{equation}and
\begin{equation}\label{def of C7}
  C_7^{-1}\leq\frac{ m_{f_\omega^n U(\omega)}(B_{f_\omega^nU(\omega)}(x,\delta))}{ m_{f_\omega^n V(\omega)}(\bar{\psi}_{\theta^n\omega}(B_{f_\omega^nU(\omega)}(x,\delta)))} \leq C_7,
\end{equation}
where $\bar{\psi}_{\theta^n\omega}=f_\omega^n\circ\psi_\omega\circ f_{\theta^n\omega}^{-n}:f_\omega^nU(\omega)\rightarrow f_\omega^nV(\omega)$ is the holonomy map induced by the local stable manifolds.
\end{lemma}
\begin{proof}[Proof of Lemma \ref{Lemma close}]
Notice that for any $x\in U(\omega)$, $\psi_\omega(x)= V(\omega)\cap W^s_\epsilon(x,\omega)$, so we have
\begin{equation}\label{d fnU fn V}
  d(f_\omega^nx,f_\omega^n\psi_\omega(x))\leq e^{-\lambda n}d(x,\psi_\omega(x)).
\end{equation}
 By \eqref{d(fnTxU,Eu)}, \eqref{d(fnTxV,Eu)} and \eqref{d fnU fn V}, we know that $f_\omega^nU(\omega)$ and $f_\omega^nV(\omega)$ will be $C^1$-close to the unstable manifold uniformly for points on $f_\omega^nU(\omega)$ and $f_\omega^nV(\omega)$ as $n$ goes to infinity. We let $n$ be determined later. For any $x\in f_\omega^nU(\omega)$ and $\delta\in (0,\delta_0)$, by employing the exponential map, we may pretend that $B_{f_\omega^nU(\omega)}(x,\delta)$ and $\bar{\psi}_{\theta^n\omega}(B_{f_\omega^nU(\omega)}(x,\delta))$ lie in Euclidean space $E^u(x,\omega)\oplus E^s(x,\omega)$. Denote $P^u:E^u(x,\omega)\oplus E^s(x,\omega)\to E^u(x,\omega)$ to be the projection onto the first coordinate. Let
 \begin{equation*}
   \Psi_n^1,\Psi_n^2:E^u(x,\omega)\to  E^s(x,\omega)
 \end{equation*}be $C^2$ maps, whose graph represent $B_{f_\omega^nU(\omega)}(x,\delta)$ and $\bar{\psi}_{\theta^n\omega}(B_{f_\omega^nU(\omega)}(x,\delta))$ respectively. For any $\eta\in (0,\delta)$, since $f_\omega^nU(\omega)$ and $f_\omega^nV(\omega)$ are $C^1$-close to the unstable manifolds, there exists $N_\eta(\omega)\in\mathbb{N}$ independent of $x\in U(\omega)$ such that for any $n\geq N_\eta(\omega)$,
 \begin{equation}\label{lip Dpsi}
   \|D\Psi_n^i|_{E^u(x,\omega)(2\mathcal{P}\delta)}\|\leq \frac{1}{2},\mbox{ and } Lip(D\Psi_n^i|_{E^u(x,\omega)(2\mathcal{P}\delta)})\leq L+1, \mbox{ for }i=1,2,
 \end{equation}and
 \begin{equation}\label{eq less than C7}
\|(\Psi_n^1-\Psi_n^2)|_{E^u(x,\omega)(2\mathcal{P}\delta)}\|_{C^1}+\mathcal{P} \sup_{z\in B_{f_\omega^nU(\omega)}(x,\delta) } \|z-\bar{\psi}_{\theta^n\omega}(z)\|<\eta,
 \end{equation}
 where constant $L$ is given in Lemma \ref{lemma stable unstable}, and $\mathcal{P}$ is the constant in \eqref{def of P}.

For any $n\geq N_\eta(\omega)$, $x\in f_\omega^nU(\omega)$ and $z\in B_{f_\omega^nU(\omega)}(x,\delta)$,  we have
 \begin{equation}\label{eq edge}
   \|P^u(z)-P^u(\bar{\psi}_{\theta^n\omega}(z))\|\leq \mathcal{P}\|z-\bar{\psi}_{\theta^n\omega}(z)\|<\eta.
 \end{equation}Combing \eqref{eq edge} and the inequality $\|(\Psi_n^1-\Psi_n^2)|_{E^u(x,\omega)(2\mathcal{P}\delta)}\|_{C^1}<\eta$, one has
 \begin{equation*}
  B_{f_\omega^nV(\omega)}(\bar{\psi}_{\theta^n\omega} (x),\frac{1}{2}\delta)\subset  \bar{\psi}_{\theta^n\omega}(B_{f_\omega^nU(\omega)}(x,\delta))\subset  B_{f_\omega^nV(\omega)}(\bar{\psi}_{\theta^n\omega} (x),2\delta),
 \end{equation*}provided $\eta$ sufficiently small and $n$ correspondingly large.
 Inequality \eqref{eq edge} also implies that the difference between the area of $P^u(B_{f_\omega^nU(\omega)}(x,\delta))$ and the area of $P^u(\bar{\psi}_{\theta^n\omega}(B_{f_\omega^nU(\omega)}(x,\delta)))$ is up to a scale of $\eta$. By letting $\eta$ sufficiently small and pick $n$ correspondingly large, and using the regularity of Lebesgue measure, the Lebesgue measure of
 \begin{equation*}\label{measure close to 1}
   P^u(B_{f_\omega^nU(\omega)}(x,\delta))\cup P^u(\bar{\psi}_{\theta^n\omega}(B_{f_\omega^nU(\omega)}(x,\delta)))\backslash \left(P^u(B_{f_\omega^nU(\omega)}(x,\delta))\cap P^u(\bar{\psi}_{\theta^n\omega}(B_{f_\omega^nU(\omega)}(x,\delta)))\right)
 \end{equation*}is less than or equal to $C\cdot \eta$ for some constant $C$.

Moreover, for any $e\in P^u(B_{f_\omega^nU(\omega)}(x,\delta))\cap P^u(\bar{\psi}_{\theta^n\omega}(B_{f_\omega^nU(\omega)}(x,\delta)))$, we have
 \begin{align*}
    & \|D_{e}\Psi_n^1-D_{e}\Psi_n^2\|\leq \|\Psi_n^1-\Psi_n^2\|_{C^1}\overset{\eqref{eq less than C7}}<\eta.
 \end{align*}
 Denote $\tilde{\Psi}_n^i:E^u(x,\omega)\to E^u(x,\omega)\oplus E^s(x,\omega)$ by $e\mapsto (e,\Psi_n^i(e))$ for $i=1,2$. Then for any $e\in P^u(B_{f_\omega^nU(\omega)}(x,\delta))\cap P^u(\bar{\psi}_{\theta^n\omega}(B_{f_\omega^nU(\omega)}(x,\delta)))$, we have
 \begin{align*}
 &\left\|(D_e\tilde{\Psi}_n^1)^*(D_e\tilde{\Psi}_n^1)-(D_e\tilde{\Psi}_n^2)^*(D_e\tilde{\Psi}_n^2)\right\|\\
 \leq&\left\|(D_e\tilde{\Psi}_n^1)^*(D_e\tilde{\Psi}_n^1)-(D_e\tilde{\Psi}_n^1)^*(D_e\tilde{\Psi}_n^2)\right\|+\left\|(D_e\tilde{\Psi}_n^1)^*(D_e\tilde{\Psi}_n^2)-(D_e\tilde{\Psi}_n^2)^*(D_e\tilde{\Psi}_n^2)\right\|\\
\overset{\eqref{lip Dpsi}}\leq & \frac{3}{2}\eta+\frac{3}{2}\eta=3\eta,
 \end{align*}where $*$ means the adjoint matrix. As a consequence, for such points $e$,
 \begin{equation*}
   \left|Jac(\tilde{\Psi}_n^1)(e)-Jac(\tilde{\Psi}_n^2)(e)\right|=\left|\sqrt{\det((D_e\tilde{\Psi}_n^1)^*(D_e\tilde{\Psi}_n^1))}-\sqrt{\det((D_e\tilde{\Psi}_n^2)^*(D_e\tilde{\Psi}_n^))}\right|\leq C\eta,
 \end{equation*}for some constant $C$. By \eqref{lip Dpsi}, we have
 \begin{equation*}
 (1+\frac{1}{2})^{\dim E^u(\omega,x)} >  Jac(\tilde{\Psi}_n^i)(e)\geq \sqrt{1-\frac{1}{4}}=\frac{\sqrt{3}}{2}\mbox{ for }i=1,2.
 \end{equation*}
 Therefore, for any $e\in P^u(B_{f_\omega^nU(\omega)}(x,\delta))\cap P^u(\bar{\psi}_{\theta^n\omega}(B_{f_\omega^nU(\omega)}(x,\delta)))$, we have
 \begin{align}\label{eq defCstar}
   &\left|\log\frac{Jac(\tilde{\Psi}_n^1)(e)}{Jac(\tilde{\Psi}_n^2)(e)}\right|
   \leq C^*\cdot \eta,
 \end{align}where $C^*$ is some constant. Now by area formula, we obtain
 \begin{equation}\label{eq leq 1}
   \begin{split}
   &m_{f_\omega^n U(\omega)}((B_{f_\omega^nU(\omega)}(x,\delta)))\\
   =&\int _{P^u(B_{f_\omega^nU(\omega)}(x,\delta))}Jac(\tilde{\Psi}_n^1)(e)dm(e)\\
   \leq&\int _{P^u(B_{f_\omega^nU(\omega)}(x,\delta))\cap P^u(\bar{\psi}_{\theta^n\omega}(B_{f_\omega^nU(\omega)}(x,\delta)))}Jac(\tilde{\Psi}_n^1)(e)dm(e)+(3/2)^{\dim E^u}C\eta\\
   \leq &e^{C^*\eta}\int _{P^u(B_{f_\omega^nU(\omega)}(x,\delta))\cap P^u(\bar{\psi}_{\theta^n\omega}(B_{f_\omega^nU(\omega)}(x,\delta)))}Jac(\tilde{\Psi}_n^2)(e)dm(e)+(3/2)^{\dim E^u}C\eta\\
   \leq & e^{C^*\eta}\int _{ P^u(\bar{\psi}_{\theta^n\omega}(B_{f_\omega^nU(\omega)}(x,\delta)))}Jac(\tilde{\Psi}_n^2)(e)dm(e)+(3/2)^{\dim E^u}C\eta\\
 =  &e^{C^*\eta} \cdot m_{f_\omega^n V(\omega)}(\bar{\psi}_{\theta^n\omega}(B_{f_\omega^nU(\omega)}(x,\delta)))+(3/2)^{\dim E^u}C\eta.
   \end{split}
 \end{equation}
Symmetrically, we also have
 \begin{equation}\label{eq leq 2}
 m_{f_\omega^n V(\omega)}(\bar{\psi}_{\theta^n\omega}(B_{f_\omega^nU(\omega)}(x,\delta)))\leq e^{C^*\eta}m_{f_\omega^n U(\omega)}((B_{f_\omega^nU(\omega)}(x,\delta)))+(3/2)^{\dim E^u}C\eta.
 \end{equation} Notice that $f_\omega^nU(\omega)$ and $f_\omega^nV(\omega)$ are $\dim (E^u(x,\omega))$-submanifold of $M$ satisfying \eqref{lip Dpsi}, therefore $m_{f_\omega^n U(\omega)}((B_{f_\omega^nU(\omega)}(x,\delta)))$ and $m_{f_\omega^n V(\omega)}(\bar{\psi}_{\theta^n\omega}(B_{f_\omega^nU(\omega)}(x,\delta)))> m_{f_\omega^n V(\omega)}((B_{f_\omega^nV(\omega)}(\bar{\psi}_{\theta^n\omega}(x),\frac{1}{2}\delta)))$ are bounded below uniformly by scale of $\delta$. Finally, for any $C_7>1$, \eqref{def of C7} is a consequence of \eqref{eq leq 1} and \eqref{eq leq 2} by letting $\eta$ sufficiently small and $n$ correspondingly large. This finishes the proof.
\end{proof}

Now we pick any $\delta\in (0,\delta_0)$, $C_7>1$, and fix
\begin{equation}\label{def N omega}
  N=N(\omega)=\max\{ N_1(\omega),N_2(\omega)\},
\end{equation}where $N_1(\omega)$ is chosen satisfying \eqref{contained in O}, and $N_2(\omega)$ is chosen in Lemma \ref{Lemma close} corresponding to $C_7>1$.
 Let $\{B_{f_\omega^NU(\omega)}(x,\delta)\}_{x\in f_\omega^NA}$ be a cover of $f_\omega^NA$ by $\delta-$balls centered at points in $f_\omega^NA$. By the Besicovitch covering lemma (see, e.g., \cite{Besicovitch}), there exists a finite subcover $\{B_i\}_{i=1}^k\subset \{B_{f_\omega^NU(\omega)}(x,\delta)\}_{x\in f_\omega^NA}$ of $f_\omega^NA$ and a number $C^\prime=C^\prime(\dim(E^u))$ satisfying
\begin{equation}\label{def of Cprime}
  \mbox{there is no point in $f_\omega^NA$ that lies in more than number $C^\prime$ of the $B_i$'s.}
\end{equation}
Next, we claim that there exists a number $C_8(\omega)>1$ independent of $N(\omega)$ 
such that
\begin{equation}\label{claim C8}
  C_8(\omega)^{-1}m_{V(\omega)}(\psi_\omega(f_{\theta^N\omega}^{-N}B_i))\leq m_{U(\omega)}(f_{\theta^N\omega}^{-N}B_i) \leq C_8(\omega)m_{V(\omega)}(\psi_\omega(f_{\theta^N\omega}^{-N}B_i)).
\end{equation}
If the above claim is true, then we arrive
\begin{align*}
  m_{V(\omega)}(\psi_\omega(A))&\leq \sum_{i=1}^{k}m_{V(\omega)}(\psi_\omega(f_{\theta^N\omega}^{-N}B_i))\leq  \sum_{i=1}^{k}C_{8}(\omega)m_{U(\omega)}(f_{\theta^N\omega}^{-N}B_i)\\
  &\overset{\eqref{contained in O},\eqref{def of Cprime}}\leq C^\prime C_{8}(\omega)m_{U(\omega)}(\mathcal{O})\\
  &\overset{\eqref{measure O leq 2 m A}}\leq 2C^\prime C_{8}(\omega)m_{U(\omega)}(A),
\end{align*}i.e., \eqref{eq aim ac} holds. Hence $\psi_\omega$ is absolutely continuous.

So it is left to prove the claim \eqref{claim C8}. We  still need the following lemma.
\begin{lemma}\label{limit of H}
For any $x\in U(\omega)$, denote \begin{equation}\label{definition of }
          H_\omega^n(x,\psi_\omega(x),T_xU(\omega),T_{\psi_\omega(x)}V(\omega)):=\frac{|\det(D_xf_\omega^n|_{T_xU(\omega)})|}{|\det(D_{\psi_\omega(x)}f_\omega^n|_{T_{\psi_\omega(x)}V(\omega)})|}.
        \end{equation}There exists a number $C_{9}(\omega)>0$ 
        such that  for any $n\in \mathbb{N}$, $x\in U(\omega)$,
        \begin{equation*}
         C_9(\omega)^{-1}\leq  H_\omega^n(x,\psi_\omega(x),T_xU(\omega),T_{\psi_\omega(x)}V(\omega))\leq C_9(\omega).
        \end{equation*}As a consequence, the limit
        \begin{equation}\label{def H omega limit}
          H_\omega(x,\psi_\omega(x),T_xU(\omega),T_{\psi_\omega(x)}V(\omega)):=\lim_{n\rightarrow \infty}H_\omega^n(x,\psi_\omega(x),T_xU(\omega),T_{\psi_\omega(x)}V(\omega))
        \end{equation}        exists and converges uniformly for all $x\in U(\omega)$.
\end{lemma}
\begin{proof}[proof of lemma \ref{limit of H}]

For any $j\in\mathbb{N}$, by Lemma \ref{detDfE-detDfE} and Lemma \ref{Holder continuity of stableunstable distribution}, we have
\begin{align*}
   & \ \ \ \ \left|\det\left(D_{f_\omega^j x}f_{\theta^j\omega}|_{D_xf_\omega^jT_xU(\omega)}\right)-\det\left(D_{f_\omega^j\psi_\omega(x)}f_{\theta^j\omega}|_{D_{\psi_\omega(x)}f_\omega^jT_{\psi_\omega(x)}V(\omega)}\right)\right|\\
   &\leq \left|\det\left(D_{f_\omega^j x}f_{\theta^j\omega}|_{D_xf_\omega^jT_xU(\omega)}\right)-\det\left(D_{f_\omega^jx}f_{\theta^j\omega}|_{D_{x}f_\omega^jE^u(x,\omega)}\right)\right|\\
   &\ \ \ +\left|\det\left(D_{f_\omega^jx}f_{\theta^j\omega}|_{D_{x}f_\omega^jE^u(x,\omega)}\right)-\det\left(D_{f_\omega^j \psi_\omega(x)}f_{\theta^j\omega}|_{D_{\psi_\omega(x)}f_\omega^jE^u(\psi_\omega(x),\omega)}\right)\right|\\
   &\ \ \ \ \ \ \ \ +\left|\det\left(D_{f_\omega^j \psi_\omega(x)}f_{\theta^j\omega}|_{D_{\psi_\omega(x)}f_\omega^jE^u(\psi_\omega(x),\omega)}\right)-\det\left(D_{f_\omega^j\psi_\omega(x)}f_{\theta^j\omega}|_{D_{\psi_\omega(x)}f_\omega^jT_{\psi_\omega(x)}V(\omega)}\right)\right|\\
   &\overset{\eqref{psoitive     time}}\leq C_2d(D_xf_\omega^jT_xU(\omega),D_xf_\omega^jE^u(x,\omega))+C_2[d(f_\omega^jx,f_\omega^j\psi_\omega(x))+d(E^u(f_\omega^jx,\theta^j\omega),E^u(f_\omega^j\psi_\omega(x),\theta^j\omega))]\\
   &\ \ \ +C_2d(D_{\psi_\omega(x)}f_\omega^jE^u(\psi_\omega(x),\omega),D_{\psi_\omega(x)}f_\omega^jT_{\psi_\omega(x)}V(\omega))\\
   &\overset{\eqref{d(fnTxU,Eu)},\eqref{Holder continuous bundle},\eqref{d(fnTxV,Eu)}}\leq C_6(\omega)C_2e^{-\lambda j}d(T_xU(\omega),E^u(x,\omega))+2C_2C_1d(f_\omega^jx,f_\omega^j\psi_\omega(x))^{\nu_1}\\
   &\quad\quad\quad\quad\quad\quad\quad +C_6(\omega)C_2e^{-\lambda j}d(T_{\psi_\omega(x)}V(\omega),E^u(\psi_\omega(x),\theta^j\omega))\\
   &\leq 2C_1C_2C_6(\omega)e^{-\lambda j\nu_1}[d(T_xU(\omega),E^u(x,\omega))+d(T_{\psi_\omega(x)}V(\omega),E^u(\psi_\omega(x),\omega))+d(x,\psi_\omega(x))^{\nu_1}]\\
   &\leq 6C_1C_2C_6(\omega)e^{-\lambda j \nu_1}.
\end{align*}Note that by the compactness of $M$ and $\Omega$ and the continuity of $f_\omega\in $Diff$^2(M)$ on $\omega\in\Omega$, there exists a constant $C_{10}$ such that for any $y\in M$ and $F(\omega)\subset T_yM$,
  \begin{equation}\label{bound of det}
    C_{10}^{-1}\leq |\det(D_yf_\omega|_{F(\omega)})|\leq C_{10}.
  \end{equation}Denote
$ C_{11}(\omega)= 6C_6(\omega)C_1C_2C_{10},$ and we obtain that for $j\in\mathbb{N}$,
  \begin{equation}\label{eq leq C 11}
    \frac{|\det(D_{f_\omega^j x}f_{\theta^j\omega}|_{D_xf_\omega^jT_xU(\omega)})|}{|\det(D_{f_\omega^j\psi_\omega(x)}f_{\theta^j\omega}|_{D_{\psi_\omega(x)}f_\omega^jT_{\psi_\omega(x)}V(\omega)})|}\leq 1+C_{11}(\omega)e^{-\lambda j\nu_1}.
  \end{equation}Now \eqref{eq leq C 11} implies that for any $n\in\mathbb{N}$,
  \begin{align*}  H_\omega^n(x,\psi_\omega(x),T_xU(\omega),T_{\psi_\omega(x)}V(\omega))
     &=\prod_{j=0}^{n-1}\frac{|\det(D_{f_\omega^j x}f_{\theta^j\omega}|_{D_xf_\omega^jT_xU(\omega)})|}{|\det(D_{f_\omega^j\psi_\omega(x)}f_{\theta^j\omega}|_{D_{\psi_\omega(x)}f_\omega^jT_{\psi_\omega(x)}V(\omega)})|}\nonumber\\
     &\leq \prod_{j=0}^{n-1}(1+C_{11}(\omega)e^{-\lambda j\nu_1})\\
     &\leq \exp\left(\frac{C_{11}(\omega)}{1-e^{\lambda \nu_1}}\right).
  \end{align*}
 We denote $\exp(\frac{C_{11}(\omega)}{1-e^{\lambda \nu_1}})$ by $C_9(\omega)$. The same estimate holds for $H_\omega^n(\psi_\omega(x),x,T_{\psi_\omega(x)}V(\omega),T_xU(\omega))$ for any $n\in\mathbb{N}$. The proof of Lemma \ref{limit of H} is complete.
\end{proof}
\begin{lemma}\label{lemma C13}
There exists $C_{13}(\omega)>1$ independent of $N(\omega)$ such that for any  $x,y\in f_{\theta^N\omega}^{-N}B_i$ and $p,q\in \psi_\omega(f_{\theta^N\omega}^{-N}B_i)$, we have
    \begin{equation}\label{def of C17}
      C_{13}(\omega)^{-1}\leq \frac{|\det D_xf_\omega^N|_{T_xU(\omega)}|}{|\det D_yf_\omega^N|_{T_yU(\omega)}|}\leq C_{13}(\omega)
    \end{equation}and
    \begin{equation}\label{def of C18}
       C_{13}(\omega)^{-1}\leq \frac{|\det D_pf_\omega^N|_{T_pV(\omega)}|}{|\det D_qf_\omega^N|_{T_qV(\omega)}|}\leq C_{13}(\omega).
    \end{equation}
\end{lemma}
\begin{proof}[Proof of Lemma \ref{lemma C13}]
Denote $x_i=f_\omega^i(x)$ and $y_i=f_\omega^i(y)$ for $i=0,...,N$. We note that $x_N,y_N\in B_i$, therefore $d(x_N,y_N)<\delta$. By virtue of \eqref{dxfn v}, we have
\begin{equation}\label{d xiyi}
  \mbox{$d(x_i,y_i)\leq C_4(\omega)^{-1}e^{-\lambda (N-i)}d(x_N,y_N)<C_4(\omega)^{-1}e^{-\lambda (N-i)}\delta$ for $i=0,...,N$.}
\end{equation}
  For $i=0,...,N$, by Lemma \ref{detDfE-detDfE} and Lemma \ref{Holder continuity of stableunstable distribution}, we have
  \begin{equation*}\label{eq def c13}
    \begin{split}
 & \left|\det\left(D_{x_i}f_{\theta^i\omega}|_{D_xf_\omega^iT_xU(\omega)}\right)-\det\left(D_{y_i}f_{\theta^i\omega}|_{D_{y}f_\omega^iT_{y}U(\omega)}\right)\right|\\
    \leq  & \left|\det\left(D_{x_i}f_{\theta^i\omega}|_{D_xf_\omega^iT_xU(\omega)}\right)-\det\left(D_{x_i}f_{\theta^i\omega}|_{E^u(x_i,\theta^i\omega)}\right)\right|\\
    &\quad+ \left|\det\left(D_{x_i}f_{\theta^i\omega}|_{E^u(x_i,\theta^i\omega)}\right)-\det\left(D_{y_i}f_{\theta^i\omega}|_{E^u(y_i,\theta^i\omega)}\right)\right|\\
    &\quad\quad +\left|\det\left(D_{y_i}f_{\theta^i\omega}|_{E^u(y_i,\theta^i\omega)}\right)-\det\left(D_{y_i}f_{\theta^i\omega}|_{D_yf_\omega^iT_yU(\omega)}\right)\right|\\
   \overset{\eqref{psoitive     time}}\leq  &C_2d(D_xf_\omega^iT_xU(\omega),E^u(x_i,\theta^i\omega))+C_2[d(x_i,y_i)+d(E^u(x_i,\theta^i\omega),E^u(y_i,\theta^i\omega))]\\
   &\quad + C_2d(E^u(y_i,\theta^i\omega),D_yf_\omega^iT_yU(\omega))\\
\overset{\eqref{d(fnTxU,Eu)},\eqref{d(fnTxV,Eu)}} \leq  & C_2C_6(\omega)e^{-\lambda i}(d(T_xU(\omega),E^u(x,\omega))+d(T_yU(\omega),E^u(y,\omega)))+C_2(1+C_1)d(x_i,y_i)^{\nu_1}\\
\overset{\eqref{d xiyi}}\leq &2 C_2C_6(\omega)e^{-\lambda i}\cdot\left(\sup_{x\in U(\omega)}d(T_xU(\omega),E^u(x,\omega))\right)+C_2(1+C_1)C_4(\omega)^{-1}e^{-\lambda(N-i)\nu_1}\delta^{\nu_1}.
    \end{split}
  \end{equation*}
By \eqref{bound of det} and noticing $\sup_{x\in U(\omega)}d(T_xU(\omega),E^u(x,\omega))< 1$, one has
  \begin{equation}\label{eq mid over}
  \begin{split}
    \frac{|\det\left(D_{x_i}f_{\theta^i\omega}|_{D_xf_\omega^iT_xU(\omega)}\right)|}{|\det\left(D_{y_i}f_{\theta^i\omega}|_{D_{y}f_\omega^iT_{y}U(\omega)}\right)|}
    \leq& 1+C_{12}(\omega)e^{-\lambda i}+C_{12}(\omega)e^{-\lambda (N-i)\nu_1}\cdot (2\delta)^{\nu_1}\\
       \leq& 1+C_{12}(\omega)e^{-\lambda i}+C_{12}(\omega)e^{-\lambda (N-i)\nu_1},
  \end{split}
  \end{equation}where
  \begin{equation}\label{def C12}
    C_{12}(\omega):=2C_{10}C_2(1+C_1)C_4(\omega)^{-1}C_6(\omega).
  \end{equation}
  Then \eqref{def of C17} is a corollary of \eqref{eq mid over} by letting
  \begin{equation*}
    C_{13}(\omega):=e^{C_{12}(\omega)\cdot \sum_{i=0}^{\infty}e^{-\lambda i}+C_{12}(\omega)\cdot \sum_{i=0}^{\infty}e^{-\lambda i\nu_1}}.
  \end{equation*}For \eqref{def of C18}, we notice that $f_\omega^N(p),f_\omega^N(q)\in f_\omega^N(\psi_\omega(f_{\theta^N\omega}^{-N}B_i))= \bar{\psi}_{\theta^N\omega}(B_i) $ and $ \bar{\psi}_{\theta^N\omega}(B_i)$ is contained in a ball in $f_\omega^NV(\omega)$ with radius less than $2\delta$ by \eqref{map delta into 2delta}. Then \eqref{def of C18} can be proved similarly as \eqref{def of C17}.
\end{proof}

Now we are ready to prove the claim $(\ref{claim C8})$. Pick any $p_i\in B_i$, denote $q_i:=\bar{\psi}_{\theta^N\omega}(p_i)\in f_\omega^NV(\omega)$ to be the holonomy image of $p_i$, by $(\ref{def of C17})$ and change of variable, we have
\begin{align*}
 &C_{13}(\omega)^{-1}\cdot\left|\det D_{p_i}f_{\theta^N\omega}^{-N}|_{T_{p_i}f_\omega^NU(\omega)}\right|\cdot m_{f_\omega^NU(\omega)}(B_i) \leq m_{U(\omega)}(f_{\theta^N\omega}^{-N}B_i)\\
 &\ \ \ \ \ \ \ \ \ \ \ \ \leq C_{13}(\omega)\cdot \left|\det D_{p_i}f_{\theta^N\omega}^{-N}|_{T_{p_i}f_\omega^NU(\omega)}\right|\cdot m_{f_\omega^NU(\omega)}(B_i).
\end{align*}Recall that $B_i$ is a $\delta$-ball, then by Lemma \ref{limit of H} and $(\ref{def of C7})$,
\begin{align*}
 &C_7^{-1}C_9(\omega)^{-1}C_{13}(\omega)^{-1}\cdot \left|\det D_{q_i}f_{\theta^N\omega}^{-N}|_{T_{q_i}f_\omega^NV(\omega)}\right|\cdot m_{f_\omega^NV(\omega)}(\bar{\psi}_{\theta^N\omega}B_i) \leq m_{U(\omega)}(f_{\theta^N\omega}^{-N}B_i)\\
 &\ \ \ \ \ \ \ \ \ \ \ \ \leq C_7C_9(\omega)C_{13}(\omega)\cdot \left|\det D_{q_i}f_{\theta^N\omega}^{-N}|_{T_{q_i}f_\omega^NV(\omega)}\right|\cdot m_{f_\omega^NV(\omega)}(\bar{\psi}_{\theta^N\omega}B_i).
\end{align*}Again by $(\ref{def of C18})$ and change of variable, we obtain
\begin{align*}
 &C_7^{-1}C_9(\omega)^{-1}(C_{13}(\omega))^{-2}\cdot m_{V(\omega)}(\psi_\omega f_{\theta^N\omega}^{-N} B_i) \leq m_{U(\omega)}(f_{\theta^N\omega}^{-N}B_i)\\
 &\ \ \ \ \ \ \ \ \ \ \ \ \leq C_7C_9(\omega)(C_{13}(\omega))^2\cdot m_{V(\omega)}(\psi_\omega f_{\theta^N\omega}^{-N} B_i).
\end{align*}Denote $C_{8}(\omega):=C_7C_9(\omega)C_{13}(\omega)^2$, then the claim $(\ref{claim C8})$ is proved. The proof for absolute continuity is complete.

Next, we explore the Jacobian of the holonomy map. Recall that $\bar{\psi}_{\theta^n\omega}:f_\omega^nU(\omega)\rightarrow f_{\omega}^nV(\omega)$ is the holonomy map induced by the local stable manifolds. Notice that  $\psi_\omega=f_{\theta^n\omega}^{-n}\circ\bar{\psi}_{\theta^n\omega}\circ f_\omega^n$, so for $m_{U(\omega)}$-a.s. $x\in U(\omega)$,
\begin{equation}\label{jacobian 1}
  Jac(\psi_\omega)(x)=H_\omega^n(x,\psi_\omega(x), T_xU(\omega),T_{\psi_\omega(x)}V(\omega))\cdot Jac(\bar{\psi}_{\theta^n\omega})(f_\omega^n(x)).
\end{equation}
To derive the formula of $Jac(\psi_\omega)(x)$, we need the following lemma, whose condition is stronger than ones in  \cite[Theorem 3.3]{Manebook}.
\begin{lemma}\label{mane lemma}
  Let $N$ and $P$ be manifolds with $m_P(P)<\infty$ and $m_N(\partial N)=0$, and $(\phi_n)_{n\geq 0}$, $\phi_n:N\to P$, be a sequence of absolutely continuous maps with Jacobian $Jac(\phi_n)$. If $\phi_n$ converges uniformly to an absolutely continuous and injective map $\phi:N\to P$ and $Jac(\phi_n)$ (by changing its value on zero measure set) converges uniformly to an integrable function $J:N\to\mathbb{R}$, then $Jac(\phi)=J$.
\end{lemma}
Next, we will define a sequence of absolutely continuous maps $\pi_{\theta^n\omega}:f_\omega^nU(\omega)\to f_\omega^nV(\omega)$ such that the sequence of mapping
\begin{equation}\label{def psinomega}
  \psi_{n,\omega}= (f_{\theta^n\omega}^{-n}|_{f_\omega^nV(\omega)})\circ \pi_{\theta^n\omega}\circ (f_\omega^n|_{U(\omega)}):U(\omega)\to V(\omega)
\end{equation}converges to $\psi_\omega$ uniformly, and a version of $Jac(\psi_{n,\omega})$ converges uniformly to an integrable function. Here, a version of $Jac(\psi_{n,\omega})$ is defined by changing the value of it on a zero $m_{U(\omega)}$-measure set.

For $\delta>0$ sufficiently small, we cover $f_\omega^nU(\omega)$ for any $n\in\mathbb{N}$ by balls $\{B_{f_\omega^nU(\omega)}(f_\omega^n(x),\delta)\}_{x\in U(\omega)}$. Since $U(\omega)$ is pre-compact, there exists a finite cover $\{B_{f_\omega^nU(\omega)}(f_\omega^n(x_i),\delta)\}_{i=1}^{k(\theta^n\omega)}$. This open cover produces a natural finite Borel partition of $f_\omega^nU(\omega)$, named $\{P_i(\theta^n\omega)\}_{i=1}^{k(\theta^n\omega)}$, in the following way:
\begin{equation*}
 P_1(\theta^n\omega):= (f_\omega^nU(\omega))\cap B_{f_\omega^nU(\omega)}(f_\omega^n(x_i),\delta) ,
\end{equation*}and
\begin{equation*}
P_i(\theta^n\omega):=(f_\omega^nU(\omega))\cap B_{f_\omega^nU(\omega)}(f_\omega^n(x_i),\delta)\backslash (\cup_{j=1}^{i-1}B_{f_\omega^nU(\omega)}(f_\omega^n(x_j),\delta))
\end{equation*}for $i=2,...,k(\theta^n\omega)$. Without loss of generality, we assume $P_i(\theta^n\omega)\not=\emptyset$ for all $i$. This Borel partition satisfies that
\begin{equation}\label{eq measure 0}
  P_i(\theta^n\omega)\subset B_{f_\omega^nU(\omega)}(f_\omega^n(x_i),\delta),\mbox{ and }m_{f_\omega^nU(\omega)}(\cup_{i}\partial P_i(\theta^n\omega))=0.
\end{equation}Denote $\Psi_n^{1,i}$ (resp. $\Psi_n^{2,i}$)$: E^u(f_\omega^n(x_i),\theta^n\omega)(\mathcal{P}\delta)\to E^s(f_\omega^n(x_i),\theta^n\omega)$ to be the $C^2$ mapping such that $\exp_{x_i}(graph(\Psi_n^{1,i}))$ represents the component of $f_\omega^nU(\omega)$  (resp. $f_\omega^nV(\omega)$) passing through $f_\omega^n(x_i)$ (resp. $f_\omega^n(\psi_\omega(x_i))$). We define $\tilde{\pi}_{\theta^n\omega}$ restricted on $\exp_{x_i}^{-1}(P_i(\theta^n\omega))$ by
\begin{equation*}
  \tilde{\pi}_{\theta^n\omega}(e,\Psi_n^{1,i}(e))=(e,\Psi_{n}^{2,i}(e))\mbox{ for any }(e,\Psi_{n}^{1,i}(e))\in \exp_{x_i}^{-1}(P_i(\theta^n\omega)),
\end{equation*}i.e. the holonomy map induced by the flat foliation $\{e+E^s(f_\omega^n(x_i),\theta^n\omega):\ e\in E^u(f_\omega^n(x_i),\theta^n\omega)(\mathcal{P}\delta)\}$. Then let
\begin{equation*}
  \pi_{\theta^n\omega}(y)=\exp_{x_i}\circ \tilde{\pi}_{\theta^n\omega}\circ \exp_{x_i}^{-1}(y),\mbox{ for }y\in P_i(\theta^n\omega).
\end{equation*}
This mapping $\pi_{\theta^n\omega}$ is naturally absolutely continuous since it is absolutely continuous on each set in the partition. Moreover, $\pi_{\theta^n\omega}$ is $C^2$ for $x\in$interior$(P_i(\theta^n\omega))$ for all $i$. Let $\psi_{n,\omega}$ be defined as \eqref{def psinomega} by using this $\pi_{\theta^n\omega}$.

Note that $f_\omega^nU(\omega)$ is $C^1$-close to $f_\omega^n V(\omega)$ as $n\to\infty$. By the local stable manifolds theorem, for any $\epsilon>0$, there exists $N_\epsilon(\omega)$ such that
\begin{equation*}
  d_{f_\omega^nV(\omega)}(\bar{\psi}_{\theta^n\omega}(x),\pi_{\theta^n\omega}(x))<\epsilon\mbox{ for }x\in f_\omega^nU(\omega),
\end{equation*}where $\bar{\psi}_{\theta^n\omega}=f_\omega^n\circ\psi_\omega\circ f_{\theta^n\omega}^{-n}$. Therefore, for any $x\in U(\omega)$, $n\geq N_\epsilon(\omega)$, we have
\begin{equation*}
  d_{V(\omega)}(\psi_\omega(x),\psi_{n,\omega}(x))= d_{V(\omega)}(f_{\theta^n\omega}^{-n}\circ\bar{\psi}_{\theta^n\omega}(f_\omega^n(x)) , f_{\theta^n\omega}^{-n}\circ\circ \pi_{\theta^n\omega}(f_\omega^n(x)))\overset{\eqref{dxfn v}}\leq C_4(\omega)^{-1}e^{-\lambda n}\epsilon,
\end{equation*}which goes to $0$ uniformly for all $x\in U(\omega)$ as $n\to\infty$.

It is left to show that there exists a version of $Jac(\psi_{n,\omega})(x)$ converges uniformly for $x\in U(\omega)$. By \eqref{def psinomega}, for $m_{U(\omega)}$-a.s. $x\in U(\omega)$, we have
\begin{equation*}
  Jac(\psi_{n,\omega})(x)=H_\omega^n(x,\psi_{n,\omega}(x), T_xU(\omega),T_{\psi_{n,\omega}(x)}V(\omega))\cdot Jac(\pi_{\theta^n\omega})(f_\omega^n(x)).
\end{equation*}Note that \eqref{eq measure 0} implies that $m_{U(\omega)}(f_{\theta^n\omega}^{-n}(\cup_{i}\partial P_i(\theta^n\omega)))=0$. For $x\in f_{\theta^n\omega}^{-n}(\cup_{i}\partial P_i(\theta^n\omega))$, we set $Jac(\pi_{\theta^n\omega})(f_\omega^n(x))=1$. For $x\in U(\omega)$ such that $f_\omega^n(x)\in$interior$(P_i)(\theta^n\omega)$, $\pi_{\theta^n\omega}$ is $C^2$ on $f_\omega^nx$. By the definition of $\tilde{\pi}_{\theta^n\omega}$, we have
\begin{equation*}
  Jac(\tilde{\pi}_{\theta^n})(\exp_{x_i}^{-1}(f_\omega^n(x)))=\frac{Jac(\tilde{\Psi}_n^{2,i})(e)}{Jac(\tilde{\Psi}_n^{1,i})(e)}\mbox{ for }e=(\tilde{\Psi}_n^{1,i})^{-1}(\exp_{x_i}^{-1}(f_\omega^n(x))),
\end{equation*}where $\tilde{\Psi}_n^{\tau,i}(e)=(e,\Psi_n^{\tau,i}(e))$ for $\tau=1,2$. According to the proof of \eqref{eq defCstar}, for any $\eta>0$, there exists $N_\eta(\omega)\in\mathbb{N}$ such that for any $n>N_\eta(\omega)$, one has
\begin{equation*}
  Jac(\tilde{\pi}_{\theta^n})(f_\omega^n(x))=\frac{Jac(\tilde{\Psi}_n^{2,i})(e)}{Jac(\tilde{\Psi}_n^{1,i})(e)}\in (e^{-C^*\eta},e^{C^*\eta})
\end{equation*} for some constant $C^*$. Moreover, we have
\begin{equation*}
  \frac{|\det D_{\tilde{\pi}_{\theta^n}(\exp_{x_i}^{-1}(f_\omega^n(x)))}\exp_{x_i}|}{|\det D_{\exp_{x_i}^{-1}(f_\omega^n(x))}\exp_{x_i}|} \to 1 \mbox{ uniformly},\mbox{ as }n\to\infty
\end{equation*}since $d(f_\omega^n(x),\pi_{\theta^n\omega}(f_\omega^n(x)))\to 0$ uniformly. Therefore, there exists a version of $Jac(\pi_{\theta^n})(f_\omega^n(x))$ converges to $1$ uniformly as $n\to\infty$.

 Finally, we show that
 \begin{equation}\label{eq ratios}
   \frac{H_\omega^n(x,\psi_{n,\omega}(x), T_xU(\omega),T_{\psi_{n,\omega}(x)}V(\omega))}{H_\omega^n(x,\psi_{\omega}(x), T_xU(\omega),T_{\psi_{\omega}(x)}V(\omega))}=\frac{|\det D_{\psi_\omega(x)}f_\omega^n|_{T_{\psi_\omega(x)}V(\omega)}|}{|\det D_{\psi_{n,\omega}(x)}f_\omega^n|_{T_{\psi_{n,\omega}(x)}V(\omega)}|}\to 1
 \end{equation}uniformly as $n\to\infty$. For any $\epsilon>0$, we pick $N_1(\omega)\in \mathbb{N}$ such that for any $n\geq N$, $$\sum_{n\geq N_1(\omega)}C_{12}(\omega)e^{-\lambda n}\leq \epsilon,$$ where $C_{12}(\omega)$ is defined in \eqref{def C12}. Then we pick $\eta>0$ sufficiently small such that if $y,z\in U(\omega)$ satisfying $\sup_{0\leq i\leq N_1(\omega)-1}d(f_\omega^i(y),f_\omega^i(z))<\eta$,
 \begin{equation*}
\frac{|\det D_{y}f_\omega^{N_1(\omega)}|_{T_{y}V(\omega)}|}{|\det D_{z}f_\omega^{N_1(\omega)}|_{T_{z}V(\omega)}|}\in (e^{-\epsilon},e^{\epsilon}),
 \end{equation*}and moreover,
 \begin{equation*}
   C_{12}(\omega)\sum_{i=0}^{\infty}e^{-\lambda i\nu_1}\eta\leq \epsilon.
 \end{equation*}By the constriction of $\psi_{n,\omega}$, we pick $N_2(\omega)>N_1(\omega)$ such that for any $n\geq \mathbb{N}_2(\omega)$, $$\sup_{0\leq i\leq n-1}d(f_\omega^i(\psi_\omega(x)),f_\omega^i(\psi_{n,\omega}(x)))<\eta.$$ Now, according to the proof of \eqref{def of C17}, for any $n\geq N_2(\omega)$, we arrive
 \begin{align*}
   &\frac{|\det D_{\psi_\omega(x)}f_\omega^n|_{T_{\psi_\omega(x)}V(\omega)}|}{|\det D_{\psi_{n,\omega}(x)}f_\omega^n|_{T_{\psi_{n,\omega}(x)}V(\omega)}|}\\
   =&
   \frac{|\det D_{\psi_\omega(x)}f_\omega^{N_1(\omega)}|_{T_{\psi_\omega(x)}V(\omega)}|}{|\det D_{\psi_{n,\omega}(x)}f_\omega^{N_1(\omega)}|_{T_{\psi_{n,\omega}(x)}V(\omega)}|}\cdot  \frac{|\det D_{f_\omega^{N_1(\omega)}(\psi_\omega(x))}f_{\theta^{N_1(\omega)}\omega}^{n-N_1(\omega)}|_{D_{\psi_\omega(x)}f_\omega^{N_1(\omega)}T_{\psi_\omega(x)}V(\omega)}|}{|\det D_{f_\omega^{N_1(\omega)}(\psi_{n,\omega}(x))}f_{\theta^{N_1(\omega)}\omega}^{n-N_1(\omega)}|_{D_{\psi_{n,\omega}(x)}f_{\theta^{N_1(\omega)}\omega}^{n-N_1(\omega)}T_{x}V(\omega)}|}\in(e^{-3\epsilon},e^{3\epsilon}).
 \end{align*}So \eqref{eq ratios} holds.
Therefore, by Lemma \ref{mane lemma} and changing value on zero $m_{U(\omega)}$-measure set, we have
\begin{equation}\label{expression of jacobian}
  Jac(\psi_\omega)(x)=H_\omega(x,\psi_\omega(x), T_xU(\omega),T_{\psi_\omega(x)}V(\omega))\mbox{ for }x\in U(\omega).
\end{equation}
The proof of Proposition \ref{Theorem fiberwise absolutely continuous} is complete.
\end{proof}

\subsection{H\"older continuity of the stable and unstable foliations on each fiber}\label{subsection 3.5}
In this subsection, we prove the H\"older continuity of the holonomy map between two local stable leaves (or unstable leaves). This result is known in deterministic hyperbolic systems (see, e.g., \cite{PughHolderfoliation}), but we didn't find any reference to this result in RDS. We supply a proof in our settings, and the proof follows the idea given in Page 762 in \cite{KHbook}. Roughly speaking, the idea is that the action of the graph transform preserves a H\"older condition.

\begin{proposition}\label{Holder continuity of holonomy map}
  Suppose $\phi$ is $C^2$ Anosov on fibers, let $\varrho\in (0,1)$ satisfy
  \begin{equation*}
    \sup_{(p,\omega)\in M\times\Omega}\|D_pf_\omega|_{E^s(p,\omega)}\|t_{(p,\omega)}^{-\varrho}<1,\ \sup_{(p,\omega)\in M\times\Omega}\|D_pf_\omega^{-1}|_{E^u(p,\omega)}\|s_{(p,\omega)}^{-\varrho}<1
  \end{equation*}where
  \begin{align*}
   t_{(p,\omega)}&:=\inf\left\{\frac{d(f_\omega p,f_\omega q)}{d(p,q)}:\ q\in M,\ d(p,q)<\epsilon_0\right\}>0,\\
   s_{(p,\omega)}&:=\inf\left\{\frac{d(f_\omega^{-1} (p),f_\omega^{-1}(q))}{d(p,q)}: q\in M,\ d(p,q)<\epsilon_0\right\}>0,
  \end{align*}and $\epsilon_0$ is the size of local stable and unstable manifolds. Then there exists a constant $\delta_0>0$ and $H=H(\delta_0,\varrho)$ such that for any $(q,\omega)\in M\times\Omega$,
\begin{align}
  \sup\left\{\frac{\sup_{x\in E^u(p,\omega)(\delta_0)}\|h_{(p,\omega)}^u(x)-\tilde{h}_{(q,\omega)}^u(x)\|}{d(p,q)^\varrho}:\ q\in M,\ d(p,q)<\delta_0\right\}&\leq H<\infty,\label{holder hu}\\
   \sup\left\{\frac{\sup_{x\in E^u(p,\omega)(\delta_0)}\|h_{(p,\omega)}^s(x)-\tilde{h}_{(q,\omega)}^s(x)\|}{d(p,q)^\varrho}:\ q\in M,\ d(p,q)<\delta_0\right\}&\leq H<\infty,\label{holder hs}
\end{align}where $h_{(p,\omega)}^u,\tilde{h}_{(q,\omega)}^u:E^u(p,\omega)(\delta_0)\rightarrow E^s(p,\omega)$ and $Exp_{p}(graph(h_{(p,\omega)}^u))$, $ Exp_p(graph(\tilde{h}_{(q,\omega)}^u))$ represent the local unstable manifolds passing through $p,\ q$ respectively, $h_{(p,\omega)}^s,\tilde{h}_{(q,\omega)}^s:E^s(p,\omega)(\delta_0)\rightarrow E^u(p,\omega)$ and $Exp_{p}(graph(h_{(p,\omega)}^s))$, $ Exp_p(graph(\tilde{h}_{(q,\omega)}^s))$ represent the local stable manifolds passing through $p,\ q$ respectively. Furthermore, the local product structure is H\"older continuous, i.e., there exists a constant $H^\prime=H^\prime(\delta_0,\varrho)$ such that for any $(x,\omega)\in M\times\Omega,$ $y\in W^s_{\delta_0}(x,\omega)$, any $z\in M$ such that $W^s_\epsilon(z,\omega)\cap W^u_\epsilon(x,\omega)\not=\emptyset$, $W^s_\epsilon(z,\omega)\cap W^u_\epsilon(y,\omega)\not=\emptyset$, we have
\begin{equation*}
  d(W^s_\epsilon(z,\omega)\cap W^u_\epsilon(x,\omega),W^s_\epsilon(z,\omega)\cap W^u_\epsilon(y,\omega))\leq H^\prime d(x,y)^\varrho;
\end{equation*}for any $(x,\omega)\in M\times\Omega,$ $y\in W^u_{\delta_0}(x,\omega)$, any $z\in M$ such that $W^u_\epsilon(z,\omega)\cap W^s_\epsilon(x,\omega)\not=\emptyset$, $W^u_\epsilon(z,\omega)\cap W^s_\epsilon(y,\omega)\not=\emptyset$, we have
\begin{equation*}
  d(W^u_\epsilon(z,\omega)\cap W^s_\epsilon(x,\omega),W^u_\epsilon(z,\omega)\cap W^s_\epsilon(y,\omega))\leq H^\prime d(x,y)^\varrho.
\end{equation*}
\end{proposition}

\begin{proof}
We first prove $(\ref{holder hu})$.
Recall that for each point $p\in M$, there exist a neighborhood $N_p\subset M$ and constant $\epsilon$ such that the exponential map $Exp_p:B_\epsilon(0)\subset Tp M\rightarrow M$ is a $C^\infty-$diffeomorphism and $N_p\subset Exp_p(B_\epsilon(0))$. Now for all $p\in M$ and $\omega\in\Omega$, consider any continuous function $g_{(p,\omega)}:E^u(p,\omega)(\epsilon)\rightarrow E^s(p,\omega)$ with $g_{(p,\omega)}(0)=0$, where $E^u(p,\omega)(\epsilon)$ is the $\epsilon-$disk in $E^u(p,\omega)$ centered at the origin. Define the special norm by
\begin{equation*}
  \|g_{(p,\omega)}\|_*=\sup\left\{\frac{\|g_{(p,\omega)}(x)\|}{\|x\|}:\ x\in E^u(p,\omega)(\epsilon),\ x\not=0\right\}.
\end{equation*}Define
\begin{equation*}
  G_{(p,\omega)}^*:=\{g_{(p,\omega)}:E^u(p,\omega)(\epsilon)\rightarrow E^s(p,\omega)| \ g_{(p,\omega)}(0)=0\mbox{ and }\|g_{(p,\omega)}\|_*<\infty\},
\end{equation*}and
\begin{equation}\label{def Gpomega}
  G_{(p,\omega)}:=\left\{g_{(p,\omega)}\in G_{(p,\omega)}^*:\ Lip(g_{(p,\omega)})\leq \frac{e^{-2\lambda}+1}{2}\right\}.
\end{equation}
\begin{sublemma}
  $G_{(p,\omega)}^*$ equipped with $\|\cdot\|_*$ is a Banach space and $G_{(p,\omega)}$ is a closed subset.
\end{sublemma}
  The above sublemma is a corollary of \cite[Lemma III.3]{MS}. $\{G_{(p,\omega)}\}_{(p,\omega)\in M\times\Omega}$ gives a bundle $G$ on $M\times\Omega$ with fiber $G_{(p,\omega)}$ on $(p,\omega)\in M\times\Omega$.

   Now we define $f_{(p,\omega)}: T_pM(\epsilon)\rightarrow T_{f_\omega p}M$ by the local representation of $f_\omega$ with respect to exponential maps $Exp_p$ and $Exp_{f_\omega p}$, i.e.,
\begin{equation*}
  f_{(p,\omega)}(v)=Exp_{f_{\omega}p}^{-1}\circ f_\omega \circ Exp_p(v),\ \forall v\in T_pM(\epsilon).
\end{equation*}For any $(p,\omega)\in M\times \Omega$, we define a bundle map $\phi^*_{(p,\omega)}$ on $G_{(p,\omega)}$ satisfying for any $g_{(p,\omega)}\in G_{(p,\omega)}$,
\begin{equation*}
  graph(\phi^*_{(p,\omega)}g_{(p,\omega)})=f_{(p,\omega)}(graph(g_{(p,\omega)}))\cap (E^u(f_\omega p,\theta\omega)(\epsilon)\oplus E^s(f_\omega p,\theta\omega)).
\end{equation*}Note that the definition of $G_{(p,\omega)}$ is relying on $\epsilon$. The next sublemma says that there exists $\epsilon>0$ such that $\phi^*_{(p,\omega)}$ maps $G_{(p,\omega)}$ into $G_{(f_\omega p,\theta\omega)}$ and this map is a contraction.
\begin{sublemma}\label{Lemma phi star contraction}
There exists a $\epsilon_0>0$ such that for any $\epsilon\in(0,\epsilon_0)$, $(p,\omega)\in M\times \Omega$, one has $\phi^*_{(p,\omega)}g_{(p,\omega)}\in G_{(f_\omega p,\theta\omega)}$, and it is a fiber contraction, i.e., for any $g_{(p,\omega)},\ g_{(p,\omega)}^\prime\in G_{(p,\omega)}$, we have
  \begin{equation*}
    \|\phi^*_{(p,\omega)}g_{(p,\omega)}-\phi^*_{(p,\omega)}g_{(p,\omega)}^\prime\|_*\leq \frac{e^{-2\lambda_0}+1}{2} \|g_{(p,\omega)}-g_{(p,\omega)}^\prime\|_*,
  \end{equation*}
\end{sublemma}
\begin{proof}[Proof of Sublemma \ref{Lemma phi star contraction}]
  Fix any $\epsilon^\prime>0$ such that
  \begin{equation}\label{eq pick epsilon prime}
  \frac{e^{-\lambda_0}+2\mathcal{P}\epsilon^\prime}{e^{\lambda_0}-2\mathcal{P}\epsilon^\prime}\leq \frac{e^{-2\lambda_0}+1}{2}.
  \end{equation}
By compactness of $\Omega$ and $M$ and the continuity of $f_\omega$ on $\omega$, for the above $\epsilon^\prime$, we can pick a $\epsilon_0 >0$ sufficiently small such that  for any $\epsilon<\epsilon_0$,
\begin{equation*}
  Lip((f_{(p,\omega)}-D_pf_\omega)|_{T_pM_(\epsilon)})<\epsilon^\prime.
\end{equation*}
  Pick any $g\in G_{(p,\omega)}$, let $f_{(p,\omega)}(x,g(x))$ have decomposition
  \begin{equation*}
    f_{(p,\omega)}(x,g(x))=(f_{(p,\omega),1}(x,g(x)),f_{(p,\omega),2}(x,g(x)))
  \end{equation*} with respect to $T_{f_\omega p}M=E^u(f_\omega p,\theta\omega)\oplus E^s(f_\omega p,\theta\omega)$, and we denote $h_{(p,\omega)}:=f_{(p,\omega),1}\circ(id,g):E^u(p,\omega)(\epsilon)\to E^u(f_\omega p,\theta\omega)$.
Note that
\begin{align*}
  f_{(p,\omega),1}(x,g(x)) & =P(E^u(f_\omega p,\theta\omega))\circ f_{(p,\omega)}\circ (id,g)(x),\\
  D_pf_\omega|_{E^u(p,\omega)}(x)&=P(E^u(f_\omega p,\theta\omega))\circ D_pf_\omega\circ (id,g)(x),
\end{align*}where $P(E^u(f_\omega p,\theta\omega)):T_{f_\omega p}M\to E^u(f_\omega p,\theta\omega)$ is the projection with respect to the above decomposition. Then we have
\begin{align}
   &\ \ \ \ Lip(f_{(p,\omega),1}\circ (id,g)-D_pf_\omega|_{E^u(p,\omega)(\epsilon)})\nonumber\\
   &\leq \mathcal{P}\cdot Lip ((f_{(p,\omega)}-D_pf_\omega)|_{T_pM(\epsilon)})\cdot Lip(id,g)\nonumber\\
   &<\mathcal{P}\epsilon^\prime,\label{f1-dfu}
\end{align}where $\mathcal{P}$ is the constant in \eqref{def of P}. By the Lipschitz Inverse function theorem (see, e.g, \cite[Theorem I.2]{MS}), $h_{(p,\omega)}$ is a homeomorphism to its image and moreover,
\begin{align}
  Lip(h_{(p,\omega)}^{-1})\leq & \frac{1}{\|D_pf_\omega|_{E^u(p,\omega)}^{-1}\|^{-1}-Lip(f_{(p,\omega),1}\circ (id,g)-D_pf_\omega|_{E^u(p,\omega)(\epsilon)})}\nonumber \\
 \overset{\eqref{f1-dfu}}< &\frac{1}{m(D_pf_\omega|_{E^u(p,\omega)})-\mathcal{P}\epsilon^\prime}\label{Lip of h inverse},
\end{align}where $m(D_pf_\omega|_{E^u(p,\omega)})=\|D_pf_\omega|_{E^u(p,\omega)}^{-1}\|^{-1}$ denotes the co-norm.
Then for any $g\in G_{(p,\omega)}$, we have
\begin{equation}\label{expression of phi g}
  (\phi_{(p,\omega)}^*g)(y)=f_{(p,\omega),2}(h_{(p,\omega)}^{-1}(y),g(h_{(p,\omega)}^{-1}(y)))\mbox{ for }y\in E^u(f_\omega p,\theta\omega)(\epsilon).
\end{equation}

Note that for $x\in E^u(p,\omega)(\epsilon)$,
\begin{align*}
  f_{(p,\omega),2}(x,g(x)) &=P(E^s(f_\omega p,\theta\omega))\circ f_{(p,\omega)}\circ (id,g)(x),\\
  D_pf_\omega|_{E^s(p,\omega)}(g(x))&=P(E^s(f_\omega p,\theta\omega))\circ D_pf_\omega\circ (id,g)(x),
\end{align*}then we have
\begin{align}
  Lip(f_{(p,\omega),2}\circ (id,g)) &\leq Lip(f_{(p,\omega),2}\circ(id,g)-D_pf_\omega|_{E^s(p,\omega)}\circ g)+Lip(D_pf_\omega|_{E^s(p,\omega)}\circ g)\nonumber \\
  &\leq \|P(E^s(f_\omega p,\theta\omega))\|\cdot Lip(f_{(p,\omega)}-D_pf_\omega)\cdot Lip((id,g))+\|D_pf_\omega|_{E^s(p,\omega)}\|\nonumber\\
  &<\mathcal{P}\epsilon^\prime+\|D_pf_\omega|_{E^s(p,\omega)}\|.\label{Lip f idg}
\end{align}Combining $(\ref{Lip of h   inverse})$, $(\ref{expression of phi g})$and $(\ref{Lip f idg})$,
\begin{align*}
  Lip(\phi^*_{(p,\omega)}g)&\leq Lip(f_{(p,\omega),2}\circ(id,g))\cdot Lip(h_{(p,\omega)}^{-1})\leq \frac{\|D_pf_\omega|_{E^s(p,\omega)}\|+\mathcal{P}\epsilon^\prime}{m(D_pf_\omega|_{E^u(p,\omega)})-\mathcal{P}\epsilon^\prime}\\
  &\leq \frac{e^{-\lambda_0}+\mathcal{P}\epsilon^\prime}{e^{\lambda_0}-\mathcal{P}\epsilon^\prime}\overset{\eqref{eq pick epsilon prime}}< \frac{e^{-2\lambda_0}+1}{2}.
\end{align*}Obviously that $(\phi^*_{(p,\omega)}g)(0)=0$, so we have shown that $\phi^*_{(p,\omega)}$ maps $G_{(p,\omega)}$ to $G_{(f_\omega p,\theta\omega)}$.

Next, we show that $\phi^*_{(p,\omega)}$ is fiber-contraction. We note that
\begin{equation*}
  \frac{\|D_pf_\omega|_{E^s(p,\omega)}\|+2\mathcal{P}\epsilon^\prime}{m(D_pf_\omega|_{E^u(p,\omega)})-2\mathcal{P}\epsilon^\prime}\leq \frac{e^{-\lambda_0}+2\mathcal{P}\epsilon^\prime}{e^{\lambda_0}-2\mathcal{P}\epsilon^\prime}
 \overset{\eqref{eq pick epsilon prime}}\leq \frac{e^{-2\lambda_0}+1}{2}.
\end{equation*} Therefore, it is sufficient to show that for any $g,g^\prime\in G_{(p,\omega)}$, for all $x\in E^u(p,\omega)(\epsilon)$,
\begin{equation}\label{phi* fiber contraction}
 \frac{\|(\phi^*_{(p,\omega)}g)(f_{(p,\omega),1}(x,g(x)))-(\phi^*_{(p,\omega)}g^\prime)(f_{(p,\omega),1}(x,g(x)))\|}{\|f_{(p,\omega),1}(x,g(x))\|}
 \leq \frac{\|D_pf_\omega|_{E^s(p,\omega)}\|+2\mathcal{P}\epsilon^\prime}{m(D_pf_\omega|_{E^u(p,\omega)})-2\mathcal{P}\epsilon^\prime}\cdot \|g-g^\prime\|_*,
\end{equation} since $h_{(p,\omega)}=f_{(p,\omega),1}\circ (id,g)$ is a homeomorphism.

Notice that
\begin{align}
 &\ \ \ \ \|f_{(p,\omega),2}(x,g(x))-f_{(p,\omega),2}(x,g^\prime(x))\|\nonumber\\
 &\leq \left\|(f_{(p,\omega),2}-P(E^s(f_\omega p,\theta\omega))D_pf_\omega)(x,g(x))-(f_{(p,\omega),2}-P(E^s(f_\omega p,\theta\omega))D_pf_\omega)(x,g^\prime(x))\right\|\nonumber\\
 &\ \ \ \ \ \ \ \ \ \ \ \ \ +\|P(E^s(f_\omega p,\theta\omega))D_pf_\omega(x,g^\prime(x))-P(E^s(f_\omega p,\theta\omega))D_pf_\omega(x,g(x))\|\nonumber\\
 &\leq \left\{Lip(f_{(p,\omega),2}-P(E^s(f_\omega p,\theta\omega))D_pf_\omega)+\|D_pf_\omega|_{E^s(p,\omega)}\|\right\}\cdot\|g(x)-g^\prime(x)\|\nonumber\\
 &=\left\{Lip(P(E^s(f_\omega p,\theta\omega))f_{(p,\omega)}-P(E^s(f_\omega p,\theta\omega))D_pf_\omega)+\|D_pf_\omega|_{E^s(p,\omega)}\|\right\}\cdot\|g(x)-g^\prime(x)\|\nonumber\\
 &< \left(\mathcal{P}\epsilon^\prime+\|D_pf_\omega|_{E^s(p,\omega)}\|\right)\|g(x)-g^\prime(x)\|\label{f2xg-f2xgprime},
\end{align}and
\begin{align}
   &\ \ \ \ \|f_{(p,\omega),1}(x,g(x))-f_{(p,\omega),1}(x,g^\prime(x))\|\nonumber\\
   &\leq \|(f_{(p,\omega),1}-P(E^u(f_\omega p,\theta\omega))D_pf_\omega)(x,g(x))-(f_{(p,\omega),1}-P(E^u(f_\omega p,\theta\omega))D_pf_\omega)(x,g^\prime(x))\|\nonumber\\
   &\ \ \ \ \ \ \ \ \ \ \ \ +\|P(E^u(f_\omega p,\theta\omega)D_pf_\omega)(x,g(x))-P(E^u(f_\omega p,\theta\omega)D_pf_\omega)(x,g^\prime(x))\|\nonumber\\
   &\leq Lip(f_{(p,\omega),1}-P(E^u(f_\omega p,\theta\omega))D_pf_\omega)\|g(x)-g^\prime(x)\|\nonumber\\
   &= Lip(P(E^u(f_\omega p,\theta\omega))f_{(p,\omega)}-P(E^u(f_\omega p,\theta\omega))D_pf_\omega)\|g(x)-g^\prime(x)\|\nonumber\\
   &<\mathcal{P}\epsilon^\prime\|g(x)-g^\prime(x)\|\label{f1xg-f1xgprime}.
\end{align}Then $(\ref{f1xg-f1xgprime})$ and $(\ref{f2xg-f2xgprime})$ imply that
\begin{align}
   &\ \ \ \ \|(\phi^*_{(p,\omega)}g)(f_{(p,\omega),1}(x,g(x)))-(\phi^*_{(p,\omega)}g^\prime)(f_{(p,\omega),1}(x,g(x)))\|\\ &=\|f_{(p,\omega),2}(x,g(x))-(\phi^*_{(p,\omega)}g^\prime)(f_{(p,\omega),1}(x,g(x)))\|\nonumber\\
   &\leq \|f_{(p,\omega),2}(x,g(x))-f_{(p,\omega),2}(x,g^\prime(x))\|+\|f_{(p,\omega),2}(x,g^\prime(x))-(\phi^*_{(p,\omega)}g^\prime)(f_{(p,\omega),1}(x,g(x)))\|\nonumber\\
   &\overset{\eqref{f2xg-f2xgprime}}\leq  \left(\mathcal{P}\epsilon^\prime+\|D_pf_\omega|_{E^s(p,\omega)}\|\right)\|g(x)-g^\prime(x)\|\nonumber\\
   &\ \ \ \ \ \ \ +\|(\phi^*_{(p,\omega)}g^\prime)(f_{(p,\omega),1}(x,g^\prime(x)))-(\phi^*_{(p,\omega)}g^\prime)(f_{(p,\omega),1}(x,g(x)))\|\nonumber\\
   &\leq \left(\mathcal{P}\epsilon^\prime+\|D_pf_\omega|_{E^s(p,\omega)}\|\right)\|g(x)-g^\prime(x)\|+Lip(\phi^*_{(p,\omega)}g^\prime)\|f_{(p,\omega),1}(x,g(x))-f_{(p,\omega),1}(x,g^\prime(x))\|\nonumber\\
   &\overset{\eqref{f1xg-f1xgprime}}<\left(2\mathcal{P}\epsilon^\prime+\|D_pf_\omega|_{E^s(p,\omega)}\|\right)\|g(x)-g^\prime(x)\|\label{nunomcontra}.
\end{align}
On the other hand,
\begin{align}
  \|f_{(p,\omega),1}(x,g(x))\| &=\|(f_{(p,\omega),1}- P(E^u(f_\omega p,\theta\omega))D_pf_\omega)(x,g(x))+P(E^u(f_\omega p,\theta\omega))D_pf_\omega(x,g(x))\|\nonumber\\
  &\overset{\eqref{f1-dfu}}\geq m(D_pf_\omega|_{E^u(p,\omega)})\|x\|-\mathcal{P}\epsilon^\prime\|(x,g(x))\|\nonumber\\
  &\geq (m(D_pf_\omega|_{E^u(p,\omega)})-2\mathcal{P}\epsilon^\prime)\|x\|\label{denomenat},
\end{align}where we use $\|g(x)\|=\|g(x)-g(0)\|\leq \|x\|$ in the last step.
Now $(\ref{phi* fiber contraction})$ follows by \eqref{nunomcontra}, \eqref{denomenat} and \eqref{eq pick epsilon prime}.
\end{proof}
By Sublemma \ref{Lemma phi star contraction}, $\phi^*:G\to G$ defined by
\begin{equation*}
  \{(p,\omega,g_{(p,\omega)}):\ (p,\omega)\in M\times \Omega,\ g_{(p,\omega)}\in G_{(p,\omega)}\}\mapsto \{(f_\omega(p),\theta\omega,\phi^*_{(p,\omega)}g_{(p,\omega)}):\ (p,\omega)\in M\times \Omega\}
\end{equation*}is a fiber contraction. Then there exists a unique section $g^*:M\times \Omega\to G$ which is invariant under $\phi^*$ in the sense that for all $(p,\omega)\in M\times \Omega$,
\begin{equation*}
  g^*_{(f_\omega(p),\theta\omega)}=\phi^*_{(p,\omega)}g^*_{(p,\omega)}.
\end{equation*}In fact, we can consider the space $\Sigma$ of all sections $g:M\times \Omega\to G$. Note that
\begin{equation*}
  \sup_{(p,\omega)\in M\times \Omega}\{\|g_{(p,\omega)}-g^\prime_{(p,\omega)}\|_*:\ g_{(p,\omega)},g_{(p,\omega)}^\prime\in G_{(p,\omega)}\}\overset{\eqref{def Gpomega}}\leq e^{-2\lambda}+1<\infty.
\end{equation*}Therefore, $\tilde{d}(g,g^\prime)=\sup\{\|g_{(p,\omega)}-g^\prime_{(p,\omega)}\|_*:\ (p,\omega)\in M\times \Omega\}$ for any $g,g^\prime\in\Sigma$ gives a metric on $\Sigma$. Moreover, $\Sigma$ is complete since $G_{(p,\omega)}$ is closed. By sublemma \ref{Lemma phi star contraction}, the action of $\phi^*$ on $\Sigma$ is clearly a contraction mapping. The unique fixed point of $\phi^*$ is the unique invariant section.
By virtue of contraction mapping theorem, one also can obtain the unique invariant section $g^*$ by iterating any section $g\in G$, i.e.,
\begin{equation}\label{fixed point}
  g^*_{(p,\omega)}=\lim_{n\to\infty}((\phi^*)^ng)_{(p,\omega)}=\lim_{n\rightarrow\infty}\phi^*_{(f_\omega^{-1}p,\theta^{-1}\omega)}\cdots\phi^*_{(f_\omega^{-n}p,\theta^{-n}\omega)}g_{(f_{\omega}^{-n}p,\theta^{-n}\omega)},
\end{equation}for any $(p,\omega)\in M\times \Omega$.

By the stable and unstable manifolds theorem, we know that the local unstable manifold passing through $p$ on the fiber $M\times \{\omega\}$ is exactly $Exp_p(graph(g_{(p,\omega)}^*))$.
Next, we will show that the bundle map $\phi^*$ preserves a local H\"older property for an appropriate H\"older exponent.

Since $E^u(x,\omega)$ and $E^s(x,\omega)$ are uniformly continuous depending on $(x,\omega)\in M\times \Omega$, with the help of local coordinate charts, we may pick a sufficiently small $\delta_0\in(0,\frac{\epsilon}{2\mathcal{P}})$ such that whenever $d(p,q)<\delta_0$, for any $\omega\in\Omega$, $g_{(q,\omega)}:E^u(q,\omega)(\epsilon)\rightarrow E^s(q,\omega)$ with $Lip(g_{(q,\omega)})<\frac{e^{-2\lambda_0}+1}{2}$ can be viewed as a Lipschitz mapping from $E^u(p,\omega)(\delta_0)$ to $E^s(p,\omega)$ with Lipschtiz constant less than $1$ with respect to the coordinate $E^u(p,\omega)\oplus E^s(p,\omega)$. We use notation $(g_{(q,\omega)})^{(p,\omega)}$ to represent $g_{(q,\omega)}$ in the coordinate $E^u(p,\omega)\oplus E^s(p,\omega)$.

From now on, we fix this $\delta_0$. We pick $N>0$ depending on $\delta_0$ such that
\begin{equation}\label{eq pick of N}
  e^{-N\lambda}<\frac{\delta_0}{2\sup_{(x,\omega)\in M\times\Omega}\|D_xf_\omega^{-1}\|}.
\end{equation}For any constant $K>0$ and $\varrho\in(0,1)$, we define
\begin{align*}
  G(\delta_0,\varrho,e^{-N\lambda},K) &:=\left\{g\in G:\ \sup_{x\in E^u(p,\omega)(\delta_0)}\|g_{(p,\omega)}(x)-(g_{(q,\omega)})^{(p,\omega)}(x)\|\leq K d(p,q)^\varrho,\right.\\
  &\ \ \ \ \ \ \ \ \ \ \ \ \left.\mbox{ whenever }e^{-N\lambda}< d(p,q)< \delta_0\right\}.
\end{align*}
\begin{sublemma}\label{lemma 3.4}
  There exists a constant $C=C(\delta_0)$ such that $G\subset G(\delta_0,\varrho,e^{-N\lambda},C(\delta_0)e^{N\lambda\varrho}).$
\end{sublemma}
\begin{proof}[Proof of Sublemma \ref{lemma 3.4}]
  Notice that both the Lipschitz constants of $g_{(p,\omega)}$ and $(g_{(q,\omega)})^{(p,\omega)}$ are less than $1$, and $d(p,q)<\delta_0$, hence there exists a constant $C=C(\delta_0)>0$ such that
  \begin{equation*}
    \sup_{x\in E^u(p,\omega)(\delta_0)}\|g_{(p,\omega)}(x)-(g_{(q,\omega)})^{(p,\omega)}(x)\|\leq C.
  \end{equation*}Notice that $d(p,q)>e^{-N\lambda}$, so we have
  \begin{align*}
     \sup_{x\in E^u(p,\omega)(\delta_0)}\|g_{(p,\omega)}(x)-(g_{(q,\omega)})^{(p,\omega)}(x)\|\leq C(\delta_0)\leq C(\delta_0)e^{N\lambda\varrho}d(p,q)^\varrho.
  \end{align*}
\end{proof}

   For any $g_{(q,\omega)}\in G_{(q,\omega)}$, the proof of Sublemma \ref{Lemma phi star contraction} implies that $\phi^*_{(q,\omega)}g_{(q,\omega)}\in G_{(f_\omega q,\theta\omega)}$ and  $Lip(\phi^*_{(q,\omega)}g_{(q,\omega)})\leq \frac{e^{-2\lambda}+1}{2}$. If $d(f_\omega p,f_\omega q)<\delta_0$, then $Lip((\phi^*_{(q,\omega)}g_{(q,\omega)})^{(f_\omega p,\theta\omega)})<1$ by the choice of $\delta_0.$
\begin{sublemma}\label{lemma keep Kalpha}
  Let $g\in G$, if $d(p,q)<\delta_0$, $d(f_\omega p,f_\omega q)<\delta_0$, and $\sup_{x\in E^u(p,\omega)(\delta_0)}\|g_{(p,\omega)}(x)-(g_{(q,\omega)})^{(p,\omega)}(x)\|\leq K d(p,q)^\varrho$ for some constant $K$, then
  \begin{equation}
    \sup_{x\in E^u(f_\omega p,\theta \omega)(\delta_0)}\left\|\left(\phi^*_{(p,\omega)}g_{(p,\omega)}\right)(x)-\left(\phi^*_{(q,\omega)}g_{(q,\omega)}\right)^{(f_\omega p,\theta\omega)}(x)\right\|\leq K d(f_\omega p,f_\omega q)^\varrho\label{keep Kalpha}
  \end{equation}provided $\varrho\in (0,1)$ satisfying
  \begin{equation}\label{DpfomegaEs W-varrho}
    \sup_{(p,\omega)\in M\times\Omega}\|D_pf_\omega|_{E^s(p,\omega)}\|t_{(p,\omega)}^{-\varrho}<1,
  \end{equation}where
  \begin{equation}\label{def of Wpomega}
   t_{(p,\omega)}:=\inf\left\{\frac{d(f_\omega p,f_\omega q)}{d(p,q)}:\ q\in M,\ d(p,q)<\epsilon\right\}>0.
  \end{equation}
\end{sublemma}
\begin{proof}[Proof of Sublemma \ref{lemma keep Kalpha}]
  We use the same notations as in the proof of Sublemma \ref{Lemma phi star contraction}. By $(\ref{DpfomegaEs W-varrho})$, we can pick a constant $\epsilon^\prime>0$ sufficiently small satisfying both \eqref{eq pick epsilon prime} and
  \begin{equation}\label{condition on Wpomeg}
      \sup_{(p,\omega)\in M\times\Omega}(2\mathcal{P}\epsilon^\prime+\|D_pf_\omega\|_{E^s(p,\omega)})t_{(p,\omega)}^{-\varrho}<1.
  \end{equation}

Notice that following fact:
\begin{equation*}\label{f p 2=phi q 1}
  f_{(p,\omega),2}(x,(g_{(q,\omega)})^{(p,\omega)}(x))=(\phi^*_{(q,\omega)}g_{(q,\omega)})^{(f_\omega p,\theta\omega)}(f_{(p,\omega),1}(x,(g_{(q,\omega)})^{(p,\omega)}(x))),
\end{equation*}
 then we have
  \begin{align*}
    &\ \ \ \ \|(\phi^*_{(p,\omega)}g_{(p,\omega)})(f_{(p,\omega),1}(x,g_{(p,\omega)}(x)))-(\phi^*_{(q,\omega)}g_{(q,\omega)})^{(f_\omega p,\theta\omega)}(f_{(p,\omega),1}(x,g_{(p,\omega)}(x)))\|\\
    &=\|f_{(p,\omega),2}(x,g_{(p,\omega)}(x))-(\phi^*_{(q,\omega)}g_{(q,\omega)})^{(f_\omega p,\theta\omega)}(f_{(p,\omega),1}(x,g_{(p,\omega)}(x)))\|\\
    &\leq \|f_{(p,\omega),2}(x,g_{(p,\omega)}(x))-f_{(p,\omega),2}(x,(g_{(q,\omega)})^{(p,\omega)}(x))\|\\
    &\ \ \ \ \ \ \ \ +\|(\phi^*_{(q,\omega)}g_{(q,\omega)})^{(f_\omega p,\theta\omega)}(f_{(p,\omega),1}(x,(g_{(q,\omega)})^{(p,\omega)}(x)))-(\phi^*_{(q,\omega)}g_{(q,\omega)})^{(f_\omega p,\theta\omega)}(f_{(p,\omega),1}(x,g_{(p,\omega)}(x)))\|\\
    &\overset{\eqref{f2xg-f2xgprime}}\leq(\mathcal{P}\epsilon^\prime+\|D_pf_\omega|_{E^s(p,\omega)}\|)\|g_{(p,\omega)}(x)-(g_{(q,\omega)})^{(p,\omega)}(x)\|\\
    &\ \ \ \ \ \ \ \ \ \ +Lip(\phi^*_{(q,\omega)}g_{(q,\omega)})^{(f_\omega p,\theta\omega)})\|f_{(p,\omega),1}(x,(g_{(q,\omega)})^{(p,\omega)}(x))-f_{(p,\omega),1}(x,g_{(p,\omega)}(x))\|\\
    &\overset{\eqref{f1xg-f1xgprime}}\leq(2\mathcal{P}\epsilon^\prime+\|D_pf_\omega|_{E^s(p,\omega)}\|)\sup_{x\in E^u(p,\omega)(\delta_0)}\|g_{(p,\omega)}(x)-(g_{(q,\omega)})^{(p,\omega)}(x)\|\\
    &\leq(2\mathcal{P}\epsilon^\prime+\|D_pf_\omega|_{E^s(p,\omega)}\|)K d(p,q)^\varrho\\
    &\leq(2\mathcal{P}\epsilon^\prime+\|D_pf_\omega|_{E^s(p,\omega)}\|)K t_{(p,\omega)}^{-\varrho}d(f_\omega p,f_\omega q)^\varrho\\
    &\overset{\eqref{condition on Wpomeg}}\leq Kd(f_\omega p,f_\omega q)^\varrho,
  \end{align*}where we note that in inequality \eqref{f2xg-f2xgprime} and \eqref{f1xg-f1xgprime}, $g^\prime$ can be replaced by $(g_{(q,\omega)})^{(p,\omega)}$. Notice that $h_{(p,\omega)}=f_{(p,\omega),1}\circ (id,g_{(p,\omega)})$ is a homeomorphism, hence we get $(\ref{keep Kalpha})$.
\end{proof}

Now consider
\begin{equation*}
  R_n(\omega):=\left\{(f_{\theta^{-n}\omega}^n p, f_{\theta^{-n}\omega}^nq)\in M\times M|\ \max_{0\leq k\leq  n}d_(f_{\theta^{-n}\omega}^kp,f_{\theta^{-n}\omega}^kq)<\delta_0,\ e^{-N\lambda}<d(p,q)<\delta_0\right\},
\end{equation*}and let $S_n(\omega)=\cup_{i=0}^nR_i(\omega)$. Applying Sublemma \ref{lemma keep Kalpha} inductively and notice Sublemma \ref{lemma 3.4}, 
we see that for any $g\in G$, $(s,t)\in S_n(\omega)$,
\begin{equation*}
 \sup_{x\in E^u(s,\omega)(\delta_0)} \left\|((\phi^*)^ng)_{(s,\omega)}(x)-(((\phi^*)^ng)_{(t,\omega)})^{(s,\omega)}(x)\right\|\leq C(\delta_0)e^{N\lambda\varrho}d(p,q)^\varrho,
\end{equation*}where $(\phi^*)^n$ is the $n$-th iteration of $\phi^*$.

\begin{sublemma}\label{sublemma new}
$  \{(s,t)\in M\times M|\ t\not\in W^u_{\delta_0}(s,\omega),\ d(s,t)<\delta_0\}\subset \cup_{n=0}^\infty S_n(\omega).$
\end{sublemma}
\begin{proof}[Proof of Sublemma \ref{sublemma new}]
  Now for any
  $(s,t)\in \{(s,t)\in M\times M|\ t\not\in W^u_{\delta_0}(s,\omega),\ d(s,t)<\delta_0\}$.
 If $d(s,t)>e^{-N\lambda}$, then $(s,t)\in R_0(\omega)$.
If $d(s,t)\leq e^{-N\lambda}$, We note that $t\in W^s_{\delta_0}([t,s]_\omega,\omega)$, where $[t,s]_\omega=W_{loc}^s(t,\omega)\cap W_{loc}^u(s,\omega)$. Then
    \begin{equation*}
      d(f_\omega^{-n}(s),f_\omega^{-n}(t))\geq d(f_\omega^{-n}([t,s]_\omega),f_\omega^{-n}(t))- d(f_\omega^{-n}([t,s]_\omega),f_\omega^{-n}(s)),
    \end{equation*}when tends to $\infty$ as $n\to\infty$. By \eqref{eq pick of N}, there exists a time $l\in\mathbb{N}$ such that $d(f_\omega^{-l}s,f_\omega^{-l}t)\in (e^{-N\lambda},\delta_0)$ and $d(f_\omega^{-i}s,f_\omega^{-i}t)\leq e^{-N\lambda}<\delta_0$ for $i\in\{0,1,...,l-1\}$. So $(s,t)\in R_{l}(\omega)$. In both cases, $(s,t)\in \cup_{n=0}^\infty S_n(\omega)$. This finishes the proof of Sublemma \ref{sublemma new}.
\end{proof}

    Hence the fix point obtained by $(\ref{fixed point})$ has the property that for any $p,q\in M$, $d(p,q)<\delta_0$ and $q\not \in W^u_{\delta_0}(p,\omega)$,
\begin{equation}\label{Holder continuous on base point}
  \sup_{x\in E^u(p,\omega)(\delta_0)}\|g^*_{(p,\omega)}(x)-(g^*_{(q,\omega)})^{(p,\omega)}(x)\|\leq C(\delta_0)e^{N\lambda\varrho}d(p,q)^\varrho:=H(\delta_0,\varrho)d(p,q)^\varrho,
\end{equation}and we note that \eqref{Holder continuous on base point} is automatically true if $q\in W_{\delta_0}^u(p,\omega)$. Hence \eqref{holder hu} holds. A similar proof can be applied to the stable manifold by reversing time.

Now let $y_0,y_1\in graph(g^*_{(p,\omega)}|_{E^u(p,\omega)(\delta_0)})$, let $q\in M$ and $d(p,q)<\delta_0$. Let $$z_0:=(P(E^u(p,\omega)y_0,(g^*_{(q,\omega)})^{(p,\omega)}(P(E^u(p,\omega)y_0)))$$ and $$z_1:=(P(E^u(p,\omega)y_1,(g^*_{(q,\omega)})^{(p,\omega)}(P(E^u(p,\omega)y_1))).$$
Then by  \eqref{holder hu},
\begin{equation}\label{eq ratio}
  \|z_1-y_1\|\leq H(\delta_0,\varrho)\|z_0-y_0\|^\varrho.
\end{equation}Denote
\begin{equation*}
  w_0:=Exp_p^{-1}(W^s_{loc}(Exp_p(y_0),\omega))\cap graph((g^*_{(q,\omega)})^{(p,\omega)}),
\end{equation*}and
\begin{equation*}
  w_1:=Exp_p^{-1}(W^s_{loc}(Exp_p(y_1),\omega))\cap graph((g^*_{(q,\omega)})^{(p,\omega)}).
\end{equation*}
\begin{sublemma}\label{sublemma ratio}
  When $\delta_0$ sufficiently small, there exists a constant $\hat{C}$ such that
\begin{equation*}
   \hat{C}^{-1}\leq \frac{\|y_1-z_1\|}{\|y_1-w_1\|}, \frac{\|y_0-z_0\|}{\|y_0-w_0\|}\leq \hat{C}.
\end{equation*}
\end{sublemma}
\begin{proof}[Proof of Sublemma \ref{sublemma ratio}]
 Fix any number $L^\prime>0$. When $\delta_0>0$ is sufficiently small, then we have
 \begin{equation*}
  l:= \sup\{\|D_x (g^*_{(q,\omega)})^{(p,\omega)}\|:\ x\in E^u(p,\omega)(\delta_0),\ d(p,q)<\delta_0,\ \omega\in \Omega\}<\frac{1}{L^\prime},
 \end{equation*}since $\|D_0 (g^*_{p,\omega})^{(p,\omega)}\|=0$ and $C^1$ continuity of local unstable manifolds as in (3) of Lemma \ref{lemma stable unstable}. By the similar reason, the local stable manifold $Exp_p^{-1}(W^s_{\delta_0}(Exp_p(y),\omega))$ for any $y\in graph(g^*_{(p,\omega)}|_{E^u(p,\omega)(\delta_0)})$, as a graph of $C^2$ function from $E^s(p,\omega)$ to $E^u(p,\omega)$, has coslope$\leq L^\prime$, provided $\delta_0>0$ sufficiently small.

 Now for $y_0$, $z_0$ and $w_0$ defined as above, then $w_0$ lies in an area constrained by the following two cones:
 \begin{equation*}
   z_0+\{(x,y)\in E^u(p,\omega)\oplus E^s(p,\omega):\ \|y\|\leq l \|x\|\}
 \end{equation*}and
 \begin{equation*}
   y_0+\{(x,y)\in E^u(p,\omega)\oplus E^s(p,\omega):\ \|x\|\leq L^\prime \|x\|\}.
 \end{equation*}By the knowledge of trigonometry, one can show that the following number satisfies the conclusion of Sublemma \ref{sublemma ratio}:
 \begin{equation*}
   \hat{C}=\frac{2(1+(L^\prime)^2)}{1-lL^\prime}
 \end{equation*}The case for $y_1,$ $z_1$ and $w_1$ is similar.
\end{proof}
By \eqref{eq ratio} and Sublemma \ref{sublemma ratio}, we obtain
\begin{equation*}
  \|y_1-w_1\|\leq \hat{C}^{1+\varrho}H(\delta_0,\varrho)\|y_0-w_0\|^\varrho,
\end{equation*}i.e., the fiber holonomy map between local stable manifolds is uniformly $\varrho-$H\"older continuous at a small scale. A similar result holds for fiber holonomy map between local unstable manifolds. Let $H^\prime=\hat{C}^{1+\varrho}H(\delta_0,\varrho)$, then the proof of Proposition \ref{Holder continuity of holonomy map} is complete.

\end{proof}

\subsection{Properties of the holonomy map between a pair of local stable leaves}\label{subsection 3.6}
In this subsection, the properties of the Holonomy maps are further discussed.

For each $\omega\in\Omega$, $\tilde{\gamma}(\omega)$ and $\gamma(\omega)$ is said to be a pair of nearby local stable leaves if $\tilde{\gamma}(\omega)$ and $\gamma(\omega)$ are local stable manifolds and the fiber holonomy map $\psi_\omega:\tilde{\gamma}(\omega)\rightarrow\gamma(\omega)$ by $\psi_\omega(x)=W^u_{\epsilon}(x,\omega)\cap \gamma(\omega)$ for $x\in \tilde{\gamma}(\omega)$ is a homeomorphism. From the proof of Proposition \ref{Theorem fiberwise absolutely continuous}, we know that $\psi_\omega$ is absolutely continuous and we have
\begin{equation}\label{Def jac psi omega}
  Jac(\psi_\omega)(x)=\lim_{n\rightarrow\infty}\frac{|\det(D_xf_\omega^{-n}|_{E^s(x,\omega)})|}{|\det(D_{\psi_\omega(x)}f_\omega^{-n}|_{E^s(\psi_\omega(x),\omega)})|},
\end{equation}where $Jac(\psi_\omega)$ denotes the Jacobian of $\psi_\omega$, and $f_\omega^{-n}$ is defined in \eqref{def f omega n}. The time in \eqref{Def jac psi omega} goes backward since the holonomy map in this section is induced by unstable manifolds, while in the proof of Proposition \ref{Theorem fiberwise absolutely continuous}, the holonomy map is induced by local stable manifolds.

In the following, we restrict the size of local stable and unstable manifolds $W^s_\epsilon(x,\omega)$, $W^u_\epsilon(x,\omega)$ satisfying $\epsilon\leq \min\{\epsilon_0,\delta_0\}$ to guarantee the H\"older continuity of the stable and unstable foliations, where $\delta_0$ is the constant in Proposition \ref{Holder continuity of holonomy map}.
\begin{lemma}\label{property of fiberwise holonomy map}
  There exist constants $a_0^\prime>0$ and $\nu_0\in(0,1)$ only depending on system $\phi$ such that for any $\psi_\omega:\tilde{\gamma}(\omega)\rightarrow\gamma(\omega)$  fiber holonomy map of  a pair of nearby random local stable leaves, the followings hold:
  \begin{enumerate}
    \item $\psi_\omega$ and $\log Jac(\psi_\omega)$ are $(a_0^\prime,\nu_0)$-H\"older continuous;
    \item $|\log Jac(\psi_\omega)(y)|\leq a_0^\prime d(y,\psi_\omega(y))^{\nu_0}$ for every $y\in\tilde{\gamma}(\omega)$;
    \item $d(f_{\omega}^{-1}x,f_{\omega}^{-1}\psi_\omega(x))\leq e^{-\lambda}d(x,\psi_\omega(x))$.
  \end{enumerate}
\end{lemma}
\begin{proof}[Proof of Lemma \ref{property of fiberwise holonomy map}]
 In Proposition \ref{Holder continuity of holonomy map}, we have proved that $\psi_\omega$ is $(H^\prime,\varrho)-$H\"older continuous for all $\omega\in\Omega$. To prove the other statements, we need the following fact: there exists a constant $C_{13}>0$ such that for any $n\geq 0$, $x\in\tilde{\gamma}(\omega)$,
 \begin{equation}\label{def C13}
   \frac{\left|det(D_{f_\omega^{-n}x}f_{\theta^{-n}\omega}^{-1}|_{E^s(f_\omega^{-n}x,\theta^{-n}\omega)})\right|}{\left|det(D_{f_\omega^{-n}\psi_\omega(x)}f_{\theta^{-n}\omega}^{-1}|_{E^s(f_\omega^{-n}\psi_\omega(x),\theta^{-n}\omega)})\right|}\leq 1+ C_{13}e^{-\lambda n \nu_1}d(x,\psi_\omega(x))^{\nu_1}.
 \end{equation}In fact, this is a consequence of \eqref{bound of det} and the following inequality
 \begin{align*}
    & \ \ \ \ \left|det(D_{f_\omega^{-n}x}f_{\theta^{-n}\omega}^{-1}|_{E^s(f_\omega^{-n}x,\theta^{-n}\omega)})-det(D_{f_\omega^{-n}\psi_\omega(x)}f_{\theta^{-n}\omega}^{-1}|_{E^s(f_\omega^{-n}\psi_\omega(x),\theta^{-n}\omega)})\right|\\
    &\overset{\eqref{negative     time}}\leq C_2[d(f_\omega^{-n}x,f_\omega^{-n}\psi_\omega(x))+d(E^s(f_\omega^{-n}x,\theta^{-n}\omega),E^s(f_\omega^{-n}\psi_\omega(x),\theta^{-n}\omega))]\\
    &\overset{\eqref{Holder continuous bundle}}\leq C_2e^{-\lambda n} d(x,\psi_\omega(x))+C_2C_1e^{-\lambda n\nu_1}d(x,\psi_\omega(x))^{\nu_1}\\
    &\leq (C_2+C_2C_1)e^{-\lambda n \nu_1}d(x,\psi_\omega(x))^{\nu_1}.
 \end{align*}

We first prove statement (2).  As a consequence of \eqref{Def jac psi omega} and \eqref{def C13}, we have
 \begin{equation}\label{eq jac psiomega}
   \begin{split}
Jac(\psi_\omega)(x)&\leq\prod_{j=0}^{\infty} (1+C_{13}e^{-\lambda j{\nu_1}}d(x,\psi_\omega(x))^{\nu_1})\\
&\leq \sum_{j=0}^\infty C_{13}e^{-\lambda j{\nu_1}}d(x,\psi_\omega(x))^{\nu_1}\mbox{ for any }x\in\tilde{\gamma}(\omega).
   \end{split}
 \end{equation}Symmetrically, we have
 \begin{equation*}
   \frac{1}{Jac(\psi_\omega)(x)}=Jac(\psi_\omega^{-1})(\psi_\omega(x)) \leq \sum_{j=0}^\infty C_{13}e^{-\lambda j{\nu_1}}d(x,\psi_\omega(x))^{\nu_1}\mbox{ for any }x\in\tilde{\gamma}(\omega).
 \end{equation*}Therefore,
 \begin{equation*}
   |Jac(\psi_\omega)(x)|\leq \sum_{j=0}^\infty C_{13}e^{-\lambda j{\nu_1}}d(x,\psi_\omega(x))^{\nu_1}\mbox{ for any }x\in\tilde{\gamma}(\omega).
 \end{equation*}

Next, we prove the H\"older continuity of $\log Jac(\psi_\omega)$. Pick any $x,y\in \tilde{\gamma}(\omega)$, we consider two cases: (case 1) $d(x,\psi_\omega(x))\leq d(x,y)$ and (case 2) $d(x,\psi_\omega(x))>d(x,y)$.

 In (case 1), applying Proposition $\ref{Holder continuity of holonomy map}$, we have
 \begin{align}
  |\log Jac(\psi_\omega)(x)-\log Jac(\psi_\omega)(y)|  &\overset{\eqref{eq jac psiomega}}\leq \sum_{j=0}^{\infty}C_{13}e^{-\lambda j\nu_1}d(x,\psi_\omega(x))^{\nu_1}+\sum_{j=0}^{\infty}C_{13}e^{-\lambda j\nu_1}d(y,\psi_\omega(y))^{\nu_1}\nonumber\\
  & \leq\sum_{j=0}^{\infty}C_{13}e^{-\lambda j\nu_1}d(x,\psi_\omega(x))^{\nu_1}+\sum_{j=0}^{\infty}C_{13}(H^\prime)^{\nu_1}e^{-\lambda j\nu_1}d(x,\psi_\omega(x))^{\nu_1\varrho}\nonumber\\
  &\leq \left(\sum_{j=0}^{\infty}C_{13}e^{-\lambda j\nu_1}+\sum_{j=0}^{\infty}C_{13}(H^\prime)^{\nu_1}e^{-\lambda j\nu_1}\right)d(x,y)^{\nu_1\varrho}\nonumber\\
  &:=S_1d(x,y)^{\nu_1\varrho}\label{definition of S1},
 \end{align}where $S_1:=\sum_{j=0}^{\infty}C_{13}e^{-\lambda j\nu_1}+\sum_{j=0}^{\infty}C_{13}(H^\prime)^{\nu_1}e^{-\lambda j\nu_1}$.

  In (case 2), i.e., $d(x,\psi_\omega (x))>d(x,y)$. Since $\psi_\omega(x)\in W_\epsilon^u(x,\omega)$, and $x,y$ lie in local stable manifold $\tilde{\gamma}(\omega)$, there exists an integer $m>0$ such that
 \begin{align}\label{eq m-1 da}
   d(f_\omega^{-k}x,f_\omega^{-k}\psi_\omega(x)) &>d(f_\omega^{-k}x,f_\omega^{-k}y)\ \mbox{for }0\leq k\leq m-1,
 \end{align}and
 \begin{equation}\label{eq m xiao}
   d(f_\omega^{-m}x,f_\omega^{-m}\psi_\omega(x))\leq d(f_\omega^{-m}x,f_\omega^{-m}y).
 \end{equation}Note that in \eqref{eq m-1 da},
 \begin{equation}\label{still close}
   d(f_\omega^{-k}x,f_\omega^{-k}y)<d(f_\omega^{-k}x,f_\omega^{-k}\psi_\omega(x)) \leq d(x,\psi_\omega(x))\leq \epsilon,\mbox{ for }0\leq k\leq m-1.
 \end{equation}  Denote $$\beta:=\sup\left\{\frac{d(f_\omega^{-1}x,f_\omega^{-1}y)}{d(x,y)}:\ d(x,y)\leq \epsilon,\ \omega\in\Omega\right\}\in(1,\infty),$$ and $$\eta:=\inf\left\{\frac{d(f_\omega^{-1}x,f_\omega^{-1}y)}{d(x,y)}:\ d(x,y)\leq \epsilon,\ \omega\in\Omega\right\}\in (0,1),$$ then by \eqref{eq m xiao}, we have
  \begin{equation*}
    \eta^md(x,\psi_\omega(x))\leq \beta^md(x,y).
  \end{equation*}As a consequence, $m\geq (\log \frac{d(x,\psi_\omega(x))}{d(x,y)})/\log(\beta/\eta)$, and
  \begin{equation}\label{eq e -m leq}
    e^{-m}\leq d(x,y)^{\frac{1}{\log(\beta/\eta)}}d(x,\psi_\omega(x))^{-\frac{1}{\log{\beta/\eta}}}.
  \end{equation}
  Note that
  \begin{equation}\label{eq jac over jac}
    \frac{Jac(\psi_\omega)(x)}{Jac(\psi_\omega)(y)}=\frac{|\det D_xf_\omega^{-m}|_{E^s(x,\omega)}|}{|\det D_yf_\omega^{-m}|_{E^s(y,\omega)}|}\cdot\frac{|\det D_{\psi_\omega(y)}f_\omega^{-m}|_{E^s(\psi_\omega(y),\omega)}|}{|\det D_{\psi_\omega(x)}f_\omega^{-m}|_{E^s(\psi_\omega(x),\omega)}|}\cdot\frac{Jac(f_\omega^{-m}\psi_\omega f_{\theta^{-m}\omega}^m)(f_\omega^{-m}x)}{Jac(f_\omega^{-m}\psi_\omega f_{\theta^{-m}\omega}^m)(f_\omega^{-m}y)},
  \end{equation}where $f_\omega^{-m}\psi_\omega f_{\theta^{-m}\omega}^m=f_\omega^{-m}\circ\psi_\omega\circ f_{\theta^{-m}\omega}^m$ is the holonomy map between $f_\omega^{-m}\tilde{\gamma}(\omega)$ and $f_\omega^{-m}\gamma(\omega)$, and we omit the composition notation for short. Therefore, we need to estimate the right side of \eqref{eq jac over jac}. Similar as the proof of \eqref{def C13}, we have
  \begin{equation*}
  \left|\log|\det D_xf_\omega^{-m}|_{E^s(x,\omega)}| -\log|\det D_yf_\omega^{-m}|_{E^s(y,\omega)}|\right| \leq \sum_{k=0}^{m-1}C_{13}d(f_\omega^{-k}x,f_\omega^{-k}y)^{\nu_1},
  \end{equation*}and therefore,
  \begin{align*}
    &\ \ \ \ \left|\log|\det D_xf_\omega^{-m}|_{E^s(x,\omega)}| -\log|\det D_yf_\omega^{-m}|_{E^s(y,\omega)}|\right| \nonumber\\
    &
    \leq \sum_{k=0}^{m-1}C_{13}e^{-\lambda(m-1-k)\nu_1}d(f_\omega^{-(m-1)}x,f_\omega^{-(m-1)}y)^{\nu_1}\\
    &\overset{\eqref{eq m-1 da}}\leq \left(\sum_{k=0}^{m-1}C_{13}e^{-\lambda(m-1-k)\nu_1}\right)d(f_\omega^{-(m-1)}x,f_\omega^{-(m-1)}\psi_\omega(x))^{\nu_1}\\
    &\leq \left(\sum_{k=0}^{m-1}C_{13}e^{-\lambda(m-1-k)\nu_1}\right)e^{-\lambda(m-1)\nu_1}d(x,\psi_\omega(x))^{\nu_1}.
  \end{align*}Denote $S_2:=(\sum_{k=0}^{\infty}C_{13}e^{-k\lambda\nu_1})e^{\lambda\nu_1}$, and note that $\frac{\lambda}{\log(\beta/\eta)}<1$, then
  \begin{align*}
     &\left|\log|\det D_xf_\omega^{-m}|_{E^s(x,\omega)}| -\log|\det D_yf_\omega^{-m}|_{E^s(y,\omega)}|\right| \\
     \leq & S_2e^{-\lambda m\nu_1}d(x,\psi_\omega(x))^{\nu_1}\overset{\eqref{eq e -m leq}}\leq S_2d(x,\psi_\omega(x))^{\nu_1-\frac{\lambda\nu_1}{\log(\beta/\eta)}}d(x,y)^{\frac{\lambda\nu_1}{\log(\beta/\eta)}}
\leq S_2d(x,y)^{\frac{\lambda\nu_1}{\log(\beta/\eta)}}.
  \end{align*}
  Similar to above, we have
  \begin{align*}
     & \ \ \ \ \left|\log |\det D_{\psi_\omega(y)}f_\omega^{-m}|_{E^s(\psi_\omega(y),\omega)}|-\log|\det D_{\psi_\omega(x)}f_\omega^{-m}|_{E^s(\psi_\omega(x),\omega)}| \right|\\
     &\leq \sum_{k=0}^{m-1}C_{13}d(f_\omega^{-k}\psi_\omega(x),f_\omega^{-k}\psi_\omega(y))^{\nu_1}\leq \sum_{k=0}^{m-1}C_{13}e^{-\lambda(m-1-k)\nu_1}d(f_\omega^{-(m-1)}\psi_\omega(x),f_\omega^{-(m-1)}\psi_\omega(y))^{\nu_1}
\end{align*}By \eqref{still close}, both $d(f_\omega^{-(m-1)}(x),f_\omega^{-(m-1)}(y))< \epsilon$ and $d(f_\omega^{-(m-1)}(x),f_\omega^{-(m-1)}(\psi_\omega(x)))<\epsilon$, which implies $d(f_\omega^{-(m-1)}(\psi_\omega(x)),f_\omega^{-(m-1)}(\psi_\omega(y)))<H^\prime d(f_\omega^{-(m-1)}(x),f_\omega^{-(m-1)}(y))^\varrho$ by the H\"older continuity of local product structure. Therefore, we have
\begin{align*}
&\left|\log |\det D_{\psi_\omega(y)}f_\omega^{-m}|_{E^s(\psi_\omega(y),\omega)}|-\log|\det D_{\psi_\omega(x)}f_\omega^{-m}|_{E^s(\psi_\omega(x),\omega)}| \right|\\
 \leq    & \sum_{k=0}^{m-1}C_{13}e^{-\lambda(m-1-k)\nu_1}(H^\prime)^{\nu_1}d(f_\omega^{-(m-1)}x,f_\omega^{-(m-1)}y)^{\nu_1\varrho}\\
\overset{\eqref{eq m-1 da}}\leq     & \sum_{k=0}^{m-1}C_{13}e^{-\lambda(m-1-k)\nu_1}(H^\prime)^{\nu_1}d(f_\omega^{-(m-1)}x,f_\omega^{-(m-1)}\psi_\omega(x))^{\nu_1\varrho}\\
\leq &\sum_{k=0}^{\infty}C_{13}e^{-\lambda k\nu_1}(H^\prime)^{\nu_1}e^{-\lambda(m-1)\nu_1\varrho}d(x,\psi_\omega(x))^{\nu_1\varrho}\\
    :=& S_3e^{-\lambda m \nu_1\varrho}d(x,\psi_\omega(x))^{\nu_1\varrho}\\
   \overset{\eqref{eq e -m leq}}\leq  & S_3d(x,\psi_\omega(x))^{\nu_1\varrho-\frac{\lambda\nu_1\varrho}{\log(\beta/\eta)}}d(x,y)^{\frac{\lambda\nu_1\varrho}{\log(\beta/\eta)}}\\
   \leq  &S_3 d(x,y)^{\frac{\lambda\nu_1\varrho}{\log(\beta/\eta)}},
  \end{align*}where   $S_3:=\sum_{k=0}^{\infty}C_{13}e^{-k\lambda\nu_1}H^{\nu_1}e^{\lambda \nu_1\varrho}$.
 Note that $f_\omega^m\circ \psi_\omega\circ f_{\theta^{-m}\omega}^m$ is the holonomy map from $f_\omega^{-m}\tilde{\gamma}(\omega)$ to $f_\omega^{-m}\gamma(\omega)$ and satisfy \eqref{eq m xiao}, hence similar to $(\ref{definition of S1})$, we have
  \begin{align*}
    &\ \ \ \ \left|\log Jac(f_\omega^{-m}\psi_\omega f_{\theta^{-m}\omega}^m)(f_\omega^{-m}x)-\log Jac(f_\omega^{-m}\psi_\omega f_{\theta^{-m}\omega}^m)(f_\omega^{-m}y)\right|\\
    &\leq \sum_{j=0}^{\infty}C_{13}e^{-\lambda j\nu_1}d(f_\omega^{-m}x,f_\omega^{-m}\psi_\omega(x))^{\nu_1}+\sum_{j=0}^{\infty}C_{13}e^{-\lambda j\nu_1}d(f_\omega^{-m}y,f_\omega^{-m}\psi_\omega(y))^{\nu_1}\\
    &\leq \sum_{j=0}^{\infty}C_{13}e^{-\lambda j\nu_1}e^{-\lambda  \nu_1}\left(d(f_\omega^{-(m-1)}x,f_\omega^{-(m-1)}\psi_\omega(x))^{\nu_1}+ d(f_\omega^{-(m-1)}y,f_\omega^{-(m-1)}\psi_\omega(y))^{\nu_1}\right).\\
  \end{align*}Notice \eqref{still close}, $f_\omega^{-(m-1)}x $ and $f_\omega^{-(m-1)}y$ are still close, and we can use the H\"older continuity of local product structure to obtain $$d(f_\omega^{-(m-1)}y,f_\omega^{-(m-1)}\psi_\omega(y))<H^\prime d(f_\omega^{-(m-1)}x,f_\omega^{-(m-1)}\psi_\omega(x))^\varrho.$$ We continue the estimate
  \begin{align*}
   &\ \ \ \ \left|\log Jac(f_\omega^{-m}\psi_\omega f_{\theta^{-m}\omega}^m)(f_\omega^{-m}x)-\log Jac(f_\omega^{-m}\psi_\omega f_{\theta^{-m}\omega}^m)(f_\omega^{-m}y)\right|\\
    &\leq \sum_{j=0}^{\infty}C_{13}e^{-\lambda j\nu_1}e^{-\lambda\nu_1}\left(d(f_\omega^{-(m-1)}x,f_\omega^{-(m-1)}\psi_\omega(x))^{\nu_1}+(H^\prime)^{\nu_1} d(f_\omega^{-(m-1)}x,f_\omega^{-(m-1)}\psi_\omega(x))^{\nu_1\varrho}\right)\\
    &\leq \sum_{j=0}^{\infty}C_{13}e^{-\lambda (j+1)\nu_1}(1+(H^\prime)^{\nu_1})e^{-(m-1)\lambda\nu_1\varrho}d(x,\psi_\omega(x))^{\nu_1\varrho}\\
    &:= S_4e^{-\lambda m\nu_1\varrho}d(x,\psi_\omega(x))^{\nu_1\varrho}\\
    &\overset{\eqref{eq e -m leq}}\leq S_4d(x,\psi_\omega(x))^{\nu_1\varrho-\frac{\lambda\nu_1\varrho}{\log(\beta/\eta)}}d(x,y)^{\frac{\lambda\nu_1\varrho}{\log(\beta/\eta)}}\\
    &\leq  S_4d(x,y)^{\frac{\lambda\nu_1\varrho}{\log(\beta/\eta)}},
  \end{align*}where $S_4:=\sum_{j=0}^{\infty}C_{13}e^{-\lambda (j+1)\nu_1}(1+(H^\prime)^{\nu_1} )e^{\lambda\nu_1\varrho}.$ Hence by \eqref{eq jac over jac} and the above estimates, in (case 2), we obtain
  \begin{equation*}
    \left|\log Jac(\psi_\omega)(x)-\log Jac(\psi_\omega)(y)\right|\leq \max\{S_2,S_3,S_4\}d(x,y)^{\frac{\lambda\nu_1\varrho}{\log(\beta/\eta)}}.
  \end{equation*}
  Now we define $a_0^\prime:=\max\{H^\prime, S_1,S_2,S_3,S_4,\sum_{j=0}^\infty C_{13}e^{-\lambda j{\nu_1}}\}$ and $\nu_0:=\min\{\frac{\lambda\nu_1\varrho}{\log(\beta/\eta)},\nu_1\}=\frac{\lambda\nu_1\varrho}{\log(\beta/\eta)}$.
 Then property $(1)$ and $(2)$ are proved.

 Property $(3)$ follows  the definition of holonomy map and contraction on local unstable manifolds when reverse time. The proof of Lemma \ref{property of fiberwise holonomy map} is complete.
 \end{proof}

\subsection{Measure disintegration on rectangles}\label{subsection 3.7}
We call $R(\omega)\subset M$ a rectangle if it is foliated by local stable manifolds and it has the local product structure. By the lemma \ref{property of fiberwise holonomy map}, for any rectangle $ R(\omega)$ in small scale, the holonomy map between stable manifolds (resp. unstable manifolds) lying in $R(\omega)$ is absolutely continuous. 
As a consequence of this absolute continuity and Fubini's theorem, in \cite[Chapter III section 6]{PDQM}, the authors proved that on each local stable leaves, the disintegration of the Riemannian volume measure is equivalent to the inherited Riemannian measure. In here, we explore the corresponding Radon-Nikodym derivative of these two measures. The proof of this statement is parallel to the treatment of \cite[Theorem 7.8]{Pesin2004}.

\begin{proposition}\label{proposition 7.1}
    There exist constant $a_0^{\prime\prime}$ only depending on the system such that for each $\omega\in\Omega$ and any rectangle $R(\omega)=[W_{\epsilon}^u(x_0,\omega),W_\epsilon^s(x_0,\omega)]$, there exists a measurable function $H(\omega):R(\omega)\rightarrow\mathbb{R}^+$ such that for any bounded measurable function $\psi:M\rightarrow\mathbb{R}$, there is a disintegration
  \begin{equation}\label{def of H omega}
    \int_{R(\omega)}\psi(x)dm(x)=\int\int_{\gamma(\omega)}\psi(x)H(\omega)|_{\gamma(\omega)}(x)dm_{\gamma(\omega)}(x)d\tilde{m}_{R(\omega)}(\gamma(\omega)),
  \end{equation} where $\gamma(\omega)$ denote the stable foliations in $R(\omega)$, and $\tilde{m}_{R(\omega)}$ is the quotient measure induced by Riemannian volume measure in the space of local stable leaves in $R(\omega)$. Moreover, for any local stable leaves $\gamma(\omega)\subset R(\omega)$, one has
  \begin{equation*}
    |\log H(\omega)(x)-\log H(\omega)(y)|\leq a_0^{\prime\prime}d(x,y)^{\nu_0},\ \forall x,y\in\gamma(\omega).
  \end{equation*}
\end{proposition}
\begin{proof}[Proof of Proposition \ref{proposition 7.1}]
  By using the exponential map, we can pretend $R(\omega)$ to be a subset of $T_{x_0}M$. Moreover, $R(\omega)\subset E^u(x_0,\omega)(r)\oplus E^u(x_0,\omega)^\perp(r)$ for some $r>0$, where $E^u(x_0,\omega)(r)$ and $ E^u(x_0,\omega)^\perp(r)$ are $r$-disks. We consider the following partition of $E^u(x_0,\omega)(r)\oplus E^u(x_0,\omega)^\perp(r)$
  \begin{equation*}
    \mathcal{D}_\omega=\{D_\omega(\eta)=\eta+E^u(x_0,\omega)(r)|\ \eta\in E^u(x_0,\omega)^\perp(r)\}.
  \end{equation*}We denote $m_{D_\omega(\eta)}$ and $m_{E^u(x_0,\omega)^\perp}$ to be the Lebesgue measure on $D_\omega(\eta)$ and $E^u(x_0,\omega)^\perp$ respectively. Then for any Borel measurable set $A\subset R(\omega)$, we have
  \begin{align*}
    m(A) &=\int_{E^u(x_0,\omega)^\perp}\int_{D_\omega(\eta)}1_A(\xi,\eta)dm_{D_\omega(\eta)}(\xi)dm_{E^u(x_0,\omega)^\perp}(\eta).
  \end{align*}We may shrink $\epsilon$ and $r$ to make sure that $D_\omega(\eta)$ are transverse to the local stable leaves for all $\eta\in  E^u(x_0,\omega)^\perp(r)$. We denote $\psi_{\omega;0,\eta}^s:D_\omega(0)\to D_\omega(\eta)$ to be the holonomy map induced by the local stable manifolds. Denote $\xi=\psi_{\omega;0,\eta}^s(\xi^\prime)$ for $\xi^\prime\in D_\omega(0)$, then we have
  \begin{equation*}
    m(A)=\int_{E^u(x_0,\omega)^\perp}\int_{D_\omega(0)}1_A(\xi,\eta)\cdot Jac(\psi_{\omega;0,\eta}^s)(\xi^\prime)dm_{D_\omega(0)}(\xi^\prime)dm_{E^u(x_0,\omega)^\perp}(\eta).
  \end{equation*}Applying Fubini's theorem, we obtain
  \begin{align*}
  m(A)=\int_{D_\omega(0)}\int_{E^u(x_0,\omega)^\perp}1_A(\xi,\eta)\cdot Jac(\psi_{\omega;0,\eta}^s)(\xi^\prime)dm_{E^u(x_0,\omega)^\perp}(\eta)dm_{D_\omega(0)}(\xi^\prime).
  \end{align*}By the local stable manifold theorem, there exists a $C^2$ function $h_{(x_0,\omega)}^s:E^u(x_0,\omega)^\perp(r)\to W_r^s(x_0,\omega)$. By changing of coordinate $\eta^\prime=h_{(x_0,\omega)}^s(\eta)$ for $\eta\in E^u(x_0,\omega)^\perp(r)$, we have
  \begin{equation*}
    m(A)=\int_{D_\omega(0)}\int_{W^s_r(x_0,\omega)}1_A(\xi,\eta)\cdot Jac(\psi_{\omega;0,\eta}^s)(\xi^\prime)\cdot \frac{1}{Jac(h_{(x_0,\omega)}^s)(\eta)} dm_{W^s_r(x_0,\omega)}(\eta^\prime)dm_{D_\omega(0)}(\xi^\prime).
  \end{equation*}Denote $\psi_{\omega;\xi^\prime,0}^u:W_r^s(\xi^\prime,\omega)\to W_r^s(x_0,\omega)$ to be the holonomy map induced by the local unstable manifolds, where $W_r^s(\xi^\prime,\omega)$ is the local stable manifold passing through $\xi^\prime\in E^u(x_0,\omega)$. By change of coordinate $\zeta=(\psi_{\omega;\xi^\prime;0}^u)^{-1}(\eta^\prime)$ for $\eta^\prime\in W_r^s(x_0,\omega)$, we have
  \begin{equation*}
     m(A)=\int_{D_\omega(0)}\int_{W^s_r(\xi,\omega)}1_A(\xi,\eta)\cdot Jac(\psi_{\omega;0,\eta}^s)(\xi^\prime)\cdot \frac{Jac(\psi_{\omega;\xi^\prime,0}^u)(\zeta)}{Jac(h_{(x_0,\omega)}^s)(\eta)} dm_{W^s_r(\xi^\prime,\omega)}(\zeta)dm_{D_\omega(0)}(\xi^\prime),
  \end{equation*}which implies that in \eqref{def of H omega},
  \begin{itemize}
      \item the space of local stable leaves is identified with $D_\omega(0)$ and  $\tilde{m}_{R(\omega)}=m_{D(\omega)(0)}$.
\item $ H(\omega)(\zeta)=Jac(\psi_{\omega;0,\eta}^s)(\xi^\prime)\cdot \frac{Jac(\psi_{\omega;\xi^\prime,0}^u)(\zeta)}{Jac(h_{(x_0,\omega)}^s)(\eta)}$ for $\zeta\in W_r^s(\xi^\prime,\omega)$, where $\eta=(h_{(x_0,\omega)}^s)^{-1}(\psi_{\omega;\xi^\prime,0}^u(\zeta))$;
  \end{itemize}
 Next, let us verify the H\"older continuity of $\log H(\omega)$ on each local stable manifold $W_r^s(\xi^\prime,\omega)$. Pick any $\zeta_1,\zeta_2\in W_r^s(\xi^\prime,\omega)$ and corresponding $\eta_1=(h_{(x_0,\omega)}^s)^{-1}(\psi_{\omega;\xi^\prime,0}^u(\zeta_1))$, $\eta_2=(h_{(x_0,\omega)}^s)^{-1}(\psi_{\omega;\xi^\prime,0}^u(\zeta_2))$. By  \eqref{expression of jacobian}, we have
  \begin{equation*}
    Jac(\psi_{\omega;0,\eta}^s)(\xi^\prime)=\lim_{n\rightarrow\infty}\frac{|\det D_{\xi^\prime} f_\omega^n|_{E^u(x_0,\omega)}|}{|\det D_{\psi_{\omega;0,\eta}^s(\xi^\prime)}f_\omega^n|_{E^u(x_0,\omega)}|}.
  \end{equation*} We notice that $\zeta_1=\psi_{\omega;0,\eta_1}^s(\xi^\prime)$ and $\zeta_2=\psi_{\omega;0,\eta_2}^s(\xi^\prime)$. Therefore,
  \begin{align*}
    \frac{Jac(\psi_{\omega;0,\eta_1}^s)(\xi^\prime)}{Jac(\psi_{\omega;0,\eta_2}^s)(\xi^\prime)}&=\lim_{n\to\infty}\frac{|\det D_{\psi_{\omega;0,\eta_2}^s(\xi^\prime)}f_\omega^n|_{E^u(x_0,\omega)}|}{|\det D_{\psi_{\omega;0,\eta_1}^s(\xi^\prime)}f_\omega^n|_{E^u(x_0,\omega)}|}=\lim_{n\to\infty}\frac{|\det D_{\zeta_2}f_\omega^n|_{E^u(x_0,\omega)}|}{|\det D_{\zeta_1}f_\omega^n|_{E^u(x_0,\omega)}|}\\
    &\leq \prod_{i=0}^\infty(1+C_{10}C_2d(f_\omega^i(\zeta_1),f_\omega^i(\zeta_2)))\\
    &\leq  \prod_{i=0}^\infty(1+C_{10}C_2e^{-\lambda i}d(\zeta_1,\zeta_2)),
  \end{align*}where $C_2$ and $C_{10}$ comes from Lemma \ref{detDfE-detDfE} and $\eqref{bound of det}$. Hence, we have
  \begin{equation}\label{log H 1}
    |\log Jac(\psi_{\omega;0,\eta_1}^s)(\xi^\prime)-\log Jac(\psi_{\omega;0,\eta_2}^s)(\xi^\prime)|\leq \frac{C_{10}C_2}{1-e^{-\lambda}}d(\zeta_1,\zeta_2).
  \end{equation}Note that  $\psi_{\omega;\xi^\prime,0}^u:W_r^s(\xi^\prime,\omega)\to W_r^s(x_0,\omega)$ is the holonomy map between local stable manifolds. By Lemma \ref{property of fiberwise holonomy map}, we have
  \begin{equation}\label{log H 2}
    |\log Jac(\psi_{\omega;\xi^\prime,0}^u)(\zeta_1)-\log Jac(\psi_{\omega;\xi^\prime,0}^u)(\zeta_2)|\leq a_0^\prime d(\zeta_1,\zeta_2)^{\nu_0}.
  \end{equation}By Lemma \ref{lemma stable unstable}, there exists a constant $C_s>0$ such that
  \begin{equation}\label{log H 3}
  \begin{split}
    &|\log Jac(h_{(x_0,\omega)}^s)(\eta_1)-\log Jac(h_{(x_0,\omega)}^s)(\eta_2)|\\
  \leq
    & C_s\|\eta_1-\eta_2\|
    \leq  C_s d(\psi_{\omega;\xi^\prime,0}^u(\zeta_1),\psi_{\omega;\xi^\prime,0}^u(\zeta_2))
    \leq C_s a_0^\prime d(\zeta_1,\zeta_2)^{\nu_0}.
  \end{split}
  \end{equation}Inequalities \eqref{log H 1}, \eqref{log H 2} and \eqref{log H 3} imply that
  \begin{equation*}
    |\log H(\omega)(\zeta_1)-\log H(\omega)(\zeta_2)|\leq \max\left\{\frac{C_{10}C_2}{1-e^{-\lambda}},a_0^\prime,C_sa_0^\prime\right\}d(\zeta_1,\zeta_2)^{\nu_0}:=a_0^{\prime\prime}d(\zeta_1,\zeta_2)^{\nu_0}.
  \end{equation*}The proof of Proposition \ref{proposition 7.1} is complete.
\end{proof}
\section{Proof of Main Result}\label{section 4}
In this section, we prove Theorem \ref{random exponential decay}. The proof is based on the study of the fiber transfer operator $L_{\omega}$, which is defined by
\begin{equation}\label{L varphi}
  L_{\omega}\varphi: M \rightarrow\mathbb{R},\ (L_{\omega}\varphi)(x):=\frac{\varphi((f_{\omega})^{-1}x)}{|\det D_{(f_{\omega})^{-1}(x)}f_{\omega}|}
\end{equation}for any bounded and measurable observable $\varphi:M\rightarrow \mathbb{R}$. We denote
\begin{equation*}
  L_{\omega}^n:=L_{\theta^{n-1}\omega}\circ\cdots\circ L_{\theta\omega}\circ L_{\omega}\mbox{ for any }\omega\in\Omega \mbox{ and }n\in\mathbb{N}.
\end{equation*}
We first construct the suitable convex cone of observables $C_\omega$ on each fiber in Subsection \ref{subsection 4.1}. Then in Subsection \ref{subsection 4.2}, we prove that the transfer operator $L_\omega$ maps $C_\omega$ into $C_{\theta\omega}$. Moreover, $L^N_\omega C_\omega$ has finite diameter with respect to the projective metric on the cone $C_{\theta^N\omega}$, and this diameter is independent of $\omega$, where $N$ comes from the topological mixing on fibers property. Birkhoff's inequality implies the contraction of $L^N_\omega:C_\omega\to C_{\theta^N\omega}$ for all $\omega\in\Omega$. In Subsection \ref{subsection 4.3}, we explore the relationship between the unique random SRB measure and the operator $L_{\theta^{-n}\omega}^n$. We show  the exponential decay of past and future correlations in Subsection \ref{subsection 4.4} and Subsection \ref{subsection 4.5} respectively by using the contraction of $L^n_{\theta^{-n}\omega}$ and $L_\omega^n$ for $n\geq N$.

Before starting the proof, we recall some constants that will be needed later on.
Let $K_1$ be the constant in Lemma \ref{Lip jac Es} such that  for any $\omega\in\Omega$, $z\in M$ ,
\begin{equation}\label{eq def K1}
  \left|\log|\det(D_xf_\omega|_{E^s(x,\omega)})|-\log|\det(D_yf_\omega|_{E^s(y,\omega)})|\right|\leq K_1d(x,y)\mbox{ for any }x,y\in W_\epsilon^s(z,\omega).
\end{equation}
By the compactness of $\Omega$ and $M$ and the continuity of $f_\omega\in$Diff$^2(M)$ on $\omega$, there exists a constant $K_2>0$ such that for any $x,y\in M$,
\begin{equation}\label{Lip constant of log det Dxfomega}
  \left|\log|\det D_xf_\omega|-\log |\det D_yf_\omega|\right|\leq K_2d(x,y).
\end{equation}
Let $a_0:=\max\{ a_0^\prime, a_0^{\prime\prime}\}$, then Lemma \ref{property of fiberwise holonomy map} and Proposition \ref{proposition 7.1} hold for constants $(a_0,\nu_0)$.
This $\nu_0$ is the desired $\nu_0$ in the statement of Theorem \ref{random exponential decay}.
Now Let's pick any $\kappa,\nu\in (0,1)$ satisfying $0<\kappa+\nu<\nu_0$ as in the statement of Theorem \ref{random exponential decay}, and pick $\kappa_1\in(0,1)$ an auxiliary constant closing to $1$ such that
\begin{equation}\label{mu+nuleq nu0}
  0<\kappa+\nu<\kappa_1\nu_0.
\end{equation}
Now we are going to prove Theorem \ref{random exponential decay} for fixed $\kappa,\nu$.
\subsection{Construction of Birkhoff cone}\label{subsection 4.1}
In this subsection, we will first construct convex cones of density functions on each local stable leaf. With the help of these convex cones of density functions on each local stable leaf, we can define our desired  convex cone of observables on each fiber. Our construction of convex cone of obervables is inspired by \cite{Liv95A} and \cite{Viana}.
 The definitions of the convex cone in a topological vector space, projective metric on the convex cone, and Birkhoff's inequality are recalled in the Appendix \ref{section convex cone}, which also can be found in \cite{Liv95A,Viana,Cas17}.

Denote $m^s$ to be the inherited Riemannian volume measure on local stable manifolds. Then by the continuity of local stable manifolds, there exists a constant  $A(\epsilon)>0$ such that $A(\epsilon)\leq m^s(W^s_\epsilon(x,\omega))$ for any $(x,\omega)\in M\times \Omega$.
In the following, for any $\omega\in\Omega$, we say that the local stable leaf $\gamma(\omega)\subset M$ having size between $\frac{A(\epsilon)}{4J^2}$ and $A(\epsilon)$ if $\gamma(\omega)$ is connected, and $m^s(\gamma(\omega))\in(\frac{A(\epsilon)}{4J^2},A(\epsilon))$, where the notation $\gamma(\omega)$ indicates that this local stable leaf is located on the fiber $M\times\{\omega\}$. Since $\gamma(\omega)$ is a connected curve, this means that the length of local stable manifold has length bounded from below and above. Applying the continuity of local stable manifolds again, there exists $\epsilon^*>0$ such that for any $\gamma(\omega)$ of size between $\frac{A(\epsilon)}{4J^2}$ and $A(\epsilon)$, there exists  $(x,\omega)\in M\times \Omega$ satisfying
\begin{equation}\label{eq def epsilon star}
W_{\epsilon^*}^s(x,\omega)\subset \gamma(\omega).
\end{equation}

For some constant $a>0$, and a local stable leaf $\gamma(\omega)$ having size between $\frac{A(\epsilon)}{4J^2}$ and $A(\epsilon)$, we define $D(a,\kappa,\gamma(\omega))$ to be the collection of all bounded and measurable function $\rho(\cdot,\omega):\gamma(\omega)\rightarrow \mathbb{R}$ satisfying the following conditions:
\begin{enumerate}[label=\textbf{(D\arabic*)}]
  \item \label{D1} $\rho(x,\omega)>0$ for $x\in\gamma(\omega)$;
  \item \label{D2} for any $x,y\in \gamma(\omega)$, $|\log\rho(x,\omega)-\log\rho(y,\omega)|\leq ad(x,y)^\kappa$.
\end{enumerate}
We note that condition \ref{D2} implies that any $\rho(\cdot,\omega)\in D(a,\kappa,\gamma(\omega))$ is continuous on $\gamma(\omega)$.  In the left of this paper, we use the notation $\rho(\omega)$ to represent $\rho(\cdot,\omega)\in D(a,\kappa,\gamma(\omega))$ for convenience.
\begin{lemma}\label{def convex cone density}
  $D(a,\kappa,\gamma(\omega))$ is a convex cone (see Definition \ref{definition convex cone} in the Appendix).
\end{lemma}
\begin{proof}[Proof of Lemma \ref{def convex cone density}]
 For any $\rho(\omega)\in D(a,\kappa,\gamma(\omega))$ and $t\in \mathbb{R}^+$, then $t\rho(x,\omega)>0$ for $x\in \gamma(\omega)$ and
 \begin{equation*}
   |\log t\rho(x,\omega)-\log t\rho(y,\omega)|=|\log\rho(x,\omega)-\log\rho(y,\omega)|\leq ad(x,y)^\kappa.
 \end{equation*}Hence $t\rho(\omega)\in D(a,\kappa,\gamma(\omega))$.

For any $\rho_i(\omega)\in D(a,\kappa,\gamma(\omega))$ and $t_i\in \mathbb{R}^+$ for $i=1,2$, it is clear that $t_1\rho_1(x,\omega)+t_2\rho_2(x,\omega)>0$ for $x\in \gamma(\omega)$. By \ref{D2}, we have
 \begin{equation*}
   e^{-ad(x,y)^\kappa}\leq  \frac{\rho_i(x,\omega)}{\rho_i(y,\omega)}\leq e^{ad(x,y)^\kappa}\mbox{ for }x,y\in\gamma(\omega)\mbox{ and }i=1,2.
 \end{equation*}As a consequence, we have
 \begin{equation*}
  e^{-ad(x,y)^\kappa} \leq \frac{t_1\rho_1(x,\omega)+t_2\rho_2(x,\omega)}{t_1\rho_1(y,\omega)+t_2\rho_2(y,\omega)}\leq e^{ad(x,y)^\kappa}\mbox{ for }x,y\in\gamma(\omega),
 \end{equation*}which is equivalent to
 \begin{equation*}
   |\log(t_1\rho_1(x,\omega)+t_2\rho_2(x,\omega))-\log (t_1\rho_1(y,\omega)+t_2\rho_2(y,\omega))|\leq ad(x,y)^\kappa.
 \end{equation*}Hence $t_1\rho_1(\omega)+t_2\rho_2(\omega)\in D(a,\kappa,\gamma(\omega))$.

 To check the last condition for convex cone, we pick any $\rho(\omega)\in \overline{D(a,\kappa,\gamma(\omega))}\cap -\overline{D(a,\kappa,\gamma(\omega))}$. Here the closure means the ``integral closure" in a weaker sense, not the closure in the topological vector space, see Def \ref{definition convex cone}. Then there exists $\rho_i(\omega)\in D(a,\kappa,\gamma(\omega))$ and $t_n^i\downarrow 0$ for $i=1,2$ such that $\rho(\omega)+t_n^1\rho_1(\omega)\in D(a,\kappa,\gamma(\omega))$ and $-\rho(\omega)+t_n^2\rho_2(\omega)\in D(a,\kappa,\gamma(\omega))$. By \ref{D1}, for any $x\in\gamma(\omega)$,  we have $\rho(x,\omega)+t_n^1\rho_1(x,\omega)>0$ and $-\rho(x,\omega)+t_n^2\rho_2(x,\omega)>0$. Letting $n\to\infty$, we arrive $\rho(x,\omega)\geq 0$ and $\rho(x,\omega)\leq 0$. Therefore, one must have $\rho(\omega)\equiv 0$. By definition, $D(a,\kappa,\gamma(\omega))$ is a convex cone.
\end{proof}

 Next, we will introduce the projective metric $d_{\gamma(\omega)}^{a,\kappa}$ on $D(a,\kappa,\gamma(\omega))$ according to Def. \ref{projective metric for cone}.
For any $\rho_1(\omega),\rho_2(\omega)\in D(a,\kappa,\gamma(\omega))$, we define
\begin{align*}
  \alpha_{\gamma(\omega)}^{a,\kappa}(\rho_1(\omega),\rho_2(\omega)) &:=\sup\{t>0:\ \rho_2(\omega)-t\rho_1(\omega)\in D(a,\kappa,\gamma(\omega))\};  \\
  \beta_{\gamma(\omega)}^{a,\kappa}(\rho_1(\omega),\rho_2(\omega)) &=\inf\{s>0:\ s\rho_1(\omega)-\rho_2(\omega)\in D(a,\kappa,\gamma(\omega))\},
\end{align*}with the convention that $\sup\emptyset=0$ and $\inf\emptyset=+\infty$. Now let us compute these two quantities. For any $t>0$ such that $\rho_2(\omega)-t\rho_1(\omega)\in D(a,\kappa,\gamma(\omega))$, then by \ref{D1}, one must have
\begin{align*}
  \rho_2(x,\omega)-t\rho_1(x,\omega) &>0\mbox{ for all }x\in \gamma(\omega),
\end{align*}which is equivalent to $t<\frac{\rho_2(x,\omega)}{\rho_1(x,\omega)}$ for all $x\in\gamma(\omega)$; and by \ref{D2} one also must have
\begin{equation*}
  e^{-ad(x,y)^\kappa}\leq \frac{\rho_2(x,\omega)-t\rho_1(x,\omega)}{\rho_2(y,\omega)-t\rho_1(y,\omega)}\leq e^{ad(x,y)^\kappa}\mbox{ for all }x,y\in\gamma(\omega),
\end{equation*}which is equivalent to
\begin{equation*}
  t\leq\inf_{x,y\in\gamma(\omega), x\not=y}\left\{\frac{\exp(ad(x,y)^\kappa)\rho_2(x,\omega)-\rho_2(y,\omega)}{\exp(ad(x,y)^\kappa)\rho_1(x,\omega)-\rho_1(y,\omega)},\frac{\exp(ad(x,y)^\kappa)\rho_2(y,\omega)-\rho_2(x,\omega)}{\exp(ad(x,y)^\kappa)\rho_1(y,\omega)-\rho_1(x,\omega)}\right\}.
\end{equation*}Therefore, we have
\begin{equation*}
  t\leq \inf\left\{\frac{\rho_2(x,\omega)}{\rho_1(x,\omega)},\frac{\exp(ad(x,y)^\kappa)\rho_2(x,\omega)-\rho_2(y,\omega)}{\exp(ad(x,y)^\kappa)\rho_1(x,\omega)-\rho_1(y,\omega)}:\ x,y\in\gamma(\omega),x\not=y\right\}.
\end{equation*}The case that $\{t>0:\rho_2(\omega)-t\rho_1(\omega)\in D(a,\kappa,\gamma(\omega))\}=\emptyset$ is equivalent to the right term of the above is equal to 0, and so $\alpha_{\gamma(\omega)}^{a,\kappa}(\rho_1(\omega),\rho_2(\omega))=\sup\emptyset=0$. Hence, we obtain
\begin{align}
 \alpha_{\gamma(\omega)}^{a,\kappa}(\rho_1(\omega),\rho_2(\omega))&=\inf\left\{\frac{\rho_2(x,\omega)}{\rho_1(x,\omega)},\frac{\exp(ad(x,y)^\kappa)\rho_2(x,\omega)-\rho_2(y,\omega)}{\exp(ad(x,y)^\kappa)\rho_1(x,\omega)-\rho_1(y,\omega)}:\ x,y\in\gamma(\omega),x\not=y\right\}\label{expression of alpha12},
 \end{align}Similarly, we can obtain
 \begin{align}\label{expression of betaak}
 \beta_{\gamma(\omega)}^{a,\kappa}(\rho_1(\omega),\rho_2(\omega))&=\sup\left\{\frac{\rho_2(x,\omega)}{\rho_1(x,\omega)},\frac{\exp(ad(x,y)^\kappa)\rho_2(x,\omega)-\rho_2(y,\omega)}{\exp(ad(x,y)^\kappa)\rho_1(x,\omega)-\rho_1(y,\omega)}:\ x,y\in\gamma(\omega),x\not=y\right\}.
\end{align}
 Now define
\begin{equation}\label{distance between density}
  d_{\gamma(\omega)}^{a,\kappa}(\rho_1(\omega),\rho_2(\omega))=\log\frac{\beta_{\gamma(\omega)}^{a,\kappa}(\rho_1(\omega),\rho_2(\omega))}{\alpha_{\gamma(\omega)}^{a,\kappa}(\rho_1(\omega),\rho_2(\omega))},
\end{equation}with the convention that $d_{\gamma(\omega)}^{a,\kappa}(\rho_1(\omega),\rho_2(\omega))=\infty$ if $\alpha_{\gamma(\omega)}^{a,\kappa}(\rho_1(\omega),\rho_2(\omega))=0$ or $\beta_{\gamma(\omega)}^{a,\kappa}(\rho_1(\omega),\rho_2(\omega))=\infty$.
By the property of projective metric (see Proposition \ref{proposition property of pm}), the followings hold:
\begin{enumerate}[label=\textbf{(P\arabic*)}]
  \item \label{P1} $d_{\gamma(\omega)}^{a,\kappa}(\rho_1(\omega),\rho_2(\omega))=d_{\gamma(\omega)}^{a,\kappa}(\rho_2(\omega),\rho_1(\omega))$;
  \item \label{P2} $d_{\gamma(\omega)}^{a,\kappa}(\rho_1(\omega),\rho_2(\omega))\leq d_{\gamma(\omega)}^{a,\kappa}(\rho_1(\omega),\rho_3(\omega))+d_{\gamma(\omega)}^{a,\kappa}(\rho_3(\omega),\rho_2(\omega))$;
  \item \label{P3} $ d_{\gamma(\omega)}^{a,\kappa}(\rho_1(\omega),\rho_2(\omega))=0$ if and only if there exists a constant $t\in\mathbb{R}^+$ such that $\rho_1(\omega)=t\rho_2(\omega)$.
\end{enumerate}Note that \ref{P2} and \ref{P3} implies that
\begin{equation}\label{P2P3}
  d_{\gamma(\omega)}^{a,\kappa}(\rho_1(\omega),\rho_2(\omega))=d_{\gamma(\omega)}^{a,\kappa}(t_1\rho_1(\omega),t_2\rho_2(\omega)) \mbox{ for any $t_1,t_2\in\mathbb{R}^+.$}
\end{equation}

It is also convenient to introduce the following convex cone. Denote $D_+(\gamma(\omega))$ by the collection of all bounded and measurable functions $\zeta(\omega):\gamma(\omega)\rightarrow\mathbb{R}$ such that $\zeta(x,\omega)>0$ for $x\in \gamma(\omega)$. It is clear that $D_+(\gamma(\omega))$ is a convex cone. For any $\zeta_1(\omega),\zeta_2(\omega)\in D_+(\gamma(\omega))$, we define
\begin{align}
  \alpha_{+,\gamma(\omega)}(\zeta_1(\omega),\zeta_2(\omega)) & =\sup\{t>0:\ \zeta_2(\omega)-t\zeta_1(\omega)\in D_+(\gamma(\omega))\};\\
  \beta_{+,\gamma(\omega)}(\zeta_1(\omega),\zeta_2(\omega)) &=\inf\{s>0:\ s\zeta_1(\omega)-\zeta_2(\omega)\in D_+(\gamma(\omega))\},
\end{align}with the convention that $\sup\emptyset=0$ and $\inf\emptyset=+\infty$, and
 the projective metric on $D_+(\gamma(\omega))$  by
\begin{align}
  d_{+,\gamma(\omega)}(\zeta_1(\omega),\zeta_2(\omega)) &:=\log \frac{\beta_{+,\gamma(\omega)}(\zeta_1(\omega),\zeta_2(\omega))}{\alpha_{+,\gamma(\omega)}(\zeta_1(\omega),\zeta_2(\omega))}\label{d+gammaomegarho1rho2},
\end{align} with the convention that $d_{+,\gamma(\omega)}(\zeta_1(\omega),\zeta_2(\omega)) =\infty$ if $\alpha_{+,\gamma(\omega)}(\zeta_1(\omega),\zeta_2(\omega))$ $=0$ or $\beta_{+,\gamma(\omega)}(\zeta_1(\omega),\zeta_2(\omega))=\infty.$
By computation, we have
\begin{align}
  \alpha_{+,\gamma(\omega)}(\zeta_1(\omega),\zeta_2(\omega))&=\inf\left\{\frac{\zeta_2(x,\omega)}{\zeta_1(x,\omega)}:\ x\in \gamma(\omega)\right\},\label{alpha +}\\
   \beta_{+,\gamma(\omega)}(\zeta_1(\omega),\zeta_2(\omega))&=\sup\left\{\frac{\zeta_2(x,\omega)}{\zeta_1(x,\omega)}:\ x\in \gamma(\omega)\right\}.\label{beta +}
\end{align}
It is clear that $D(a,\kappa,\gamma(\omega))\subset D_+(\gamma(\omega))$, and $d_{+,\gamma(\omega)}(\rho_1(\omega),\rho_2(\omega))\leq d_{\gamma(\omega)}^{a,\kappa}(\rho_1(\omega),\rho_2(\omega))$ for any $\rho_1(\omega),\rho_2(\omega)\in D(a,\kappa,\gamma(\omega))$.


By using density functions on $\gamma(\omega)$, we can define the corresponding density function on the pullback of $\gamma(\omega)$. For any $\omega\in\Omega$, and a  local stable leaf $\gamma(\omega)$ having size between $\frac{A(\epsilon)}{4J^2}$ and $A(\epsilon)$, since the local stable manifold is a curve, we can divide $f_\omega^{-1}\gamma(\omega)$ into connected local stable manifolds with size between $\frac{A(\epsilon)}{4J^2}$ and $A(\epsilon)$, named $\gamma_i(\theta^{-1}\omega)$ for $i$ belonging to a finite index set such that $\gamma_i(\theta^{-1}\omega)\cap \gamma_j(\theta^{-1}\omega)=\partial\gamma_i(\theta^{-1}\omega)\cap \partial \gamma_j(\theta^{-1}\omega)$ for $i\not=j$. By using any continuous density functions $\rho(\omega)$ on $\gamma(\omega)$, we define the corresponding density function $\rho_i(\theta^{-1}\omega)$ on $\gamma_i(\theta^{-1}\omega)$ by
\begin{align}
  \rho_i(x,\theta^{-1}\omega)&:=\frac{|\det D_xf_{\theta^{-1}\omega}|_{E^s(x,\theta^{-1}\omega)}|}{|\det D_xf_{\theta^{-1}\omega}|}\rho(f_{\theta^{-1}\omega}x,\omega)\mbox{ for }x\in\gamma_i(\theta^{-1}\omega)\label{definition of rhoi}.
\end{align} For any bounded and measurable function $\varphi:M\rightarrow\mathbb{R}$, by changing of variable, we have
\begin{equation}\label{Lphirho=phirhoi}
  \begin{split}
  &\ \ \ \ \int_{\gamma(\omega)}(L_{\theta^{-1}\omega}\varphi)(y)\rho(y,\omega)dm_{\gamma(\omega)}(y)\\
  &=\sum_{i}\int_{f_{\theta^{-1}\omega}\gamma_i(\theta^{-1}\omega)}\frac{\varphi((f_{\theta^{-1}\omega})^{-1}y)}{|\det D_{(f_{\theta^{-1}\omega})^{-1}y}f_{\theta^{-1}\omega}|}\cdot \rho(y,\omega)dm_{\gamma(\omega)}(y)\\
  &=\sum_{i}\int_{\gamma_i(\theta^{-1}\omega)}\frac{\varphi(x)}{|\det D_xf_{\theta^{-1}\omega}|}\cdot \rho(f_{\theta^{-1}\omega}x,\omega)\cdot |\det D_xf_{\theta^{-1}\omega}|_{E^s(x,\theta^{-1}\omega)}|dm_{\gamma_i(\theta^{-1}\omega)}(x)\\
  &=\sum_{i}\int_{\gamma_i(\theta^{-1}\omega)}\varphi(x)\rho_i(x,\theta^{-1}\omega)dm_{\gamma_i(\theta^{-1}\omega)}(x).
  \end{split}
\end{equation}

By using density function on $\gamma(\omega)$, we also can define the corresponding density function on the holonomy image of $\gamma(\omega)$.  Given pair of local stable leaves $\gamma(\omega)$ and $\tilde{\gamma}(\omega)$, with the help of holonomy map $\psi_\omega^u:\tilde{\gamma}(\omega)\rightarrow\gamma(\omega)$ induced by the local unstable manifolds, for every continuous density function $\rho(\omega)$ on $\gamma(\omega)$, we associate the density $\tilde{\rho}(\omega)$ on $\tilde{\gamma}(\omega)$ by
\begin{equation}\label{definition of tilde rho}
  \tilde{\rho}(x,\omega)=\rho(\psi_\omega(x),\omega)\cdot Jac(\psi_\omega^u)(x)\mbox{ for }x\in\tilde{\gamma}(\omega).
\end{equation}
By the Radon-Nikodym theorem, we have
\begin{equation}\label{eq change variable}
  \int_{\tilde{\gamma}(\omega)}\tilde{\rho}(x,\omega)dm_{\tilde{\gamma}(\omega)}(x)=\int_{\gamma(\omega)}\rho(y,\omega)dm_{\gamma(\omega)}(y).
\end{equation} We can define the distance between $\tilde{\gamma}(\omega)$ and $\gamma(\omega)$ by
\begin{equation}\label{distance between local stable}
  d_u(\tilde{\gamma}(\omega),\gamma(\omega)):=\sup\{d(x,\psi_\omega^u(x)):\ x\in\tilde{\gamma}(\omega)\},
\end{equation}where we use subscript $u$ to indicate that the distance is induced by the local unstable manifolds.

Recall that constants $K_1,K_2,\kappa,\kappa_1,a_0$ and $\nu_0$ are picked at the beginning of Section \ref{section 4}. Let $a_1$ be any number  such that
\begin{equation}\label{assumption on a1}
 \frac{K_1+K_2}{1-e^{-\lambda\kappa_1\nu_0}}<a_1.
\end{equation}Let $a\in\mathbb{R}$ be any number such that
\begin{equation}\label{assumption on a}
  a_1a_0^{\kappa_1}+a_0<\frac{a}{2}.
\end{equation}Let $D(a_1,\kappa,\gamma(\omega))$, $D(\frac{a}{2},\kappa,\gamma(\omega))$, and $D(\frac{a}{2},\kappa_1\nu_0,\gamma(\omega))$ be convex cones defined just like $D(a,\kappa,\gamma(\omega))$. The relations between $\rho(\omega)$ and $\tilde{\rho}(\omega)$, $\rho(\omega)$ and $\rho_i(\theta^{-1}\omega)$ are given in the following lemma.
\begin{lemma}\label{rhoi}
There are $\lambda_1=\lambda_1(a_1,\kappa)>0$ and $\Lambda_1=\Lambda_1(\lambda_1,a)<1$ such that
\begin{enumerate}
  \item if $\rho(\omega)\in D(a_1,\kappa,\gamma(\omega))$, then $\rho_i(\theta^{-1}\omega)\in D(e^{-\lambda_1}a_1,\kappa,\gamma_i(\theta^{-1}\omega))\subset D(a_1,\kappa,\gamma_i(\theta^{-1}\omega))$;
  \item if $\rho(\omega)\in D(\frac{a}{2},\kappa,\gamma(\omega))$, then $\rho_i(\theta^{-1}\omega)\in D(e^{-\lambda_1}\frac{a}{2},\kappa,\gamma_i(\theta^{-1}\omega))\subset D(\frac{a}{2},\kappa,\gamma_i(\theta^{-1}\omega))$;
  \item if $\rho(\omega)\in D(\frac{a}{2},\kappa_1\nu_0,\gamma(\omega))$, then $\rho_i(\theta^{-1}\omega)\in D(e^{-\lambda_1}\frac{a}{2},\kappa_1\nu_0,\gamma_i(\theta^{-1}\omega))\subset D(\frac{a}{2},\kappa_1\nu_0,\gamma_i(\theta^{-1}\omega))$;
   \item if $\rho(\omega)\in D(a,\kappa,\gamma(\omega))$, then $\rho_i(\theta^{-1}\omega)\in D(e^{-\lambda_1}a,\kappa,\gamma_i(\theta^{-1}\omega))\subset D(a,\kappa,\gamma_i(\theta^{-1}\omega))$;
\end{enumerate}furthermore,
\begin{enumerate}
  \item[(5)] for any $\rho(\omega),\varsigma(\omega)\in D(a,\kappa,\gamma(\omega))$, we have
      \begin{equation}\label{eq d leq Lamda1 d}
        d_{\gamma_i(\theta^{-1}\omega)}^{a,\kappa}(\rho_i(\theta^{-1}\omega),\varsigma_i(\theta^{-1}\omega))\leq \Lambda_1d_{\gamma(\omega)}^{a,\kappa}(\rho(\omega),\varsigma(\omega)),
      \end{equation} where $\rho_i(\theta^{-1}\omega)$ and $\varsigma_i(\theta^{-1}\omega)$ are density functions defined as in \eqref{definition of   rhoi} on $\tilde{\gamma}_i(\theta^{-1}\omega)$, and $d_{\gamma_i(\theta^{-1}\omega)}^{a,\kappa}$ (resp. $ d_{\gamma(\omega)}^{a,\kappa}$) is the projective metric on $D(a,\kappa,\gamma_i(\theta^{-1}\omega))$ (resp. $D(a,\kappa,\gamma(\omega))$).
\end{enumerate}Moreover,
\begin{equation}\label{tilder rho}
  \mbox{if } \rho(\omega)\in D(a_1,\kappa_1,\gamma(\omega)), \mbox{then } \tilde{\rho}(\omega)\in D(a/2,\kappa_1\nu_0,\tilde{\gamma}(\omega))\subset D(a/2,\kappa,\tilde{\gamma}(\omega)),
\end{equation}where $\tilde{\rho}(\omega)$ is defined in \eqref{definition of tilde rho}.
\end{lemma}
\begin{proof}[proof of lemma \ref{rhoi}]
We first prove (1). Let $\rho(\omega)\in D(a_1,\kappa,\gamma(\omega))$. Recall that $\rho_i(\theta^{-1}\omega)$ is defined in \eqref{definition of rhoi} on $\gamma_i(\theta^{-1}\omega)\subset f_\omega^{-1}\gamma(\omega)$.
Clearly, $\rho_i(x,\theta^{-1}\omega)>0$ for all $x\in\gamma_i(\theta^{-1}\omega)$. By \eqref{assumption on a1}, we can pick $\lambda_1>0$ closing to $0$ so that
\begin{equation}\label{pick lambda1}
  a_1>\frac{K_1+K_2}{e^{-\lambda_1}-e^{-\lambda\kappa_1\nu_0}}\overset{\eqref{mu+nuleq nu0}}> \frac{K_1+K_2}{e^{-\lambda_1}-e^{-\lambda \kappa}}>0.
\end{equation} Then for any $x,y\in\gamma_i(\theta^{-1}\omega)$
\begin{align*}
   &\ \ \ \ \left|\log \rho_i(x,\theta^{-1}\omega)-\log \rho_i(y,\theta^{-1}\omega)\right|\\
   &\leq \left|\log\rho(f_{\theta^{-1}\omega}x,\omega)-\log\rho(f_{\theta^{-1}\omega}y,\omega)\right|+\left|\log|\det D_xf_{\theta^{-1}\omega}|_{E^s(x,\theta^{-1}\omega)}|-\log|\det D_yf_{\theta^{-1}\omega}|_{E^s(y,\theta^{-1}\omega)}|\right|\\
   &\ \ \ \ \ \ \ \ \ +\left|\log|\det D_xf_{\theta^{-1}\omega}|-\log|\det D_yf_{\theta^{-1}\omega}|\right|\\
   &\overset{\eqref{eq def K1},\eqref{Lip constant of log det Dxfomega}}\leq a_1d(f_{\theta^{-1}\omega}x,f_{\theta^{-1}\omega}y)^{\kappa}+K_1d(x,y)+K_2d(x,y)\\
   &\leq a_1 e^{-\lambda\kappa}d(x,y)^\kappa+(K_1+K_2)d(x,y)\\
   &\overset{\eqref{pick lambda1}}\leq a_1 e^{-\lambda_1}d(x,y)^\kappa,
\end{align*} This proves part $(1)$. Similar proof can be applied to (2), (3) and (4) by using the fact
\begin{equation*}
   a>\frac{a}{2}>a_1\overset{\eqref{pick lambda1}}>\frac{K_1+K_2}{e^{-\lambda_1}-e^{-\lambda\kappa_1\nu_0}}> \frac{K_1+K_2}{e^{-\lambda_1}-e^{-\lambda \kappa}}>0.
\end{equation*}

Next, we prove (5).
By (4), we have a linear operator that maps from convex cone $D(a,\kappa,\gamma(\omega))$ to convex cone $ D(a,\kappa,\gamma_i(\theta^{-1}\omega))$ given by $\rho(\omega)\mapsto \rho_i(\theta^{-1}\omega)$.
By Birkhoff's inequality (Proposition \ref{birkhoff inequality}), if
\begin{equation*}
  R_i(\theta^{-1}\omega):=\sup\{d_{\gamma_i(\theta^{-1}\omega)}^{a,\kappa}(\rho_i(\theta^{-1}\omega),\varsigma_i(\theta^{-1}\omega)):\ \rho(\omega),\varsigma(\omega)\in D(a,\kappa,\gamma(\omega))\}<\infty,
\end{equation*}then
\begin{equation}\label{dgammailessthan dgamma}
  d_{\gamma_i(\theta^{-1}\omega)}^{a,\kappa}(\rho_i(\theta^{-1}\omega),\varsigma_i(\theta^{-1}\omega))\leq (1-e^{-R_i(\theta^{-1}\omega)})d_{\gamma(\omega)}^{a,\kappa}(\rho(\omega),\varsigma(\omega)).
\end{equation}
To estimate $R_i(\theta^{-1}\omega)$, it suffices to estimate the diameter of $D(e^{-\lambda_1}a,\kappa,\gamma_i(\theta^{-1}\omega))$ in $D(a,\kappa,\gamma_i(\theta^{-1}\omega))$ under the projective metric $d_{\gamma_i(\theta^{-1}\omega)}^{a,\kappa}$, since $\rho_i(\theta^{-1}\omega),\varsigma_i(\theta^{-1}\omega)\in D(e^{-\lambda_1}a,\kappa,\gamma_i(\theta^{-1}\omega)).$

Pick any  $\zeta_1(\theta^{-1}\omega),\ \zeta_2(\theta^{-1}\omega)\in D(e^{-\lambda_1}a,\kappa,\gamma_i(\theta^{-1}\omega))$, then $\zeta_1(\theta^{-1}\omega),\zeta_2(\theta^{-1}\omega)\in D_+(\gamma_i(\theta^{-1}\omega))$ automatically. For any $x,y\in\gamma_i(\theta^{-1}\omega)$, $x\not=y$,
\begin{align*}
  \frac{\exp(ad(x,y)^\kappa)-\zeta_2(y,\theta^{-1}\omega)/\zeta_2(x,\theta^{-1}\omega)}{\exp(ad(x,y)^\kappa)-\zeta_1(y,\theta^{-1}\omega)/\zeta_1(x,\theta^{-1}\omega)} &\geq   \frac{\exp(ad(x,y)^\kappa)-\exp(e^{-\lambda_1}ad(x,y)^\kappa)}{\exp(ad(x,y)^\kappa)-\exp(-e^{-\lambda_1}ad(x,y)^\kappa)}\\
  &\geq \tau_1,
\end{align*}where $$\tau_1=\inf\left\{\frac{z-z^{\exp(-\lambda_1)}}{z-z^{-\exp(-\lambda_1)}}:\ z>1\right\}=\lim_{z\to 1}\frac{z-z^{\exp(-\lambda_1)}}{z-z^{-\exp(-\lambda_1)}}=\frac{1-\exp(-\lambda_1)}{1+\exp(-\lambda_1)}\in(0,1).$$ Comparing
\begin{align*}
  &\ \ \ \ \alpha_{\gamma_i(\theta^{-1}\omega)}^{a,\kappa}(\zeta_1(\theta^{-1}\omega),\zeta_2(\theta^{-1}\omega))\\
  &=\inf\left\{\frac{\zeta_2(x,\theta^{-1}\omega)}{\zeta_1(x,\theta^{-1}\omega)},\frac{\exp(ad(x,y)^\kappa)\zeta_2(x,\theta^{-1}\omega)-\zeta_2(y,\theta^{-1}\omega)}{\exp(ad(x,y)^\kappa)\zeta_1(x,\theta^{-1}\omega)-\zeta_1(y,\theta^{-1}\omega)}:\ x,y\in\gamma_i(\theta^{-1}\omega),x\not=y\right\}
\end{align*} and
\begin{equation*}
  \alpha_{+,\gamma_i(\theta^{-1}\omega)}(\zeta_1(\theta^{-1}\omega),\zeta_2(\theta^{-1}\omega))\overset{\eqref{alpha +}}=\inf\left\{\frac{\zeta_2(x,\theta^{-1}\omega)}{\zeta_1(x,\theta^{-1}\omega)}:\ x\in\gamma_i(\theta^{-1}\omega)\right\},
\end{equation*} we have
\begin{align*}
   & \ \ \ \ \frac{\alpha_{\gamma_i(\theta^{-1}\omega)}^{a,\kappa}(\zeta_1(\theta^{-1}\omega),\zeta_2(\theta^{-1}\omega))}{\alpha_{+,\gamma_i(\theta^{-1}\omega)}(\zeta_1(\theta^{-1}\omega),\zeta_2(\theta^{-1}\omega))}\\
   &\geq \inf\left\{1, \frac{\exp(ad(x,y)^\kappa)-\zeta_2(y,\theta^{-1}\omega)/\zeta_2(x,\theta^{-1}\omega)}{\exp(ad(x,y)^\kappa)-\zeta_1(y,\theta^{-1}\omega)/\zeta_1(x,\theta^{-1}\omega)}:\ x,y\in\gamma_i(\theta^{-1}\omega),x\not=y\right\}\\
   &\geq \tau_1.
\end{align*}
Similarly, let $$\tau_2=\sup\left\{\frac{z-z^{-\exp(-\lambda_1)}}{z-z^{\exp(-\lambda_1)}}:\ z>1\right\}=\lim_{z\to 1}\frac{z-z^{-\exp(-\lambda_1)}}{z-z^{\exp(-\lambda_1)}}=\frac{1+\exp(-\lambda_1)}{1-\exp(-\lambda_1)}\in (1,\infty),$$ we have
\begin{equation*}
  \beta_{\gamma_i(\theta^{-1}\omega)}^{a,\kappa}(\zeta_1(\theta^{-1}\omega),\zeta_2(\theta^{-1}\omega))\leq \tau_2\beta_{+,\gamma_i(\theta^{-1}\omega)}(\zeta_1(\theta^{-1}\omega),\zeta_2(\theta^{-1}\omega)),
\end{equation*}where $ \beta_{\gamma_i(\theta^{-1}\omega)}^{a,\kappa}(\zeta_1(\theta^{-1}\omega),\zeta_2(\theta^{-1}\omega))$ and $\beta_{+,\gamma_i(\theta^{-1}\omega)}(\zeta_1(\theta^{-1}\omega),\zeta_2(\theta^{-1}\omega))$ are given as \eqref{expression of betaak} and \eqref{beta +}. Thus, we conclude
\begin{equation}
  d_{\gamma_i(\theta^{-1}\omega)}^{a,\kappa}(\zeta_1(\theta^{-1}\omega),\zeta_2(\theta^{-1}\omega))\leq d_{+,\gamma_i(\theta^{-1}\omega)}(\zeta_1(\theta^{-1}\omega),\zeta_2(\theta^{-1}\omega))+\log(\tau_2/\tau_1).\label{dgammaileqdgamma+i}
\end{equation}
Next, we estimate $d_{+,\gamma_i(\omega)}(\zeta_1(\theta^{-1}\omega),\zeta_2(\theta^{-1}\omega))$ for $\zeta_1(\theta^{-1}\omega),\zeta_2(\theta^{-1}\omega)\in D(e^{-\lambda_1}a,\alpha,\gamma_i(\theta^{-1}\omega))$. By the third property of projective metric \ref{P3} and normalizing $\zeta_1(\theta^{-1}\omega)$, $\zeta_2(\theta^{-1}\omega)$, we can assume that
\begin{equation*}
  \int_{\gamma_i(\theta^{-1}\omega)}\zeta_1(x,\theta^{-1}\omega)dm_{\gamma_i(\theta^{-1}\omega)}(x)=\int_{\gamma_i(\theta^{-1}\omega)}\zeta_2(x,\theta^{-1}\omega)dm_{\gamma_i(\theta^{-1}\omega)}(x)=1,
\end{equation*}without changing $d_{+,\gamma_i(\omega)}(\zeta_1(\theta^{-1}\omega),\zeta_2(\theta^{-1}\omega))$. Therefore, there exists points $y_1,y_2\in \gamma_i(\theta^{-1}\omega)$ such that $\zeta_1(y_1,\theta^{-1}\omega)=1$ and $\zeta_2(y_2,\theta^{-1}\omega)=1$ by continuity.  Then \ref{D2} of $D(e^{-\lambda_1}a,\kappa,\gamma_i(\theta^{-1}\omega))$ implies for all $x\in\gamma_i(\theta^{-1}\omega)$
\begin{equation*}
  \frac{\zeta_2(x,\theta^{-1}\omega)}{\zeta_1(x,\theta^{-1}\omega)}=\frac{\zeta_2(x,\theta^{-1}\omega)/\zeta_1(y_1,\theta^{-1}\omega)}{\zeta_1(x,\theta^{-1}\omega)/\zeta_2(y_2,\theta^{-1}\omega)}\geq \frac{\exp(-e^{-\lambda
  _1}a(diam \gamma_i(\theta^{-1}\omega))^\kappa)}{\exp(e^{-\lambda
  _1}a(diam \gamma_i(\theta^{-1}\omega))^\kappa)}\geq e^{-2a},
\end{equation*}where $diam$ means the diameter. It follows that $\alpha_{+,\gamma_i(\theta^{-1}\omega) }(\zeta_1(\theta^{-1}\omega),\zeta_2(\theta^{-1}\omega))\geq e^{-2a}$. Similarly, $\beta_{+,\gamma_i(\theta^{-1}\omega) }(\zeta_1(\theta^{-1}\omega),\zeta_2(\theta^{-1}\omega))\leq e^{2a}$. Therefore, $ d_{+,\gamma_i(\theta^{-1}\omega)}(\zeta_1(\theta^{-1}\omega),\zeta_2(\theta^{-1}\omega))\leq 4a$. So we have
\begin{align}
  d_{\gamma_i(\theta^{-1}\omega)}^{a,\kappa}(\zeta_1(\theta^{-1}\omega),\zeta_2(\theta^{-1}\omega))&\overset{\eqref{dgammaileqdgamma+i}}\leq d_{+,\gamma_i(\theta^{-1}\omega)}(\zeta_1(\theta^{-1}\omega),\zeta_2(\theta^{-1}\omega))+\log(\tau_2/\tau_1)\nonumber\\
  &\leq 4a+\log(\tau_2/\tau_1)\label{bound of d gamma rho1rho2}.
\end{align}Since $\zeta_1(\theta^{-1}\omega),\zeta_2(\theta^{-1}\omega)\in D(e^{-\lambda_1}a,\kappa,\gamma_i(\theta^{-1}\omega))$ are arbitrary chosen, $R_i(\theta^{-1}\omega)\leq 4a+\log(\tau_2/\tau_1)$.  By $(\ref{dgammailessthan dgamma})$, let $\Lambda_1=1-e^{-(4a+\log(\tau_2/\tau_1))}$. The proof of (5) is complete.

Finally, let's prove (\ref{tilder rho}). Let $\rho(\omega)\in D(a_1,\kappa_1,\gamma(\omega))$ be arbitrarily chosen, and $\tilde{\rho}(\omega)$ on $\tilde{\gamma}(\omega)$ be defined by \eqref{definition of tilde rho}.
It is clear that $\tilde{\rho}(x,\omega)>0$. Moreover, for any $x,y\in\tilde{\gamma}(\omega)$, we have
\begin{align*}
   &\ \ \ \ \ |\log\tilde{\rho}(x,\omega)-\log \tilde{\rho}(y,\omega)|\\
   &\leq |\log \rho(\psi_\omega(x),\omega)-\log \rho(\psi_\omega(y),\omega)|+|\log Jac(\psi_\omega)(x)-\log Jac(\psi_\omega)(y)|\\
   &\leq a_1d(\psi_\omega(x),\psi_\omega(y))^{\kappa_1}+a_0d(x,y)^{\nu_0}\\
   &\leq a_1a_0^{\kappa_1}d(x,y)^{\kappa_1\nu_0}+a_0d(x,y)^{\nu_0}\\
   &\leq (a_1a_0^{\kappa_1}+a_0)d(x,y)^{\kappa_1\nu_0}\\
   &\overset{\eqref{assumption on a}}\leq \frac{a}{2}d(x,y)^{\kappa_1\nu_0}.
\end{align*}So $\tilde{\rho}(\omega)\in D(a/2,\kappa_1\nu_0,\tilde{\gamma}(\omega)).$ The proof of Lemma \ref{rhoi} is complete.
\end{proof}
\begin{remark}\label{remark4.1}
 We can apply the proof of statement (1) inductively to show the following conclusion. For any local stable leaf $\gamma(\omega)$ having size between $\frac{A(\epsilon)}{4J^2}$ and $A(\epsilon)$, any $\rho(\omega)\in D(a_1,\kappa,\gamma(\omega))$, and any local stable leaf $\gamma^*(\theta^{-n}\omega)\subset f_\omega^{-n}\gamma(\omega)$, we can define $\rho^*(\theta^{-n}\omega)$ on $\gamma^*(\theta^{-n}\omega)$ by
  \begin{equation*}
    \rho^*(x,\theta^{-n}\omega)=\frac{|\det D_xf_{\theta^{-n}\omega}^n|_{E^s(x,\theta^{-n}\omega)}|}{|\det D_xf_{\theta^{-n}\omega}^n|}\rho(f_{\theta^{-n}\omega}^nx,\omega)\mbox{ for }x\in\gamma^*(\theta^{-n}\omega).
  \end{equation*}Then $\rho^*(\theta^{-n}\omega)\in D(e^{-\lambda_1}a_1,\kappa,\gamma^*(\theta^{-n}\omega))$.
\end{remark}

From now on, we fix parameters $a_1$ and $a$ in \eqref{assumption on a1} and \eqref{assumption on a}, and $\lambda_1,\ \Lambda_1$ in Lemma \ref{rhoi}. Now, we use convex cones $D(a_1,\kappa_1,\gamma(\omega))$, $D(\frac{a}{2},\kappa,\gamma(\omega))$ and $D(a,\kappa,\gamma(\omega))$ to construct the convex cone of observables on the fiber $\{\omega\}$. Let $b>1$ and $c>0$ be parameters to be determined later. Recall that $\nu$ is the constant picked at the beginning of Sec. \ref{section 4}.
For any $\omega\in\Omega$, we define $C_\omega(b,c,\nu)$ to be the collection of all bounded measurable functions $\varphi:M\rightarrow \mathbb{R}$ satisfying:
  \begin{enumerate}[label=\textbf{(C\arabic*)}]
    \item \label{C1} $\int_{\gamma(\omega)}\varphi(x)\rho(x,\omega)dm_{\gamma(\omega)}(x)>0$ for every local stable submanifold $\gamma(\omega)\subset M$ size between $\frac{A(\epsilon)}{4J^2}$ and $A(\epsilon)$, and every $\rho(\omega)\in D(a/2,\kappa,\gamma(\omega))$ satisfying $\int_{\gamma(\omega)}\rho(x,\omega)dm_{\gamma(\omega)}(x)=1$;
    \item \label{C2} $|\log\int_{\gamma(\omega)}\varphi(x)\rho(x,\omega)dm_{\gamma(\omega)}(x)-\log\int_{\gamma(\omega)}\varphi(x)\varsigma(x,\omega)dm_{\gamma(\omega)}(x)|\leq b\cdot d_{\gamma(\omega)}^{a,\kappa}(\rho(\omega),\varsigma(\omega))$ for every local stable submanifold $\gamma(\omega)\subset M$ having size between $\frac{A(\epsilon)}{4J^2}$ and $A(\epsilon)$, any $\rho(\omega)$, $\varsigma(\omega)\in D(a/2,\kappa,\gamma(\omega))\subset D(a,\kappa,\gamma(\omega))$ satisfying $\int_{\gamma(\omega)}\rho(x,\omega)dm_{\gamma(\omega)}(x)=\int_{\gamma(\omega)}\varsigma(x,\omega)dm_{\gamma(\omega)}(x)=1$.
    \item\label{C3}$|\log\int_{\gamma(\omega)}\varphi(x)\rho(x,\omega)dm_{\gamma(\omega)}(x)-\log\int_{\tilde{\gamma}(\omega)}\varphi(x)\tilde{\rho}(x,\omega)dm_{\tilde{\gamma}(\omega)}(x)|\leq cd_u(\gamma(\omega),\tilde{\gamma}(\omega))^\nu$ for every pair of local stable leaves $\gamma(\omega),\tilde{\gamma}(\omega)\subset M$ having size between $\frac{A(\epsilon)}{4J^2}$ and $A(\epsilon)$, any $\rho(\omega)\in D(a_1,\kappa_1,\gamma(\omega))$ and $\tilde{\rho}(\omega) $ corresponding to $\rho(\omega)$ defined as $(\ref{definition of tilde rho})$.
  \end{enumerate}
\begin{remark}
Parameters $a$, $a_1$, $\kappa_1$, $b$ and $c$ are constructed to prove the contraction of the transfer operator on the convex cone of observables. We just need to guarantee that all auxiliary parameters only depend on $\kappa$ and $\nu$.
\end{remark}
\begin{remark}\label{nonegative function s C2 auto}
  Note that \ref{C2} is automatically fulfilled if $\varphi$ is nonnegative. In fact, notice that $$\int_{\gamma(\omega)}\rho(x,\omega)dm_{\gamma(\omega)}(x)=\int_{\gamma(\omega)}\varsigma(x,\omega)dm_{\gamma(\omega)}(x)=1,$$
  then we have
  \begin{equation*}
    \inf_{y\in \gamma(\omega)}\left\{\frac{\rho(y,\omega)}{\varsigma(y,\omega)}\right\}\leq 1,
  \end{equation*}otherwise $\int_{\gamma(\omega)}\rho(\omega)dm_{\gamma(\omega)}>\int_{\gamma(\omega)}\varsigma(\omega)dm_{\gamma(\omega)}$, a contradiction.  So  we have
  \begin{align}
    \frac{\rho(x,\omega)}{\varsigma(x,\omega)}&\leq \sup_{x\in \gamma(\omega)}\left\{\frac{\rho(x,\omega)}{\varsigma(x,\omega)}\right\}/\inf_{y\in \gamma(\omega)}\left\{\frac{\rho(y,\omega)}{\varsigma(y,\omega)}\right\}=\exp(d_{+,\gamma(\omega)}(\rho(\omega),\varsigma(\omega)))\nonumber\\
    &\leq \exp(d_{\gamma(\omega)}^{a,\kappa}(\rho(\omega),\varsigma(\omega)))\label{rho over zeta},
  \end{align}where $d_{+,\gamma(\omega)}$ is the projective metric on $D_+(\gamma(\omega))$ defined as \eqref{d+gammaomegarho1rho2}. Switch $\rho$ and $\varsigma$, we get $\frac{\varsigma(x,\omega)}{\rho(x,\omega)}\leq \exp(d_{\gamma(\omega)}^{a,\kappa}(\rho(\omega),\varsigma(\omega)))$. Therefore, we have
  \begin{equation*}
  e^{-d_{\gamma(\omega)}^{a,\kappa}(\rho(\omega),\varsigma(\omega))}  \leq \frac{\int_{\gamma(\omega)}\varphi(x)\rho(x,\omega)dm_{\gamma(\omega)}(x)}{\int_{\gamma(\omega)}\varphi(x)\varsigma(x,\omega)dm_{\gamma(\omega)}(x)}\leq e^{d_{\gamma(\omega)}^{a,\kappa}(\rho(\omega),\varsigma(\omega))}.
  \end{equation*}
  $(C2) $ is a consequence of the above inequality as long as $b>1$.

Besides, it is clear that positive constant functions belong to $C_\omega(b,c,\nu)$ for all $\omega\in\Omega$.
\end{remark}
\begin{lemma}\label{Cbcv convex cone}
 For each $\omega\in\Omega$, $C_\omega(b,c,\nu)$ is a convex cone (see Def. \ref{definition convex cone} in the Appendix).
\end{lemma}
\begin{proof}[Proof of Lemma \ref{Cbcv convex cone}]
For any $\varphi \in C_\omega(b,c,\nu)$, and $t>0$, it is clear that $t\varphi \in C_\omega(b,c,\nu)$ .

Now, we prove the convexity, i.e., $\varphi_1,\varphi_2\in C_\omega(b,c,\nu)$ and $t_1,t_2>0$,  we are going to prove $t_1\varphi+t_2\varphi_2\in C_\omega(b,c,\nu).$ \ref{C1} is clearly fulfilled. For every local stable submanifold $\gamma(\omega)\subset M$ having size between $\frac{A(\epsilon)}{4J^2}$ and $A(\epsilon)$, any $\rho(\omega)$, $\varsigma(\omega)\in D(a/2,\kappa,\gamma(\omega))\subset D(a,\kappa,\gamma(\omega))$ satisfying $\int_{\gamma(\omega)}\rho(x,\omega)dm_{\gamma(\omega)}(x)=\int_{\gamma(\omega)}\varsigma(x,\omega)dm_{\gamma(\omega)}(x)=1$, by \ref{C2}, we have
\begin{equation*}
  e^{-bd_{\gamma(\omega)}^{a,\kappa}(\rho(\omega),\varsigma(\omega))}\leq \frac{\int_{\gamma(\omega)}\varphi_i(x)\rho(x,\omega)dm_{\gamma(\omega)}(x)}{\int_{\gamma(\omega)}\varphi_i(x)\varsigma(x,\omega)dm_{\gamma(\omega)}(x)}\leq e^{bd_{\gamma(\omega)}^{a,\kappa}(\rho(\omega),\varsigma(\omega))}
\end{equation*} The above implies that
\begin{equation*}
  e^{-bd_{\gamma(\omega)}^{a,\kappa}(\rho(\omega),\varsigma(\omega))}\leq \frac{\int_{\gamma(\omega)}\left(t_1\varphi_1(x)+t_2\varphi_2(x)\right)\rho(x,\omega)dm_{\gamma(\omega)}}{\int_{\gamma(\omega)}\left(t_1\varphi_1(x)+t_2\varphi_2(x)\right)\varsigma(x,\omega)dm_{\gamma(\omega)}}\leq e^{bd_{\gamma(\omega)}(\rho(\omega),\varsigma(\omega))}.
\end{equation*}So \ref{C2} is verified. Similarly, \ref{C3} can  also be verified. Therefore, $t_1\varphi+t_2\varphi_2\in C_\omega(b,c,\nu).$

Now, we prove that $-\overline{C_\omega(b,c,\nu)}\cap \overline{C_\omega(b,c,v)}=\{0\}$. Here, the closure means the ``integral closure" in a weaker sense, see Def. \ref{definition convex cone}.  Suppose $\varphi\in -\overline{C_\omega(b,c,\nu)}\cap \overline{C_\omega(b,c,v)}$, then there exists $\varphi_1,\varphi_2\in C_\omega(b,c,\nu)$ and $t_n^1,t_n^2\downarrow 0$ such that $\varphi+t_n^1\varphi_1\in C_\omega(b,c,\nu)$ and $-\varphi+t_n^2\varphi_2\in C_\omega(b,c,\nu)$ for each $n\in\mathbb{N}$.
Hence, for any  local stable leaf $\gamma(\omega)$ having size between $\frac{A(\epsilon)}{4J^2}$ and $A(\epsilon)$, and $\rho(\omega)\in D(a/2,\kappa,\gamma(\omega))$, we have
\begin{align*}
  \int_{\gamma(\omega)}(\varphi+t_n^1\varphi_1)(x)\rho(x,\omega)dm_{\gamma(\omega)} &>0;\\
   \int_{\gamma(\omega)}(-\varphi+t_n^2\varphi_2)(x)\rho(x,\omega)dm_{\gamma(\omega)} &>0.
\end{align*}Letting $n\rightarrow \infty$, by bounded convergence theorem, one must have
\begin{equation}\label{integration is 0}
  \int_{\gamma(\omega)}\varphi(x)\rho(x,\omega)dm_{\gamma(\omega)}(x)=0.
\end{equation}
Now pick $g\in C^{0,\kappa}(M)$ any $\kappa-$H\"older continuous function. Let
\begin{equation*}
  |g|_{\kappa}:=\sup_{x\not= y\in M}\frac{|g(x)-g(y)|}{d(x,y)^\kappa},\ g^+:=\frac{1}{2}(|g|+g),\ g^-:=\frac{1}{2}(|g|-g),
\end{equation*}and $B=\frac{2(|g|_\kappa+1)}{a}$. Then it is clear that
\begin{equation*}
  \log(g^++B), \mbox{ and }\log(g^-+B)
\end{equation*}are $(a/2,\kappa)-$H\"older continuous, i.e.,
\begin{equation*}
  \left|\log(g^\pm(x)+B)-\log(g^\pm(y)+B)\right|\leq \frac{a}{2}d(x,y)^\kappa,\ \forall x,y\in M.
\end{equation*}
Then $(g^+(\cdot)+B)|_{\gamma(\omega)},\ (g^-(\cdot)+B)|_{\gamma(\omega)}$ are in $D(a/2,\kappa,\gamma(\omega))$ by checking the definition. By $(\ref{integration is 0})$ and the linearity of integration, we have
\begin{equation*}
  \int_{\gamma(\omega)}\varphi(x)g(x)dm_{\gamma(\omega)}(x)=0.
\end{equation*}
 We can pick $g\in C^{0,\kappa}(M)$ $L^1-$approximating $\varphi$ since $\varphi$ is bounded and measurable, hence we have $\int_{\gamma(\omega)}\varphi^2(x)dm_{\gamma(\omega)}(x)=0$. So $\varphi(x)=0$ for $x\in \gamma(\omega)$. Since $\gamma(\omega)\subset M$ is an arbitrary local stable leaf, $\varphi\equiv 0$. The proof of Lemma \ref{Cbcv convex cone} is complete.
\end{proof}

Now $C_\omega(b,c,\nu)$ is a convex cone, so we can define the projective metric on $C_\omega(b,c,\nu)$. For any $\varphi_1,\varphi_2\in C_\omega(b,c,\nu),$ define
\begin{equation*}
  \alpha_\omega(\varphi_1,\varphi_2):=\sup\{t>0:\ \varphi_2-t\varphi_1\in C_\omega(b,c,\nu)\},
\end{equation*}
\begin{equation*}
  \beta_\omega(\varphi_1,\varphi_2):=\inf\{s>0:\ s\varphi_1-\varphi_2\in C_\omega(b,c,\nu)\},
\end{equation*}with the convention that $\sup\emptyset=0$ and $\inf\emptyset=+\infty$, and let
\begin{equation*}
  d_\omega(\varphi_1,\varphi_2):=\log \frac{\beta_\omega(\varphi_1,\varphi_2)}{\alpha_\omega(\varphi_1,\varphi_2)},
\end{equation*}with the convention that $d_\omega(\varphi_1,\varphi_2)=\infty$ if $\alpha_\omega(\varphi_1,\varphi_2)=0$ or $\beta_\omega(\varphi_1,\varphi_2)=\infty$.
\begin{lemma}\label{lemma metric on cone}
   For any $\varphi_1,\varphi_2\in C_\omega(b,c,\nu)$,
   \begin{align}\label{alpha omega phi1phi2}
 \alpha_\omega(\varphi_1,\varphi_2)  &=\inf\left\{\frac{\int_{\gamma(\omega)}\varphi_2\rho^\prime(\omega) dm_{\gamma(\omega)}}{\int_{\gamma(\omega)}\varphi_1\rho^\prime (\omega) dm_{\gamma(\omega)}},\frac{\int_{\gamma(\omega)}\varphi_2\rho^\prime(\omega) dm_{\gamma(\omega)}}{\int_{\gamma(\omega)}\varphi_1\rho^\prime(\omega) dm_{\gamma(\omega)}}\xi_\omega(\rho^\prime,\rho^{\prime\prime},\varphi_1,\varphi_2)\right.\\
& \left. \frac{\int_{\gamma(\omega)}\varphi_2\rho(\omega) dm_{\gamma(\omega)}}{\int_{\gamma(\omega)}\varphi_1\rho(\omega) dm_{\gamma(\omega)}}\eta_\omega(\rho,\tilde{\rho},\varphi_1,\varphi_2),\frac{\int_{\tilde{\gamma}(\omega)}\varphi_2\tilde{\rho}(\omega)dm_{\tilde{\gamma}(\omega)}}{\int_{\tilde{\gamma}(\omega)}\varphi_1\tilde{\rho}(\omega)dm_{\tilde{\gamma}(\omega)}}\eta_\omega(\tilde{\rho},\rho,\varphi_1,\varphi_2)\right\}\nonumber,
\end{align}and
\begin{align}\label{beta omega phi1phi2}
 \beta_\omega(\varphi_1,\varphi_2)  &=\sup\left\{\frac{\int_{\gamma(\omega)}\varphi_2\rho^\prime(\omega) dm_{\gamma(\omega)}}{\int_{\gamma(\omega)}\varphi_1\rho^\prime (\omega) dm_{\gamma(\omega)}},\frac{\int_{\gamma(\omega)}\varphi_2\rho^\prime(\omega) dm_{\gamma(\omega)}}{\int_{\gamma(\omega)}\varphi_1\rho^\prime (\omega) dm_{\gamma(\omega)}}\xi_\omega(\rho^\prime,\rho^{\prime\prime},\varphi_1,\varphi_2)\right.\\
& \left. \frac{\int_{\gamma(\omega)}\varphi_2\rho(\omega) dm_{\gamma(\omega)}}{\int_{\gamma(\omega)}\varphi_1\rho(\omega) dm_{\gamma(\omega)}}\eta_\omega(\rho,\tilde{\rho},\varphi_1,\varphi_2),\frac{\int_{\tilde{\gamma}(\omega)}\varphi_2\tilde{\rho}(\omega)dm_{\tilde{\gamma}(\omega)}}{\int_{\tilde{\gamma}(\omega)}\varphi_1\tilde{\rho}(\omega)dm_{\tilde{\gamma}(\omega)}}\eta_\omega(\tilde{\rho},\rho,\varphi_1,\varphi_2)\right\}\nonumber,
\end{align}where
\begin{equation}\label{xi omega}
  \xi_\omega(\rho^\prime,\rho^{\prime\prime},\varphi_1,\varphi_2):=\frac{\exp(bd_{\gamma(\omega)}^{a,\kappa}(\rho^\prime(\omega),\rho^{\prime\prime}(\omega)))-\int_{\gamma(\omega)}\varphi_2\rho^{\prime\prime}(\omega) dm_{\gamma(\omega)}/\int_{\gamma(\omega)}\varphi_2\rho^\prime(\omega) dm_{\gamma(\omega)}}{\exp(bd_{\gamma(\omega)}^{a,\kappa}(\rho^\prime(\omega),\rho^{\prime\prime}(\omega)))-\int_{\gamma(\omega)}\varphi_1\rho^{\prime\prime}(\omega) dm_{\gamma(\omega)}/\int_{\gamma(\omega)}\varphi_1\rho^\prime(\omega) dm_{\gamma(\omega)}},
\end{equation}
\begin{equation}\label{eta omega}
  \eta_\omega(\rho,\tilde{\rho},\varphi_1,\varphi_2):=\frac{\exp(cd_u(\gamma(\omega),\tilde{\gamma}(\omega))^\nu)-\int_{\tilde{\gamma}}\varphi_2\tilde{\rho}(\omega)dm_{\tilde{\gamma}(\omega)}/\int_{\gamma(\omega)}\varphi_2\rho(\omega) dm_{\gamma(\omega)}}{\exp(cd_u(\gamma(\omega),\tilde{\gamma}(\omega))^\nu)-\int_{\tilde{\gamma}}\varphi_1\tilde{\rho}(\omega)dm_{\tilde{\gamma}(\omega)}/\int_{\gamma(\omega)}\varphi_1\rho(\omega) dm_{\gamma(\omega)}},
\end{equation}
\begin{equation}\label{eta omega2}
  \eta_\omega(\tilde{\rho},\rho,\varphi_1,\varphi_2):=\frac{\exp(cd_u(\gamma(\omega),\tilde{\gamma}(\omega))^\nu)-\int_{\gamma(\omega)}\varphi_2\rho(\omega) dm_{\gamma(\omega)}/\int_{\tilde{\gamma}}\varphi_2\tilde{\rho}(\omega)dm_{\tilde{\gamma}(\omega)}}{\exp(cd_u(\gamma(\omega),\tilde{\gamma}(\omega))^\nu)-\int_{\gamma(\omega)}\varphi_1\rho(\omega) dm_{\gamma(\omega)}/\int_{\tilde{\gamma}}\varphi_1\tilde{\rho}(\omega)dm_{\tilde{\gamma}(\omega)}},
\end{equation}
and the inf and sup runs over all $\rho^\prime(\omega),\rho^{\prime\prime}(\omega)\in D(a/2,\kappa,\gamma(\omega))$ with $\int_{\gamma(\omega)}\rho^\tau(x,\omega)dm_{\gamma(\omega)}(x)=1$ for $\tau=\prime,\prime\prime$, every pair of local stable leaves $\gamma(\omega)$ and $\tilde{\gamma}(\omega)$ having size between $\frac{A(\epsilon)}{4J^2}$ and $A(\epsilon)$, $\rho(\omega)\in D(a_1,\kappa_1,\gamma(\omega))$ and corresponding $\tilde{\rho}(\omega)\in D(a/2,\kappa,\tilde{\gamma}(\omega))$.
\end{lemma}
\begin{proof}[Proof of Lemma \ref{lemma metric on cone}]
 We only compute $\alpha_\omega$, as the case for $\beta_\omega$ is similar. For any $t>0$ such that $\varphi_2-t\varphi_1\in C_\omega(b,c,\nu)$, then $\varphi_2-t\varphi_1$ must satisfy \ref{C1}, \ref{C2} and \ref{C3}.

 First, for every local stable submanifold $\gamma(\omega)\subset M$ having size between $\frac{A(\epsilon)}{4J^2}$ and $A(\epsilon)$, and every $\rho(\omega)\in D(a/2,\kappa,\gamma(\omega))$ satisfying $\int_{\gamma(\omega)}\rho(x,\omega)dm_{\gamma(\omega)}(x)=1$, one must have $          \int_{\gamma(\omega)}(\varphi_2-t\varphi_1)\rho(\omega)dm_{\gamma(\omega)}>0,$ which implies
        \begin{equation*}
          t<\frac{\int_{\gamma(\omega)}\varphi_2\rho(\omega)dm_{\gamma(\omega)}}{\int_{\gamma(\omega)}\varphi_1\rho(\omega)dm_{\gamma(\omega)}}.
        \end{equation*}

Second,  for every local stable submanifold $\gamma(\omega)\subset M$ having size between $\frac{A(\epsilon)}{4J^2}$ and $A(\epsilon)$, and $\rho(\omega)$, $\varsigma(\omega)\in D(a/2,\kappa,\gamma(\omega))\subset D(a,\kappa,\gamma(\omega))$ satisfying $\int_{\gamma(\omega)}\rho(x,\omega)dm_{\gamma(\omega)}(x)=\int_{\gamma(\omega)}\varsigma(x,\omega)dm_{\gamma(\omega)}(x)=1$, one must have
        \begin{align*}
        e^{-b\cdot d_{\gamma(\omega)}^{a,\kappa}(\rho(\omega),\varsigma(\omega))}\leq   \frac{\int_{\gamma(\omega)}(\varphi_2-t\varphi_1)\rho(\omega)dm_{\gamma(\omega)}}{\int_{\gamma(\omega)}(\varphi_2-t\varphi_1)\varsigma(\omega)dm_{\gamma(\omega)}},
          \leq e^{b\cdot d_{\gamma(\omega)}^{a,\kappa}(\rho(\omega),\varsigma(\omega))}
        \end{align*}which implies
        \begin{equation*}
          t\leq \frac{e^{b\cdot d_{\gamma(\omega)}^{a,\kappa}(\rho(\omega),\varsigma(\omega))}\cdot\int_{\gamma(\omega)}\varphi_2\varsigma(\omega)dm_{\gamma(\omega)}-\int_{\gamma(\omega)}\varphi_2\rho(\omega)dm_{\gamma(\omega)}}{e^{b\cdot d_{\gamma(\omega)}^{a,\kappa}(\rho(\omega),\varsigma(\omega))}\cdot\int_{\gamma(\omega)}\varphi_1\varsigma(\omega)dm_{\gamma(\omega)}-\int_{\gamma(\omega)}\varphi_1\rho(\omega)dm_{\gamma(\omega)}}
        \end{equation*}and
        \begin{equation*}
          t\leq \frac{ e^{b\cdot d_{\gamma(\omega)}^{a,\kappa}(\rho(\omega),\varsigma(\omega))}\cdot\int_{\gamma(\omega)}\varphi_2\rho(\omega)dm_{\gamma(\omega)}-\int_{\gamma(\omega)}\varphi_2\varsigma(\omega)dm_{\gamma(\omega)}}{ e^{b\cdot d_{\gamma(\omega)}^{a,\kappa}(\rho(\omega),\varsigma(\omega))}\cdot \int_{\gamma(\omega)}\varphi_1\rho(\omega)dm_{\gamma(\omega)}-\int_{\gamma(\omega)}\varphi_1\varsigma(\omega)dm_{\gamma(\omega)}}.
        \end{equation*}

        Third, for every pair of local stable leaves $\gamma(\omega),\tilde{\gamma}(\omega)\subset M$ having size between $\frac{A(\epsilon)}{4J^2}$ and $A(\epsilon)$, and $\gamma(\omega)$ is the holonomy image of $\tilde{\gamma}(\omega)$, $\rho(\omega)\in D(a_1,\kappa_1,\gamma(\omega))$ and $\tilde{\rho}(\omega) $ corresponding to $\rho(\omega)$ defined as $(\ref{definition of tilde rho})$, one must have
        \begin{equation*}
          e^{-cd_u(\gamma(\omega),\tilde{\gamma}(\omega))^\nu}\leq
          \frac{\int_{\gamma(\omega)}(\varphi_2-t\varphi_1)\rho(\omega)dm_{\gamma(\omega)}}{\int_{\tilde{\gamma}(\omega)}(\varphi_2-t\varphi_1)\tilde{\rho}(\omega)dm_{\tilde{\gamma}(\omega)}}\leq
         e^{ cd_u(\gamma(\omega),\tilde{\gamma}(\omega))^\nu},
        \end{equation*}which implies
        \begin{equation*}
          t\leq \frac{e^{ cd_u(\gamma(\omega),\tilde{\gamma}(\omega))^\nu}\cdot \int_{\tilde{\gamma}(\omega)}\varphi_2\tilde{\rho}(\omega)dm_{\tilde{\gamma}(\omega)}-\int_{\gamma(\omega)}\varphi_2\rho(\omega)dm_{\gamma(\omega)}}{e^{ cd_u(\gamma(\omega),\tilde{\gamma}(\omega))^\nu}\cdot \int_{\tilde{\gamma}(\omega)}\varphi_1\tilde{\rho}(\omega)dm_{\tilde{\gamma}(\omega)}-\int_{\gamma(\omega)}\varphi_1\rho(\omega)dm_{\gamma(\omega)}}
        \end{equation*}and
        \begin{equation*}
 t\leq \frac{e^{ cd_u(\gamma(\omega),\tilde{\gamma}(\omega))^\nu}\cdot \int_{\gamma(\omega)}\varphi_2\rho(\omega)dm_{\gamma(\omega)}-\int_{\tilde{\gamma}(\omega)}\varphi_2\tilde{\rho}(\omega)dm_{\tilde{\gamma}(\omega)}}{e^{ cd_u(\gamma(\omega),\tilde{\gamma}(\omega))^\nu}\cdot \int_{\gamma(\omega)}\varphi_1\rho(\omega)dm_{\gamma(\omega)}-\int_{\tilde{\gamma}(\omega)}\varphi_1\tilde{\rho}(\omega)dm_{\tilde{\gamma}(\omega)}}.
        \end{equation*}Organizing, we obtain \eqref{alpha omega phi1phi2}.
\end{proof}

\subsection{Contraction of the fiber transfer operator}\label{subsection 4.2}
In this subsection, we will prove that the fiber transfer operator $L_\omega$ maps $C_\omega(b,c,\nu)$ into $C_{\theta\omega}(b,c,\nu)$ for all $\omega\in\Omega$. Moreover, the diameter of $L^N_\omega C_\omega(b,c,\nu)$ with respect to the projective metric on $C_{\theta^N\omega}(b,c,\nu)$ is finite, and this diameter is independent of $\omega\in\Omega$, where the number $N$ comes from the topological mixing on fibers property and will be constructed in (\ref{N property 1}) and (\ref{N property 2}). Birkhoff's inequality (Proposition \ref{birkhoff inequality} in the Appendix) implies the contraction of the fiber transfer operator $L^N_\omega:C_\omega(b,c,\nu)\to C_{\theta^N\omega}(b,c,\nu)$.

\begin{lemma}\label{invariance of C(b,c,v)}
Recall that $\Lambda_1\in(0,1)$ is given in  Lemma \ref{rhoi}.
Let $\lambda_2\in(\max\{\Lambda_1,e^{-\lambda\nu}\},1)$ be any number,  for any $b>b_0(\lambda_2,\Lambda_1)=\frac{1}{\lambda_2-\Lambda_1}$, there exists $c_0=c_0(b,\nu)$ such that for any $c>c_0$ and for all $\omega\in\Omega$, we have $L_{\theta^{-1}\omega}(C_{\theta^{-1}\omega}(b,c,\nu))\subset C_{\omega}(\lambda_2b,\lambda_2c,\nu)\subset C_{\omega}(b,c,\nu)$. Recall that the fiber transfer operator $L_{\theta^{-1}\omega}$ is defined by
\begin{equation*}
    (L_{\theta^{-1}\omega}\varphi)(x)=\frac{\varphi((f_{\theta^{-1}\omega})^{-1}x)}{|\det D_{(f_{\theta^{-1}\omega})^{-1}(x)}f_{\theta^{-1}\omega}|},
  \end{equation*}for any bounded and measurable $\varphi:M\to \mathbb{R}$.
\end{lemma}
\begin{proof}[Proof of Lemma \ref{invariance of C(b,c,v)}]
   Pick any $\omega\in\Omega$ and any $\varphi\in C_{\theta^{-1}\omega}(b,c,\nu)$. It is clear that $L_{\theta^{-1}\omega}\varphi:M\rightarrow\mathbb{R}$ is bounded and measurable.

  Let us first verify condition \ref{C1} for $L_{\theta^{-1}\omega}\varphi$. Pick
 $\gamma(\omega)$ any local stable leaf having size between $\frac{A(\epsilon)}{4J^2}$ and $A(\epsilon)$, and any $\rho(\omega)\in D(a/2,\kappa,\gamma(\omega))$ satisfying $\int_{\gamma(\omega)}\rho(x,\omega)dm_{\gamma(\omega)}=1$. Since $(f_{\theta^{-1}\omega})^{-1}\gamma(\omega)$ is a curve, we can divide $(f_{\theta^{-1}\omega})^{-1}\gamma(\omega)$ into connected local stable submanifolds of size between $\frac{A(\epsilon)}{4J^2}$ and $A(\epsilon)$, named $\gamma_i(\theta^{-1}\omega)$ for $i$ belonging to a finite index set such that $\gamma_i(\theta^{-1}\omega)\cap \gamma_j(\theta^{-1}\omega)=\partial\gamma_i(\theta^{-1}\omega)\cap \partial \gamma_j(\theta^{-1}\omega)$ for $i\not=j$. Let $\rho_i(\theta^{-1}\omega)$ be defined as $(\ref{definition of rhoi})$ on $\gamma_i(\theta^{-1}\omega)$. By Lemma \ref{rhoi},
  \begin{equation*}
      \rho_i(\theta^{-1}\omega)\in D(e^{-\lambda_1}a/2,\kappa,\gamma_i(\theta^{-1}\omega))\subset D(a/2,\kappa,\gamma_i(\theta^{-1}\omega)).
 \end{equation*}Hence by $(\ref{Lphirho=phirhoi})$, we have
  \begin{align*}
  &\ \ \ \ \int_{\gamma(\omega)}(L_{\theta^{-1}\omega}\varphi)(y)\rho(y,\omega) dm_{\gamma(\omega)}(y)\\
  &=\sum_{i}\int_{\gamma_i(\theta^{-1}\omega)}\varphi(x)\rho_i(x,\theta^{-1}\omega)dm_{\gamma_i(\theta^{-1}\omega)}(x)\\
  &=\sum_{i}\int_{\gamma_i(\theta^{-1}\omega)}\rho_i(\theta^{-1}\omega)dm_{\gamma_i(\theta^{-1}\omega)}\cdot\int_{\gamma_i(\theta^{-1}\omega)}\varphi(x)\frac{\rho_i(x,\theta^{-1}\omega)}{\int_{\gamma_i(\theta^{-1}\omega)}\rho_i(\theta^{-1}\omega)dm_{\gamma_i(\theta^{-1}\omega)}}dm_{\gamma_i(\theta^{-1}\omega)}(x)>0,
  \end{align*}where the last inequality holds since $\varphi\in C_{\theta^{-1}\omega}(b,c,\nu)$ and using condition \ref{C1}
  for $\varphi$.

  Secondly, let us verify condition \ref{C2} for $L_{\theta^{-1}\omega}\varphi$. For any $\rho(\omega), \zeta(\omega)\in D(a/2,\kappa,\gamma(\omega))$ such that
  \begin{equation*}
    \int_{\gamma(\omega)}\rho(x,\omega)dm_{\gamma(\omega)}(x)=1 \mbox{ and } \int_{\gamma(\omega)}\zeta(x,\omega)dm_{\gamma(\omega)}(x)=1.
  \end{equation*}We divide $f_\omega^{-1}\gamma(\omega)$ into $\{\gamma_i(\theta^{-1}\omega)\}$ as before. Let $\rho_i(\theta^{-1}\omega)$  and $\zeta_i(\theta^{-1}\omega)$ be density functions defined as $(\ref{definition of rhoi})$ on $\gamma_i(\theta^{-1}\omega)$ corresponding to $\rho(\omega)$ and $\zeta(\omega)$ respectively. We normalize these density functions by
  \begin{align*}
    \rho^\prime_i(x,\theta^{-1}\omega) &=\rho_i(x,\theta^{-1}\omega)/\int_{\gamma_i(\theta^{-1}\omega)}\rho_i(y,\theta^{-1}\omega)dm_{\gamma_i(\theta^{-1}\omega)}(y)\ \mbox{ for }x\in\gamma_i(\theta^{-1}\omega),\\
    \zeta^\prime_i(x,\theta^{-1}\omega) &=\zeta_i(x,\theta^{-1}\omega)/\int_{\gamma_i(\theta^{-1}\omega)}\zeta_i(y,\theta^{-1}\omega)dm_{\gamma_i(\theta^{-1}\omega)}(y)\ \mbox{ for }x\in\gamma_i(\theta^{-1}\omega).\\
  \end{align*}
 By Lemma \ref{rhoi}, we have $\rho^\prime_i(\theta^{-1}\omega),\zeta^\prime_i(\theta^{-1}\omega)\in D(e^{-\lambda_1}\frac{a}{2},\kappa,\gamma_i(\theta^{-1}\omega))\subset D(\frac{a}{2},\kappa,\gamma_i(\theta^{-1}\omega))$.
  Then
  \begin{align*}
     &\ \ \ \ \int_{\gamma(\omega)}(L_{\theta^{-1}\omega}\varphi)(y)\zeta(y,\omega)dm_{\gamma(\omega)}(y)\\
     &\overset{\eqref{Lphirho=phirhoi}}=\sum_{i}\int_{\gamma_i(\theta^{-1}\omega)}\varphi(x)\zeta_i(x,\theta^{-1}\omega)dm_{\gamma_i(\theta^{-1}\omega)}(x)\\
     &=\sum_{i}\int_{\gamma_i(\theta^{-1}\omega)}\zeta_i(x,\theta^{-1}\omega)dm_{\gamma_i(\theta^{-1}\omega)}(x)\int_{\gamma_i(\theta^{-1}\omega)}\varphi(x)\zeta_i^\prime(x,\theta^{-1}\omega)dm_{\gamma_i(\theta^{-1}\omega)}(x)\\
     &\overset{\ref{C2}}\leq \sum_{i}\int_{\gamma_i(\theta^{-1}\omega)}\zeta_i(\theta^{-1}\omega)dm_{\gamma_i(\theta^{-1}\omega)}\cdot \exp(bd_{\gamma_i(\theta^{-1}\omega)}^{a,\kappa}(\rho_i^\prime(\theta^{-1}\omega),\zeta_i^\prime(\theta^{-1}\omega)))\cdot \int_{\gamma_i(\theta^{-1}\omega)}\varphi\rho_i^\prime(\theta^{-1}\omega)dm_{\gamma_i(\theta^{-1}\omega)}\\
     &\overset{\eqref{P2P3}}= \sum_{i}\frac{\int_{\gamma_i(\theta^{-1}\omega)}\zeta_i(\theta^{-1}\omega)dm_{\gamma_i(\theta^{-1}\omega)}}{\int_{\gamma_i(\theta^{-1}\omega)}\rho_i(\theta^{-1}\omega)dm_{\gamma_i(\theta^{-1}\omega)}}\cdot \exp(bd_{\gamma_i(\theta^{-1}\omega)}^{a,\kappa}(\rho_i(\theta^{-1}\omega),\zeta_i(\theta^{-1}\omega)))\cdot \int_{\gamma_i(\theta^{-1}\omega)}\varphi\rho_i(\theta^{-1}\omega)dm_{\gamma_i(\theta^{-1}\omega)}\\
     &\overset{\eqref{eq d leq Lamda1 d}}\leq \sum_{i}\frac{\int_{\gamma_i(\theta^{-1}\omega)}\zeta_i(\theta^{-1}\omega)dm_{\gamma_i(\theta^{-1}\omega)}}{\int_{\gamma_i(\theta^{-1}\omega)}\rho_i(\theta^{-1}\omega)dm_{\gamma_i(\theta^{-1}\omega)}}\cdot \exp(b\Lambda_1d_{\gamma(\omega)}^{a,\kappa}(\rho(\omega),\zeta(\omega)))\cdot \int_{\gamma_i(\theta^{-1}\omega)}\varphi\rho_i(\theta^{-1}\omega)dm_{\gamma_i(\theta^{-1}\omega)}.
  \end{align*}
  Note that
  \begin{equation*}
    \frac{\zeta_i(x,\theta^{-1}\omega)}{\rho_i(x,\theta^{-1}\omega)}\overset{\eqref{definition of   rhoi}}=\frac{\zeta(f_{\theta^{-1}\omega}x,\omega)}{\rho(f_{\theta^{-1}\omega}x,\omega)}\overset{\eqref{rho     over zeta}}\leq \exp(d_{\gamma(\omega)}^{a,\kappa}(\rho(\omega),\zeta(\omega))) \mbox{ for }x\in \gamma_i(\theta^{-1}\omega).
  \end{equation*}We continue the estimate
  \begin{align*}
     & \ \ \ \ \int_{\gamma(\omega)}(L_{\theta^{-1}\omega}\varphi)(y)\zeta(y,\omega)dm_{\gamma(\omega)}(y)\\
     &\leq \exp\left(d_{\gamma(\omega)}^{a,\kappa}(\rho(\omega),\zeta(\omega))\right) \exp\left(b\Lambda_1d_{\gamma(\omega)}^{a,\kappa}(\rho(\omega),\zeta(\omega))\right)\sum_{i}^{}\int_{\gamma_i(\theta^{-1}\omega)}\varphi\rho_i(\theta^{-1}\omega)dm_{\gamma_i(\theta^{-1}\omega)}\\
 &= \exp\left((1+b\Lambda_1)d_{\gamma(\omega)}^{a,\kappa}(\rho(\omega),\zeta(\omega))\right)\int_{\gamma(\omega)}(L_{\theta^{-1}\omega}\varphi)(y)\rho(y,\omega)dm_{\gamma(\omega)}(y)\\
     &\leq \exp\left(b\lambda_2d_{\gamma(\omega)}^{a,\kappa}(\rho(\omega),\zeta(\omega))\right)\int_{\gamma(\omega)}(L_{\theta^{-1}\omega}\varphi)(y)\rho(y,\omega)dm_{\gamma(\omega)}(y),
  \end{align*}provided $\lambda_2\in (\Lambda_1,1)$ and $b>\frac{1}{\lambda_2-\Lambda_1}:=b_0$. Switching $\rho(\omega)$ and $\zeta(\omega)$ in the above estimates, we get
  \begin{equation*}
    \int_{\gamma(\omega)}(L_{\theta^{-1}\omega}\varphi)(y)\rho(y,\omega)dm_{\gamma(\omega)}(y)\leq \exp(b\lambda_2d_{\gamma(\omega)}^{a,\kappa}(\rho(\omega),\zeta(\omega)))\int_{\gamma(\omega)}(L_{\theta^{-1}\omega}\varphi)(y)\zeta(y,\omega)dm_{\gamma(\omega)}(y).
  \end{equation*}
  The above two estimates imply that
  \begin{equation*}
    \left|\log \int_{\gamma(\omega)}(L_{\theta^{-1}\omega}\varphi)(y)\rho(y,\omega)dm_{\gamma(\omega)}-\log \int_{\gamma(\omega)}(L_{\theta^{-1}\omega}\varphi)(y)\zeta(y,\omega)dm_{\gamma(\omega)}\right|\leq b\lambda_2d_{\gamma(\omega)}^{a,\kappa}(\rho(\omega),\zeta(\omega)).
  \end{equation*}

Next, let us verify condition \ref{C3} for $L_{\theta^{-1}\omega}\varphi$. Given any pair of local stable leaves $\gamma(\omega)$ and $\tilde{\gamma}(\omega)$. We first divide $(f_{\theta^{-1}\omega})^{-1}\gamma(\omega)$ into connected local stable manifolds of size between $\frac{A(\epsilon)}{4J}$ and $\frac{A(\epsilon)}{2J}$, named $\gamma_i(\theta^{-1}\omega)$, such that $\gamma(\omega)=\cup f_{\theta^{-1}\omega}\gamma_i(\theta^{-1}\omega)$. Let $\tilde{\gamma}_i(\theta^{-1}\omega)$ be the holonomy image of $\gamma_i(\theta^{-1}\omega)$ inside of $(f_{\theta^{-1}\omega})^{-1}\tilde{\gamma}(\omega)$. Naturally, we have $\tilde{\gamma}(\omega)=\cup f_{\theta^{-1}\omega}\tilde{\gamma}_i(\theta^{-1}\omega)$. Note that the Jacobian of holonomy map between local stable manifolds is bounded above by $J$  and bounded below by $J^{-1}$. Therefore, $\tilde{\gamma}_i(\theta^{-1}\omega)$ have size between $\frac{A(\epsilon)}{4J^2}$ and $A(\epsilon)$. For any $\rho(\omega)\in D(a_1,\kappa_1,\gamma(\omega))$, let $\tilde{\rho}(\omega)$ be defined as $(\ref{definition of tilde rho})$, and we already see that  $\tilde{\rho}(\omega)\in D(\frac{a}{2},\kappa_1\nu_0,\tilde{\gamma}(\omega))\subset D(a,\kappa,\tilde{\gamma}(\omega)).$ Let
\begin{align*}
  \rho_i(x,\theta^{-1}\omega) & =\frac{|\det D_xf_{\theta^{-1}\omega}|_{E^s(x,\theta^{-1}\omega)}|}{|\det D_xf_{\theta^{-1}\omega}|}\rho(f_{\theta^{-1}\omega}x,\omega)\mbox{ for }x\in\gamma_i(\theta^{-1}\omega)\\
\end{align*} and
\begin{align*}
 (\tilde{\rho})_i(\theta^{-1}\omega)&=\frac{|\det D_xf_{\theta^{-1}\omega}|_{E^s(x,\theta^{-1}\omega)}|}{|\det D_xf_{\theta^{-1}\omega}|}\tilde{\rho}(f_{\theta^{-1}\omega}x,\omega)\mbox{ for }x\in\gamma_i(\theta^{-1}\omega).
\end{align*}Then by \eqref{Lphirho=phirhoi}, we have
\begin{align*}
  \int_{\gamma(\omega)}(L_{\theta^{-1}\omega}\varphi)(y)\rho(y,\omega)dm_{\gamma(\omega)}(y) & =\sum_{i}\int_{\gamma_i(\theta^{-1}\omega)}\varphi(x)\rho_i(x,\theta^{-1}\omega)dm_{\gamma_i(\theta^{-1}\omega)}(x),\\
  \int_{\tilde{\gamma}(\omega)}(L_{\theta^{-1}\omega}\varphi)(y)\tilde{\rho}(y,\omega)dm_{\tilde{\gamma}(\omega)}(y) & =\sum_{i}\int_{\tilde{\gamma}_i(\theta^{-1}\omega)}\varphi(x)(\tilde{\rho})_i(x,\theta^{-1}\omega)dm_{\tilde{\gamma}_i(\theta^{-1}\omega)}(x).\\
\end{align*}By Lemma \ref{rhoi}, $\rho_i(\theta^{-1}\omega)\in D(e^{-\lambda_1}a_1,\kappa_1,\gamma_i(\theta^{-1}\omega))\subset D(a_1,\kappa_1,\gamma_i(\theta^{-1}\omega))$.
Let  $\psi_{\theta^{-1}\omega}^i$ be the holonomy map between $\tilde{\gamma}_i(\theta^{-1}\omega)$ and $\gamma_i(\theta^{-1}\omega)$, and define
\begin{equation*}
  \tilde{\rho_i}(x,\theta^{-1}\omega)=\rho_i(\psi_{\theta^{-1}\omega}^i(x),\theta^{-1}\omega)\cdot |Jac(\psi_{\theta^{-1}\omega}^i)(x)|\mbox{ for }x\in \tilde{\gamma}_i(\theta^{-1}\omega)
\end{equation*}to be the density function defined by the holonomy map corresponding to $\rho_i(\theta^{-1}\omega)$. Note that $(\tilde{\rho})_i(\theta^{-1}\omega)$ and $ \tilde{\rho_i}(\theta^{-1}\omega)$ are different, and we keep in mind that $(\tilde{\rho})_i(\theta^{-1}\omega)$ is obtained by first applying the action of holonomy map then applying the action of pull back to $\rho(\omega)$, while $ \tilde{\rho_i}(\theta^{-1}\omega)$ is obtained by reversing the order of these two actions.
Since $\varphi\in C_{\theta^{-1}\omega}(b,c,\nu)$ and $\rho_i(\theta^{-1}\omega)\in D(a_1,\kappa_1,\gamma_i(\theta^{-1}\omega))$, by \ref{C3}, we conclude for each $i$,
\begin{align}
  &\ \ \ \ \left|\log \int_{\gamma_i(\theta^{-1}\omega)}\varphi(x)\rho_i(x,\theta^{-1}\omega)dm_{\gamma_i(\theta^{-1}\omega)}-\log \int_{\tilde{\gamma}_i(\theta^{-1}\omega)}\varphi(x)\tilde{\rho_i}(x,\theta^{-1}\omega)dm_{\tilde{\gamma}_i(\theta^{-1}\omega)}\right|\label{logphirho-logphirho}\\
  &\leq cd_u(\gamma_i(\theta^{-1}\omega),\tilde{\gamma}_i(\theta^{-1}\omega))^\nu\nonumber\\
  &\overset{\eqref{distance between local stable}}\leq c\cdot e^{-\lambda\nu}d_u(\gamma(\omega),\tilde{\gamma}(\omega))^\nu.\nonumber
\end{align} To finish the proof of the condition \ref{C3}, we need the following sublemma:
\begin{sublemma}\label{sublemma}
  There exists a number $K_0=K_0(a,b)>0$ depending on $a,\ b$ and system constants such that for each $i$, the following inequality holds
  \begin{align}\label{2inequality}
   &\ \ \ \  \left|\log\int_{\tilde{\gamma}_i(\theta^{-1}\omega)}\varphi(x)\tilde{\rho}_i(x,\theta^{-1}\omega)dm_{\tilde{\gamma}_i(\theta^{-1}\omega)}-\log\int_{\tilde{\gamma}_i(\theta^{-1}\omega)}\varphi(x)(\tilde{\rho})_i(x,\theta^{-1}\omega)dm_{\tilde{\gamma}_i(\theta^{-1}\omega)}\right|\\
   &\leq K_0d_u(\gamma(\omega),\tilde{\gamma}(\omega))^{\nu}\nonumber.
  \end{align}
\end{sublemma}

We postpone the proof of Sublemma \ref{sublemma} and finish the proof of Lemma \ref{invariance of C(b,c,v)}. We combine $(\ref{logphirho-logphirho})$ and $(\ref{2inequality})$ to obtain
\begin{align*}
   & \ \ \ \ \left|\log\int_{\gamma_i(\theta^{-1}\omega)}\varphi(x)\rho_i(x,\theta^{-1}\omega)dm_{\gamma_i(\theta^{-1}\omega)}-\log\int_{\tilde{\gamma}_i(\theta^{-1}\omega)}\varphi(x)(\tilde{\rho})_i(x,\theta^{-1}\omega)dm_{\tilde{\gamma}_i(\theta^{-1}\omega)}\right|\\
   &\leq (c\cdot e^{-\lambda \nu}+K_0)d_u(\gamma(\omega),\tilde{\gamma}(\omega))^{\nu}.
\end{align*}As a consequence,
\begin{align*}
   & \ \ \ \ \left|\log\int_{\gamma(\omega)}(L_{\theta^{-1}\omega}\varphi)(y)\rho(y,\omega)dm_{\gamma(\omega)}(y) -\log\int_{\tilde{\gamma}(\omega)}(L_{\theta^{-1}\omega}\varphi)(y)\tilde{\rho}(y,\omega)dm_{\tilde{\gamma}(\omega)}(y) \right|\\
   &=\left|\log \sum_{i}\int_{\gamma_i(\theta^{-1}\omega)}\varphi(x)\rho_i(x,\theta^{-1}\omega)dm_{\gamma_i(\theta^{-1}\omega)}(x)-\log \sum_{i}\int_{\tilde{\gamma}_i(\theta^{-1}\omega)}\varphi(x)(\tilde{\rho})_i(x,\theta^{-1}\omega)dm_{\tilde{\gamma}_i(\theta^{-1}\omega)}(x)   \right| \\
   &\leq (c\cdot e^{-\lambda\nu}+K_0)d_u(\gamma(\omega),\tilde{\gamma}(\omega))^\nu\\
   &\leq \lambda_2 cd_u(\gamma(\omega),\tilde{\gamma}(\omega))^\nu,
\end{align*}provided $\lambda_2\in (e^{-\lambda\nu},1)$, and $c\geq \frac{K_0(a,b)}{\lambda_2-\exp(-\lambda\nu)}:=c_0$. The proof of Lemma \ref{invariance of C(b,c,v)} is complete.
\end{proof}

\begin{proof}[Proof of Sublemma \ref{sublemma}]
Applying Lemma \ref{rhoi} and $(\ref{tilder rho})$ to $\rho(\omega)\in D(a_1,\kappa_1,\gamma(\omega))$, we see that $(\tilde{\rho})_i(\theta^{-1}\omega)$, $ \tilde{\rho_i}(\theta^{-1}\omega)$ both belong to 
$D(a/2,\kappa_1\nu_0,\tilde{\gamma}_i(\theta^{-1}\omega))\subset
D(a/2,\kappa,\tilde{\gamma}_i(\theta^{-1}\omega))$. We fix $i$ and prove Sublemma \ref{sublemma}. Without ambiguity, we denote $\rho^\prime(\theta^{-1}\omega):=(\tilde{\rho})_i(\theta^{-1}\omega)$ and $\rho^{\prime\prime}(\theta^{-1}\omega):=\tilde{\rho}_i(\theta^{-1}\omega)$ for short.

We normalize the random density $\rho^\prime(\theta^{-1}\omega)$ and $\rho^{\prime\prime}(\theta^{-1}\omega)$ by letting
\begin{equation*}
 \bar{\rho}^\prime(x,\theta^{-1}\omega):= \frac{\rho^\prime(x,\theta^{-1}\omega)}{\int_{\tilde{\gamma}_i(\theta^{-1}\omega)}\rho^\prime(x,\theta^{-1}\omega)dm_{\tilde{\gamma}_i(\theta^{-1}\omega)}(x)},\ \bar{\rho}^{\prime\prime}(x,\theta^{-1}\omega):= \frac{\rho^{\prime\prime}(x,\theta^{-1}\omega)}{\int_{\tilde{\gamma}_i(\theta^{-1}\omega)}\rho^{\prime\prime}(x,\theta^{-1}\omega)dm_{\tilde{\gamma}_i(\theta^{-1}\omega)}(x)}.
\end{equation*}Then by condition \ref{C2}, we have
\begin{equation}\label{estimate lemma 4.1 1}
  \begin{split}
 &\left|\log \int_{\tilde{\gamma}_i(\theta^{-1}\omega)}\varphi(x)\rho^\prime(x,\theta^{-1}\omega)dm_{\tilde{\gamma}_i(\theta)^{-1}\omega}-\log \int_{\tilde{\gamma}_i(\theta^{-1}\omega)}\varphi(x)\rho^{\prime\prime}(x,\theta^{-1}\omega)dm_{\tilde{\gamma}_i(\theta)^{-1}\omega}\right|\\
 \leq &\left|\log \int_{\tilde{\gamma}_i(\theta^{-1}\omega)}\rho^\prime(x,\theta^{-1}\omega)dm_{\tilde{\gamma}_i(\theta^{-1}\omega)}(x)-\log \int_{\tilde{\gamma}_i(\theta^{-1}\omega)}\rho^{\prime\prime}(x,\theta^{-1}\omega)dm_{\tilde{\gamma}_i(\theta^{-1}\omega)}(x)\right|\\
 &\quad+ bd_{\tilde{\gamma}_i(\theta^{-1}\omega)}^{a,\kappa}(\bar{\rho}^\prime(\theta^{-1}\omega),\bar{\rho}^{\prime\prime}(\theta^{-1}\omega)).
  \end{split}
\end{equation}
Next, we are going to estimate terms in the right hand of the above inequality.
By definition, for $x\in\tilde{\gamma}_i(\theta^{-1}\omega)$,we have expressions
\begin{align*}
  \rho^\prime(x,\theta^{-1}\omega) &=\frac{\left|\det D_xf_{\theta^{-1}\omega}|_{E^s(x,\theta^{-1}\omega)}\right|}{\left|\det D_xf_{\theta^{-1}\omega}\right|}\cdot \rho (\psi_\omega f_{\theta^{-1}\omega}(x),\omega)|Jac(\psi_\omega)(f_{\theta^{-1}\omega}x)|;\\
  \rho^{\prime\prime}(x,\theta^{-1}\omega)&=\frac{\left|\det D_{\psi_{\theta^{-1}\omega}^i(x)}f_{\theta^{-1}\omega}|_{E^s(\psi_{\theta^{-1}\omega}^i(x),\theta^{-1}\omega)}\right|}{\left|\det D_{\psi_{\theta^{-1}\omega}^i(x)}f_{\theta^{-1}\omega}\right|}\cdot\rho(f_{\theta^{-1}\omega}\psi_{\theta^{-1}\omega}^i(x),\omega)\cdot |Jac(\psi_{\theta^{-1}\omega}^i)(x)|.
\end{align*}
By definition of holonomy map, we have
\begin{equation}\label{3termofestimate}
  \rho(\psi_\omega f_{\theta^{-1}\omega}(x),\omega)=\rho(f_{\theta^{-1}\omega}\psi_{\theta^{-1}\omega}^i(x),\omega)\mbox{ for }x\in\tilde{\gamma}_i(\theta^{-1}\omega).
\end{equation}
By Lemma \ref{property of fiberwise holonomy map} (2), for $x\in\tilde{\gamma}_i(\theta^{-1}\omega)$, we have
\begin{align}
   \left|\log |Jac(\psi_\omega)(f_{\theta^{-1}\omega}x)|-\log  |Jac(\psi_{\theta^{-1}\omega}^i)(x)|\right|
   &\leq a_0d(f_{\theta^{-1}\omega}(x),\psi_\omega f_{\theta^{-1}\omega}(x))^{\nu_0}+a_0d(x,\psi_{\theta^{-1}\omega}^i(x))^{\nu_0}\nonumber\\
   &\leq a_0(1+e^{-\lambda\nu_0})d_u(\gamma(\omega),\tilde{\gamma}(\omega))^{\nu_0}\label{4thtermofestimate}.
\end{align}Combing Lemma \ref{Holder continuity of stableunstable distribution} and Lemma \ref{detDfE-detDfE}, for all $x,y\in M$, $\omega\in\Omega$, we have
\begin{equation}\label{det dfEs-detdfEs}
  \left||\det D_xf_\omega|_{E^s(x,\omega)}|-|\det D_yf_\omega|_{E^s(y,\omega)}|\right|\leq (C_2+C_2C_1)d(x,y)^{\nu_0}.
\end{equation}
Then, by $(\ref{det dfEs-detdfEs})$ and \eqref{bound of det}, for $x\in \tilde{\gamma}_i(\theta^{-1}\omega)$, we have
\begin{align}
 &\ \ \ \ \left|\log |\det D_xf_{\theta^{-1}\omega}|_{E^s(x,\theta^{-1}\omega)}|-\log|\det D_{\psi_{\theta^{-1}\omega}^i(x)}f_{\theta^{-1}\omega}|_{E^s(\psi_{\theta^{-1}\omega}^i(x),\theta^{-1}\omega)}|\right|\nonumber\\  &\leq C_{10}(C_2+C_2C_1) d(x,\psi_{\theta^{-1}\omega}^i(x))^{\nu_0}\nonumber\\
 &\leq C_{10}(C_2+C_2C_1)  d_u(\gamma_i(\theta^{-1}\omega),\tilde{\gamma}_i(\theta^{-1}\omega))^{\nu_0}\nonumber\\
 &\leq C_{10}(C_2+C_2C_1) e^{-\lambda\nu_0}d_u(\gamma(\omega),\tilde{\gamma}(\omega))^{\nu_0}\label{1thtermofestimate}.
\end{align}
Applying $(\ref{Lip constant of log det Dxfomega})$, we have
\begin{equation}\label{2ndtermofestimate}
  \left|\log|\det D_{\psi_{\theta^{-1}\omega}^ix}f_{\theta^{-1}\omega}|-\log|\det D_xf_{\theta^{-1}\omega}|\right|\leq K_2d(x,\psi_{\theta^{-1}\omega}^ix)\leq K_2e^{-\lambda}d_u(\gamma(\omega),\tilde{\gamma}(\omega)).
\end{equation}
Then $(\ref{3termofestimate}) $, $(\ref{4thtermofestimate})$, $(\ref{1thtermofestimate})$ and $(\ref{2ndtermofestimate})$ imply
\begin{align}\label{estimate rho prime rho pp}
  e^{-K_3d_u(\gamma(\omega),\tilde{\gamma}(\omega))^{\nu_0}}\leq \frac{\rho^\prime(x,\theta^{-1}\omega)}{\rho^{\prime\prime}(x,\theta^{-1}\omega)} \leq e^{K_3d_u(\gamma(\omega),\tilde{\gamma}(\omega))^{\nu_0}}\mbox{ for any }x\in\tilde{\gamma}_i(\theta^{-1}\omega),
\end{align}where $K_3=\max\{a_0(1+e^{-\lambda\nu_0}),C_{10}(C_2+C_2C_1) e^{-\lambda\nu_0},K_2e^{-\lambda}\}$, and therefore,
\begin{align}
   & \ \ \ \ \left|\log \int_{\tilde{\gamma}_i(\theta^{-1}\omega)}\rho^\prime(x,\theta^{-1}\omega)dm_{\tilde{\gamma}_i(\theta^{-1}\omega)}(x)-\log \int_{\tilde{\gamma}_i(\theta^{-1}\omega)}\rho^{\prime\prime}(x,\theta^{-1}\omega)dm_{\tilde{\gamma}_i(\theta^{-1}\omega)}(x)\right|\nonumber\\
     &\leq K_3d_u(\gamma(\omega),\tilde{\gamma}(\omega))^{\nu_0}\label{2nd term of inequality}.
\end{align}

 Next, we estimate $d_{\tilde{\gamma}_i(\theta^{-1}\omega)}^{a,\kappa}(\bar{\rho}^\prime(\theta^{-1}\omega),\bar{\rho}^{\prime\prime}(\theta^{-1}\omega))$. The inequality \eqref{estimate rho prime rho pp} also implies that
\begin{equation}\label{estimate on d+tilde gamma i}
  d_{+,\tilde{\gamma}_i(\theta^{-1}\omega)}(\bar{\rho}^\prime(\theta^{-1}\omega),\bar{\rho}^{\prime\prime}(\theta^{-1}\omega))=d_{+,\tilde{\gamma}_i(\theta^{-1}\omega)}(\rho^\prime(\theta^{-1}\omega),\rho^{\prime\prime}(\theta^{-1}\omega))\overset{\eqref{d+gammaomegarho1rho2}}\leq 2K_3d_u(\gamma(\omega),\tilde{\gamma}(\omega))^{\nu_0}.
\end{equation}Similar as the proof of $(\ref{dgammaileqdgamma+i})$, we have an estimate
\begin{equation}\label{d leq d+log}
  d_{\tilde{\gamma}_i(\theta^{-1}\omega)}^{a,\kappa}(\bar{\rho}^\prime(\theta^{-1}\omega),\bar{\rho}^{\prime\prime}(\theta^{-1}\omega))\leq d_{+,\tilde{\gamma}_i(\theta^{-1}\omega)}(\bar{\rho}^\prime(\theta^{-1}\omega),\bar{\rho}^{\prime\prime}(\theta^{-1}\omega))+\log\left(\hat{\tau}_2(\theta^{-1}\omega)/\hat{\tau}_1(\theta^{-1}\omega)\right),
\end{equation}where
\begin{align*}
  \hat{\tau}_1(\theta^{-1}\omega)&= \inf_{x\not=y\in\tilde{\gamma}_i(\theta^{-1}\omega)}\left\{1,\frac{\exp(ad(x,y)^\kappa)-\rho^{\prime\prime}(y,\theta^{-1}\omega)/\rho^{\prime\prime}(x,\theta^{-1}\omega)}{\exp(ad(x,y)^\kappa)-\rho^{\prime}(y,\theta^{-1}\omega)/\rho^{\prime}(x,\theta^{-1}\omega)}\right\},
\end{align*}and
\begin{align*}
 \hat{\tau}_2(\theta^{-1}\omega)&= \sup_{x\not=y\in\tilde{\gamma}_i(\theta^{-1}\omega)}\left\{1,\frac{\exp(ad(x,y)^\kappa)-\rho^{\prime\prime}(y,\theta^{-1}\omega)/\rho^{\prime\prime}(x,\theta^{-1}\omega)}{\exp(ad(x,y)^\kappa)-\rho^{\prime}(y,\theta^{-1}\omega)/\rho^{\prime}(x,\theta^{-1}\omega)}\right\}.
\end{align*}So we need to estimate $\log\left(\hat{\tau}_2(\theta^{-1}\omega)/\hat{\tau}_1(\theta^{-1}\omega)\right).$
Denote
\begin{align*}
  B_1(x,y,\theta^{-1}\omega) &:=\frac{\rho^{\prime}(y,\theta^{-1}\omega)}{\rho^{\prime}(x,\theta^{-1}\omega)}\cdot \exp(-ad(x,y)^\kappa)\mbox{ for } x,y\in \tilde{\gamma}_i(\theta^{-1}\omega),x\not=y,\\
B_2(x,y,\theta^{-1}\omega) &:=\frac{\rho^{\prime\prime}(y,\theta^{-1}\omega)}{\rho^{\prime\prime}(x,\theta^{-1}\omega)}\cdot \exp(-ad(x,y)^\kappa)\mbox{ for }x,y\in \tilde{\gamma}_i(\theta^{-1}\omega),x\not=y,
\end{align*}then
\begin{equation}\label{relation tau b}
\begin{split}
\frac{\exp(ad(x,y)^\kappa)-\rho^{\prime\prime}(y,\theta^{-1}\omega)/\rho^{\prime\prime}(x,\theta^{-1}\omega)}{\exp(ad(x,y)^\kappa)-\rho^{\prime}(y,\theta^{-1}\omega)/\rho^{\prime}(x,\theta^{-1}\omega)}&=\frac{1-B_2(x,y,\theta^{-1}\omega)}{1-B_1(x,y,\theta^{-1}\omega)}.
\end{split}
\end{equation} Since $\rho^\prime(\theta^{-1}\omega),\rho^{\prime\prime}(\theta^{-1}\omega)\in D(a/2,\kappa_1\nu_0,\tilde{\gamma}_i(\theta^{-1}\omega))\subset D(a/2,\kappa,\tilde{\gamma}_i(\theta^{-1}\omega))$,
\begin{align}
 \log B_1(x,y,\theta^{-1}\omega)&=\log \rho^\prime(y,\theta^{-1}\omega)-\log\rho^\prime(x,\theta^{-1}\omega)-ad(x,y)^\kappa\nonumber\\
 &\leq \frac{a}{2}d(x,y)^\kappa-ad(x,y)^\kappa\nonumber\\
 &\leq -\frac{a}{2}d(x,y)^\kappa<0.\label{bound of B1}
\end{align}As a consequence, $B_1(x,y,\theta^{-1}\omega)\leq e^{-\frac{a}{2}d(x,y)^\kappa}<1$. Similarly, $B_2(x,y,\theta^{-1}\omega)\leq e^{-\frac{a}{2}d(x,y)^\kappa}<1$. Hence, on the one hand
\begin{align}
 &\ \ \ \ |B_1(x,y,\theta^{-1}\omega)-B_2(x,y,\theta^{-1}\omega)|\nonumber\\
 &\leq \max\{B_1(x,y,\theta^{-1}\omega),B_2(x,y,\theta^{-1}\omega)\}\cdot|\log B_1(x,y,\theta^{-1}\omega)-\log B_2(x,y,\theta^{-1}\omega)|\nonumber\\
 &\leq |\log\rho^\prime(x,\theta^{-1}\omega)-\log \rho^{\prime\prime}(x,\theta^{-1}\omega)|+|\log \rho^\prime(y,\theta^{-1}\omega)-\log \rho^{\prime\prime}(y,\theta^{-1}\omega)|\nonumber\\
 &\overset{\eqref{estimate rho prime rho pp}}\leq 2K_3d_u(\gamma(\omega),\tilde{\gamma}(\omega))^{\nu_0}\nonumber\\
 &\overset{\eqref{mu+nuleq nu0}}\leq 2K_3d_u(\gamma(\omega),\tilde{\gamma}(\omega))^{\kappa+\nu}\label{on one hand},
\end{align}where in the first inequality, we use the following inequality
\begin{equation*}
  |z_1-z_2|\leq \max\{z_1,z_2\}|\log z_1-\log z_2|\mbox{ for any }z_1,z_2\in(0,1).
\end{equation*}On the other hand, recall that $\rho^\prime(\theta^{-1}\omega)$, $ \rho^{\prime\prime}(\theta^{-1}\omega)$ both belong to
$D(a/2,\kappa_1\nu_0,\tilde{\gamma}_i(\theta^{-1}\omega))$, we also have
\begin{align}
 &\ \ \ \ |B_1(x,y,\theta^{-1}\omega)-B_2(x,y,\theta^{-1}\omega)|\nonumber\\
 &\leq \max\{B_1(x,y,\theta^{-1}\omega),B_2(x,y,\theta^{-1}\omega)\}\cdot |\log B_1(x,y,\theta^{-1}\omega)-\log B_2(x,y,\theta^{-1}\omega)|\nonumber\\
  & \leq |\log\rho^\prime(x,\theta^{-1}\omega)-\log \rho^{\prime}(y,\theta^{-1}\omega)|+|\log \rho^{\prime\prime}(x,\theta^{-1}\omega)-\log \rho^{\prime\prime}(y,\theta^{-1}\omega)|\nonumber\\
  &\leq 2\cdot \frac{a}{2}d(x,y)^{\kappa_1\nu_0}\nonumber\\
  &\leq ad(x,y)^{\kappa+\nu}\label{on the other hand}.
\end{align}Estimates $(\ref{on one  hand})$ and $(\ref{on the other hand})$ imply that
\begin{align}
   |B_1(x,y,\theta^{-1}\omega)-B_2(x,y,\theta^{-1}\omega)| &\leq \max\{a,2K_3\}d_u(\gamma(\omega),\tilde{\gamma}(\omega))^\nu\cdot d(x,y)^\kappa\nonumber\\
   &:=K_4d_u(\gamma(\omega),\tilde{\gamma}(\omega))^\nu\cdot d(x,y)^\kappa.
\end{align}Then
\begin{align}
  \left|\log \frac{1-B_2(x,y,\theta^{-1}\omega)}{1-B_1(x,y,\theta^{-1}\omega)}\right|&\leq \frac{|B_1(x,y,\theta^{-1}\omega)-B_2(x,y,\theta^{-1}\omega)|}{1-\max\{B_1(x,y,\theta^{-1}\omega),B_2(x,y,\theta^{-1}\omega)\}}\nonumber\\
  &\overset{\eqref{bound of B1}}\leq \frac{K_4d_u(\gamma(\omega),\tilde{\gamma}(\omega))^\nu\cdot d(x,y)^\kappa}{1-\exp(-\frac{a}{2}d(x,y)^\kappa)}\nonumber\\
  &\leq K_5 d_u(\gamma(\omega),\tilde{\gamma}(\omega))^\nu,\label{log1-b21-b1}
\end{align}where $K_5:=K_4\cdot \sup_{z\in(0,1)}\frac{z^\kappa}{1-\exp(-\frac{a}{2}z^\kappa)}<\infty$, and in the first ``$\leq$", we use the following fact
\begin{equation*}
  |\log(1-z_1)-\log(1-z_2)|\leq \frac{|z_1-z_2|}{1-\max\{z_1,z_2\}}\mbox{ for any }z_1,z_2\in(0,1).
\end{equation*} Hence we have
\begin{equation}\label{1st term of inequality}
  |\log \hat{\tau}_2(\theta^{-1}\omega)/\hat{\tau}_1(\theta^{-1}\omega)|\overset{\eqref{relation tau b},\eqref{log1-b21-b1}}\leq 2K_5d_u(\gamma(\omega),\tilde{\gamma}(\omega))^\nu.
\end{equation}Let $K_0=b(2K_3+2K_5)+K_3$, then by \eqref{estimate    lemma 4.1 1}, $(\ref{2nd term    of inequality})$, \eqref{estimate on d+tilde gamma i}, $(\ref{d leq d+log})$ and $(\ref{1st term of inequality})$, Sublemma \ref{sublemma} is proved.
\end{proof}

From now on, we fix any  $\lambda_2\in(\max\{\Lambda_1,e^{-\lambda\nu}\},1)$, $b>b_0(\lambda_2,\Lambda_1)=\frac{1}{\lambda_2-\Lambda_1}$ and $c>c_0=c_0(b,\nu)$ as in Lemma \ref{invariance of C(b,c,v)}. We have shown that
\begin{equation*}
  L_{\theta^{-1}\omega}(C_{\theta^{-1}\omega}(b,c,\nu))\subset C_\omega(\lambda_2b,\lambda_2c,\nu)\subset C_\omega(b,c,\nu),\ \forall \omega\in\Omega.
\end{equation*}

Recall that $\epsilon^*$ is a constant satisfying \eqref{eq def epsilon star}. Let $\delta\in(0,\epsilon^*/8)$ be the constant in Lemma \ref{local product structure} corresponding to $\epsilon^*/8$, i.e., for any $x,y\in M$, $d(x,y)<\delta$, for all $\omega\in \Omega$, we have
\begin{equation*}
  W^s_{\epsilon^*/8}(x,\omega)\cap W^u_{\epsilon^*/8}(y,\omega)\not=\emptyset.
\end{equation*} Now let $\{B_{\delta/4}(x)\}_{x\in M}$ be an open cover of $M$. Pick a subcover $\{B_{\delta/4}(x_i)\}_{i=1}^l$ by the compactness of $M$. Now by the definition of topological mixing on fibers, there exists a $N\in\mathbb{N}$ such that for any $n\geq N$, $\omega\in\Omega$,
\begin{equation}\label{N property 1}
  \phi^n(B_{\delta/4}(x_i)\times \{\omega\} )\cap(B_{\delta/4}(x_j)\times \{\theta^n\omega\} )\not=\emptyset\ \mbox{for any }1\leq i,j\leq l.
\end{equation}Moreover, we pick $N$ large enough such that
\begin{equation}\label{N property 2}
  e^{\lambda N}\epsilon^*\geq 24\epsilon,\ \mbox{and}\
  e^{-\lambda N}\epsilon^*<2\delta.
\end{equation}From now on, we fix this constant $N$. Next, we are going to show that the diameter of image of $L^N_{\theta^{-N}\omega}:C_{{\theta^{-N}\omega}}(b,c,\nu)\rightarrow C_\omega(b,c,\nu)$ with respect to the projective metric on $C_\omega(b,c,\nu)$ is finite, and this finite diameter is independent of $\omega\in\Omega$, where we recall
\begin{equation*}
  L^N_{\theta^{-N}\omega}= L_{\theta^{-1}\omega}\circ \cdots\circ L_{\theta^{-(N-1)}\omega}\circ L_{\theta^{-N}\omega}.
\end{equation*}
\begin{lemma}\label{lemma diamter of LN finite}
  There exists a constant $D_2=D_2(\lambda_2,a,b,c,N)>0$ such that for any $\omega\in\Omega$,
  \begin{equation}
    \sup\left\{d_\omega(\varphi_1,\varphi_2):\ \varphi_1,\varphi_2\in L^N_{\theta^{-N}\omega} C_{\theta^{-N}\omega}(b,c,\nu)\right\}\leq D_2<\infty,
  \end{equation}where $d_\omega$ is the projective metric on $C_\omega(b,c,\nu)$.
\end{lemma}
\begin{proof}[Proof of Lemma \ref{lemma diamter of LN finite}]
  By Lemma \ref{invariance of C(b,c,v)},  $L_{\theta^{-1}\omega}(C_{\theta^{-1}\omega}(b,c,\nu))\subset C_\omega(\lambda_2b,\lambda_2c,\nu)\subset C_\omega(b,c,\nu)$, for all $\omega\in\Omega$. For any $\varphi_1,\varphi_2\in L^N_{\theta^{-N}\omega}C_{\theta^{-N}\omega}(b,c,\nu)\subset C_\omega(\lambda_2b,\lambda_2c,\nu)$, we need to estimate $\alpha_\omega(\varphi_1,\varphi_2)$ and $\beta_\omega(\varphi_1,\varphi_2)$ as in $\eqref{alpha omega phi1phi2}$ and $\eqref{beta omega phi1phi2}$.

   By $(\ref{xi omega})$ and condition \ref{C2}, for any $\rho^\prime(\omega),\rho^{\prime\prime}(\omega)\in D(a/2,\kappa,\gamma(\omega))$ with $\int_{\gamma(\omega)}\rho^\tau(\omega)dm_{\gamma(\omega)}=1$ for $\tau=\prime,\prime\prime$, one has
  \begin{align*}
  \xi_\omega(\rho^\prime,\rho^{\prime\prime},\varphi_1,\varphi_2)&=\frac{\exp(bd_{\gamma(\omega)}^{a,\kappa}(\rho^\prime(\omega),\rho^{\prime\prime}(\omega)))-\int_{\gamma(\omega)}\varphi_2\rho^{\prime\prime} dm_{\gamma(\omega)}/\int_{\gamma(\omega)}\varphi_2\rho^\prime dm_{\gamma(\omega)}}{\exp(bd_{\gamma(\omega)}^{a,\kappa}(\rho^\prime(\omega),\rho^{\prime\prime}(\omega)))-\int_{\gamma(\omega)}\varphi_1\rho^{\prime\prime} dm_{\gamma(\omega)}/\int_{\gamma(\omega)}\varphi_1\rho^\prime dm_{\gamma(\omega)}}\\
  &\geq \frac{\exp(bd_{\gamma(\omega)}^{a,\kappa}(\rho^\prime(\omega),\rho^{\prime\prime}(\omega)))-\exp(b\lambda_2d_{\gamma(\omega)}^{a,\kappa}(\rho^\prime(\omega),\rho^{\prime\prime}(\omega)))}{\exp(bd_{\gamma(\omega)}^{a,\kappa}(\rho^\prime(\omega),\rho^{\prime\prime}(\omega)))-\exp(-b\lambda
  _2d_{\gamma(\omega)}^{a,\kappa}(\rho^\prime(\omega),\rho^{\prime\prime}(\omega)))}\\
  &\geq \tau_3,
  \end{align*}where $\tau_3:=\inf\{\frac{z-z^{\lambda_2}}{z-z^{-\lambda_2}}:\ z>1\}=\frac{1-\lambda_2}{1+\lambda_2}\in(0,1)$. Similarly, we have
  \begin{equation*}
    \xi_\omega(\rho^\prime,\rho^{\prime\prime},\varphi_1,\varphi_2)\leq \tau_4
  \end{equation*}for $\tau_4=\sup\{\frac{z-z^{-\lambda_2}}{z-z^{\lambda_2}}:\ z>1\}=\frac{1+\lambda_2}{1-\lambda_2}\in(1,\infty)$. Likewise, $\eta_\omega(\rho,\tilde{\rho},\varphi_1,\varphi_2),\eta_\omega(\tilde{\rho},\rho,\varphi_1,\varphi_2)\in[\tau_3,\tau_4]$ for any $\rho(\omega)\in D(a_1,\kappa_1,\gamma(\omega))$ and corresponding $\tilde{\rho}(\omega)\in D(a/2,\kappa,\tilde{\gamma}(\omega))$, where $\eta_\omega(\rho,\tilde{\rho},\varphi_1,\varphi_2)$ and $\eta_\omega(\tilde{\rho},\rho,\varphi_1,\varphi_2)$ are defined in \eqref{eta omega} and \eqref{eta omega2} respectively.

  Let $C_{+,\omega}$ be the collection of all bounded measurable functions $\varphi:M\rightarrow \mathbb{R}$ only satisfying condition \ref{C1}, i.e., $\int_{\gamma(\omega)}\varphi(x)\rho(x,\omega)dm_{\gamma(\omega)}(x)>0$ for every local stable submanifold $\gamma(\omega)\subset M$ having size between $\frac{A(\epsilon)}{4J^2}$ and $A(\epsilon)$, and every $\rho(\omega)\in D(a/2,\kappa,\gamma(\omega))$ satisfying $\int_{\gamma(\omega)}\rho(x,\omega)dm_{\gamma(\omega)}(x)=1$. Next, we introduce the projective metric on $C_{+,\omega}$. For $\varphi_1,\varphi_2\in C_{+,\omega}$, we define
\begin{align*}
  \alpha_{+,\omega}(\varphi_1,\varphi_2) &:=\sup\{t>0:\varphi_2-t\varphi_1\in C_{+,\omega}\},\\
  \beta_{+,\omega}(\varphi_1,\varphi_2)&:=\inf\{s>0:\ s\varphi_1-\varphi_2\in C_{+,\omega}\},
\end{align*} with the convention that $\sup\emptyset=0$ and $\inf\emptyset=+\infty$, and let
\begin{equation}\label{def of d+omega}
  d_{+,\omega}(\varphi_1,\varphi_2):=\log \frac{\beta_{+,\omega}(\varphi_1,\varphi_2)}{\alpha_{+,\omega}(\varphi_1,\varphi_2)}
\end{equation}with the convention that $d_{+,\omega}(\varphi_1,\varphi_2)=\infty$ if $\alpha_{+,\omega}(\varphi_1,\varphi_2)=0$ or $\beta_{+,\omega}(\varphi_1,\varphi_2)=\infty$.
 By similar computation as before, we have
\begin{align}
  \alpha_{+,\omega}(\varphi_1,\varphi_2)  & =\inf\left\{\frac{\int_{\gamma(\omega)}\varphi_2\rho(\omega) dm_{\gamma(\omega)}}{\int_{\gamma(\omega)}\varphi_1\rho(\omega) dm_{\gamma(\omega)}}\right\},\label{alpha + omega phi1phi2}
\end{align}and
\begin{align}
   \beta_{+,\omega}(\varphi_1,\varphi_2)  & =\sup\left\{\frac{\int_{\gamma(\omega)}\varphi_2\rho(\omega) dm_{\gamma(\omega)}}{\int_{\gamma(\omega)}\varphi_1\rho(\omega) dm_{\gamma(\omega)}}\right\},\label{beta + omega phi1phi2}
\end{align}where the supremum and infimum runs over all local stable leaf $\gamma(\omega)$ having size between $\frac{A(\epsilon)}{4J^2}$ and $A(\epsilon)$, and any $\rho(\omega)\in D(\frac{a}{2},\kappa,\gamma(\omega))$ with $\int_{\gamma(\omega)}\rho(x,\omega)dm_{\gamma(\omega)}=1$.

   Comparing $(\ref{alpha omega phi1phi2})$ with $(\ref{alpha   + omega phi1phi2})$ and $(\ref{beta omega phi1phi2})$ with $(\ref{beta +    omega phi1phi2})$, and noticing that $D(a_1,\kappa_1,\gamma(\omega))\subset D(\frac{a}{2},\kappa,\gamma(\omega))$, we have for $\varphi_1,\varphi_2\in L^N_{\theta^{-N}\omega}C_{\theta^{-N}\omega}(b,c,\nu)\subset C_\omega(\lambda_2b,\lambda_2c,\nu)$,
  \begin{align*}
   \alpha_\omega(\varphi_1,\varphi_2)  &\geq \tau_3\alpha_{+,\omega}(\varphi_1,\varphi_2),\mbox{ and }
   \beta_\omega(\varphi_1,\varphi_2)\leq\tau_4\beta_{+,\omega}(\varphi_1,\varphi_2).
  \end{align*}As a consequence, we have
  \begin{equation}\label{domgealeq d +omega}
    d_\omega(\varphi_1,\varphi_2)\leq d_{+,\omega}(\varphi_1,\varphi_2)+\log \frac{\tau_4}{\tau_3}.
  \end{equation}Note that  for $\varphi_1,\varphi_2\in L^N_{\theta^{-N}\omega}C_{\theta^{-N}\omega}(b,c,\nu)$, we have
  \begin{align}
   d_{+,\omega}(\varphi_1,\varphi_2) &= \log\sup\left\{ \frac{\int_{\gamma^{\prime\prime}(\omega)}\varphi_2\rho{\prime\prime}(\omega) dm_{\gamma{\prime\prime}(\omega)}/\int_{\gamma{\prime\prime}(\omega)}\varphi_1\rho{\prime\prime}(\omega) dm_{\gamma{\prime\prime}(\omega)}}{\int_{\gamma^\prime(\omega)}\varphi_2\rho^\prime (\omega) dm_{\gamma^\prime(\omega)}/\int_{\gamma^\prime(\omega)}\varphi_1\rho^\prime (\omega) dm_{\gamma^\prime(\omega)}}\right\}\label{d+omega12}\\
    &= \log\sup\left\{\frac{\int_{\gamma^{\prime\prime}(\omega)}\varphi_2\rho{\prime\prime}(\omega) dm_{\gamma{\prime\prime}(\omega)}}{\int_{\gamma^\prime(\omega)}\varphi_2\rho^\prime(\omega) dm_{\gamma^\prime(\omega)}}\cdot \frac{\int_{\gamma{\prime}(\omega)}\varphi_1\rho{\prime} (\omega) dm_{\gamma{\prime}(\omega)}}{\int_{\gamma{\prime\prime}(\omega)}\varphi_1\rho{\prime\prime} (\omega) dm_{\gamma{\prime\prime}(\omega)}}\right\},\nonumber
  \end{align}where the sup runs over all random local stable leaves $\gamma^\tau(\omega)$ having size between $\frac{A(\epsilon)}{4J^2}$ and $A(\epsilon)$, $\rho^\tau(\omega)\in D(\frac{a}{2},\kappa,\gamma^\tau(\omega))$  with $\int_{\gamma^\tau(\omega)}\rho^\tau(x,\omega)dm_{\gamma^\tau(\omega)}$=1 for $\tau=\prime,\prime\prime$. Next, we are going to estimate
  \begin{equation}\label{phirho2/phirho1}
    \frac{\int_{\gamma^{\prime\prime}(\omega)}\varphi(x)\rho{\prime\prime}(x,\omega) dm_{\gamma{\prime\prime}(\omega)}(x)}{\int_{\gamma^\prime(\omega)}\varphi(x)\rho^\prime(x,\omega) dm_{\gamma^\prime(\omega)}(x)}
  \end{equation}for $\varphi\in L^N_{\theta^{-N}\omega}C_{\theta^{-N}\omega}(b,c,\nu)$ and any $\rho^\prime(\omega)\in D(a/2,\kappa,\gamma^\prime(\omega))$, $\rho^{\prime\prime}(\omega)\in D(a/2,\kappa,\gamma^{\prime\prime}(\omega))$ with $\int_{\gamma^{\prime\prime}(\omega)}\rho^{\prime\prime}(\omega)dm_{\gamma^{\prime\prime}(\omega)}=\int_{\gamma^{\prime}(\omega)}\rho^{\prime}(\omega)dm_{\gamma^{\prime}(\omega)}=1$.
  For $\varphi\in L_{\theta^{-N}\omega}^NC_{\theta^{-N}\omega}(b,c,\nu)$, let
  \begin{align}\label{def of bark1bark2}
    \bar{k}_1(\omega)  =\left(\int_{\gamma^\prime(\omega)}\varphi(x)dm_{\gamma^\prime(\omega)}(x)\right)^{-1},\
    \bar{k}_2(\omega) =\left(\int_{\gamma^{\prime\prime}(\omega)}\varphi(x)dm_{\gamma^{\prime\prime}(\omega)}(x)\right)^{-1},
  \end{align}and we define
  \begin{align*}
    k_1(x,\omega)&:=\bar{k}_1(\omega)/\int_{\gamma^\prime(\omega)}\bar{k}_1(\omega)dm_{\gamma^\prime(\omega)} \mbox{ for }x\in \gamma^\prime(\omega),\\
    k_2(x,\omega)&:=\bar{k}_2(\omega)/\int_{\gamma^{\prime\prime}(\omega)}\bar{k}_2(\omega)dm_{\gamma^{\prime\prime}(\omega)}  \mbox{ for }x\in \gamma^{\prime\prime}(\omega).
  \end{align*}
  By construction, the constant function $k_1(\omega)\in D(a/2,\kappa,\gamma^\prime(\omega))$ and $k_2(\omega)\in D(a/2,\kappa,\gamma^{\prime\prime}(\omega))$ and $\int_{\gamma^\prime(\omega)}k_1(x,\omega)dm_{\gamma^\prime(\omega)}=1$ and $\int_{\gamma^{\prime\prime}(\omega)}k_2(x,\omega)dm_{\gamma^{\prime\prime}(\omega)}=1$. Now by \ref{C2},
  \begin{align}
     & \ \ \ \ \frac{\int_{\gamma^{\prime\prime}(\omega)}\varphi(x)\rho{\prime\prime}(x,\omega) dm_{\gamma{\prime\prime}(\omega)}(x)}{\int_{\gamma^\prime(\omega)}\varphi(x)\rho^\prime(x,\omega) dm_{\gamma^\prime(\omega)}(x)}\nonumber\\
     &=\frac{\int_{\gamma^{\prime\prime}(\omega)}\varphi(x)\rho{\prime\prime}(x,\omega) dm_{\gamma{\prime\prime}(\omega)}(x)}{\int_{\gamma^{\prime\prime}(\omega)}\varphi(x)k_2(x,\omega) dm_{\gamma{\prime\prime}(\omega)}(x)}\cdot \frac{\int_{\gamma^{\prime}(\omega)}\varphi(x)k_1(x,\omega) dm_{\gamma^\prime(\omega)}(x)}{\int_{\gamma^\prime(\omega)}\varphi(x)\rho^\prime(x,\omega) dm_{\gamma^\prime(\omega)}(x)}\cdot\frac{\int_{\gamma^{\prime}(\omega)}\bar{k}_1(\omega)dm_{\gamma^{\prime}(\omega)}}{\int_{\gamma^{\prime\prime}(\omega)}\bar{k}_2(\omega)dm_{\gamma^{\prime\prime}(\omega)}}\nonumber\\
     &\leq e^{\lambda_2 b d_{\gamma^{\prime\prime}(\omega)}^{a,\kappa}(\rho^{\prime\prime}(\cdot,\omega),k_2(\cdot,\omega))}\cdot e^{\lambda_2bd_{\gamma^{\prime}(\omega)}^{a,\kappa}(\rho^{\prime}(\cdot,\omega),k_1(\cdot,\omega))}\cdot\frac{\int_{\gamma^{\prime}(\omega)}\bar{k}_1(\omega)dm_{\gamma^{\prime}(\omega)}}{\int_{\gamma^{\prime\prime}(\omega)}\bar{k}_2(\omega)dm_{\gamma^{\prime\prime}(\omega)}}.\label{eqvarrho2varrho1}
  \end{align}
  \begin{sublemma}\label{sublemma 7.2.}
    There exists a constant $D_1=D_1(a,b,c,N)<\infty$ such that
    \begin{equation}\label{definition of D1}
      \frac{\int_{\gamma^{\prime}(\omega)}\bar{k}_1(\omega)dm_{\gamma^{\prime}(\omega)}}{\int_{\gamma^{\prime\prime}(\omega)}\bar{k}_2(\omega)dm_{\gamma^{\prime\prime}(\omega)}}\leq D_1.
    \end{equation}
  \end{sublemma}
  We postpone the proof of Sublemma \ref{sublemma 7.2.} and finish the proof of Lemma \ref{lemma diamter of LN finite}.
  Now we let $\tau_5=\sup\{\frac{z-z^{-1/2}}{z-z^{1/2}}:\ z>1\}\in (1,\infty)$ and $\tau_6=\inf\{\frac{z-z^{1/2}}{z-z^{-1/2}}:\ z>1\}\in(0,1)$. Similar to $(\ref{bound of d gamma rho1rho2})$, the diameter of $D(a/2,\kappa,\gamma(\omega))$ with respect to the projective metric on $D(a,\kappa,\gamma(\omega))$ is bounded by $4a+\log \tau_5/\tau_6$. Therefore,
  \begin{align*}
    d_{\gamma^{\prime\prime}(\omega)}^{a,\kappa}(\rho^{\prime\prime}(\omega),\bar{k}_2(\omega)) &\leq 4a+\log\tau_5/\tau_6,\mbox{ and }
    d_{\gamma^{\prime}(\omega)}^{a,\kappa}(\rho^{\prime}(\omega),\bar{k}_1(\omega))\leq 4a+\log\tau_5/\tau_6.
  \end{align*}Hence
   \begin{equation*}
     \frac{\int_{\gamma^{\prime\prime}(\omega)}\varphi(x)\rho{\prime\prime}(x,\omega) dm_{\gamma{\prime\prime}(\omega)}(x)}{\int_{\gamma^\prime(\omega)}\varphi(x)\rho^\prime(x,\omega) dm_{\gamma^\prime(\omega)}(x)}\overset{\eqref{eqvarrho2varrho1}}\leq e^{2\lambda_2 b(4a+\log\tau_5/\tau_6)}D_1.
   \end{equation*} As a consequence, we have $$d_{+,\omega}(\varphi_1,\varphi_2)\overset{\eqref{d+omega12}}\leq \log( e^{4\lambda_2 b(4a+\log\tau_5/\tau_6)}D_1^2),$$ and
$    d_\omega(\varphi_1,\varphi_2)\overset{\eqref{domgealeq d +omega}}\leq \log(e^{4\lambda_2 b(4a+\log\tau_5/\tau_6)}D_1^2)+\log\tau_4/\tau_3:=D_2,$
which finishes the proof of Lemma \ref{lemma diamter of LN finite}.
\end{proof}
\begin{proof}[Proof of Sublemma \ref{sublemma 7.2.}]
For any $\omega\in\Omega$ and $\psi\in C_\omega(b,c,\nu)$, we define
\begin{equation*}
  \|\psi\|_{\omega,+}=\sup\frac{\int_{\gamma(\omega)}\psi(x)\rho(x,\omega)dm_{\gamma(\omega)}(x)}{\int_{\gamma(\omega)}\rho(x,\omega)dm_{\gamma(\omega)}(x)},\  \|\psi\|_{\omega,-}=\inf\frac{\int_{\gamma(\omega)}\psi(x)\rho(x,\omega)dm_{\gamma(\omega)}(x)}{\int_{\gamma(\omega)}\rho(x,\omega)dm_{\gamma(\omega)}(x)},
\end{equation*}where the supremum and infimum runs over all local stable leaf $\gamma(\omega)\subset M$ having size between $\frac{A(\epsilon)}{4J^2}$ and $A(\epsilon)$, and any $\rho(\omega)\in D(a_1,\kappa,\gamma(\omega)).$

By noticing that constant function $\textbf{1}$ belongs to $D(a_1,\kappa,\gamma(\omega))$, we have
\begin{equation*}
    \frac{\int_{\gamma^{\prime}(\omega)}\bar{k}_1(\omega)dm_{\gamma^{\prime}(\omega)}}{\int_{\gamma^{\prime\prime}(\omega)}\bar{k}_2(\omega)dm_{\gamma^{\prime\prime}(\omega)}}=
    \left(\frac{\int_{\gamma^{\prime\prime}(\omega)}\varphi\cdot\textbf{1} dm_{\gamma^{\prime\prime}(\omega)}}{\int_{\gamma^{\prime\prime}(\omega)}\textbf{1} dm_{\gamma^{\prime\prime}(\omega)}}\right)\big/\left(\frac{\int_{\gamma^{\prime}(\omega)}\varphi\cdot\textbf{1} dm_{\gamma^{\prime}(\omega)}}{\int_{\gamma^{\prime\prime}(\omega)} \textbf{1}dm_{\gamma^{\prime}(\omega)}}\right)\leq \frac{\|\varphi\|_{\omega,+}}{\|\varphi\|_{\omega,-}},
\end{equation*}and here $\bar{k}_1(\omega)$ and $\bar{k}_2(\omega)$ was defined in \eqref{def of bark1bark2} by using $\varphi\in L_{\theta^{-N}\omega}^NC_{\theta^{-N}\omega}(b,c,\nu)\subset C_\omega(b,c,\nu)$.
Therefore, Lemma \ref{sublemma 7.2.} is proved if there exists $D_1=D_1(a,b,c,N)>0$ independent of $\omega\in\Omega$ such that
\begin{equation}\label{Lnvarphi+ over Lnvarphi-}
  \frac{\|L_{\theta^{-N}\omega}^N\varphi\|_{\omega,+}}{\|L_{\theta^{-N}\omega}^N\varphi\|_{\omega,-}}\leq D_1\mbox{ for any $\varphi\in C_{\theta^{-N}\omega}(b,c,\nu)$.}
\end{equation}

 We need some preliminary inequalities before we start the proof of \eqref{Lnvarphi+ over Lnvarphi-}. By the continuity of $f_\omega$ in $\omega$, there exists a constant $K_6>0$ such that
\begin{equation}\label{bounds of Dxfomega}
  K_6^{-1}\leq |\det D_xf_\omega|\leq K_6\mbox{ for all }(x,\omega)\in M\times\Omega.
\end{equation}We recall that for $\varphi \in C_{\theta^{-n}\omega}(b,c,\nu)$,
\begin{equation*}
  (L_{\theta^{-n}\omega}^n\varphi)(x)=\frac{\varphi(f_\omega^{-n}x)}{|\det D_{f_\omega^{-n}x}f_{\theta^{-n}\omega}^n|}.
\end{equation*} Then \eqref{bounds of Dxfomega} implies that
\begin{equation}\label{bounds of Ln 1}
 (K_6)^{-n}\leq \inf_{x\in M} |(L_{\theta^{-n}\omega}^n \textbf{1})(x)|\leq \sup_{x\in M} |(L_{\theta^{-n}\omega}^n \textbf{1})(x)|\leq (K_6)^n \mbox{ for any }\omega\in\Omega,\ n\in\mathbb{N}.
\end{equation} For any $n\geq 1$ and $\gamma(\omega)\subset M$ local stable leaf having size between $\frac{A(\epsilon)}{4J^2}$ and $A(\epsilon)$, we divide $f_{\omega}^{-n}\gamma(\omega)$ into finite pieces of connected local stable leaves having size between $\frac{A(\epsilon)}{4J^2}$ and $A(\epsilon)$, named $\gamma_i(\theta^{-n}\omega)$ for $i$ belonging to a finite index set such that $\gamma_i(\theta^{-n}\omega)\cap \gamma_j(\theta^{-n}\omega)=\partial\gamma_i(\theta^{-n}\omega)\cap \partial \gamma_j(\theta^{-n}\omega)$ for $i\not=j$. For any $\rho(\omega)\in D(a_1,\kappa,\gamma(\omega))$, we define
\begin{equation*}
  \rho_i(x,\theta^{-n}\omega)=\frac{|\det D_xf_{\theta^{-n}\omega}^n|_{E^s(x,\theta^{-n}\omega)}|}{|\det D_xf_{\theta^{-n}\omega}^n|}\rho(f_{\theta^{-n}\omega}^nx,\omega)\mbox{ for }x\in\gamma_i(\theta^{-n}\omega).
\end{equation*}By Remark \ref{remark4.1}, we have $\rho_i(\theta^{-n}\omega)\in D(a_1,\kappa,\gamma_i(\theta^{-n}\omega))$. Likewise $(\ref{Lphirho=phirhoi})$, we have
\begin{align*}
  \int_{\gamma(\omega)}(L^n_{\theta^{-n}\omega}\varphi)(x)\rho(x,\omega)dm_{\gamma(\omega)}(x)&=
  \sum_{i}\int_{\gamma_i(\theta^{-n}\omega)}\varphi(x)\rho_i(x,\theta^{-n}\omega)dm_{\gamma_i(\theta^{-n}\omega)}(x).
\end{align*}Now we have
\begin{align*}
&\ \ \ \ \ \frac{\int_{\gamma(\omega)}(L^n_{\theta^{-n}\omega}\varphi)(x)\rho(x,\omega)dm_{\gamma(\omega)}(x)}{\int_{\gamma(\omega)}\rho(x,\omega)dm_{\gamma(\omega)}(x)}=
\frac{\sum_{i}\int_{\gamma_i(\theta^{-n}\omega)}\varphi(x)\rho_i(x,\theta^{-n}\omega)dm_{\gamma_i(\theta^{-n}\omega)}(x)}{\int_{\gamma(\omega)}\rho(x,\omega)dm_{\gamma(\omega)}(x)}\\
&\leq \frac{\sum_{i}\int_{\gamma_i(\theta^{-n}\omega)}\rho_i(x,\theta^{-n}\omega)dm_{\gamma_i(\theta^{-n}\omega)}(x)\cdot\|\varphi\|_{\theta^{-n}\omega,+}}{\int_{\gamma(\omega)}\rho(x,\omega)dm_{\gamma(\omega)}(x)}\\
&=\frac{\int_{\gamma(\omega)}(L^n_{\theta^{-n}\omega}\textbf{1})(x)\rho(x,\omega)dm_{\gamma(\omega)}(x)}{\int_{\gamma(\omega)}\rho(x,\omega)dm_{\gamma(\omega)}(x)}\cdot\|\varphi\|_{\theta^{-n}\omega,+}\\
&\overset{\eqref{bounds of Ln 1}}\leq (K_6)^n\cdot\|\varphi\|_{\theta^{-n}\omega,+},
\end{align*}where the last inequality holds since $\rho(x,\omega)>0$ for all $x\in\gamma(\omega)$ by \ref{D1}. Since $\gamma(\omega)$ and $\rho(\omega)\in D(a_1,\kappa,\gamma(\omega))$ are arbitrary, we get
\begin{equation}\label{Ln varphi leq Kn varphi}
  \|L^n_{\theta^{-n}\omega}\varphi\|_{\omega,+}\leq (K_6)^n\|\varphi\|_{\theta^{-n}\omega,+}.
\end{equation}

Next, there exists a constant $D_3>0$ such that for any $\omega\in \Omega$, any $\gamma(\omega)\subset M$ local stable leaf having size between $\frac{A(\epsilon)}{4J^2}$ and $A(\epsilon)$, $\rho_1(\omega),\rho_2(\omega)\in D(a/2,\kappa,\gamma(\omega))$, any $\varphi\in C_\omega(b,c,\nu)$, one has
\begin{equation}\label{Definition of D3}
  \sup_{z\in \gamma(\omega)}\frac{\rho_2(z,\omega)\int_{\gamma(\omega)}\varphi(x)\rho_1(x,\omega)dm_{\gamma(\omega)}(x)}{\rho_1(z,\omega)\int_{\gamma(\omega)}\varphi(x)\rho_2(x,\omega)dm_{\gamma(\omega)}(x)}\leq D_3<\infty.
\end{equation}In fact, for any $z\in \gamma(\omega)$, then by condition \ref{C2} and the fact that the diameter of $D(a/2,\kappa,\gamma(\omega))$ in $D(a,\kappa,\gamma(\omega))$ with respect to the projective metric $d_{\gamma(\omega)}^{a,\kappa}$ is bounded by $4a+\log(\tau_5/\tau_6)$, we have
\begin{align*}
   &\ \ \ \ \ \frac{\rho_2(z,\omega)}{\rho_1(z,\omega)}\cdot\frac{\int_{\gamma(\omega)}\varphi(x)\rho_1(x,\omega)dm_{\gamma(\omega)}(x)/\int_{\gamma(\omega)}\rho_1(\omega)dm_{\gamma(\omega)}}{\int_{\gamma(\omega)}\varphi(x)\rho_2(x,\omega)dm_{\gamma(\omega)}(x)/\int_{\gamma(\omega)}\rho_2(\omega)dm_{\gamma(\omega)}}\cdot\frac{\int_{\gamma(\omega)}\rho_1(\omega)dm_{\gamma(\omega)}}{\int_{\gamma(\omega)}\rho_2(\omega)dm_{\gamma(\omega)}}\\
    &\leq \frac{\rho_2(z,\omega)}{\int_{\gamma(\omega)}\rho_2(x,\omega)dm_{\gamma(\omega)}(x)}\cdot e^{bd_{\gamma(\omega)}^{a,\kappa}(\rho_1(\omega),\rho_2(\omega))}\cdot\frac{\int_{\gamma(\omega)}\rho_1(x,\omega)dm_{\gamma(\omega)}(x)}{\rho_1(z,\omega)}\\
  &\leq \frac{\rho_2(z,\omega)}{\int_{\gamma(\omega)}\rho_2(x,\omega)dm_{\gamma(\omega)}(x)}\cdot e^{(4a+\log\tau_5/\tau_6)b}\cdot\frac{\int_{\gamma(\omega)}\rho_1(x,\omega)dm_{\gamma(\omega)}(x)}{\rho_1(z,\omega)}\\
   &\leq e^{a/2(diam(\gamma(\omega)))^\kappa}\cdot e^{a/2(diam(\gamma(\omega)))^\kappa}\cdot e^{(4a+\log\tau_5/\tau_6)b}\\
   &=e^{a+b(4a+\log\tau_5/\tau_6)}:=D_3.
\end{align*}

Now we are in the position to prove $(\ref{Lnvarphi+ over Lnvarphi-})$. Recall that $N\in\mathbb{N}$ satisfying $(\ref{N property 1})$ and $(\ref{N property 2})$. For any $\omega\in\Omega$ and $\varphi\in C_{\theta^{-N}\omega}(b,c,\nu)$, we choose $\gamma_*(\theta^{-N}\omega)$ a local stable leaf having size between $\frac{A(\epsilon)}{4J^2}$ and $A(\epsilon)$, and $\rho_*(\theta^{-N}\omega)\in D(a_1,\kappa,\gamma_*(\theta^{-N}\omega))$ such that
\begin{equation*}
  \frac{\int_{\gamma_*(\theta^{-N}\omega)}\varphi(x)\rho_*(x,\theta^{-N}\omega)dm_{\gamma_*(\theta^{-N}\omega)}}{\int_{\gamma_*(\theta^{-N}\omega)}\rho_*(x,\theta^{-N}\omega)dm_{\gamma_*(\theta^{-N}\omega)}}\geq \frac{1}{2}\|\varphi\|_{\theta^{-N}\omega,+}.
\end{equation*}  We pick $x_*(\theta^{-N}\omega)\in\gamma_*(\theta^{-N}\omega)$ such that
\begin{equation*}
  W_{\epsilon^*}^s(x_*(\theta^{-N}\omega),\theta^{-N}\omega)\subset \gamma_*(\theta^{-N}\omega)\subset W_{\epsilon}^s(x_*(\theta^{-N}\omega),\theta^{-N}\omega).
\end{equation*}
To avoid the size of the holonomy image of $\gamma_*(\theta^{-N}\omega)$ being too large, we break $\gamma_*(\theta^{-N}\omega)$ into pieces of size between $\frac{A(\epsilon)}{4J}$ and $\frac{A(\epsilon)}{J}$. The number of pieces is at most $\frac{A(\epsilon)}{A(\epsilon)/J}+1=J+1$. We can pick one of pieces, named $\gamma_*^1(\theta^{-N}\omega)$, such that
\begin{equation}\label{pick of rhostar}
   \frac{\int_{\gamma_*^1(\theta^{-N}\omega)}\varphi(x)\rho_*(x,\theta^{-N}\omega)dm_{\gamma_*(\theta^{-N}\omega)}}{\int_{\gamma_*(\theta^{-N}\omega)}\rho_*(x,\theta^{-N}\omega)dm_{\gamma_*(\theta^{-N}\omega)}}\geq \frac{1}{2(J+1)}\|\varphi\|_{\theta^{-N}\omega,+}.
\end{equation}

Pick any $\gamma(\omega)\subset M$ local stable leaf having size between $\frac{A(\epsilon)}{4J^2}$ and $A(\epsilon)$, and $x(\omega)\in\gamma(\omega)$ such that
\begin{equation*}
  W_{\epsilon^*}^s(x(\omega),\omega)\subset \gamma(\omega)\subset    W_{\epsilon}^s(x(\omega),\omega).
\end{equation*} Recall that $\{B_{\delta/4}(x_i)\}_{i=1}^l$ is a cover of $M$. Then there exists $i,j\in\{1,...,l\}$ such that $x_*(\theta^{-N}\omega)\in B_{\delta/4}(x_i)$ and $x(\omega)\in B_{\delta/4}(x_j)$. Then by the choice of $N$, $\phi^N(B_{\delta/4}(x_i)\times\{\theta^{-N}\omega\})\cap (B_{\delta/4}(x_j)\times\{\omega\})\not=\emptyset$. Pick $y(\omega)\in f_{\theta^{-N}\omega}^NB_{\delta/4}(x_i)\cap B_{\delta/4}(x_j)$, then
\begin{equation*}
  d(y(\omega),x(\omega))\leq d(y(\omega),x_j)+d(x_j,x(\omega))\leq \delta/4+\delta/4<\delta.
\end{equation*}Then
\begin{equation*}
 y_1(\omega):=W^u_{\epsilon^*/8}(y(\omega),\omega)\cap W^s_{\epsilon^*/8}(x(\omega),\omega)\subset W^u_{\epsilon^*/8}(y(\omega),\omega)\cap\gamma(\omega)
\end{equation*}exists.
Note that by $(\ref{N property 2})$, we have
\begin{equation*}
  d(f_{\omega}^{-N}y(\omega),f_{\omega}^{-N}y_1(\omega))\leq e^{-\lambda N}\cdot \frac{\epsilon^*}{8}\leq \delta/4.
\end{equation*}So
\begin{align*}
 d(x_*(\theta^{-N}\omega),f_{\omega}^{-N}y_1(\omega))&\leq d(x_*(\theta^{-N}\omega),x_i)+d(x_i,f_{\omega}^{-N}y(\omega))+d(f_{\omega}^{-N}y(\omega),f_{\omega}^{-N}y_1(\omega))\\
 &\leq \delta/4+\delta/4+\delta/4<\delta.
\end{align*}As a consequence,
\begin{equation}\label{get close}
  W^s_{\epsilon^*/8}(x_*(\theta^{-N}\omega),\theta^{-N}\omega)\cap W^u_{\epsilon^*/8}(f_{\omega}^{-N}y_1(\omega),\theta^{-N}\omega)\not=\emptyset.
\end{equation}Note that  $f_\omega^{-N}y_1(\omega)\in f_\omega^{-N}\gamma(\omega)$, and $e^{\lambda N}\epsilon^*/8\overset{\eqref{N property 2}}\geq 3\epsilon$, therefore, $W^s_{3\epsilon}(f_\omega^{-N}y_1(\omega),\theta^{-N}\omega)\subset f_\omega^{-N}\gamma(\omega)$. Moreover,  \eqref{get close} implies that $\gamma_*(\omega)$ is sufficient close to $W_{3\epsilon}^s(f_\omega^{-N}y_1(\omega),\theta^{-N}\omega)\subset f_\omega^{-N}\gamma(\omega)$. As a consequence, $f_{\omega}^{-N}\gamma(\omega)$ contains a holonomy image of $\gamma_*^1(\theta^{-N}\omega)$, named $\gamma_1(\theta^{-N}\omega)$. Since the size of $\gamma_*^1(\theta^{-N}\omega)$ is between $\frac{A(\epsilon)}{4J}$ and $\frac{A(\epsilon)}{J}$, by holonomy, the size of  $\gamma_1(\theta^{-N}\omega)$ is between $\frac{A(\epsilon)}{4J^2}$ and $A(\epsilon)$.


Now for any $\rho(\omega)\in D(a_1,\kappa,\gamma(\omega))$, define
\begin{equation*}
  \rho_1(x,\theta^{-N}\omega)=\frac{|\det D_xf_{\theta^{-N}\omega}^N|_{E^s(x,\theta^{-N}\omega)}|}{|\det D_xf_{\theta^{-N}\omega}^N|}\rho(f_{\theta^{-N}\omega}^Nx,\omega)\mbox{ for }x\in\gamma_1(\theta^{-N}\omega).
\end{equation*}By Remark \ref{remark4.1}, $\rho_1(\theta^{-N}\omega)\in D(e^{-\lambda_1}a_1,\kappa,\gamma_1(\theta^{-N}\omega))\subset D(a_1,\kappa,\gamma_1(\theta^{-N}\omega))$.
Let $\tilde{\rho}_1(\theta^{-N}\omega)$ be the density function on $\gamma_*^1(\theta^{-N}\omega)$ defined as $(\ref{definition of tilde rho})$ corresponding to $\rho_1(\theta^{-N}\omega)$, and therefore $\tilde{\rho}_1(\theta^{-N}\omega)\in D(a/2,\kappa,\gamma_*^1(\theta^{-N}\omega))$ by Lemma \ref{rhoi}.
Then
\begin{align*}
 &\frac{\int_{\gamma(\omega)}(L^N_{\theta^{-N}\omega}\varphi)(x)\rho(x,\omega)dm_{\gamma(\omega)}}{\int_{\gamma(\omega)}\rho(x,\omega)dm_{\gamma(\omega)}}\\
\geq & \frac{\int_{\gamma_1(\theta^{-N}\omega)}\varphi(x)\rho_1(x,\theta^{-N}\omega)dm_{\gamma_1(\theta^{-N}\omega)}}{\int_{\gamma(\omega)}\rho(x,\omega)dm_{\gamma(\omega)}}\\
\overset{\ref{C3}}\geq & \frac{\int_{\gamma_*^1(\theta^{-N}\omega)}\varphi(x)\tilde{\rho}_1(x,\theta^{-N}\omega)dm_{\gamma_*(\theta^{-N}\omega)}}{\int_{\gamma(\omega)}\rho(x,\omega)dm_{\gamma(\omega)}}\cdot e^{-cd_u(\gamma_1(\theta^{-N}\omega),\gamma_*(\theta^{-N}\omega))^\nu}.
\end{align*}Pick any $z\in\gamma_*^1(\theta^{-N}\omega)$, and we note that $\rho_*(\theta^{-N}\omega)\in D(a_1,\kappa,\gamma_*^1(\theta^{-N}\omega))\subset D(a/2,\kappa,\gamma_*^1(\theta^{-N}\omega))$ and $\tilde{\rho}_1(\theta^{-N}\omega)\in D(a/2,\kappa,\gamma_*^1(\theta^{-N}\omega))$, then we continue the estimate
\begin{align*}
 &\frac{\int_{\gamma(\omega)}(L^N_{\theta^{-N}\omega}\varphi)(x)\rho(x,\omega)dm_{\gamma(\omega)}}{\int_{\gamma(\omega)}\rho(x,\omega)dm_{\gamma(\omega)}}\\
\overset{\eqref{Definition of D3}}\geq  & \frac{\int_{\gamma_*^1(\theta^{-N}\omega)}\varphi(x)\rho_*(x,\theta^{-N}\omega)dm_{\gamma_*(\theta^{-N}\omega)}}{\int_{\gamma(\omega)}\rho(x,\omega)dm_{\gamma(\omega)}}\cdot e^{-cd_u(\gamma_1(\theta^{-N}\omega),\gamma_*^1(\theta^{-N}\omega))^\nu}\cdot D_3^{-1}\cdot \frac{\tilde{\rho}_1(z,\theta^{-N}\omega)}{\rho_*(z,\theta^{-N}\omega)}\\
 \overset{\eqref{pick of rhostar}}\geq & \frac{\int_{\gamma_*^1(\theta^{-N}\omega)}\rho_*(x,\theta^{-N}\omega)dm_{\gamma_*(\theta^{-N}\omega)}\frac{1}{2(J+1)}\|\varphi\|_{\theta^{-N}\omega,+}}{\int_{\gamma(\omega)}\rho(x,\omega)dm_{\gamma(\omega)}}\cdot e^{-cd_u(\gamma_1(\theta^{-N}\omega),\gamma_*^1(\theta^{-N}\omega))^\nu}\cdot D_3^{-1}\cdot \frac{\tilde{\rho}_1(z,\theta^{-N}\omega)}{\rho_*(z,\theta^{-N}\omega)}\\
 \overset{\eqref{Definition of D3}}\geq & \frac{\int_{\gamma_*^1(\theta^{-N}\omega)}\tilde{\rho}_1(x,\theta^{-N}\omega)dm_{\gamma_*(\theta^{-N}\omega)}\|\varphi\|_{\theta^{-N}\omega,+}}{2(J+1)\int_{\gamma(\omega)}\rho(x,\omega)dm_{\gamma(\omega)}}\cdot e^{-cd_u(\gamma_1(\theta^{-N}\omega),\gamma_*^1(\theta^{-N}\omega))^\nu}\cdot D_3^{-2}\\
\overset{\eqref{eq change variable}}=  &\frac{1}{2(J+1)}e^{-cd_u(\gamma_1(\theta^{-N}\omega),\gamma_*^1(\theta^{-N}\omega))^\nu}\cdot D_3^{-2}\cdot\|\varphi\|_{\theta^{-N}\omega,+}\cdot \frac{\int_{\gamma_1(\theta^{-N}\omega)}\rho_1(x,\theta^{-N}\omega)dm_{\gamma_1(\theta^{-N}\omega)}}{\int_{\gamma(\omega)}\rho(x,\omega)dm_{\gamma(\omega)}}\\
=  &\frac{1}{2(J+1)}e^{-cd_u(\gamma_1(\theta^{-N}\omega),\gamma_*^1(\theta^{-N}\omega))^\nu}\cdot D_3^{-2}\cdot\|\varphi\|_{\theta^{-N}\omega,+}\cdot \frac{\int_{f_{\theta^{-N}\omega}^N\gamma_1(\theta^{-N}\omega)}(L^N_{\theta^{-N}\omega}1)(x)\rho(x,\omega)dm_{\gamma(\omega)}}{\int_{\gamma(\omega)}\rho(x,\omega)dm_{\gamma(\omega)}}\\
 \overset{\eqref{bounds of Ln 1}}\geq & \frac{1}{2(J+1)}e^{-cd_u(\gamma_1(\theta^{-N}\omega),\gamma_*^1(\theta^{-N}\omega))^\nu}\cdot D_3^{-2}\cdot\|\varphi\|_{\theta^{-N}\omega,+}\cdot (K_6)^{-N}\cdot \frac{\int_{f_{\theta^{-N}\omega}^N\gamma_1(\theta^{-N}\omega)}\rho(x,\omega)dm_{\gamma(\omega)}}{\int_{\gamma(\omega)}\rho(x,\omega)dm_{\gamma(\omega)}}.
\end{align*}Pick any $t\in f_{\theta^{-N}\omega}^N\gamma_1(\theta^{-N}\omega)\subset\gamma(\omega)$, then we continue the estimate
\begin{align*}
&\frac{\int_{\gamma(\omega)}(L^N_{\theta^{-N}\omega}\varphi)(x)\rho(x,\omega)dm_{\gamma(\omega)}}{\int_{\gamma(\omega)}\rho(x,\omega)dm_{\gamma(\omega)}}\\
\geq  & \frac{e^{-cd_u(\gamma_1(\theta^{-N}\omega),\gamma_*^1(\theta^{-N}\omega))^\nu}}{2(J+1)}\cdot D_3^{-2}\cdot\|\varphi\|_{\theta^{-N}\omega,+}\cdot (K_6)^{-N}\cdot \frac{\int_{f_{\theta^{-N}\omega}^N\gamma_1(\theta^{-N}\omega)}\rho(x,\omega)/\rho(t,\omega)dm_{\gamma(\omega)}}{\int_{\gamma(\omega)}\rho(x,\omega)/\rho(t,\omega)dm_{\gamma(\omega)}}\\
\geq  & \frac{e^{-cd_u(\gamma_1(\theta^{-N}\omega),\gamma_*^1(\theta^{-N}\omega))^\nu}}{2(J+1)}\cdot D_3^{-2}\cdot\|\varphi\|_{\theta^{-N}\omega,+}\cdot (K_6)^{-N}\cdot e^{-a_1(diam(\gamma(\omega)))^\kappa\cdot 2} \cdot \frac{\int_{f_{\theta^{-N}\omega}^N\gamma_1(\theta^{-N}\omega)}dm_{\gamma(\omega)}}{\int_{\gamma(\omega)}dm_{\gamma(\omega)}}\\
 \geq & \frac{e^{-c\epsilon^\nu-2a_1\epsilon^\kappa}}{2(J+1)}\cdot D_3^{-2}\cdot\|\varphi\|_{\theta^{-N}\omega,+}\cdot (K_6)^{-N} \cdot \left(\inf_{(x,\omega)\in M\times\Omega}m(D_xf_\omega|_{E^s(x,\omega)})\right)^N\cdot \frac{\int_{\gamma_1(\theta^{-N}\omega)}dm_{\gamma_1(\theta^{-N}\omega)}}{\int_{\gamma(\omega)}dm_{\gamma(\omega)}},
\end{align*}where $m(D_xf_\omega|_{E^s(x,\omega)})=\|(D_xf_\omega|_{E^s(x,\omega)})^{-1}\|^{-1}$. Note that both $\gamma_1(\theta^{-N}\omega)\subset M$ and $\gamma(\omega)\subset M$ are local stable leaf having size between $\frac{A(\epsilon)}{4J^2}$ and $A(\epsilon)$. Therefore,
\begin{equation*}
   \frac{\int_{\gamma_1(\theta^{-N}\omega)}dm_{\gamma_1(\theta^{-N}\omega)}}{\int_{\gamma(\omega)}dm_{\gamma(\omega)}}\geq \frac{1}{4J^2}.
\end{equation*}We continue the estimate
\begin{align*}
  &\ \ \ \ \frac{\int_{\gamma(\omega)}(L^N_{\theta^{-N}\omega}\varphi)(x)\rho(x,\omega)dm_{\gamma(\omega)}}{\int_{\gamma(\omega)}\rho(x,\omega)dm_{\gamma(\omega)}}\\
  &\overset{\eqref{Ln varphi leq Kn varphi}}\geq \frac{1}{8J^2(J+1)}e^{-c\epsilon^\nu-a\epsilon^\kappa}\cdot D_3^{-2}\cdot (K_6)^{-2N} \cdot\left(\inf_{(x,\omega)\in M\times\Omega}m(D_xf_\omega|_{E^s(p,\omega)})\right)^N\cdot\|L^N_{\theta^{-N}\omega}\varphi\|_{\omega,+}\\
  &:=(D_1)^{-1}\|L^N_{\theta^{-N}\omega}\varphi\|_{\omega,+}.
\end{align*}Since $\gamma(\omega)$ and $\rho(\omega)\in D(a_1,\kappa,\gamma(\omega))$ are arbitrary, we have $\|L^N_{\theta^{-N}\omega}\varphi\|_{\omega,-}\geq (D_1)^{-1} \|L^N_{\theta^{-N}\omega}\varphi\|_{\omega,+}.$
Hence $(\ref{Lnvarphi+ over Lnvarphi-})$ is proved. The proof of Sublemma \ref{sublemma 7.2.} is complete.
\end{proof}

Note that the Lemma \ref{lemma diamter of LN finite} is proved for all $\omega\in\Omega$, so we also have
\begin{equation}
  \sup\{d_{\theta^N\omega}(L_\omega^N\varphi_1,L_\omega^N\varphi_2):\varphi_1,\varphi_2\in C_\omega(b,c,\nu)\}\leq D_2,\mbox{ for all }\omega\in\Omega,
\end{equation} where $d_{\theta^N\omega}$ is the projective metric on $C_{\theta^N\omega}(b,c,\nu)$.

\begin{lemma}\label{lemma D4}
  There exist a number $D_4$ and a number $\Lambda\in(0,1)$ both depending on $D_2$ and $N$ such that for all $n\geq N$, for all $\omega\in\Omega$,
  \begin{align}
     d_\omega(L^n_{\theta^{-n}\omega}\varphi_{\theta^{-n}\omega}^1,L^n_{\theta^{-n}\omega}\varphi_{\theta^{-n}\omega}^2)&\leq D_4\Lambda^n\mbox{ for any }\varphi_{\theta^{-n}\omega}^1,\varphi_{\theta^{-n}\omega}^2\in C_{\theta^{-n}\omega}(b,c,\nu);\label{contraction of L n}\\
    d_{\theta^n\omega}(L_\omega^n\varphi_\omega^1,L_\omega^n\varphi_\omega^2)&\leq D_4\Lambda^n\mbox{ for any }\varphi_{\omega}^1,\varphi_{\omega}^2\in C_{\omega}(b,c,\nu),\label{contraction of Ln from omega}
  \end{align}where $d_\omega$ and $d_{\theta^n\omega}$ are projective metric on $C_\omega(b,c,\nu)$ and $C_{\theta^n\omega}(b,c,\nu)$ respectively.
\end{lemma}
\begin{proof}
Now we have a linear operator $L^N_{\theta^{-N}\omega}$ maps cone $C_{\theta^{-N}\omega}(b,c,\nu)$ into cone $C_\omega(b,c,\nu)$ with finite diameter of $L^N_{\theta^{-N}\omega}(C_{\theta^{-N}\omega}(b,c,\nu))$ in $C_\omega(b,c,\nu)$, then we apply Birkhoff's inequality (Proposition \ref{birkhoff inequality}) to obtain that for all $\omega\in\Omega$,
\begin{equation}\label{iterate form}
  d_\omega(L^N_{\theta^{-N}\omega}\varphi_1,L^N_{\theta^{-N}\omega}\varphi_2)\leq (1-e^{-D_2}) d_{\theta^{-N}\omega}(\varphi_1,\varphi_2),\mbox{ for all }\varphi_1,\varphi_2\in C_{\theta^{-N}\omega}(b,c,\nu).
\end{equation} 
Now, for any $\omega\in\Omega$, $n\geq N$ and $\varphi^1_{\theta^{-n}\omega},\varphi^2_{\theta^{-n}\omega}\in C_{\theta^{-n}\omega}(b,c,\nu)$, we write $n=(k+1)N+r$ for some $r\{0,...,N-1\}$ and $k\geq 0$. Then
\begin{align*}
 &d_\omega(L^n_{\theta^{-n}\omega}\varphi_{\theta^{-n}\omega}^1,L^n_{\theta^{-n}\omega}\varphi_{\theta^{-n}\omega}^2)\\
\overset{\eqref{iterate form}}\leq & (1-e^{-D_2})^kd_{\theta^{-n+N+r}\omega}(L_{\theta^{-n}\omega}^{N+r}\varphi_{\theta^{-n}\omega}^1,L_{\theta^{-n}\omega}^{N+r}\varphi_{\theta^{-n}\omega}^2)\\
= & (1-e^{-D_2})^{[\frac{n-N}{N}]}d_{\theta^{-n+N+r}\omega}(L^N_{\theta^{-n+r}\omega}L_{\theta^{-n}\omega}^r\varphi^1_{\theta^{-n}\omega},L^N_{\theta^{-n+r}\omega}L_{\theta^{-n}\omega}^r\varphi^2_{\theta^{-n}\omega})\nonumber\\
 \leq & \Lambda^{n-2N}D_2=\frac{D_2}{\Lambda^{2N}}\Lambda^n:=D_4\Lambda^n,
\end{align*}where $\Lambda=(1-e^{-D_2})^{\frac{1}{N}}<1$.
Similarly, for all $\omega\in\Omega$, $n\geq N$ and $\varphi_\omega^1,\varphi_\omega^2\in C_\omega(b,c,\nu)$, one has
\begin{align*}
   d_{\theta^n\omega}(L_\omega^n\varphi_\omega^1,L_\omega^n\varphi_\omega^2)&\leq
   (1-e^{-D_2})^{[\frac{n-N}{N}]}d_{\theta^{N+r}\omega}(L^N_{\theta^r\omega}L_\omega^r\varphi^1_{\omega},L^N_{\theta^r\omega}L_\omega^r\varphi^2_{\omega})
   \leq D_4\Lambda^n.
\end{align*}The proof of Lemma \ref{lemma D4} is complete.
\end{proof}

\subsection{Construction of the random SRB measure}\label{subsection 4.3}
 In this subsection, we will prove that the sequence $(f_{\theta^{-n}\omega}^n)_*m$ converges with respect to the weak$^*$ topology on $Pr(M)$ by using the contraction of $L^n_{\theta^{-n}\omega}$ when $n\geq N$. Moreover, we will prove that the random probability measure $\mu_\omega$ defined by the weak$^*$ limit of $(f_{\theta^{-n}\omega}^n)_*m$ is $\phi-$invariant.

Before we introduce the next lemma, we need some preparations.
For any $\omega\in\Omega$, since the local stable leaves form a partition in a neighborhood of a point on each $M$, we can cover $M$ on the fiber $\{\omega\}$ by finite rectangles, i.e., $\mathcal{R}(\omega)=\{R_1(\omega),...,R_i(\omega),...,R_{k(\omega)}(\omega)\},\ k(\omega)<\infty,
$ satisfying
\begin{enumerate}
  \item[(A)] each $R_i(\omega)$ is a proper rectangle, i.e., it is the closure of its interior;
  \item[(B)] ${\rm interior}(R_i(\omega))\cap {\rm interior}(R_j(\omega))=\emptyset$ if $i\not=j$;
  \item[(C)] each $R_i(\omega)$ is foliated by local stable leaves having size between $\frac{A(\epsilon)}{4J^2}$ and $A(\epsilon)$.
\end{enumerate}
The conditions (A) and (B) of this cover is much weaker than the conditions of random Markov partition constructed in \cite[Sec. 3]{Gund99}. We can obtain (C) by cutting and pasting some sets of the random Markov partition if necessary.

 By Proposition \ref{proposition 7.1}, for any $i\in\{1,...,k(\omega)\}$, there exists a function $H_i(\omega):R_i(\omega)\rightarrow \mathbb{R}^+$ with $\log H_i(\omega)$ $(a_0,\nu_0)-$H\"older continuous on each local stable leaf and for all bounded measurable functions $\psi:M\rightarrow\mathbb{R}$, we have disintegration
\begin{equation}\label{disintegration}
  \int_{R_i(\omega)}\psi(x)dm(x)=\int\int_{\gamma^i(\omega)}\psi(x)H_i(\omega)(x)|_{\gamma^i(\omega)}dm_{\gamma^i(\omega)}(x)d\tilde{m}_{R_i(\omega)}(\gamma^i(\omega)),
\end{equation}where $\gamma^i(\omega)$ denotes the stable leaves in $R_i(\omega)$ and $\tilde{m}_{R_i(\omega)}$ is the quotient measure induced by Riemannian volume measure in the space of local stable leaves in $R_i(\omega)$.

\begin{lemma}\label{Cauchy sequence lemma}
 For any $\omega\in\Omega$, given any positive function sequence $\{\varphi_n\}_{n\in\mathbb{N}}\subset C_\omega(b,c,\nu)$ satisfying
  \begin{equation}\label{integral =1}
    \int_M\varphi_n(x)dm(x)=1\ \mbox{for all }n\in\mathbb{N},
  \end{equation}and
  \begin{equation*}
    d_{+,\omega}(\varphi_n,\varphi_m)\rightarrow 0\mbox{ exponentially as }n,m\rightarrow \infty,
  \end{equation*}where $d_{+,\omega}$ is the projective metric on $C_{+,\omega}$ defined in \eqref{def of d+omega}. Then for any continuous function $\psi:M\rightarrow \mathbb{R}$, the sequence
  $\left\{\int_M\varphi_n(x)\psi(x)dm(x)\right\}_{n\in\mathbb{N}}$ is a Cauchy sequence.
\end{lemma}
\begin{proof}
In the case that $\psi:M\to\mathbb{R}$ is a constant function, $\left\{\int_M\varphi_n(x)\psi(x)dm(x)\right\}_{n\in\mathbb{N}}$ is a constant sequence, therefore Cauchy. In the following, we only need to consider the case that $\psi$ is nonconstant continuous function.

  First, we consider any positive nonconstant continuous function $\psi:M\rightarrow \mathbb{R}$, satisfying
  \begin{equation*}
    |\log\psi|_{\kappa}=\sup_{x,y\in M,x\not=y}\frac{|\log\psi(x)-\log\psi(y)|}{d(x,y)^\kappa}<\frac{a}{4}.
  \end{equation*}
 Let $R_i(\omega)$ and $H_i(\omega)$ be defined as above for $i\in\{1,...,k(\omega)\}$. Note that $\psi(\cdot)H_i(\omega)(\cdot)|_{\gamma^i(\omega)}$ is strictly positive on any local stable leaf $\gamma^i(\omega)\subset  R_i(\omega)$. Moreover, $\log(\psi(\cdot)H_i(\omega)(\cdot))$ is $(a/2,\kappa)-$H\"older continuous on $\gamma^i(\omega)$ since both $\log\psi$ and $\log H_i(\omega)$ are $(a/2,\kappa)$-H\"older continuous on $\gamma^i(\omega)$. Hence $(\psi\cdot H_i(\omega))|_{\gamma^i(\omega)}\in D(\frac{a}{2},\kappa,\gamma^i(\omega))$. Therefore, by the representation of $\beta_{+,\omega}(\varphi_k,\varphi_l)$ and $\alpha_{+,\omega}(\varphi_k,\varphi_l)$ as in $(\ref{beta +    omega phi1phi2})$ and $(\ref{alpha   + omega phi1phi2})$, for all $i\in\{1,...,k(\omega)\}$, any local stable leaf $\gamma^i(\omega)\subset R_i(\omega)$ and $k,l\in\mathbb{N}$, we have
 \begin{equation}\label{over leq beta}
   \begin{split}
   &\ \ \ \ \frac{\int_{\gamma^i(\omega)}\varphi_k(x)\psi(x)H_i(\omega)(x)|_{\gamma^i(\omega)}dm_{\gamma^i(\omega)}(x)}{\int_{\gamma^i(\omega)}\varphi_l(x)\psi(x)H_i(\omega)(x)|_{\gamma^i(\omega)}dm_{\gamma^i(\omega)}(x)}\\
   &=\frac{\int_{\gamma^i(\omega)}\varphi_k(x)\psi(x)H_i(\omega)(x)|_{\gamma^i(\omega)}dm_{\gamma^i(\omega)}(x)/\int_{\gamma^i(\omega)}\psi(x)H_i(\omega)(x)|_{\gamma^i(\omega)}dm_{\gamma^i(\omega)}(x)}{\int_{\gamma^i(\omega)}\varphi_l(x)\psi(x)H_i(\omega)(x)|_{\gamma^i(\omega)}dm_{\gamma^i(\omega)}(x)/\int_{\gamma^i(\omega)}\psi(x)H_i(\omega)(x)|_{\gamma^i(\omega)}dm_{\gamma^i(\omega)}(x)}\\
    &\in[ \alpha_{+,\omega}(\varphi_k,\varphi_l), \beta_{+,\omega}(\varphi_k,\varphi_l)].
   \end{split}
 \end{equation}
 By the assumption that $
    \int_M\varphi_k(x)dm(x)=\int_{M}\varphi_l(x)dm(x)=1$ and $(\ref{disintegration})$, there exists a $\hat{i}$ and a local stable leaf $\gamma^{\hat{i}}(\omega)\subset R_{\hat{i}}(\omega)$ such that
  \begin{equation}\label{over leq 1}
    \int_{\gamma^{\hat{i}}(\omega)}\varphi_k(x)H_{\hat{i}}(\omega)(x)|_{\gamma^{\hat{i}}(\omega)}dm_{\gamma^i(\omega)}(x)\leq \int_{\gamma^{\hat{i}}(\omega)}\varphi_l(x)H_{\hat{i}}(\omega)(x)|_{\gamma^{\hat{i}}(\omega)}dm_{\gamma^i(\omega)}(x).
  \end{equation}Otherwise,
   \begin{align*}
     \int_M\varphi_kdm=&\sum_{i=1}^{k(\omega)}\int\int_{\gamma^i(\omega)}\varphi_k(x)H_i(\omega)(x)|_{\gamma^i(\omega)}dm_{\gamma^i(\omega)}(x)d\tilde{m}_{R_i(\omega)}\\
     >& \sum_{i=1}^{k(\omega)}\int\int_{\gamma^i(\omega)}\varphi_l(x)H_i(\omega)(x)|_{\gamma^i(\omega)}dm_{\gamma^i(\omega)}(x)d\tilde{m}_{R_i(\omega)}=\int_M\varphi_ldm,
   \end{align*}
   a contradiction.
Now for any $i$ and local stable leaf $\gamma^i(\omega)\subset R_i(\omega)$, we have
\begin{equation}\label{over leq d+omega}
\begin{split}
   &\ \ \ \ \ \frac{\int_{\gamma^i(\omega)}\varphi_k(x)\psi(x)H_i(\omega)(x)|_{\gamma^i(\omega)}dm_{\gamma^i(\omega)}(x)}{\int_{\gamma^i(\omega)}\varphi_l(x)\psi(x)H_i(\omega)(x)|_{\gamma^i(\omega)}dm_{\gamma^i(\omega)}(x)}\\
   &\overset{\eqref{over leq     beta}}\leq \beta_{+,\omega}(\varphi_k,\varphi_l)= \frac{\beta_{+,\omega}(\varphi_k,\varphi_l)}{\alpha_{+,\omega}(\varphi_k,\varphi_l)}\cdot \alpha_{+,\omega}(\varphi_k,\varphi_l) \\
   &\overset{\eqref{over     leq beta}}\leq \frac{\beta_{+,\omega}(\varphi_k,\varphi_l)}{\alpha_{+,\omega}(\varphi_k,\varphi_l)}\cdot \frac{\int_{\gamma^{\hat{i}}(\omega)}\varphi_k(x)H_{\hat{i}}(\omega)(x)|_{\gamma^{\hat{i}}(\omega)}dm_{\gamma^i(\omega)}(x)}{\int_{\gamma^{\hat{i}}(\omega)}\varphi_l(x)H_{\hat{i}}(\omega)(x)|_{\gamma^{\hat{i}}(\omega)}dm_{\gamma^i(\omega)}(x)}\\
   &\overset{\eqref{over leq 1}}\leq \frac{\beta_{+,\omega}(\varphi_k,\varphi_l)}{\alpha_{+,\omega}(\varphi_k,\varphi_l)}\cdot 1\\
   &=\exp(d_{+,\omega}(\varphi_k,\varphi_l)),\mbox{ for all }k,l\geq 1.
\end{split}
\end{equation}
By assumption, $d_{+,\omega}(\varphi_k,\varphi_l)\to 0$ exponentially as $n,m\to\infty.$ Now pick $N^\prime>0$ such that for any $k,l>N^\prime$, $d_{+,\omega}(\varphi_k,\varphi_l)<\frac{1}{2}$, then we have
\begin{equation}\label{cauchy 1}
  \begin{split}
   & \ \ \ \ \left|\int_M\varphi_k(x)\psi(x)dm(x)-\int_M\varphi_l(x)\psi(x)dm(x)\right|\\
   &=\left|\int_M\varphi_l(x)\psi(x)dm(x)\right|\cdot \left|\frac{\int_M\varphi_k(x)\psi(x)dm(x)}{\int_M\varphi_l(x)\psi(x)dm(x)}-1\right|\\
   &\leq \sup_{x\in M}|\psi(x)|\cdot \left|\frac{\sum_{i=1}^{k(\omega)}\int_{R_i(\omega)}\varphi_k(x)\psi(x)dm(x)}{\sum_{i=1}^{k(\omega)}\int_{R_i(\omega)}\varphi_l(x)\psi(x)dm(x)}-1\right|\\
   &=\|\psi\|_{C^0(M)}\cdot\left|\frac{\sum_{i=1}^{k(\omega)}\int\int_{\gamma^i(\omega)}\varphi_k(x)\psi(x)H_i(\omega)(x)|_{\gamma^i(\omega)}dm_{\gamma^i(\omega)}(x)d\tilde{m}_{R_i(\omega)}}{\sum_{i=1}^{k(\omega)}\int\int_{\gamma^i(\omega)}\varphi_l(x)\psi(x)H_i(\omega)(x)|_{\gamma^i(\omega)}dm_{\gamma^i(\omega)}(x)d\tilde{m}_{R_i(\omega)}}-1\right|\\
   &\overset{\eqref{over leq d+omega}}\leq \|\psi\|_{C^0(M)}\cdot\left(e^{d_{+,\omega}(\varphi_k,\varphi_l)}-1\right)\\
   &\leq 2\|\psi\|_{C^0(M)}\cdot d_{+,\omega}(\varphi_k,\varphi_l),
  \end{split}
\end{equation}
where in the first $\leq $ we use the positivity of $\varphi_l$ and $\int\varphi_l dm=1$.
Hence $\{\int_M\varphi_n(x)\psi(x)dm(x)\}_{n\in\mathbb{N}}$ is a Cauchy sequence in this case.

Secondly, for any nonconstant function $\psi\in C^{0,\kappa}(M)$, let
\begin{equation*}
  B=\frac{5|\psi|_{\kappa}}{a}>0,
\end{equation*}where $|\psi|_{\kappa}:=\sup_{x,y\in M,x\not=y}\frac{|\psi(x)-\psi(y)|}{d(x,y)^\kappa}$. We define
\begin{equation*}
  \psi_B^+:=\frac{1}{2}(|\psi|+\psi)+B,\ \psi_B^-:=\frac{1}{2}(|\psi|-\psi)+B.
\end{equation*}Then for any $x,y\in M$, we have
\begin{equation*}
  |\psi_B^\pm(x)-\psi_B^\pm(y)|\leq |\psi|_\kappa d(x,y)^\kappa,
\end{equation*}and therefore, we obtain
\begin{equation*}
  \left|\frac{\psi_B^\pm(x)}{\psi_B^\pm(y)}-1\right|\leq \frac{a}{5}d(x,y)^\kappa.
\end{equation*}
Switch $x$ and $y$ to get
\begin{equation}\label{eq Holder log psi}
    |\log\psi_B^\pm|_{\kappa}=\sup_{x,y\in M,x\not=y}\frac{|\log\psi_B^\pm(x)-\log\psi_B^\pm(y)|}{d(x,y)^\kappa}<\frac{a}{4}.
  \end{equation}
 Then we apply $(\ref{cauchy 1})$ and the linearity of integration, for any $k,l>N^\prime$, one has
\begin{align*}
  &\ \ \ \ \left|\int_M\varphi_k(x)\psi(x)dm(x)-\int_M\varphi_l(x)\psi(x)dm(x)\right|\\
  &\leq \left|\int_M\varphi_k(x)\psi_B^+(x)dm(x)-\int_M\varphi_l(x)\psi_B^+(x)dm(x)\right|+\left|\int_M\varphi_k(x)\psi_B^-(x)dm(x)-\int_M\varphi_l(x)\psi_B^-(x)dm(x)\right|\\
 &\leq 2(\|\psi_B^+\|_{C^0(M)}+\|\psi_B^-\|_{C^0(M)})d_{+,\omega}(\varphi_k,\varphi_l)\\
  &\leq \left(4\|\psi\|_{C^0(M)}+\frac{20}{a}|\psi|_\kappa\right)d_{+,\omega}(\varphi_k,\varphi_l)\\
  &\leq \max\{4,\frac{20}{a}\}\|\psi\|_{C^{0,\kappa}(M)}\cdot d_{+,\omega}(\varphi_k,\varphi_l).
\end{align*}

Finally, for any nonconstant continuous function $\psi:M\rightarrow \mathbb{R}$, for any $\epsilon>0$, we can pick a nonconstant function $\tilde{\psi}\in C^{0,\kappa}(M)$ such that
\begin{equation*}
  \sup_{x\in M}|\psi(x)-\tilde{\psi}(x)|<\epsilon/4.
\end{equation*}
Now, pick $N^{\prime\prime}>N^\prime>0$ depending on $\tilde{\psi}$ and $\epsilon$ such that for all $k,l\geq N^{\prime\prime}$
\begin{equation*}
  \max\left\{4,\frac{20}{a}\right\}\cdot\|\tilde{\psi}\|_{C^{0,\kappa}(M)}\cdot d_{+,\omega}(\varphi_k,\varphi_l)<\epsilon/2.
\end{equation*}Then for any $k,l\geq N^{\prime\prime}$, one has
\begin{align*}
   &\ \ \ \ \ \left|\int_M\varphi_k(x)\psi(x)dm(x)-\int_M\varphi_l(x)\psi(x)dm(x)\right|\\
   &\leq \left|\int_M\varphi_k(x)\tilde{\psi}(x)dm(x)-\int_M\varphi_l(x)\tilde{\psi}(x)dm(x)\right|+\int_M\varphi_k|\psi-\tilde{\psi}|dm+\int_M\varphi_l|\psi-\tilde{\psi}|dm\\
   &\overset{\eqref{integral =1}}\leq  \max\left\{4,\frac{20}{a}\right\}\cdot\|\tilde{\psi}\|_{C^{0,\kappa}(M)}\cdot d_{+,\omega}(\varphi_k,\varphi_l)+\epsilon/4+\epsilon/4\\
   &\leq \epsilon,
\end{align*}where we note that $\varphi_k$ and $\varphi_l$ are positive functions.
Hence, for any continuous function $\psi:M\rightarrow \mathbb{R}$, the sequence $\{\int_M\varphi_n(x)\psi(x)dm(x)\}_{n\in\mathbb{N}}$ is a Cauchy sequence. The proof of Lemma \ref{Cauchy sequence lemma} is complete.
\end{proof}
For any measurable function $\varphi:M\rightarrow \mathbb{R}$, we define the fiber Koopman operator
\begin{equation}\label{U varphi}
 U_{\omega}\varphi: M\rightarrow\mathbb{R} ,\ (U_{\omega}\varphi)(x):=\varphi(f_{\theta^{-1}\omega}x).
\end{equation}We denote
\begin{equation*}
  U_\omega^{n}:=U_{\theta^{-(n-1)}\omega}\circ\cdots \circ U_{\theta^{-1}\omega} \circ U_{\omega}\mbox{ for all }n\in\mathbb{N} \mbox{ and }\omega\in\Omega.
\end{equation*}
For any $\omega\in\Omega$ and any bounded measurable functions $\varphi_1,\varphi_2,$  by changing variable, we have
\begin{align}
  \int_M(L_{\theta^{-1}\omega}\varphi_1)(y)\varphi_2(y)dm(y) &=\int_{M}\frac{\varphi_1((f_{\theta^{-1}\omega})^{-1}y)}{|\det D_{(f_{\theta^{-1}\omega})^{-1}(y)}f_{\theta^{-1}\omega}|}\varphi_2(y)dm(y)\nonumber\\
  &=\int_{M}\frac{\varphi_1(x)}{|\det D_xf_{\theta^{-1}\omega}|}\varphi_2(f_{\theta^{-1}\omega}x)|\det D_xf_{\theta^{-1}\omega}|dm(x)\nonumber\\
  &=\int_M\varphi_1(x)(U_{\omega}\varphi_2)(x)dm(x)\label{L to U}.
\end{align} Let $\textbf{1}$ be the constant function $\textbf{1}(x)\equiv 1$, then $\textbf{1}\in C_\omega(b,c,\nu)$ for all $\omega\in\Omega$ by the Remark \ref{nonegative function s C2 auto}.
Now consider $\varphi_n=L^n_{\theta^{-n}\omega}\textbf{1}$ for $n\geq N$ and notice that for all $\omega\in \Omega$,
\begin{align*}
  \int_M (L^n_{\theta^{-n}\omega}\textbf{1})(x)dm(x) =\int_M \textbf{1}(x)(U^n_\omega1)(x)dm(x)
  =\int_M \textbf{1} dm(x)
  =1.
\end{align*}Moreover, by $(\ref{contraction of L n})$, we have
 \begin{equation}
   d_{+,\omega}(L^n_{\theta^{-n}\omega}\textbf{1},L^{n+k}_{\theta^{-(n+k)}\omega}\textbf{1})\leq d(L^n_{\theta^{-n}\omega}\textbf{1},L^{n}_{\theta^{-n}\omega}(L^k_{\theta^{-(n+k)}\omega}\textbf{1}))\leq \Lambda^n\cdot D_4\mbox{ for }n\geq N.
 \end{equation} Hence the positive functions sequence $\{\varphi_n=L^n_{\theta^{-n}\omega}\textbf{1}\}_{n\in\mathbb{N}}\subset C_\omega(b,c,\nu)$ satisfies conditions of Lemma \ref{Cauchy sequence lemma}. So for any $g\in C^0(M)$, $\{\int_M (L_{\theta^{-n}\omega}^n\textbf{1})(x)g(x)dm\}_{n\in\mathbb{N}}$ is a Cauchy sequence. Now define $\mathcal{F}_\omega:C^0(M)\rightarrow \mathbb{R}$ by $
   \mathcal{F}_\omega(g)=\lim_{n\rightarrow\infty}\int_M (L^n_{\theta^{-n}\omega}\textbf{1})(x)g(x)dm(x).$ It is clear that $\mathcal{F}_\omega$ is a positive linear functional on $C^0(M)$. By the Riesz representation theorem, there exists a regular Borel measure $\mu_\omega$ such that
 \begin{equation}\label{definition of mu omega}
   \int_M g(x)d\mu_\omega(x)=\lim_{n\rightarrow \infty}\int_M(L^n_{\theta^{-n}\omega}\textbf{1})(x)g(x)dm(x).
 \end{equation}Moreover, $\mu_\omega$ is a probability measure since $\mu_\omega(M)=\lim_{n\to\infty}\int_M (L^n_{\theta^{-n}\omega}\textbf{1})(x)dm(x)=1$.

 Note that for each $g\in C^0(M)$, $\omega\mapsto \int_Mg(x)d\mu_\omega(x)$ is measurable because of the measurability of $\omega\mapsto \int_M(L^n_{\theta^{-n}\omega}\textbf{1})(x)g(x)dm(x)$. For any closed set $B\subset M$, let $g_k(x):=1-\min\{kd(x,B),1\}$ for $k\in \mathbb{N}$ where $d(x,B):=\inf\{d(x,y):\ y\in B\}$, then $g_k(x)\in C^0(M)$ and $g_k(x)\searrow 1_B(x)$. Then by Monotone convergence theorem, we have
 \begin{equation*}
   \mu_\omega(B)=\lim_{k\rightarrow\infty}\int_M g_k(x)d\mu_\omega(x)=\lim_{k\rightarrow\infty}\lim_{n\rightarrow\infty}\int_M(L^n_{\theta^{-n}\omega}\textbf{1})(x,\omega)g_k(x)dm(x).
 \end{equation*}Hence $\omega\mapsto \mu_\omega(B)$ is measurable for any closed set $B\subset M$. By the definition in Section \ref{subsection 2.2}, $\omega\mapsto\mu_\omega$ defines a random probability measure.

Now for any continuous $g:M\rightarrow\mathbb{R}$,
\begin{align*}
  \int_Mg(f_\omega x)d\mu_\omega &=\lim_{n\rightarrow \infty}\int_M(L_{\theta^{-n}\omega}^n\textbf{1})(x)g(f_\omega x)dm(x)\overset{\eqref{L to U}}=\lim_{n\rightarrow \infty}\int_M g(f_\omega f_{\theta^{-n}\omega}^nx)dm(x)\\
  &=\lim_{n\rightarrow \infty}\int_M g(f_{\theta^{-(n+1)}\theta\omega}^{n+1}x)dm(x)=\lim_{n\rightarrow \infty}\int_M (L_{\theta^{-(n+1)}\theta\omega}^{n+1}\textbf{1})(x)g(x)dm(x)\\
  &=\int_M g(x)d\mu_{\theta\omega}.
\end{align*}
 Thus, the random probability measure $\mu_\omega$ is $\phi-$invariant, i.e.,
 \begin{equation}\label{measure invariant}
   (f_\omega)_*\mu_\omega=\mu_{\theta\omega} \mbox{ for all }\omega\in\Omega.
 \end{equation}

 Notice that for any $g\in C^0(M)$, we have
 \begin{align*}
   \int_Mg(x)d\mu_\omega(x)&=\lim_{n\rightarrow \infty}\int_M(L^n_{\theta^{-n}\omega}\textbf{1})(x)g(x)dm(x)\overset{\eqref{L to U}}=\lim_{n\rightarrow\infty}\int_M U_\omega^ngdm\\
   &=\lim_{n\rightarrow\infty}\int_M g(f_{\theta^{-n}\omega}^nx)dm(x)
   =\lim_{n\rightarrow\infty}\int_M g(y)d(f_{\theta^{-n}\omega}^n)_*m(y).
 \end{align*}So $\mu_\omega$ is actually the weak$^*-$limit of $(f_{\theta^{-n}\omega}^n)_*m$.

\begin{remark}\label{another definitio of mu omega}
For each $\omega\in\Omega$, $k\in\mathbb{N}$ and any positive function $\varphi_{\theta^{-k}\omega}\in C_{\theta^{-k}\omega}(b,c,\nu)$ such that $\int_M\varphi_{\theta^{-k}\omega}(x)dm=1$, then we must have
  \begin{equation}
    \lim_{n\rightarrow\infty}\int_M(L^n_{\theta^{-n}\omega}\varphi_{\theta^{-n}\omega})(x)g(x)dm(x)=\int_M g(x)d\mu_\omega(x),\ \forall g\in C^0(M).
  \end{equation} In fact, we define the sequence $\hat{\varphi}_n\in C_\omega(b,c,\nu)$ by $\hat{\varphi}_{2k}=L^k_{\theta^{-k}\omega}\textbf{1}$, $\hat{\varphi}_{2k+1}=L^k_{\theta^{-k}\omega}\varphi_{\theta^{-k}\omega}$ for all $k\geq N$. By noticing
  \begin{align*}
    d_{+,\omega}(\hat{\varphi}_{2k},\hat{\varphi}_{2k+1}) &\leq d_\omega(L^k_{\theta^{-k}\omega}\textbf{1},L^k_{\theta^{-k}\omega}\varphi_{\theta^{-k}\omega})\overset{\eqref{contraction of      L n}}\leq \Lambda^{k}D_4
  \end{align*}and by $(\ref{L   to U})$,
  \begin{equation*}
    \int_M(L^k_{\theta^{-k}\omega}\varphi_{\theta^{-k}\omega})(x)dm(x)=\int_M\varphi_{\theta^{-k}\omega}(x)(U^k_\omega\textbf{1})(x)dm(x)=\int_M\varphi_{\theta^{-k}\omega}(x)dm(x)=1.
  \end{equation*} So $\{\hat{\varphi}_n\}_{n\in\mathbb{N}}\subset C_\omega(b,c,\nu)$ satisfying the condition of Lemma \ref{Cauchy sequence lemma}. Thus, the sequence $\{\int_M\hat{\varphi}_{n}(x)g(x)dm(x)\}$ is a Cauchy sequence for all $g\in C^0(M).$ As a consequence, we have
  \begin{align*}
    \int_Mg(x)d\mu_\omega(x) &=\lim_{n\rightarrow\infty}\int_{M}(L^n_{\theta^{-n}\omega}\textbf{1})(x)g(x)dm(x)=\lim_{n\rightarrow\infty}\int_{M}(L^n_{\theta^{-n}\omega}\varphi_{\theta^{-n}\omega})(x)g(x)dm(x).
  \end{align*}
\end{remark}
\subsection{Proof of the exponential decay of the (quenched) past random correlations}\label{subsection 4.4}
In this subsection, we prove the exponential decay of the past random correlations.

\begin{lemma}\label{correlation step 1}
  Let $\psi:M\rightarrow\mathbb{R}^+$ be a positive function such that $\log\psi$ is $(\frac{a}{4},\kappa)$-H\"older continuous. Then there exists a constant $K(D_4)>0$  depending on $D_4$ such that for any $\omega\in\Omega$, positive function $\varphi_{\theta^{-k}\omega}\in C_{\theta^{-k}\omega}(b,c,\nu)$ for $k\in\mathbb{N}$, and $n\geq N$, the following holds:
  \begin{equation}\label{correlation 3}
    \begin{split}
    &\ \ \ \ \ \left|\int_M\psi(f_{\theta^{-n}\omega}^nx)\varphi_{\theta^{-n}\omega}(x)dm(x)-\int_M\psi(x)d\mu_\omega(x)\int_M\varphi_{\theta^{-n}\omega}(x)dm(x)\right|\\
    &\leq K(D_4)\cdot \|\psi\|_{C^0(M)} \cdot \int_M\varphi_{\theta^{-n}\omega}(x)dm(x)\cdot\Lambda^{n}.
    \end{split}
  \end{equation}
 Recall that $D_4$ comes from Lemma \ref{lemma D4}, and $N$ is constructed satisfying \eqref{N property 1} and \eqref{N property 2}.
\end{lemma}
\begin{proof}
  First, we  consider that positive function $\varphi_{\theta^{-k}\omega}\in C_{\theta^{-k}\omega}(b,c,\nu)$ satisfying $\int_M\varphi_{\theta^{-k}\omega}(x)dm(x)=1$ for all $k\in\mathbb{N}$. So
  \begin{equation*}
    \int_M(L^n_{\theta^{-n}\omega}\varphi_{\theta^{-n}\omega})(x)dm(x)\overset{\eqref{L to U}}=\int_M\varphi_{\theta^{-n}\omega}\cdot U_\omega^n\textbf{1}dm=1,\mbox{ for all }n\in\mathbb{N}.
  \end{equation*}
  Then the proof of $(\ref{cauchy 1})$ indicates that for any $n\geq N,\ k\geq 0$, we have
  \begin{align*}
     & \ \ \ \ \left|\int_M\psi(x)(L^n_{\theta^{-n}\omega}\varphi_{\theta^{-n}\omega})(x)dm(x)-\int_M\psi(x)(L^{n+k}_{\theta^{-(n+k)}\omega}\varphi_{\theta^{-(n+k)}\omega})(x)dm(x)\right|\\
     &\leq \|\psi\|_{C^0(M)}\left(e^{d_{+,\omega}(L^n_{\theta^{-n}\omega}\varphi_{\theta^{-n}\omega},L^{n+k}_{\theta^{-(n+k)}\omega}\varphi_{\theta^{-(n+k)}\omega})}-1\right)\\
     &= \|\psi\|_{C^0(M)}\left(e^{d_{+,\omega}(L^n_{\theta^{-n}\omega}\varphi_{\theta^{-n}\omega},L^{n}_{\theta^{-n}\omega}L^{k}_{\theta^{-(n+k)}\omega}\varphi_{\theta^{-(n+k)}\omega})}-1\right)\\
      &\leq \|\psi\|_{C^0(M)}\left(e^{d_{\omega}(L^n_{\theta^{-n}\omega}\varphi_{\theta^{-n}\omega},L^{n}_{\theta^{-n}\omega}L^{k}_{\theta^{-(n+k)}\omega}\varphi_{\theta^{-(n+k)}\omega})}-1\right)\\
     &\overset{\eqref{contraction of      L n}}\leq \|\psi\|_{C^0(M)}\left(e^{D_4\Lambda^n}-1\right)\\
     &\leq K(D_4)\|\psi\|_{C^0(M)}\Lambda^{n}
  \end{align*}for some constant $K(D_4)$. Letting $k\rightarrow\infty$ in the above, by Remark \ref{another definitio of mu omega}, we have
  \begin{equation}\label{correlation 1}
    \left|\int_M\psi(x)(L^n_{\theta^{-n}\omega}\varphi_{\theta^{-n}\omega})(x)dm(x)-\int_M\psi(x)d\mu_\omega(x)\right|\leq K(D_4)\cdot\|\psi\|_{C^0(M)}\Lambda^{n}.
  \end{equation}Note that by $(\ref{L   to U})$, we have
  \begin{equation*}
    \int_M\psi(x)(L^n_{\theta^{-n}\omega}\varphi_{\theta^{-n}\omega})(x)dm(x)=\int_M\psi(f_{\theta^{-n}\omega}^nx)\varphi_{\theta^{-n}\omega}(x)dm(x).
  \end{equation*}
  Hence $(\ref{correlation 1})$ becomes
  \begin{equation}\label{correlation 2}
    \left|\int_M\psi(f_{\theta^{-n}\omega}^nx)\varphi_{\theta^{-n}\omega}(x)dm(x)-\int_M\psi(x)d\mu_\omega(x)\right|\leq K(D_4)\cdot\|\psi\|_{C^0(M)}\Lambda^{n}.
  \end{equation}
  Now for any positive function $\varphi_{\theta^{-n}\omega}\in C_{\theta^{-n}\omega}(b,c,\nu)$, let $\tilde{\varphi}_{\theta^{-n}\omega}(x):=\varphi_{\theta^{-n}\omega}(x)/\int_M\varphi_{\theta^{-n}\omega}(x)dm(x)$. We  replace $\varphi_{\theta^{-n}\omega}$ by $\tilde{\varphi}_{\theta^{-n}\omega}$ in $(\ref{correlation 2})$. Then \eqref{correlation     3} is proved.
\end{proof}
We still need the following lemma. Recall the constant $K_2$ in $(\ref{Lip constant of log det Dxfomega})$, and $C_2+C_2C_1$ in \eqref{det dfEs-detdfEs}.
\begin{lemma}\label{varphi L1 in Cbcv}
Pick the number $c$ in the definition of $C_\omega(b,c,\nu)$ sufficiently large satisfying not only $c>c_0(b,\nu)$ in Lemma \ref{invariance of C(b,c,v)}, but also
\begin{equation}\label{pick c}
  \max\left\{2a_0,\frac{K_2}{1-e^{-\lambda}},\frac{C_2+C_2C_1}{1-e^{-\lambda\nu_0}}\right\}<c.
\end{equation} Let $c_1$ be any constant such that
\begin{equation}\label{assumption on c1}
  1<c_1<\max\left\{2a_0,\frac{K_2}{1-e^{-\lambda}},\frac{C_2+C_2C_1}{1-e^{-\lambda\nu_0}}\right\}.
\end{equation}
 Given any positive continuous function $\varphi:M\rightarrow\mathbb{R}^+$ with
  \begin{equation*}
   |\log\varphi|_\nu= \sup_{x,y\in M,\ x\not=y}\frac{|\log\varphi(x)-\log\varphi(y)|}{d(x,y)^\nu}< c_1,
  \end{equation*}
  then $\varphi\cdot(L^l_{\theta^{-l}\omega}\textbf{1})\in C_\omega(b,c,\nu)$ for every $l\geq 1$ and all $\omega\in\Omega$.
\end{lemma}
\begin{proof}
We prove this lemma for each fixed $\omega\in\Omega$ and $l\geq 1$. First, $\varphi\cdot (L^l_{\theta^{-l}\omega}\textbf{1})$ is obviously bounded and measurable function.

  Let us verify \ref{C1}. For every local stable manifold $\gamma(\omega)$ having size between $\frac{A(\epsilon)}{4J^2}$ and $A(\epsilon)$, and any $\rho(\omega)\in D(a/2,\kappa,\gamma(\omega))$ with $\int_{\gamma(\omega)}\rho(x,\omega)dm_{\gamma(\omega)}(x)=1$, we have
  \begin{equation*}
    \int_{\gamma(\omega)}\varphi(x)(L^l_{\theta^{-l}\omega}\textbf{1})(x)\rho(x,\omega)dm_{\gamma(\omega)}(x)\geq \inf_{x\in M}\varphi(x)\cdot \int_{\gamma(\omega)}(L^l_{\theta^{-l}\omega}\textbf{1})(x)\rho(x,\omega)dm_{\gamma(\omega)}(x)>0
  \end{equation*}since $L^l_{\theta^{-l}\omega}\textbf{1}\in C_\omega(b,c,\nu)$ and $\varphi$ is positive and continuous.

  By Remark \ref{nonegative function s C2 auto}, $\varphi\cdot L^l_{\theta^{-l}\omega}\textbf{1}$ fulfills \ref{C2} since $\varphi\cdot L^l_{\theta^{-l}\omega}\textbf{1}$ is nonnegative. So it is left to verify \ref{C3}.

  Let $\gamma(\omega),\tilde{\gamma}(\omega)$ be any pair of local stable manifolds. Pick any $\rho(\omega)\in D(a_1,\kappa_1,\gamma(\omega))$, and let $\tilde{\rho}(\omega)\in D(a/2,\kappa,\tilde{\gamma}(\omega))$ be defined as $(\ref{definition of tilde rho})$ corresponding to $\rho(\omega)$. We divide $f_{\omega}^{-l}\gamma(\omega)$ into connected local stable manifolds of size between $\frac{A(\epsilon)}{4J}$ and $\frac{A(\epsilon)}{2J}$, named $\gamma_i(\theta^{-l}\omega)$, such that $\gamma(\omega)=\cup f_{\theta^{-l}\omega}^l\gamma_i(\theta^{-l}\omega)$. Let $\tilde{\gamma}_i(\theta^{-l}\omega)$ be the holonomy image of $\gamma_i(\theta^{-l}\omega)$ inside of $f_{\omega}^{-l}\tilde{\gamma}(\omega)$. Naturally, we have $\tilde{\gamma}(\omega)=\cup f_{\theta^{-1}\omega}\tilde{\gamma}_i(\theta^{-1}\omega)$. Note that the Jacobian of holonomy map between local stable manifolds is bounded above by $J$  and bounded below by $J^{-1}$. Therefore, $\tilde{\gamma}_i(\theta^{-l}\omega)$ have size between $\frac{A(\epsilon)}{4J^2}$ and $A(\epsilon)$. Denote $\psi_\omega:\tilde{\gamma}(\omega)\rightarrow \gamma(\omega)$ to be the holonomy map between $\tilde{\gamma}(\omega)$ and $\gamma(\omega)$, and $\psi^i_{\theta^{-l}\omega}:\tilde{\gamma}_i(\theta^{-l}\omega)\rightarrow\gamma_i(\theta^{-l}\omega)$ the holonomy map  induced by the local unstable manifolds. By using the definition of $L_\omega$ and changing of variables, we have
  \begin{align*}
     & \ \ \ \ \int_{\gamma(\omega)}(L^l_{\theta^{-l}\omega}\textbf{1})(x)\varphi(x)\rho(x,\omega)dm_{\gamma(\omega)}(x)\\
     &=\sum_{i}\int_{\gamma_i(\theta^{-l}\omega)}\frac{\left|\det D_xf_{\theta^{-l}\omega}^l|_{E^s(x,\theta^{-l}\omega)}\right|}{\left|\det D_xf_{\theta^{-l}\omega}^l\right|}\rho(f_{\theta^{-l}\omega}^lx,\omega)\varphi(f_{\theta^{-l}\omega}^lx)dm_{\gamma_i(\theta^{-l}\omega)}(x)\\
     &=\sum_{i}\int_{\tilde{\gamma}_i(\theta^{-l}\omega)}\frac{\left|\det D_{\psi^i_{\theta^{-l}\omega}(x)}f_{\theta^{-l}\omega}^l|_{E^s(\psi^i_{\theta^{-l}\omega}(x),\theta^{-l}\omega)}\right|}{\left|\det D_{\psi^i_{\theta^{-l}\omega}(x)}f_{\theta^{-l}\omega}^l\right|}\cdot\rho(f_{\theta^{-l}\omega}^l\psi^i_{\theta^{-l}\omega}(x),\omega)\cdot\varphi(f_{\theta^{-l}\omega}^l\psi^i_{\theta^{-l}\omega}(x))\\
     &\ \ \ \ \ \ \ \ \ \ \ \ \ \ \ \cdot Jac( \psi^i_{\theta^{-l}\omega})(x)dm_{\tilde{\gamma}_i(\theta^{-l}\omega)}(x).
  \end{align*}On the other hand, we have
  \begin{align*}
     & \ \ \ \ \ \int_{\tilde{\gamma}(\omega)}(L^l_{\theta^{-l}\omega}\textbf{1})(x)\varphi(x)\tilde{\rho}(x,\omega)dm_{\tilde{\gamma}(\omega)}(x)\\
     &=\sum_{i}\int_{\tilde{\gamma}_i(\theta^{-l}\omega)}\frac{\left|\det D_xf_{\theta^{-l}\omega}^l|_{E^s(x,\theta^{-l}\omega)}\right|}{\left|\det D_xf_{\theta^{-l}\omega}^l\right|}\tilde{\rho}(f_{\theta^{-l}\omega}^lx,\omega)\varphi(f_{\theta^{-l}\omega}^lx)dm_{\tilde{\gamma}_i(\theta^{-l}\omega)}(x)\\
     &=\sum_{i}\int_{\tilde{\gamma}_i(\theta^{-l}\omega)}\frac{\left|\det D_xf_{\theta^{-l}\omega}^l|_{E^s(x,\theta^{-l}\omega)}\right|}{\left|\det D_xf_{\theta^{-l}\omega}^l\right|}\rho(\psi_\omega(f_{\theta^{-l}\omega}^lx),\omega)\cdot Jac(\psi_\omega)(f_{\theta^{-l}\omega}^lx)\cdot \varphi(f_{\theta^{-l}\omega}^lx)dm_{\tilde{\gamma}_i(\theta^{-l}\omega)}(x).
  \end{align*}Note that $f_{\theta^{-l}\omega}^l\psi_{\theta^{-l}\omega}^i(x)=\psi_\omega(f_{\theta^{-l}\omega}^l(x))$ for $x\in\tilde{\gamma}_i(\theta^{-l}\omega)$ by the invariance of stable and unstable manifolds, so we have
  \begin{equation}\label{inequality 1}
    \rho(f_{\theta^{-l}\omega}^l\psi_{\theta^{-l}\omega}^i(x),\omega)=\rho(\psi_\omega(f_{\theta^{-l}\omega}^l\omega),\omega).
  \end{equation}
Since $\log\varphi$ is $(c_1,\nu)-$H\"older continuous, for $x\in\tilde{\gamma}_i(\theta^{-l}\omega)$
\begin{align}
  \left|\log\varphi(f_{\theta^{-l}\omega}^l\psi_{\theta^{-l}\omega}^i(x))-\log\varphi(f_{\theta^{-l}\omega}^lx)\right|&\leq c_1d(f_{\theta^{-l}\omega}^l\psi_{\theta^{-l}\omega}^i(x),f_{\theta^{-l}\omega}^l(x))^\nu\nonumber\\
  &\leq c_1d_u(\gamma(\omega),\tilde{\gamma}(\omega))^\nu\label{inequality2}.
\end{align}
By Lemma \ref{property of fiberwise holonomy map} (2) and $a_0^\prime\leq a_0$, for $x\in\tilde{\gamma}_i(\theta^{-l}\omega)$, one has
\begin{align}
 &\ \ \ \ \ \  \left|\log Jac(\psi_\omega)(f_{\theta^{-l}\omega}^lx)-\log Jac (\psi_{\theta^{-l}\omega}^i)(x)|\right| \nonumber\\
  &\leq a_0d(f_{\theta^{-l}\omega}^lx,\psi_\omega f_{\theta^{-l}\omega}^lx)^{\nu_0}+a_0d(x,\psi_{\theta^{-l}\omega}^ix)^{\nu_0}\nonumber\\
  &\leq a_0d_u(\gamma(\omega),\tilde{\gamma}(\omega))^{\nu_0}+a_0d_u(\gamma_i(\theta^{-l}\omega),\tilde{\gamma}_i(\theta^{-l}\omega))^{\nu_0}\nonumber\\
  &\leq a_0d_u(\gamma(\omega),\tilde{\gamma}(\omega))^{\nu_0}+a_0e^{-\lambda l\nu_0}d_u(\gamma(\omega),\tilde{\gamma}(\omega))^{\nu_0}\nonumber\\
  &< 2a_0d_u(\gamma(\omega),\tilde{\gamma}(\omega))^{\nu_0}\label{inequality3}.
\end{align}By $(\ref{Lip constant of log det Dxfomega})$, for $x\in \tilde{\gamma}_i(\theta^{-l}\omega)$, we deduce that
\begin{align}
   & \ \ \ \ \ \left|\log|\det D_{\psi_{\theta^{-l}\omega}^i(x)}f_{\theta^{-l}\omega}^l|-\log|\det D_xf_{\theta^{-l}\omega}^l|\right|\nonumber\\
   &\leq K_2d(x,\psi_{\theta^{-l}\omega}^i(x))+K_2d(f_{\theta^{-l}\omega}x,f_{\theta^{-l}\omega}\psi_{\theta^{-l}\omega}^i(x))+\cdots+ K_2d(f_{\theta^{-l}\omega}^{l-1}x,f_{\theta^{-l}\omega}^{l-1}\psi^i_{\theta^{-l}\omega}(x))\nonumber\\
   &\leq K_2e^{-l\lambda}d_u(\gamma(\omega),\tilde{\gamma}(\omega))+K_2e^{-(l-1)\lambda}d_u(\gamma(\omega),\tilde{\gamma}(\omega))+\cdots +K_2e^{-\lambda}d_u(\gamma(\omega),\tilde{\gamma}(\omega))\nonumber\\
   &\leq \frac{K_2}{1-e^{-\lambda}}\cdot d_u(\gamma(\omega),\tilde{\gamma}(\omega))\label{inequality 4}.
\end{align}By applying $(\ref{det dfEs-detdfEs})$, for $x\in\tilde{\gamma}_i(\theta^{-l}\omega)$,
\begin{align}
   &\ \ \ \ \ \left|\log|\det D_{\psi^i_{\theta^{-l}\omega}(x)}f_{\theta^{-l}\omega}^l|_{E^s(\psi^i_{\theta^{-l}\omega}(x),\theta^{-l}\omega)}| -\log |\det D_xf_{\theta^{-l}\omega}^l|_{E^s(x,\theta^{-l}\omega)}| \right|\nonumber\\
   &\leq (C_2+C_2C_1)\left[d(\psi^i_{\theta^{-l}\omega}(x),x)^{\nu_0}+\cdots +d(f_{\theta^{-l}\omega}^{l-1}\psi^i_{\theta^{-l}\omega}(x),f_{\theta^{-l}\omega}^{l-1}x)^{\nu_0}\right]\nonumber\\
   &\leq \frac{C_2+C_2C_1}{ 1-e^{-\lambda\nu_0}}d_u(\gamma(\omega),\tilde{\gamma}(\omega))^{\nu_0}\label{inequality5}.
\end{align}Combining $(\ref{inequality 1})$, $(\ref{inequality2})$, $(\ref{inequality3})$, $(\ref{inequality 4})$ and $(\ref{inequality5})$, we conclude
\begin{align*}
   &\ \ \ \ \ \left|\log \int_{\gamma(\omega)}(L^l_{\theta^{-l}\omega}\textbf{1})(x)\varphi(x)\rho(x,\omega)dm_{\gamma(\omega)}(x)-\log \int_{\tilde{\gamma}(\omega)}(L^l_{\theta^{-l}\omega}\textbf{1})(x)\varphi(x)\tilde{\rho}(x,\omega)dm_{\tilde{\gamma}(\omega)}(x)\right|\\
   &\leq \max\left\{c_1,2a_0,\frac{K_2}{1-e^{-\lambda}},\frac{C_2+C_2C_1}{1-e^{-\lambda\nu_0}}\right\}d_u(\gamma(\omega),\tilde{\gamma}(\omega))^\nu\\
  &\overset{\eqref{pick c}\eqref{assumption on c1}}\leq cd_u(\gamma(\omega),\tilde{\gamma}(\omega))^{\nu}.
\end{align*}Hence \ref{C3} is verified. Therefore,  $\varphi\cdot(L^l_{\theta^{-l}\omega}\textbf{1})\in C_\omega(b,c,\nu)$. The proof of Lemma \ref{varphi L1 in Cbcv} is complete.
\end{proof}
Now for any positive continuous function $\psi\in C^0(M)$ such that $\log\psi$ is $(a/4,\kappa)-$H\"older continuous, and for any positive continuous function $\varphi:M\rightarrow\mathbb{R}$ satisfying that $\log\varphi$ is $(c_1,\nu)-$H\"older continuous, by Lemma \ref{varphi L1 in Cbcv}, for each $n\in\mathbb{N}$ we have
\begin{equation*}
  \varphi\cdot L_{\theta^{-(l+n)}\omega}^{l}\textbf{1}=\varphi\cdot L_{\theta^{-l}\theta^{-n}\omega}^{l}\textbf{1}\in C_{\theta^{-n}\omega}(b,c,\nu)\mbox{ for all }l\in\mathbb{N}, \ \omega\in\Omega.
\end{equation*}
Now, we apply Lemma \ref{correlation step 1} to obtain that for all $\omega\in\Omega$, $n\geq N$, we have
\begin{align}
   & \ \ \ \ \ \left|\int_M\psi(f_{\theta^{-n}\omega}^nx)(\varphi\cdot L^l_{\theta^{-l}\theta^{-n}\omega}\textbf{1})(x)dm(x)-\int_M\psi(x)d\mu_\omega(x)\int_M(\varphi\cdot L^l_{\theta^{-l}\theta^{-n}\omega}\textbf{1})(x)dm(x)\right|\nonumber\\
   &\leq K(D_4)\|\psi\|_{C^0(M)}\int_M\varphi\cdot(L^l_{\theta^{-l}\theta^{-n}\omega}\textbf{1})dm\cdot \Lambda^{n}\label{correlation2},\mbox{ for all }l\in\mathbb{N}.
\end{align}Let $l\rightarrow\infty$, by $(\ref{definition of mu omega})$, for all $\omega\in\Omega$, $n\geq N$, we have
\begin{align}
   & \ \ \ \ \ \left|\int_{M}\psi(f_{\theta^{-n}\omega}^n x)\varphi(x)d\mu_{\theta^{-n}\omega}(x)-\int_M\psi(x)d\mu_\omega(x)\int_M\varphi(x)d\mu_{\theta^{-n}\omega}(x)\right|\nonumber\\
   &\leq K(D_4)\|\psi\|_{C^0(M)}\int_M\varphi d\mu_{\theta^{-n}\omega}\cdot\Lambda^{n}\nonumber\\
   &\leq  K(D_4)\|\psi\|_{C^0(M)}\cdot\|\varphi\|_{C^0(M)}\cdot\Lambda^{n}.\label{last2}
\end{align}

Finally, given any $\psi\in C^{0,\kappa}(M)$ and $\varphi\in C^{0,\nu}(M)$. If $\psi$ or $\varphi$ is a constant function, then by \eqref{measure invariant},
\begin{equation*}
  \left|\int_{M}\psi(f_{\theta^{-n}\omega}^n x)\varphi(x)d\mu_{\theta^{-n}\omega}(x)-\int_M\psi(x)d\mu_\omega(x)\int_M\varphi(x)d\mu_{\theta^{-n}\omega}(x)\right|=0.
\end{equation*}Therefore, to prove the exponential decay of past random correlations, it is sufficient to consider the case that both $\psi\in C^{0,\kappa}(M)$ and $\varphi\in C^{0,\nu}(M)$ are nonconstant functions.
Let
\begin{equation*}
  B_{\psi}=\frac{5|\psi|_{\kappa}}{a}>0,\ B_{\varphi}=\frac{2|\varphi|_{\nu}}{c_1}>0,
\end{equation*}and define
\begin{align*}
  \psi_{B_\psi}^+=\frac{1}{2}(|\psi|+\psi)+B_{\psi},\ &\psi_{B_\psi}^-=\frac{1}{2}(|\psi|-\psi)+B_{\psi},\\
  \varphi_{B_\varphi}^+=\frac{1}{2}(|\varphi|+\varphi)+B_{\varphi},\ & \varphi_{B_\varphi}^-=\frac{1}{2}(|\varphi|-\varphi)+B_{\varphi}.
\end{align*}Similar as \eqref{eq Holder log psi}, we can show that $\log\psi_{B_\psi}^\pm$ are $(a/4,\kappa)-$H\"older continuous, and $\log\varphi_{B_\varphi}^\pm$ are $(c_1,\nu)$-H\"older continuous. By $(\ref{last2})$ and the linearity of integration, we conclude for all $\omega\in\Omega$ and $n\geq N$,
\begin{align*}
   &\left|\int_{M}\psi(f_{\theta^{-n}\omega}^n x)\varphi(x)d\mu_{\theta^{-n}\omega}(x)-\int_M\psi(x)d\mu_\omega(x)\int_M\varphi(x)d\mu_{\theta^{-n}\omega}(x)\right|\nonumber\\
 = & \left|\int_{M}\left(\psi_{B_\psi}^+-\psi_{B_\psi}^-\right)(f_{\theta^{-n}\omega}^n x)\left(\varphi_{B_\varphi}^+-\varphi_{B_\varphi}^-\right)(x)d\mu_{\theta^{-n}\omega}(x)\right.\\
 &\quad\quad\left.-\int_M\left(\psi_{B_\psi}^+-\psi_{B_\psi}^-\right)(x)d\mu_\omega(x)\int_M\left(\varphi_{B_\varphi}^+-\varphi_{B_\varphi}^-\right)(x)d\mu_{\theta^{-n}\omega}(x)\right|\\
\leq  & 4K(D_4)\cdot\max\left\{1,\frac{5}{a}\right\}\cdot\max\left\{1,\frac{2}{c_1}\right\}\cdot\|\psi\|_{C^{0,\kappa}(M)}\cdot\|\varphi\|_{C^{0,\nu}(M)}\cdot\Lambda^{n}.
\end{align*}Note that the above is true for all $n\geq N$. Next, for $n\in\{0,...,N-1\}$, we have
\begin{align*}
   &\left|\int_{M}\psi(f_{\theta^{-n}\omega}^n x)\varphi(x)d\mu_{\theta^{-n}\omega}(x)-\int_M\psi(x)d\mu_\omega(x)\int_M\varphi(x)d\mu_{\theta^{-n}\omega}(x)\right|\\
  \leq& 2\|\psi\|_{C^0(M)}\|\varphi\|_{C^0}\leq \frac{2}{\Lambda^N}\|\psi\|_{C^0(M)}\|\varphi\|_{C^0(M)}\Lambda ^n.
\end{align*}
Therefore, we let
\begin{equation}\label{def of k}
  K:=\max\left\{ 4K(D_4)\cdot\max\left\{1,\frac{5}{a}\right\}\cdot\max\left\{1,\frac{2}{c_1}\right\},2\Lambda^{-N}\right\},
\end{equation}then
\begin{align*}
 \left|\int_{M}\psi(f_{\theta^{-n}\omega}^n x)\varphi(x)d\mu_{\theta^{-n}\omega}(x)-\int_M\psi(x)d\mu_\omega(x)\int_M\varphi(x)d\mu_{\theta^{-n}\omega}(x)\right|
  \leq K\|\psi\|_{C^{0,\kappa}(M)}\cdot\|\varphi\|_{C^{0,\nu}(M)}\cdot\Lambda^{n}.
\end{align*}for all $n\geq 0$. This finishes the proof for the past random correlations.
\subsection{Proof of the exponential decay of the (quenched) future random correlation}\label{subsection 4.5}In this subsection, we prove the exponential decay of the future random correlations.

\begin{lemma}\label{future correlation step 1}
  Let $\psi:M\rightarrow\mathbb{R}^+$ be any positive continuous function satisfying that $\log\psi$ is $(\frac{a}{4},\kappa)-$H\"older continuous. Then for each $\omega\in\Omega$, positive function $\varphi_{\omega}\in C_{\omega}(b,c,\nu)$, for any $n\geq N$, the following holds:
\begin{equation}\label{future 1}
  \left|\int_M\psi(f_\omega^n x)\varphi_\omega(x)dm-\int_M\psi(x)d\mu_{\theta^n\omega}\int_M\varphi_\omega(x)dm\right|\leq K(D_4)\cdot\|\psi\|_{C^0(M)}\int_M\varphi_\omega(x)dm\cdot \Lambda^n.
\end{equation}Recall that $K(D_4)$ is defined in Lemma \ref{correlation step 1}.
\end{lemma}
\begin{proof}
  We first prove the case that positive function $\varphi_\omega\in C_\omega(b,c,\nu)$ satisfies $\int_M\varphi_\omega(x)dm(x)=1$. Note that by Lemma \ref{invariance of C(b,c,v)}, for any $n\geq N,\ k\geq 0$, $L_\omega^n\varphi_\omega,\ L_{\theta^{-k}\omega}^{n+k}1\in C_{\theta^n\omega}(b,c,\nu)$ are positive functions. Similar proof as $(\ref{cauchy 1})$ can be applied on the fiber $\{\theta^n\omega\}$ to show that for $n\geq N$,
  \begin{align}
     & \ \ \ \ \left|\int_M\psi(x)(L^n_{\omega}\varphi_{\omega})(x)dm(x)-\int_M\psi(x)(L^{n+k}_{\theta^{-k}\omega}\textbf{1})(x)dm(x)\right|\nonumber\\
     &\leq \|\psi\|_{C^0(M)}\left(e^{d_{+,\theta^n\omega}(L^n_{\omega}\varphi_\omega,L^{n+k}_{\theta^{-k}\omega}\textbf{1})}-1\right)\nonumber\\
     &\leq \|\psi\|_{C^0(M)}\left(e^{d_{+,\theta^n\omega}(L^n_{\omega}\varphi_\omega,L^{n}_{\omega}L^{k}_{\theta^{-k}\omega}\textbf{1})}-1\right)\nonumber\\
     &\overset{\eqref{contraction of Ln from omega}}\leq   \|\psi\|_{C^0(M)}\left(e^{D_4\Lambda^n}-1\right)\nonumber\\
     &\leq K(D_4)\cdot\|\psi\|_{C^0(M)}\Lambda^{n}.\label{lemma 4.9 1}
  \end{align} Notice that $L_{\theta^{-k}\omega}^{n+k}\textbf{1}=L_{\theta^{-k}\theta^n\omega}^kL_{\theta^{-k}\omega}^n\textbf{1}$, and  $L_{\theta^{-k}\omega}^n\textbf{1}\in C_{\theta^{n-k}\omega}(b,c,\nu)=C_{\theta^{-k}\theta^n\omega}(b,c,\nu)$. Moreover, $\int_ML_{\theta^{-k}\omega}^n\textbf{1}dm=1$ for all $k\in\mathbb{N}$.
  By Remark \ref{another definitio of mu omega}, we have
  \begin{equation*}
  \lim_{k\to \infty}\int_M\psi(x)(L_{\theta^{-k}\omega}^{n+k}\textbf{1})(x)dm(x)=\lim_{k\to \infty}\int_M\psi\cdot L_{\theta^{-k}\theta^n\omega}^{k}(L_ {\theta^{-k}\omega}^{n}\textbf{1})dm=\int_M\psi(x)d\mu_{\theta^n\omega}(x).
  \end{equation*}Therefore, let $k\to \infty$ in \eqref{lemma      4.9 1}, then for $n\geq N$,
  \begin{equation}\label{future correlation 1}
    \left|\int_M\psi(x)(L^n_{\omega}\varphi_{\omega})(x)dm(x)-\int_M\psi(x)d\mu_{\theta^n\omega}(x)\right|\leq K(D_4)\cdot \|\psi\|_{C^0(M)}\Lambda^{n}.
  \end{equation}Note that by $(\ref{L   to U})$, we have
  \begin{equation*}
    \int_M\psi(x)(L^n_{\omega}\varphi_{\omega})(x)dm(x)=\int_M\psi(f_{\omega}^nx)\varphi_{\omega}(x)dm(x).
  \end{equation*}
  Hence \eqref{future correlation 1} becomes that for any $n\geq N$,
  \begin{equation}\label{future 2}
    \left|\int_M\psi(f_{\omega}^nx)\varphi_{\omega}(x)dm(x)-\int_M\psi(x)d\mu_{\theta^n\omega}(x)\right|\leq K(D_4)\cdot \|\psi\|_{C^0(M)}\Lambda^{n}.
  \end{equation}
  Now for any positive function $\varphi_\omega\in C_\omega(b,c,\nu)$, let $\tilde{\varphi}_\omega(x):=\varphi_\omega(x)/\int_M\varphi_\omega dm$. Then  $(\ref{future 1})$  can be proved by replacing $\varphi_\omega$ by $\tilde{\varphi}_\omega$ in $(\ref{future 2})$.
\end{proof}

Now assume the function $\psi:M\rightarrow\mathbb{R}$ such that $\psi>0$, and $\log\psi$ is $(a/4,\kappa)-$H\"older continuous. Let $\varphi:M\rightarrow\mathbb{R}^+$ with $\log\varphi$ is $(c_1,\nu)-$H\"older continuous. Then by the Lemma \ref{varphi L1 in Cbcv}, we have
\begin{equation*}
  \varphi\cdot L_{\theta^{-l}\omega}^{l} \textbf{1}\in C_{\omega}(b,c,\nu)\mbox{ for all }l\in\mathbb{N} \mbox{ all }\omega\in\Omega.
\end{equation*}
Now, we apply Lemma \ref{future correlation step 1} to obtain that for all $\omega\in\Omega$, $n\geq N$,
\begin{align*}
   & \ \ \ \ \ \left|\int_M\psi(f_{\omega}^nx)(\varphi\cdot L^l_{\theta^{-l}\omega}\textbf{1})(x)dm(x)-\int_M\psi(x)d\mu_{\theta^n\omega}(x)\int_M(\varphi\cdot L^l_{\theta^{-l}\omega}\textbf{1})(x)dm(x)\right|\nonumber\\
   &\leq K(D_4)\|\psi\|_{C^0(M)}\int_M\varphi\cdot(L^l_{\theta^{-l}\omega}\textbf{1})dm\cdot \Lambda^{n},\mbox{ for all }l\in\mathbb{N}.
\end{align*}Let $l\rightarrow\infty$, by $(\ref{definition of mu omega})$, for all $\omega\in\Omega$, $n\geq N$,
\begin{equation}\label{future last2}
  \begin{split}
 &\left|\int_{M}\psi(f_{\omega}^n x)\varphi(x)d\mu_{\omega}(x)-\int_M\psi(x)d\mu_{\theta^n\omega}(x)\int_M\varphi(x)d\mu_{\omega}(x)\right|\\
\leq & K(D_4)\|\psi\|_{C^0(M)}\int_M\varphi d\mu_{\omega}\cdot\Lambda^{n}\leq  K(D_4)\|\psi\|_{C^0(M)}\cdot\|\varphi\|_{C^0(M)}\cdot\Lambda^{n}.
  \end{split}
\end{equation}

Finally, given $\psi\in C^{0,\kappa}(M)$ and $\varphi\in C^{0,\nu}(M)$. If $\psi$ or $\varphi$ is a constant function, then by \eqref{measure invariant},
\begin{equation*}
 \left|\int_{M}\psi(f_{\omega}^n x)\varphi(x)d\mu_{\omega}(x)-\int_M\psi(x)d\mu_{\theta^n\omega}(x)\int_M\varphi(x)d\mu_{\omega}(x)\right|=0.
\end{equation*}Therefore, to prove the exponential decay of past random correlations, it is sufficient to consider the case that both $\psi\in C^{0,\kappa}(M)$ and $\varphi\in C^{0,\nu}(M)$ are nonconstant functions. let
\begin{equation*}
  B_{\psi}=\frac{5|\psi|_{\kappa}}{a}>0,\ B_{\varphi}=\frac{2|\varphi|_{\nu}}{c_1}>0,
\end{equation*}and define
\begin{align*}
  \psi_{B_\psi}^+=\frac{1}{2}(|\psi|+\psi)+B_{\psi},\ &\psi_{B_\psi}^-=\frac{1}{2}(|\psi|-\psi)+B_{\psi},\\
  \varphi_{B_\varphi}^+=\frac{1}{2}(|\varphi|+\varphi)+B_{\varphi},\ & \varphi_{B_\varphi}^-=\frac{1}{2}(|\varphi|-\varphi)+B_{\varphi}.
\end{align*}As before, $\log\psi_{B_\psi}^\pm$ are $(a/4,\kappa)-$H\"older continuous, and $\log\varphi_{B_\varphi}^\pm$ are $(c_1,\nu)$ H\"older continuous. By $(\ref{future    last2})$ and the linearity of integration, we conclude for $n\geq N$,
\begin{align*}
 &\left|\int_{M}\psi(f_{\omega}^n x)\varphi(x)d\mu_{\omega}(x)-\int_M\psi(x)d\mu_{\theta^n\omega}(x)\int_M\varphi(x)d\mu_{\omega}(x)\right|\\
  \leq & 4K(D_4)\cdot\max\left\{1,\frac{5}{a}\right\}\cdot\max\left\{1,\frac{2}{c_1}\right\}\cdot\|\psi\|_{C^{0,\kappa}(M)}\cdot\|\varphi\|_{C^{0,\nu}(M)}\cdot\Lambda^{n}.
\end{align*} Recall that $K$ is defined in $(\ref{def of k})$, then we arrive
\begin{equation*}
  \left|\int_{M}\psi(f_{\omega}^n x)\varphi(x)d\mu_{\omega}(x)-\int_M\psi(x)d\mu_{\theta^n\omega}(x)\int_M\varphi(x)d\mu_{\omega}(x)\right|
  \leq  K\cdot \|\psi\|_{C^{0,\kappa}(M)}\cdot\|\varphi\|_{C^{0,\nu}(M)}\cdot\Lambda^{n},
\end{equation*}for all $n\geq 0$.
This finishes the proof for the future random correlations.

\appendix
\section{}

\subsection{Examples of Anosov and mixing on Fibers System}\label{Appendix Example}
Theorem 8.1 in \cite{HLL} showed that the random dynamical systems satisfying the following conditions are topological mixing on fibers:
\begin{enumerate}
  \item[(A1)]$(\Omega,\theta)$ is a minimal irrational rotation on the compact torus;
  \item[(A2)]$\phi$ is Anosov on fibers;
  \item[(A3)]$\phi$ is topological transitive on $M\times\Omega$.
\end{enumerate}
\begin{remark}
This class of systems can't be topological mixing.
In fact, it is well-known that a factor of a topological mixing system is also topological mixing.
Notice that the irrational rotation is not topological mixing, hence $\phi$ can't be topological mixing.
\end{remark}
A typical example satisfying (A1)-(A3) is given by following.  Fiber Anosov maps on 2-d tori:  Let $\theta$ be a homeomorphism on a compact metric space $\Omega$ and let $A= \begin{pmatrix}
                                                                                                    2 &1 \\
                                                                                                    1 &1
                                                                                                  \end{pmatrix}.$ Define $\phi:\mathbb{T}^2\times \Omega\rightarrow \mathbb{T}^2\times \Omega$ by
 \begin{equation*}
   \phi\left( \begin{pmatrix}
               x \\
               y
             \end{pmatrix},\omega\right)=\left( A\begin{pmatrix}
                                                  x \\
                                                  y
                                                \end{pmatrix}+h(\omega),\theta\omega\right)=\left( \begin{pmatrix}
                                                                                                    2 &1 \\
                                                                                                    1 &1
                                                                                                  \end{pmatrix}\begin{pmatrix}
                                                                                                                 x \\
                                                                                                                 y
                                                                                                               \end{pmatrix}+\begin{pmatrix}
                                                                                                                               h_1(\omega) \\
                                                                                                                               h_2(\omega)
                                                                                                                             \end{pmatrix},\theta\omega\right),
 \end{equation*}where $h$ is a continuous map from $\Omega$ to $\mathbb{T}^2$.

Theorem 8.2 in \cite{HLL} showed that the systems satisfying the following conditions are also mixing on fibers:
\begin{enumerate}
  \item[(B1)]$(\Omega,\theta)$ is a homeomorphism on a compact metric space;
  \item[(B2)]$\phi$ is Anosov on fibers;
  \item[(B3)]There exists an $f_\omega-$invariant Borel probability measure $\nu$ with full support (i.e. supp$\nu=M$).
\end{enumerate}
A typical example satisfying (B1)-(B3) is given by following. Random composition of $2\times 2$ area-preserving positive matrices: let
  \begin{equation*}
    \left\{B_i=\begin{pmatrix}
            a_i & b_i  \\
            c_i & d_i  \\
          \end{pmatrix}\right\}_{1\leq i\leq p}
  \end{equation*}be $2\times2$ matrices with $a_i,b_i,c_i,d_i\in\mathbb{Z}^+$, and $|a_id_i-c_ib_i|=1$ for any $i\in\{1,...,p\}$.  Let $\Omega=\mathcal{S}_p:=\{1,...,p\}^\mathbb{Z}$ with the left shift operator $\theta$ be the symbolic dynamical system with $p$ symbols.
  For any $\omega=(...,\omega_{-1},\omega_0,\omega_1,...)\in \Omega$, we define $g(\omega)=B_{\omega_0}.$ Then the skew product $\tilde{\phi}:\Omega\times\mathbb{T}^2\rightarrow \Omega\times\mathbb{T}^2$ defined by
  \begin{equation*}
    \tilde{\phi}( x,\omega)=\big( g(\omega)x,\theta\omega\big)
  \end{equation*}
  is an Anosov on fibers system with continuous co-invariant splitting $\mathbb{R}^2=E^u_{\omega}\oplus E^s_{\omega} $ for $\omega\in \mathcal{S}_p$, on which
  \begin{align*}
    \|Dg_\omega v\|&\geq \kappa\|v\| \mbox{ for }v\in E^{u}_{\omega};\\
    \|Dg_\omega \eta\|&\leq \kappa^{-1}\|\eta\| \mbox{ for }\eta\in E^{s}_{\omega}
  \end{align*}for $\kappa:=\underset{1\leq i\leq p}{\min}\min\{\sqrt{a_i^2+c_i^2},\sqrt{b_i^2+d_i^2}\}\geq \sqrt{2}$ by Proposition 8.2 in \cite{HLL}.
\subsection{Random SRB measure for Anosov and Mixing on Fibers Systems}
Let $F:\mathbb{Z}\times\Omega\times M\rightarrow M$ be a continuous random dynamical system over an invertible ergodic metric dynamical systems $(\Omega,\mathcal{B},P,\theta)$.
\begin{definition}
  A map $K:\Omega\rightarrow 2^M$ is called a closed random set if $K(\omega)$ is closed for any $\omega\in\Omega$; and $\omega\mapsto d(x,K(\omega))$ is measurable for each fixed $x\in M.$ $U:\Omega\rightarrow 2^M$ is called an open random set if $U^c$ is a random closed set.
\end{definition}
\begin{definition}
  F is called random topological transitive if for any given open random sets $U$ and $V$ with $U(\omega),V(\omega)\not=\emptyset$ for all $\omega\in\Omega$, there exists a random variable $n$ taking values in $\mathbb{Z}$ such that the intersection $F(n(\omega),\theta^{-n(\omega)})U(\theta^{-n(\omega)}\omega)\cap V(\omega)\not=\emptyset$ $\P$-a.s..
\end{definition}
The following Lemma is the Lemma A.1 in \cite{HLL}.
\begin{lemma}
  If $F$ is topological mixing on fibers, then $F$ is random topological transitive.
\end{lemma}
The following theorem is Theorem 4.3 in \cite{Gund99}, which is the main result of the SRB measure for random hyperbolic systems.
\begin{theorem}
  Let $F$ be a $C^{1+\alpha}$ RDS with a random topological transitive hyperbolic attractor $\Lambda\subset M\times\Omega$. Then there exists a unique $F-$invariant measure (SRB-measure) $\nu$ supported by $\Lambda$ and characterized by each of the following:
  \begin{enumerate}
    \item[(i)]$h_\nu(F)=\int\sum\lambda_i^+d\nu$, where $\lambda_i$ are the Lyapunov exponents corresponding to $\nu$;
    \item[(ii)]$\P$-a.s. the conditional measure of $\nu_\omega$ on the unstable manifolds are absolutely continuous with respect to the Riemannian volume on these submanifolds;
    \item[(iii)]$h_\nu(F)+\int fd\nu=\sup\{h_\mu(F)+\int fd\mu: F-\mbox{invariant}\  \mu\}$ and the later is the topological pressure $\pi_F(f)$ of $f$ which satisfies $\pi_F(f)=0$;
    \item[(iv)] $\nu=\psi\tilde{\mu}$, where $\psi$ is the conjugation between $F$ on $\Lambda$ and two-sided shift $\sigma$ on $\Sigma_A$, and $\tilde{\mu}$ is the equilibrium state for the $\sigma$ and function $f\circ\psi$. The measure $\tilde{\mu}$ can be obtained as a natural extension of the probability measure $\mu$ which is invariant with respect to the one-sided shift on $\Sigma_A^+$ and such that $L_\eta^*\mu_{\theta\omega}=\mu_\omega$ $\P$-a.s. where $\eta-f\circ\psi=h-h\circ(\theta\times\sigma)$ for some random H\"older continuous function $h$;
    \item[(v)]$\nu$ can be obtained as a weak$^*$ limit $\nu_\omega=\lim_{n\rightarrow }F(n,\theta^{-n}\omega)m_{\theta^{-n}\omega}$ $\P$-a.s. for any measure $m_\omega$ absolutely continuous with respect to the Riemannian volume such that supp$m_\omega\subset U(\omega)$.
  \end{enumerate}
\end{theorem}

\subsection{Anosov on Fibers System and Partially Hyperbolic systems}
In this section, we show that Anosov on fibers system contains a class of partially hyperbolic systems.
\begin{definition}
  $(M,f)$ is called partially hyperbolic in the narrow sense if the tangent bundle admits a splitting into three continuous vector subbundles $T_xM=E^1(x)\oplus E^2(x)\oplus E^3(x)$ which satisfy
  \begin{enumerate}
    \item dominated splitting, i.e., $D_xf(E^i(x))=E^i(f(x))$ for $i=1,2,3$, and there exists a constant $c>0$ and $\lambda\in(0,1)$ such that $\|Df^n|_{E^i(x)}\|\leq c\lambda^n\|Df^n|_{E^{i+1}(x)}\|$ for $i=1,2$,
    \item $E^1(x)$ is uniformly contracted and $E^3(x)$ is uniformly expanded under the action of $D_xf$.
  \end{enumerate}
\end{definition}We denote the dominated splitting by $T_xM=E^1(x)\oplus_<E^2(x)\oplus_<E^3(x)$.

 Let $\Omega$ be a compact smooth manifold, and let $\theta:\Omega\rightarrow \Omega$ be a diffeomorphism. Denote $f(x,\omega):=f_\omega(x)$ and $\phi^{-1}(x,\omega)=(f_\omega^{-1}(x),\theta^{-1}\omega).$
\begin{proposition}\label{proposition A.1}
Assume
\begin{enumerate}
  \item[(a)]$\phi:M\times\Omega\rightarrow M\times\Omega$ is Anosov on fibers,
  \item[(b)]$f(x,\omega)$ and $f_\omega^{-1}(x)$ are $C^1$ in $\omega$,
  \item[(c)]The diffeomorphism $\theta$ satisfying:\begin{align*}
              \sup_{(x,\omega)\in M\times\Omega}\|D_xf_\omega|_{E^s(x,\omega)}\|<&\inf_{\omega\in\Omega}\|D_\omega\theta^{-1}\|^{-1}:=m_1\\
              &\leq \sup_{\omega\in\Omega}\|D_\omega\theta\|:=m_2<\inf_{(x,\omega)\in M\times\Omega}\|D_x(f_{\theta^{-1}\omega})^{-1}|_{E^u(x,\omega)}\|^{-1}.
            \end{align*}
\end{enumerate}
Then $\phi$ is partially hyperbolic in the narrow sense.
\end{proposition}
\begin{proof}
  We first show the existence of a dominated splitting.
   Note that $T_{(x,\omega)}M\times\Omega=T_xM\times T_\omega\Omega$ already has a splitting $E^u(x,\omega)\times\{0\}\oplus E^s(x,\omega)\times\{0\}\oplus \{0\}\times T_\omega\Omega$, but this splitting may not be invariant. For any $v\in T_xM\times T_\omega\Omega$, then $v=v_1+v_2+v_3$ according to the above splitting. Notice that $\|D\phi(x,\omega)v_3\|$ only depends on $\|D_\omega f(x,\omega)\|$ and $\|D_\omega\theta\|$. $\|D\phi^{-1}(x,\omega)v_3\|$ only depends on $\|D_\omega (f_{\theta^{-1}\omega})^{-1}(x)\|$ and $\|D_\omega\theta^{-1}\|$, then by the compactness of $M$ and $\Omega$, there exists a number $K$ such that
   \begin{equation*}
     \|D\phi(x,\omega)v_3\|\leq K\|v_3\|,\ \|D\phi^{-1}(x,\omega)v_3\|\leq K\|v_3\|.
   \end{equation*} We let $P(E^u(x,\omega)\times\{0\})$ denote the projection map from $T_xM\times T_\omega\Omega$ to $E^u(x,\omega)\times\{0\}$ with respect to the splitting $E^u(x,\omega)\times\{0\}\oplus E^s(x,\omega)\times\{0\}\oplus \{0\}\times T_\omega\Omega$. $P(E^s(x,\omega)\times\{0\})$ and $P(\{0\}\times T_\omega\Omega)$ are similar notations. Since $E^s(x,\omega) $ and $E^u(x,\omega)$ are uniformly continuous on $x$ and $\omega$, there exists a number $\mathcal{P}>1$ such that
   \begin{equation*}
     \sup\{\|P(E^s(x,\omega)\times\{0\})\|,\|P(E^u(x,\omega)\times\{0\})\|:\ (x,\omega)\in M\times\Omega\}<\mathcal{P}.
   \end{equation*}
   Now consider the cone
   \begin{equation*}
     C(x,\omega):=\{v\in T_xM\times T_\omega\Omega|\ \|v_2\|+b\|v_3\|\leq\|v_1\|\},
   \end{equation*}where $b$ is a number such that
   \begin{equation}\label{Appendix condition on b}
     b>\frac{2\mathcal{P}K}{e^{\lambda_0}-m_2}.
   \end{equation}
   Denote $$c_0=\max\left\{\frac{2\mathcal{P}K+bm_2}{e^{\lambda_0}b},e^{-2\lambda_0}\right\}\in (0,1).$$

   For any $v\in C(x,\omega)$, we have
   \begin{align*}
     D\phi(x,\omega)v &=D\phi(x,\omega)v_1+D\phi(x,\omega)v_2+D\phi(x,\omega)v_3\\
     &=D\phi(x,\omega)v_1+P(E^u(\phi(x,\omega))\times\{0\})D\phi(x,\omega)v_3\\
     &\ \ \ \ \ \ \ +D\phi(x,\omega)v_2+P(E^s(\phi(x,\omega))\times\{0\})D\phi(x,\omega)v_3\\
     &\ \ \ \ \ \ \ \ \ \ \ \ +P(\{0\}\times T_{\theta\omega}\Omega)D\phi(x,\omega)v_3\\
     &=(D\phi(x,\omega)v)_1+(D\phi(x,\omega)v)_2+(D\phi(x,\omega)v)_3.
   \end{align*}
  Then
  \begin{align*}
    \|(D\phi(x,\omega)v)_2\|+b\|(D\phi(x,\omega)v)_3\| &\leq e^{-\lambda_0}\|v_2\|+\mathcal{P}K\|v_3\|+b\cdot m_2\|v_3\|\\
    &=e^{-\lambda_0}\|v_2\|+(\mathcal{P}K+bm_2)\|v_3\|\\
    &= e^{-\lambda_0}\|v_2\|+(2\mathcal{P}K+bm_2)\|v_3\|-\mathcal{P}K\|v_3\|\\
    &\leq e^{\lambda_0}\left(e^{-2\lambda_0}\|v_2\|+e^{-\lambda_0}(2\mathcal{P}K+bm_2)\|v_3\|\right)-c_0\mathcal{P}K\|v_3\|\\
    &\overset{\eqref{Appendix condition on b}}<e^{\lambda_0}(c_0\|v_2\|+c_0b\|v_3\|)-c_0\mathcal{P}K\|v_3\|\\
    &\leq c_0e^{\lambda_0}\|v_1\|-c_0\mathcal{P}K\|v_3\|\\
    &\leq c_0\|(D\phi(x,\omega)v)_1\|.
  \end{align*}Hence $D\phi(x,\omega)C(x,\omega)\subset int C(\phi(x,\omega))$. By the cone-field criteria (see Theorem 2.6 in \cite{CP15}), $T_xM\times T_\omega\Omega$ has a dominated splitting $S_1(x,\omega)\oplus_<S_2(x,\omega)$ with $\dim(S_2(x,\omega))=\dim(E^u(x,\omega)\times\{0\})$. Notice that $E^u(x,\omega)\times\{0\}$ lies in $C(x,\omega)$ and it is invariant under $D\phi(x,\omega)$, so $S_2(x,\omega)=E^u(x,\omega)\times\{0\}$.

  On the other hand, consider another cone
  \begin{equation*}
    \mathcal{C}(x,\omega)=\{v\in T_xM\times T_\omega\Omega:\ \|v_1\|+d\|v_3\|\leq \|v_2\|\},
  \end{equation*}where
  \begin{equation}\label{appendix condition on d}
    d> \frac{2\mathcal{P}K}{e^{\lambda_0}-m_1^{-1}}.
  \end{equation}
  Denote $$c_1=\max\left\{\frac{2\mathcal{P}K+m_1^{-1}d}{de^{\lambda_0}},e^{-2\lambda_0}\right\}\in(0,1).$$
  For any $v\in \mathcal{C}(x,\omega)$, we have
  \begin{align*}
     D\phi^{-1}(x,\omega)v&=D\phi^{-1}(x,\omega)v_1+D\phi^{-1}(x,\omega)v_2+D\phi^{-1}(x,\omega)v_3\\
     &=D\phi^{-1}(x,\omega)v_1+P(E^u(\phi^{-1}(x,\omega))\times\{0\})D\phi^{-1}(x,\omega)v_3\\
     &\ \ \ \ \ \ \ \ +D\phi^{-1}(x,\omega)v_2+P(E^s(\phi^{-1}(x,\omega))\times\{0\})D\phi^{-1}(x,\omega)v_3\\
     &\ \ \ \ \ \ \ \ \ \ \ \ \ \ +P(\{0\}\times T_\omega\Omega)D\phi^{-1}(x,\omega)v_3\\
     &=(D\phi^{-1}(x,\omega)v)_1+(D\phi^{-1}(x,\omega)v)_2+(D\phi^{-1}(x,\omega)v)_3.
  \end{align*}Then
  \begin{align*}
    \|(D\phi^{-1}(x,\omega)v)_1\|+d\|(D\phi^{-1}(x,\omega)v)_3\| & \leq e^{-\lambda_0}\|v_1\|+\mathcal{P}K\|v_3\|+d\cdot m_1^{-1}\|v_3\|\\
    &\leq e^{-\lambda_0}\|v_1\|+(2\mathcal{P}K+dm_1^{-1})\|v_3\|-c_1\mathcal{P}K\|v_3\|\\
    &\leq e^{\lambda_0}(e^{-2\lambda_0}\|v_1\|+e^{-\lambda_0}(2\mathcal{P}K+dm_1^{-1})\|v_3\|)-c_1\mathcal{P}K\|v_3\|\\
    &\overset{\eqref{appendix condition on d}}< e^{\lambda_0}(c_1\|v_1\|+c_1d\|v_3\|)-c_1\mathcal{P}K\|v_3\|\\
    &\leq c_1e^{\lambda_0}\|v_2\|-c_1\mathcal{P}K\|v_3\|\\
    &\leq c_1\|(D\phi^{-1}(x,\omega)v)_2\|.
  \end{align*}Hence $D\phi^{-1}(x,\omega)\mathcal{C}(x,\omega)\subset int(\mathcal{C}(\phi^{-1}(x,\omega)))$. By the cone-field criteria, $T_xM\times T_\omega\Omega$ has a dominated splitting $H_1(x,\omega)\oplus_<H_2(x,\omega)$ with $\dim H_1(x,\omega)=\dim (E^s(x,\omega)\times\{0\})$. Notice that $E^s(x,\omega)\times\{0\}$ lies in cone $\mathcal{C}(x,\omega)$ and it is invariant under $D\phi(x,\omega)$, so $H_1(x,\omega)=E^s(x,\omega)\times\{0\}.$

  Now $T_xM\times T_\omega\Omega$ has two dominated splittings: $S_1(x,\omega)\oplus_< (E^u(x,\omega)\times\{0\})$ and $(E^s(x,\omega)\times\{0\})\oplus_< H_2(x,\omega)$. Then, by uniqueness of the dominated splitting (Proposition 2.2 in \cite{CP15}), we have
  \begin{equation*}
    T_xM\times T_\omega\Omega=(E^s(x,\omega)\times\{0\})\oplus_<(S_1(x,\omega)\cap H_2(x,\omega))\oplus_<(E^u(x,\omega)\times\{0\}).
  \end{equation*}
  Besides, we already know that $E^s(x,\omega)\times\{0\}$ is uniformly contracted under $D\phi(x,\omega)$ and $E^u(x,\omega)\times \{0\}$ is uniformly expanded under $D\phi(x,\omega)$. Hence $\phi$ is partially hyperbolic in the narrow sense.
\end{proof}

\subsection{Convex Cone, Projective Metric and Birkhoff's Inequality}\label{section convex cone}
In this section, we introduce the notion of projective metric associated to a convex cone in a topological vector space. The following knowledge is borrowed from \cite[Section 2.1]{Viana}, and also can be found in \cite{Liv95A}.

\begin{definition}\label{definition convex cone}
Let $E$ be a topological vector space. A subset $C\subset E$ is said to be a convex cone if
\begin{enumerate}
  \item $tv\in C$ for $v\in C$ and $t\in \mathbb{R}^+$;
  \item for any $t_1,t_2\in\mathbb{R}^+$, $v_1,v_2\in C$, then $t_1v_1+t_2v_2\in C$;
  \item $\bar{C}\cap -\bar{C}=\{0\}$, where $\bar{C}$ is the `` integral closure" in a weaker sense, and it is defined by: $w\in\bar{C}$ if and only if there are $v\in C$ and $t_n\searrow0$ such that $w+t_nv\in C$ for all $n\geq 1$.
\end{enumerate}
\end{definition}
\begin{definition}\label{projective metric for cone}
For a convex cone $C\subset E$, given any $v_1,v_2\in C$, we define
\begin{align*}
   \alpha(v_1,v_2)&:=\sup\{t>0:\ v_2-tv_1\in C\};\\
   \beta(v_1,v_2)&:=\inf\{s>0:\ sv_1-v_2\in C\},
\end{align*}with the convention that $\sup\emptyset=0$ and $\inf\emptyset=+\infty$. The projective metric between $v_1,v_2\in C$ is defined by
\begin{equation*}
  d_C(v_1,v_2)=\log\frac{\beta(v_1,v_2)}{\alpha(v_1,v_2)}
\end{equation*}with the convention that $d_C(v_1,v_2)=\infty$ if $\alpha(v_1,v_2)=0$ or $\beta(v_1,v_2)=\infty$.
\end{definition}
\begin{proposition}\label{proposition property of pm}
$d_{C}$ is a metric in the projective quotient of $C$, i.e.,
\begin{enumerate}
  \item $d_C(v_1,v_2)\geq 0$ for all $v_1,v_2\in C$(guaranteed by (3) in Definition \ref{definition convex cone});
  \item $d_C(v_1,v_2)=d_C(v_2,v_1)$ for all $v_1,v_2\in C$;
  \item $d_C(v_1,v_3)\leq d_C(v_1,v_2)+d_C(v_2,v_3)$ for all $v_1,v_2,v_3\in C$;
  \item $d_C(v_1,v_2)=0$ if and only if there exists $t\in\mathbb{R}^+$ such that $v_1=tv_2$.
\end{enumerate}
\end{proposition}
\begin{proposition}[Birkhoff's inequality]\label{birkhoff inequality}
Let $E_1,\ E_2$ be two topological vector spaces, and $C_i\subset E_i$, for $i=1,2$ be convex cones. Let $L:E_1\rightarrow E_2$ be a linear operator and assume that $L(C_1)\subset C_2$. Let $D=\sup\{d_{C_2}(L(v_1),L(v_2)):\ v_1,v_2\in C_1\}$. If $D<\infty$, then
\begin{equation*}
  d_{C_2}(L(v_1),L(v_2))\leq (1-e^{-D})d_{C_1}(v_1,v_2)\mbox{ for all }v_1,v_2\in C_1.
\end{equation*}

\end{proposition}

\subsection*{Acknowledgement}
This work was finished at the Brigham Young University and constitutes part of my PhD. I would like to thank my advisor Kening Lu for many valuable discussions. Thanks also goes to Jianyu Chen, Zeng Lian, Fan Yang and referee for valuable suggestions.

\bibliographystyle{acm}
\bibliography{randomcorrelationbib}


\end{document}